\documentclass{article}
\usepackage{epsfig,amssymb,amsmath,diagrams,amsthm,hyperref}

\usepackage{multind}
\makeindex{notations}
\makeindex{keywords}

\makeindex{notations'} 
\makeindex{keywords'} 

\numberwithin{equation}{section}

\newtheorem{theorem}{Theorem}[section]
\newtheorem{conjecture}[theorem]{Conjecture}
\newtheorem{example}[theorem]{Example}
\newtheorem{corollary}[theorem]{Corollary}

\newtheorem{ktf}[theorem]{KTF}
\newtheorem{lemma}[theorem]{Lemma}
\newtheorem{proposition}[theorem]{Proposition}

\newcommand{\color}[6]{}

\newcommand{\1}{\mathbf{1}}
\newcommand{\A}{\mathbf{A}}

\newcommand{\bs}{\backslash}
\renewcommand{\c}{\mathfrak{c}}
\newcommand{\comment}[1]{}
\newcommand{\C}{\mathbf{C}}
\newcommand{\CK}{\operatorname{CK}}
\newcommand{\dct}{the dominated convergence theorem}
\newcommand{\ddt}{\left.\frac{d}{dt}\right|_{\scriptscriptstyle t=0}}

\newcommand{\ds}{\displaystyle}
\newcommand{\e}{\varepsilon}

\newcommand{\eqdef}{\stackrel{\text{\tiny def}}{=}}

\newcommand{\fin}{\operatorname{fin}}
\newcommand{\ff}{f^{\mathtt{n}}}

\newcommand{\g}{\gamma}

\newcommand{\lie}{\mathfrak{g}}
\newcommand{\GL}{\operatorname{GL}}

\newcommand{\Ind}{\operatorname{Ind}}

\renewcommand{\Im}{\operatorname{Im}}
\newcommand{\IL}[1]{\ensuremath{\int\limits_{\scriptscriptstyle #1}}}
\newcommand{\IIL}[1]{\ensuremath{\iint\limits_{\scriptscriptstyle #1}}}

\renewcommand{\k}{\mathtt{k}}

\newcommand{\lcm}{\operatorname{lcm}}
\renewcommand{\L}{L^2_0}

\newcommand{\mat}[4]{\begin{pmatrix} {#1} & {#2} \\ {#3} & {#4}
  \end{pmatrix}}

\newcommand{\meas}{\operatorname{meas}}
\renewcommand{\mod}{\text{ mod }}
\newcommand{\n}{\mathtt{n}}

\newcommand{\ol}{\overline}
\newcommand{\olG}{\overline{G}}

\newcommand{\Q}{\mathbf{Q}}

\newcommand{\R}{\mathbf{R}}
\renewcommand{\Re}{\operatorname{Re}}
\newcommand{\sg}[1]{\left<{#1}\right>}
\newcommand{\sgn}{\operatorname{sgn}}

\newcommand{\sss}{\scriptscriptstyle}
\newcommand{\scri}{\scriptstyle}
\newcommand{\SL}{\operatorname{SL}}
\newcommand{\SO}{\operatorname{SO}}

\renewcommand{\subset}{\subseteq}
\newcommand{\Supp}{\operatorname{Supp}}
\newcommand{\tc}{\widetilde\chi}

\newcommand{\tr}{\operatorname{tr}}
\newcommand{\tT}{\tau_{\sss T}}

\newcommand{\ve}{\varepsilon}
\newcommand{\w}{\omega}

\newcommand{\Z}{\mathbf{Z}}
\newcommand{\Zhat}{\widehat{\Z}}

\newcommand{\Af}{\mathbf{A}_{\fin}}

\newcommand{\ord}{\operatorname{ord}}

\newcommand{\cont}{\operatorname{cont}}
\newcommand{\cusp}{\operatorname{cusp}}
\newcommand{\disc}{\operatorname{disc}}
\newcommand{\res}{\operatorname{res}}

\newcommand{\vpa}{\phi_{p,i,N_p}^{\chi_{1p},\chi_{2p}}}

\newcommand{\smat}[4]{\bigl(\begin{smallmatrix}{#1}&{#2}\\{#3}&{#4}\end{smallmatrix}\bigr )}
\begin{document}
\title{Kuznetsov's trace formula and the Hecke eigenvalues of Maass forms}
\author{A. Knightly and C. Li}
\maketitle
\small\normalsize

\begin{abstract}
We give an adelic treatment of the Kuznetsov trace formula as a
  relative trace formula on $\GL(2)$ over $\Q$.  The result is
  a variant which incorporates a Hecke eigenvalue
  in addition to two Fourier coefficients on the spectral side.
  We include a proof of a Weil bound for the generalized
  twisted Kloosterman sums which arise on the geometric side.
  As an application, we show that the Hecke eigenvalues of Maass forms
  at a fixed prime,
  when weighted as in the Kuznetsov formula, become equidistributed relative
  to the Sato-Tate measure in the limit as the level goes to infinity.
\end{abstract}
\tableofcontents
\pagebreak

\section{Introduction}

\subsection{Some history}
A {\em Fourier trace formula} for $\GL(2)$ is an identity between a
 product of two Fourier coefficients, averaged
  over a family of automorphic forms on $\GL(2)$, and 
  a series involving Kloosterman sums and the Bessel $J$-function.
    The first example, arising from Petersson's computation of the Fourier
  coefficients of Poincar\'e series in 1932 \cite P and his introduction
  of the inner product in 1939 \cite{P2}, has the form
\[\hskip -.1cm\frac{\Gamma(k-1)}{(4\pi\sqrt{mn})^{k-1}}
  \sum_{f\in \mathcal{F}_k(N)}\frac{a_m(f)\ol{a_n(f)}}{\|f\|^2}
  = \delta_{m,n}+2\pi i^k\sum_{c\in N\Z^+}\frac{S(m,n;c)}c
  J_{k-1}(\frac{4\pi\sqrt{mn}}c),\]
where $\mathcal{F}_k(N)$ is an orthogonal basis for the space 
  of cusp forms $S_k(\Gamma_0(N))$, and\index{keywords}{Kloosterman sum!classical}
\[S(m,n;c)=\sum_{x\ol x\equiv 1\mod c}e^{2\pi i({mx+n\ol x})/c}\]
 is a Kloosterman sum.  
Because of the existence of the Weil bound\index{keywords}{Kloosterman sum!Weil bound}
\begin{equation}\label{weil}
|S(m,n;c)|\le \tau(c)(a,b,c)^{1/2}c^{1/2}
\end{equation}
where $\tau$ is the divisor function, and the bound
\[J_{k-1}(x)\ll \min(x^{k-1},x^{-1/2})\]
  for the Bessel function,
  the Petersson formula is useful for approximating expressions involving 
  Fourier coefficients of cusp forms.
For example, Selberg used it in 1964 (\cite{Sel2})
  to obtain the nontrivial bound 
\begin{equation}\label{SelEst}a_n(f)=O(n^{(k-1)/2+1/4+\e})\end{equation}
  in the direction of the Ramanujan-Petersson conjecture 
  $a_n(f)=O(n^{(k-1)/2+\e})$ subsequently proven by Deligne.

In his paper, Selberg mentioned
  the problem of extending his method to the case of Maass forms.
  This was begun in the late 1970's independently by
  Bruggeman and Kuznetsov (\cite{Brug}, \cite{Ku}).
  The left-hand side of the above Petersson formula is now
  replaced by a sum of the form
\begin{equation}\label{N1cusp}
\sum_{u_j\in\mathcal{F}} \frac{a_m(u_j)\ol{a_n(u_j)}}{\|u_j\|^2}\frac{h(t_j)}{\cosh(\pi t_j)},
\end{equation}
  where $m,n>0$, $\mathcal{F}$ is an (orthogonal) basis of 
  Maass cusp forms of weight $k=0$
  and level $N=1$, $t_j$ is the spectral parameter defined by 
  $\Delta u_j=(\tfrac14+t_j^2)u_j$ for the Laplacian $\Delta$,
  and $h(t)$ is an even holomorphic function with sufficient decay.
  There is a companion term
  coming from the weight $0$ part of the continuous spectrum,
  describable in terms of the Eisenstein series
\[E(s,z)= \frac12\sum_{c,d\in\Z\atop{(c,d)=1}}
  \frac{y^{1/2+s}} {|cz+d|^{1+2s}}
  \qquad (\Re(s)>\tfrac12, y>0, z=x+iy).\]
  More accurately, it involves the analytic continuation to $s$ 
  on the imaginary line.  This analytic continuation is provided by
  the Fourier expansion 
\begin{align}\label{EN1}
E(s,z)= y^{1/2+s}&+y^{1/2-s}\frac{\sqrt{\pi}\,\Gamma(s)\zeta(2s)}{\Gamma(1/2+s)\zeta(1+2s)}
  \\
\notag &+\frac{2y^{1/2}\pi^{1/2+s}}{\Gamma(1/2+s)\zeta(1+2s)}\sum_{m\neq 0}
  \sigma_{2s}(m)|m|^sK_s(2\pi|m|y)e^{2\pi i mx}.
\end{align}
  Here $\sigma_{2s}(m)=\sum_{0<d|m}d^{2s}$ is the
  divisor sum, and $K_s$ is the $K$-Bessel function.
  The continuous contribution to the Kuznetsov/Bruggeman formula is the following
  integral of the product of two Fourier coefficients of $E(it,z)$ 
  against the function $h(t)$:
\begin{equation}\label{N1cont}
\frac1\pi\int_{-\infty}^\infty\frac{(m/n)^{it}\sigma_{2it}(m)\ol{\sigma_{2it}(n)}}
  {|\zeta(1+2it)|^2} h(t)dt.
\end{equation}
The Fourier trace formula is then the equality between the sum of \eqref{N1cusp} and \eqref{N1cont}
  on the so-called spectral side, with the geometric side given by
\begin{equation}\label{N1geom}
  \frac{\delta_{m,n}}{\pi^2}\int_{-\infty}^\infty h(t) \tanh(\pi t)\,t\,dt
+\frac{2i}{\pi}\sum_{c\in \Z^+}\frac{S(m,n;c)}{c}\int_{-\infty}^\infty J_{2it}
  (\frac{4\pi\sqrt{mn}}c)\frac{h(t)\,t}{\cosh(\pi t)}dt.
\end{equation}
  Using this together with the Weil bound \eqref{Weil}, Kuznetsov proved
  a mean-square estimate for the Fourier coefficients $a_n(u_j)$ (\cite{Ku}, Theorem 6),
  which immediately implies the bound
\[a_n(u_j)\ll_{j,\e} n^{1/4+\e}\]
  in the direction of the (still open) Ramanujan conjecture $a_n(u_j)=O(n^\e)$.
  (See also \cite{Brug}, \S4.)
  This extended Selberg's result \eqref{SelEst} to the case of Maass forms.

  Kuznetsov also ``inverted" the formula to give a variant in which
  a general test function appears on the geometric side
  in place of the Bessel integral.
  (Motohashi has given an interesting conceptual explanation of this,
  showing that the procedure is reversible, \cite{Mo2}.)
  This allows for important 
  applications to bounding sums of Kloosterman sums.
  Namely, Kuznetsov proved that the estimate
\begin{equation}\label{linnik}
\sum_{c\le X}\frac{S(m,n;c)}c\ll_{m,n,\e} X^{\theta+\e}
\end{equation}
  holds with $\theta=\tfrac 16$ (\cite{Ku}, Theorem 3).
  The Weil bound alone yields only $\theta=\tfrac12$,
  showing that Kuznetsov's method detects considerable 
  cancellation among the Kloosterman sums due to the oscillations in their arguments.
Linnik had conjectured in 1962 that \eqref{linnik} holds with $\theta=0$, and 
  Selberg 
  remarked that this would imply the Ramanujan-Petersson conjecture
  for holomorphic cusp forms of level $1$, (\cite{Sel2}; see also \S4 of \cite{Mu}).
  By studying the Dirichlet series
  \[Z(s,m,n)=\sum_{c}\frac{S(m,n;c)}{c^{2s}},\]
  Selberg also codified a relationship between
  sums of Kloosterman sums and the smallest eigenvalue $\lambda_1$
  of the Laplacian,
  leading him to conjecture that $\lambda_1\ge \tfrac 14$ for congruence subgroups.
  He obtained the inequality $\lambda_1\ge \tfrac 3{16}$ using the Weil bound
  \eqref{Weil}.  This inequality is also
  a consequence of the generalized Kuznetsov formula given in 1982 by
   Deshouillers and Iwaniec (\cite{DI}).
  
Fourier trace formulas have since become a staple tool in analytic number theory.
  We mention here a sampling of notable results in which they have played a role.
  Deshouillers and Iwaniec used the Kuznetsov formula to deduce
  bounds for very general weighted averages of Kloostermans sums, showing
  in particular that Linnik's conjecture holds on average (\cite{DI}, \S1.4).
  They list some interesting consequences in \S1.5 of their paper.  For example,
  there are infinitely many primes $p$ for which $p+1$ has a prime factor greater
  than $p^{21/32}$.  They also give applications to the Brun-Titchmarsh theorem
  and to mean-value theorems for primes in arithmetic progressions (see also
  \cite{Iw1}, \S12-13).

  Suppose $f(x)\in\Z[x]$ is a quadratic polynomial with negative discriminant.  
If $p$ is prime and $\nu$ is a root of $f$ in $\Z/p\Z$, then
  the fractional part $\{\tfrac\nu p\}\in [0,1)$ is independent of the choice
  of representative for $\nu$ in $\Z$.  Duke, Friedlander, and Iwaniec proved that
  for $(p,\nu)$ ranging over all such pairs, the set of these fractional parts
   is uniformly distributed in $[0,1]$, i.e.
   for any $0\le \alpha<\beta\le 1$,
\[\frac{\#\{(p,\nu)|\,p\le x, f(\nu)\equiv 0\mod p, \alpha\le \{\tfrac\nu p\}<\beta\}}
  {\#\{p\le x|\, p \text{  prime}\}} \sim (\beta-\alpha)\]
  as $x\to\infty$ (\cite{DFI}).  Their 
  proof uses the Kuznetsov formula to bound a certain related 
  Poincar\'e series via its spectral
  expansion.  See also Chapter 21 of \cite{IK}.

Applications of Fourier trace formulas to the theory of $L$-functions abound.
  Using the results of \cite{DI}, Conrey showed in 1989 that more than 40\%
  of the zeros of the Riemann zeta function are on the critical line 
  (\cite{Con}).\footnote{Conrey, Iwaniec and Soundararajan have recently
  proven that more than 56\% of the zeros of the family of 
  Dirichlet $L$-functions lie on the critical line, \cite{CIS}.}
  Motohashi's book \cite{Mo} discusses other applications to $\zeta(s)$,
  including the asymptotic formula for its fourth moment.
  In his thesis, Venkatesh used a Fourier trace formula
  to carry out the first case of Langlands' {\em Beyond Endoscopy} program
  for $\GL(2)$ (\cite{L}, \cite{V1}, \cite{V2}).
  This provided a new proof of the
  result of Labesse and Langlands characterizing as dihedral those forms for which the
  symmetric square $L$-function has a pole, as well as giving an
  asymptotic bound for the dimension of holomorphic cusp forms of weight $1$, extending
  results of Duke.
  Fourier trace formulas have also been used by many authors 
  in establishing 
  subconvexity bounds for $\GL(1)$, 
  $\GL(2)$ and Rankin-Selberg $L$-functions; see \cite{MV} and its
  references, although this definitive paper
  does not actually use trace formulas. 
  Subconvexity bounds have important arithmetic applications, notably to Hilbert's
  eleventh problem of determining the integers that are integrally represented 
  by a given quadratic form over a number field
  (\cite{IS1}, \cite{BH}).
  Other applications of Fourier trace formulas include 
  nonvanishing of $L$-functions at the central point 
  (\cite{Du}, \cite{IS}, \cite{KMV}) and the 
  density of low-lying zeros of automorphic $L$-functions 
  (starting with \cite{ILS}).

\subsection{Overview of the contents}

Zagier is apparently the first one to observe that Kuznetsov's formula
  can be obtained by integrating each variable of an automorphic kernel 
  function over the unipotent subgroup.
  His proof is detailed by Joyner in \S1 of \cite{Joy}.
  See also the description by Iwaniec on p. 258 of \cite{Iw1}, and
  the article \cite{LiX} by X.\,Li, who also
  extended the formula to the setting of Maass forms for $\SL_n(\Z)$, \cite{Gld}.
  Related investigations have been carried out by others, notably
  in the context of base change by Jacquet and Ye (cf. \cite{Ja}
  and its references).

Our primary purpose is to give a detailed account of this method over the 
  adeles of $\Q$, for Maass cusp forms of arbitrary level $N$ and nebentypus $\w'$.
  We obtain a variant of the Kuznetsov trace formula by using the
  kernel function attached to a Hecke operator $T_\n$.
  The final formula is given in Theorem \ref{main} on page \pageref{main},
  and it differs from the usual
  version by the inclusion of eigenvalues of $T_\n$ on the spectral side.
  The cuspidal term thus has the form
\begin{equation}\label{Ncusp}  \sum_{u_j\in \mathcal{F}(N)}
 \frac{\lambda_\n(u_j)\, a_{m_1}(u_j) \ol{a_{m_2}(u_j)}}
 {\|u_j\|^2} \frac{h(t_j)}{\cosh(\pi t_j)}.
\end{equation}
  This is a complement to the article \cite{KL1}, which dealt with Petersson's formula
  from the same viewpoint.  As we pointed out there, the above
  variant can alternatively be derived from the
  classical version (see Section \ref{classder} below).
It is also possible to invert the final formula to get a version with the
  test function appearing on the geometric side rather than the spectral side, although we will
  not pursue this.  See Theorem 2 of \cite{BKV} or \cite{A}, p. 135.

The incorporation of Hecke eigenvalues in \eqref{Ncusp} allows us to prove a result 
  about their distribution (Theorem \ref{dist}).  
  To state a special case, assume for simplicity that the nebentypus is trivial,
  and that the basis $\mathcal{F}(N)$ is chosen so that $a_1(u_j)=1$ for all $j$.
  Then for any prime $p\nmid N$, we prove that
  the multiset of Hecke eigenvalues $\lambda_p(u_j)$,
  when weighted by
  \[w_j=\frac{1}{\|u_j\|^2}\frac{h(t_j)}{\cosh(\pi t_j)},\]
 becomes
  equidistributed relative to the Sato-Tate measure in the limit as $N\to \infty$.
  This means that for any continuous function $f$ on $\R$,
\[\lim_{N\to\infty}\frac{\sum_{u_j\in\mathcal{F}(N)}f(\lambda_p(u_j))w_j}
  {\sum_{u_j\in \mathcal{F}(N)}w_j}=\frac1\pi\int_{-2}^2 f(x)
  \sqrt{1-\tfrac{x^2}4}dx.\]
This can be viewed as evidence for the Ramanujan conjecture, which asserts that
  $\lambda_p(u_j)\in [-2,2]$ for all $j$.
  The above result holds independently of both $p$ and the choice of $h$ from a large
  family of suitable functions.
We discuss some of the history of this problem and its relation to the
  Sato-Tate conjecture in Section \ref{Equi}.
  
The material in the first six sections can be used
  as a basis for any number of investigations of Maass forms 
  with the $\GL(2)$ trace formula.
Sections \ref{2}-\ref4 are chiefly expository.
  We begin with the goal of explaining the connection
  between the Laplace eigenvalue of a Maass form and the principal series
  representation of $\GL_2(\R)$ determined by it.  
  We then give a detailed account
  of the passage between a Maass form on the upper half-plane and its
  adelic counterpart, which is a cuspidal funcion on $\GL_2(\A)$.
  We also describe the adelic Hecke operators of weight $\k=0$ and level $N$
  corresponding to the classical ones $T_\n$.

Although similar in spirit with the derivation of Petersson's 
  formula in \cite{KL1},
  the analytic difficulties in the present case are considerably more subtle.
  Whereas in the holomorphic case the relevant Hecke operator is of finite rank,
  in the weight zero case it is not even Hilbert-Schmidt.  The
  setting for the adelic trace formula is the Hilbert space
\[L^2(\w)=\left\{\begin{array}l \phi:G(\A)\rightarrow \C\\
  \phi(z\g g)=\w(z)\phi(g) \quad(z\in Z(\A), \g\in G(\Q)),\\
  \int_{Z(\A)G(\Q)\bs G(\A)}|\phi|^2 <\infty,\end{array}\right.\]
where $G=\GL_2$, $Z$ is the center, and 
  $\w$ is a finite order Hecke character.  Relative to the right
  regular representation $R$ of $G(\A)$ on $L^2(\w)$,
  there is a spectral decomposition $L^2(\w)=L^2_{\disc}(\w)\oplus L^2_{\cont}(\w)$.
  The classical cusp forms correspond to certain elements in the discrete part, while
  the continuous part is essentially a direct integral of certain
  principal series 
  representations $H(it)$ of $G(\A)$.  We begin Section \ref6 by describing this
  in detail, following Gelbart and Jacquet \cite{GJ}.  
  For a function $f\in L^1(\ol{\w})$ attached to a classical
  Hecke operator, we then investigate the kernel
\begin{equation}\label{Kgeom}K(x,y)=\sum_{\g\in Z(\Q)\bs G(\Q)}f(x^{-1}\g y)
\end{equation}
  of the operator $R(f)$.  We assume that $f_\infty$ is bi-invariant under
  $\SO(2)$, compactly supported in $\olG(\R)^+$, and sufficiently differentiable.
  Then letting 
  $\phi$ range through an orthonormal basis
  for the subspace of vectors in $H(0)$ of weight $0$ and level $N$,
  the main result of the section is a proof that the spectral expansion
\begin{align*}
K(x,y) =\hskip .2cm&  \delta_{\w,1}\frac3\pi\int_{\olG(\A)}f(g)dg+
  \sum_{\varphi\in \mathcal{F}(N)}
  \frac{R(f)\varphi(x)\ol{\varphi(y)}}{\|\varphi\|^2}\\
 &+\frac1{4\pi}\sum_\phi\int_{-\infty}^\infty E(\pi_{it}(f)\phi_{it},x)
  \ol{E(\phi_{it},y)}dt,
\end{align*}
is absolutely convergent and valid for all $x,y$.
These are, respectively, the residual, cuspidal, and continuous components of the kernel.

In Section \ref 5, we discuss the Eisenstein series.  
  We give an explicit description of the finite set of Eisenstein series $E(\phi_s,g)$
  that contribute to the above expression for $K(x,y)$.  Their Fourier coefficients involve
  generalized divisor sums and Dirichlet $L$-values on the right
  edge of the critical strip, directly generalizing \eqref{EN1}.
  We derive bounds for these Fourier coefficients, which are useful
  for both the convergence and applications of the Kuznetsov formula.
  For this purpose we require
  lower bounds for Dirichlet $L$-functions on the right edge of
  the critical strip, reviewed in Section 2.
  (We note that more generally, in establishing absolute convergence of the 
  spectral side of Jacquet's $\GL(n)$ relative trace formula,
  Lapid makes use of 
  lower bounds for Rankin-Selberg $L$-functions on the right edge of 
  the critical strip due to Brumley, \cite{Lap}, \cite{Br}.)

In Section \ref{ftf} we integrate each variable of $K(x,y)$ against a character over 
  the unipotent group $N(\Q)\bs N(\A)$.  Using the geometric form \eqref{Kgeom}
  of the kernel, we obtain
  the geometric side of the Kuznetsov formula as a sum
  of orbital integrals whose finite parts evaluate to generalized 
  twisted Kloosterman sums, defined by
\[S_{\w'}(m_2,m_1;\n;c)=\sum_{dd'\equiv \n\mod c}\ol{\w'(d)}e^{2\pi i(dm_2+d'm_1)/c}
  \qquad(\text{for }N|c),\]
where $\w'$ is the Dirichlet character of modulus $N$ attached to $\w$.
  These sums also arise in the generalized Petersson formula of \cite{KL1}.
  After an extra averaging at the archimedean place, we obtain the $J$-Bessel
  integrals as in \eqref{N1geom}.
  Using the spectral form of the kernel we obtain the
  spectral side of the Kuznetsov formula, giving the main result, Theorem \ref{main}.
  The function $h(t)$ of \eqref{Ncusp} is the Selberg transform of
  the archimedean test function $f_\infty$.

The hypothesis that $f_\infty$ be smooth and compactly supported
  amounts to requiring that $h(iz)$ be an even Paley-Wiener function.
  This is very restrictive, ruling out well-behaved functions like the
  Gaussian $h(t)=e^{-t^2}$.  In Section \ref{Val}, we carefully study the various
  transforms involved under more relaxed hypotheses, and show that the Kuznetsov formula
  remains valid.  We start with a function $f$ on $G(\A)$ which is
  $C^m$ for $m$ sufficiently large, and has polynomial decay rather than compact
  support.  We then express $f$ as a limit of compactly supported $C^m$ 
  functions (for which we have already established the Kuznetsov formula),
  and then show that the Kuznetsov formula is preserved in the limit.
  A key step is proving that $R(f)$ is a Hilbert-Schmidt operator
  on the cuspidal subspace 
  (cf. Corollary \ref{genHS}).

In Section \ref{Klsec}, we prove the Weil bound 
\begin{equation}\label{Sb}
|S_{\chi}(a,b;\n;c)|\le \tau(\n)\tau(c)(a\n,b\n,c)^{1/2}c^{1/2}\c_\chi^{1/2},
\end{equation}
  where $\c_\chi$ is the conductor of $\chi$, and $\tau$ is the divisor function.
  Various identities relate the generalized sum to classical
  twisted Kloosterman sums $S_\chi(a,b;c)=S_\chi(a,b;1;c)$.  Therefore we reduce to
  proving a Weil bound for the latter sums.  The latter is well-known, but seems to 
  be a gap in the literature.  Furthermore, it is sometimes erroneously 
  asserted that
$|S_\chi(a,b;c)|\le \tau(c)(a,b,c)^{1/2}c^{1/2}$.
  We give a counterexample on p.\,\pageref{p3}. 
  For these reasons, we have included all of the details of the proof of \eqref{Sb}.  

  For simplicity, in this paper we only treat forms of
  weight $\k=0$ over $\Q$, and we deal only with positive integer
  Fourier coefficients for the cusp at infinity.
  There are many expositions of the Kuznetsov formula in the classical
  language which extend beyond this scope and give other applications.
  See especially \cite{DI}, \cite{CPS}, \cite{Mo3} and \cite{BM}.
The latter incorporates
  general weights and cusps over a totally real field.
  We also recommend the text of Baker \cite{Ba}.

\subsection{Acknowledgements}
We would like to thank Jon Rogawski and Eddie Herman for
  suggesting several changes which have greatly improved the exposition.
  In particular, Herman drew our attention to the thesis \cite{A} of Andersson,
  and suggested including the content of Section \ref{classder}.
We also thank Farrell Brumley, George Knightly, and Yuk-Kam Lau for
  helpful discussions.
  The first author was supported in part by the University of
  Maine Summer Faculty Research Fund and by NSF grant DMS 0902145.

\pagebreak
\section{Preliminaries}\label{2}
\subsection{Notation and Haar measure}\label{notation}

Notation and normalization of measures is the same as in \cite{KL}, where
   full details are given.
Let $G=\GL_2$,\index{notations}{G @$G$}\index{notations'}{GL@$\GL_2$} let $M=\{\smat*{0}{0}*\}\subset G$\index{notations}{M@$M$, diagonal subgroup}
  be the diagonal subgroup, and let $N=\{\smat1*01\}\subset G$\index{notations}{N@$N$, unipotent subgroup}
  be the upper triangular unipotent subgroup.
  The Borel subgroup\index{keywords'}{Borel subgroup} of upper triangular
  matrices
  is denoted $B=MN=NM$\index{notations}{B@$B=\{\smat ab0d\}$ Borel subgroup}.
  We write $\olG$\index{notations}{G bar@$\olG=G/Z$} for $G/Z$, where $Z$\index{notations}{Z@$Z$, the center of $G$} is the center of $G$,
  and generally for a subset $S\subset G$, $\ol S$\index{notations'}{Sbar@$\ol S$} denotes the image of $S$ in $\olG$.
  Let
\begin{equation}\label{Kinf}
K_\infty=\{k_\theta:=\smat{\cos\theta}{\sin\theta}{-\sin\theta}{\cos\theta}|\,\theta\in\R\}
 \index{notations}{Kinfty@$K_\infty$} \index{notations}{Ktheta@$k_\theta=\smat{\cos\theta}{\sin\theta}{-\sin\theta}{\cos\theta}$}
\end{equation}
denote the compact subgroup $\SO(2)$ of $G(\R)$.

Let $\Z^+$\index{notations}{Zplus@$\Z^+$, the set of positive integers} denote the set of positive integers and let $\R^+$ \index{notations}{Rplus@$\R^+$, the set of positive real numbers} denote the
  group of positive reals.
If $p$ is prime, we let $\Q_p$ and $\Z_p$ denote the $p$-adic numbers and $p$-adic
  integers, respectively.  For any rational integer $x>0$, we often use the notation
\[x_p=\ord_p(x),\] \index{notations}{11@$x_p=\ord_p(x), N_p=\ord_p(N)$, etc.}
so that $x=\prod_p p^{x_p}$, $N=\prod_p p^{N_p}$, etc.

  Let $\A,\Af$\index{notations}{Adele@$\A$ adeles}\index{notations}{Adelefinite@$\Af$ finite adeles} be the adeles and finite adeles of $\Q$.  Then
   $\Zhat=\prod_p\Z_p$\index{notations}{Zhat@$\Zhat=\prod \Z_p$} is an open compact subgroup of $\Af$.
  For an element $d\in\Q^*$ \index{notations}{12@$d_N$, idele $(x_p)$, $x_p=d$ for $p"|N$, $x_p=1$ otherwise}, we let
\begin{equation}\label{dN}
d_N\in\A^*
\end{equation}
be the idele which agrees with $d$ at places $p|N$ and is $1$ at all other places.

  Let $K_p=G(\Z_p)$\index{notations}{Kp@$K_p$} and $K_{\fin}=G(\Zhat)$\index{notations}{Kfin@$K_{\fin}$} denote the standard maximal compact subgroups of $G(\Q_p)$ and $G(\Af)$ respectively.
  By the Iwasawa decomposition\index{keywords'}{Iwasawa decomposition},
\[G(\A)=M(\A)N(\A)K,\]
where
\[K=K_\infty\times K_{\fin}.\]\index{notations}{Kp@$K$}
For an integer $N\ge 1$, define the following nested congruence subgroups
  of $K_{\fin}$:\index{notations}{K0N@$K_0(N)$}\index{notations}{K1N@$K_1(N)$}\index{notations}{KN@$K(N)$}
\[K_0(N)=\{\smat abcd\in K_{\fin}|\,c\equiv 0\mod N\Zhat\},\]
\[K_1(N)=\{\smat abcd\in K_0(N)|\,d\equiv 1\mod N\Zhat\},\]
\[K(N)=\{k\in K_{\fin}|\, k\equiv \smat1001\mod N\Zhat\}.\]
Each of these is open and compact in $G(\Af)$.  By the strong approximation theorem,
we have\index{keywords}{strong approximation}
\begin{equation}\label{sa}
G(\A)=G(\Q)(G(\R)^+\times K_1(N)),
\end{equation}
where as usual $G(\Q)$ is embedded diagonally in $G(\A)$, and $G(\R)^+$ is the subgroup
  of $\GL_2(\R)$ consisting of matrices with positive determinant.
We will also use the local subgroups\index{notations}{K0Np@$K_0(N)_p$}
  $K_0(N)_p=\{\smat abcd\in K_p|\, c\in N\Z_p\}$,
  and similarly for $K_1(N)_p$.\index{notations}{K1Np@$K_1(N)_p$}

We take $\Gamma_0(N)$, $\Gamma_1(N)$ and 
$\Gamma(N)$\index{notations}{Gamma0N@$\Gamma_0(N)$}\index{notations}{Gamma1N@$\Gamma_1(N)$} \index{notations}{GammaN@$\Gamma(N)$}
  to be the intersections of
  the above congruence subgroups with $\SL_2(\Z)$ as usual.
  We set
\begin{equation}\label{psiN} \index{notations}{psiN@$\psi(N)$}
\psi(N)=[K_{\fin}:K_0(N)]=[\SL_2(\Z):\Gamma_0(N)]
  =N\prod_{\overset{p|N}{\sss p\text{ prime}}}(1+\frac1p),
\end{equation}
and locally $\psi(N)=\prod_p\psi_p(N)$, where\index{notations}{psipN@$\psi_p(N)$}
\[
\psi_p(N)=[K_p:K_0(N)_p]=p^{N_p-1}(p+1).
\]

Haar measure will be normalized as follows.  See \S7 of \cite{KL} for more detail.
On $\R$ we take Lebesgue measure $dx$, and on $\R^*$ we take $\frac{dy}{|y|}$.
On $\Q_p$ we normalize by $\meas(\Z_p)=1$, and on $\Q_p^*$ we take $\meas(\Z_p^*)=1$.
  These choices determine measures on $\A$ and $\A^*\cong Z(\A)$, with $\meas(\Q\bs \A)=1$.
  We normalize $dk$ on $K_\infty$ by $\meas(K_\infty)=1$,
  and use the above measures on $\R$ and $\R^*$ to define measures on $N(\R)\cong \R$
   and $M(\R)\cong \R^*\times\R^*$.
  These choices determine a measure on $G(\R)$ by the Iwasawa decomposition:
  writing $g=mnk$, we take $dg = dm\, dn\, dk$.
  We normalize Haar measure on $G(\Q_p)$ so that $\meas(K_p)=1$, and on
  $G(\Af)$ by taking $\meas(K_{\fin})=1$.  We then
  adopt the product measure on $G(\A)=G(\R)\times G(\Af)$.
Having fixed measures on $G(\A)$ and $Z(\A)\cong\A^*$ as above, we give $\olG(\A)
  =G(\A)/Z(\A)$ the associated quotient measure.  It has
  the property that $\meas(\olG(\Q)\bs \olG(\A))=\pi/3$.
  In the quotient measure on $\olG(\Q_p)$, we have $\meas(\ol{K_p})=1$.
  We also take $\meas(\ol{K_\infty})=1$, which is {\em not} the quotient measure
  on $K_\infty/\{\pm 1\}$.

For any real number $x$, we denote
\[e(x) = e^{2\pi i x}.\] \index{notations}{e@$e(x)=e^{2\pi i x}$}
We let $\theta:\A\longrightarrow\C^*$ denote the standard character of $\A$.  It is
  defined by
\begin{equation}\label{theta}
\theta_p(x)=\begin{cases}e(-x)=e^{-2\pi ix}&\text{if }p=\infty\\
e(r_p(x))=e^{2\pi i r_p(x)}&\text{if }p<\infty,
\end{cases} \index{notations}{thetap@$\theta_p$}
\end{equation}
where $r_p(x)\in\Q$\index{notations'}{rp@$r_p(x)$} is the $p$-principal part of $x$, a number with $p$-power denominator
  characterized (up to $\Z$) by $x\in r_p(x)+\Z_p$.  Then $\theta$ is trivial\index{notations}{theta@$\theta$, character of $\Q\bs\A$}
  on $\Q$, and $\theta_{\fin}=\prod_{p<\infty}\theta_p$\index{notations}{thetafin@$\theta_{\fin}$} is trivial precisely on $\Zhat$.
  For $m\in\Q$, we define the character $\theta_m$ by
\[\theta_m(x) = \theta(-mx)=\ol{\theta(mx)}.\index{notations}{thetam@$\theta_m=\theta(-m\,\cdot\,)$}
\]
  It is well-known that every character of $\Q\bs\A$ arises in this way, i.e.
  $\Q\cong \widehat{\Q\bs\A}$ by the the map $m\mapsto\theta_m$.

If $V$ is a space of functions on a group $G$, then unless otherwise 
  specified, we denote the
  right regular action of $G$ on $V$ by $R$.  Thus for $\phi\in V$
  and $g,x\in G$,
\[R(g)\phi(x)\index{notations}{R@$R$, right regular representation}=\index{keywords}{right regular action}\phi(xg).\]

\subsection{Characters and Dirichlet $L$-functions}

For a positive integer $N$, a {\bf Dirichlet character}\index{keywords}{Dirichlet character}
  modulo $N$ is a homomorphism
\begin{equation}\label{dc}
\tc:(\Z/N\Z)^*\longrightarrow\C^*,
\end{equation}
extended to a function on $\Z$ by taking $\tc(n)=0$ if $\gcd(n,N)>1$.
  The simplest example is when \eqref{dc} is the trivial homomorphism.
  In this case we say that $\tc$ is the
  {\bf principal character}\index{keywords}{Dirichlet character!principal} modulo $N$.

If $d|N$ and $\chi'$ is a Dirichlet character modulo $d$, then it defines
  a Dirichlet character $\tc$ modulo $N$ by the composition
\begin{equation}\label{olddc}
\tc:(\Z/N\Z)^*\longrightarrow (\Z/d\Z)^*\longrightarrow\C^*,
\end{equation}
where the last arrow is $\chi'$.  We say that $\tc$ is the character of modulus
  $N$ {\bf induced}\index{keywords}{Induced character}\index{keywords}{Dirichlet character!induced}
   from $\chi'$.  Conversely, if $\tc$ is a Dirichlet character modulo $N$
  that factors through the projection to $(\Z/d\Z)^*$ for some positive $d|N$ as above,
  then we say $d$ is an {\bf induced modulus} for $\tc$.  The 
  {\bf conductor}\index{keywords}{conductor!of Dirichlet character} of $\tc$
  is the smallest induced modulus $\c_{\tc}$ for $\tc$.  Equivalently, $\c_{\tc}$
  is the smallest positive divisor of $N$ for which
  $\tc(a)=1$ whenever $\gcd(a,N)=1$ and $a\equiv 1\mod \c_{\tc}$.
  If $\c_{\tc}=N$, then
  $\tc$ is {\bf primitive}.\index{keywords}{Dirichlet character!primitive}

Write $\A^*=\Q^*(\R^+\times \Zhat^*)$.
  A {\bf Hecke character}\index{keywords}{Hecke character}\index{notations}{chi@$\chi$!Hecke character}
   is a continuous homomorphism $\chi: \A^*\longrightarrow\C^*$,
  trivial on $\Q^*$.  The restriction of $\chi$ to $\R^+$ is
  of the form $x\mapsto x^s$ for a unique complex number $s$.
  Therefore the Hecke character
\[\chi_0(a)=\chi(a)|a|^{-s} \index{notations'}{chi0@$\chi_0$}\]
  is trivial on $\Q^*\R^+$, so it has finite order (cf. Lemma 12.1 of \cite{KL};
  beware that in the bijection discussed after that lemma, {\em Dirichlet characters}
  should read {\em primitive Dirichlet characters}).
  Thus an arbitrary Hecke character is uniquely of the form $\chi_0\otimes|\cdot|^s$,
  where $\chi_0$ has finite order.
  The local components $\chi_p:\Q_p^*\longrightarrow \C^*$ 
 ($p\le \infty$)\index{notations}{chip@$\chi_p$, local component of a Hecke character}
  are given by
\[\chi_p(a)=
   \chi(1,\ldots,1,\stackrel{\sss {}^{p^{\text{th}}}}{a},1,1,\ldots),\]
so that $\chi=\prod_p\chi_p$.

 For a finite order Hecke character $\chi$, we let
\[\c_\chi\in\Z^+\]
\index{notations}{cchi@$\c_\chi$ conductor of $\chi$}denote 
  the conductor of $\chi$.\index{keywords}{conductor!of Hecke character}\index{keywords}{Hecke character!conductor}
  This is the smallest positive integer which has the property that
  $\chi(a)=1$ for all $a\in (1+\c_\chi\Zhat)\cap \Zhat^*$.
  For any $N\in \c_{\chi}\Z^+$ we can attach to $\chi$ a Dirichlet character
  $\chi'=\chi'_N$ of modulus $N$ and conductor $\c_\chi$, via
  \index{notations}{chi'@$\chi'=\chi'_N$ Dirichlet character mod $N$ attached to $\chi$} \index{notations'}{chi'N@$\chi'_N$}
\begin{equation}\label{chi'}
\chi:\Q^*(\R^+\times\Zhat^*)\longrightarrow \Zhat^*\longrightarrow
  (\Z/N\Z)^* \longrightarrow\C^*,
\end{equation}
  where the last arrow defines $\chi'$.  The case $N=\c_\chi$ defines
  a bijection between the set of finite order Hecke characters of conductor $N$
  and the set of primitive Dirichlet characters modulo $N$.
  For any integer $d$ prime to $N$, we have
\begin{equation}\label{x'}
\chi'(d)=\prod_{p|N} \chi_p(d)=\chi(d_N) \qquad(d,N)=1,
\end{equation}
with $d_N$ as in \eqref{dN}.

\begin{lemma}[Dirichlet vs. Hecke $L$-functions]
In the above situation, \index{keywords}{Dirichlet $L$-function} \index{keywords}{Hecke $L$-function}
\begin{equation}\label{L}
L(s,\ol{\chi'})=L_{N}(s,\chi),
\index{notations}{L schi@$L(s,\ol{\chi'})$} \index{notations}{L Nschi@$L_{N}(s,\chi)$}
\end{equation}
where the partial $L$-function on the right is defined by the Euler product
\begin{equation}\label{partialL}
L_{N}(s,\chi)=\prod_{p\nmid N} (1-\chi_p(p)p^{-s})^{-1}.
\end{equation}
\end{lemma}
\noindent{\em Remark:} If $N=\c_\chi$, i.e. $\chi'$ is primitive,
   then $L_N(s,\chi)=L(s,\chi)$ by definition.

\begin{proof}
 It is easy to show that $1=\chi(p)=\chi_p(p)\chi'(p)$ for any $p\nmid N$
   (\cite{KL}, (12.7)).  Therefore
\[L(s,\ol{\chi'})=\sum_{n>0}\ol{\chi'(n)}n^{-s}
  =\prod_{p\nmid N}(1-\ol{\chi'(p)}p^{-s})^{-1}=L_N(s,\chi).\qedhere\]
\end{proof}

We will need lower bounds for Dirichlet $L$-functions on the right edge of the
critical strip, since such $L$-values arise in the denominators of the
  Fourier coefficients of Eisenstein series.

\begin{theorem} \label{Ram} Let $\chi$ be a non-principal Dirichlet character modulo $N$.
Write $s=\sigma + it$. There exists a constant $c > 0$ for which the following statements hold.
\begin{enumerate}
\item If $\chi$ is non-real, then for \hskip .2cm
  $\ds1-\frac{c}{(\log (N(\lfloor|t|\rfloor+2)))^{9}}<\sigma\le 2$,
 \[ L(s,\chi)^{-1} \ll \Bigl(\log \bigl(N(\lfloor|t|\rfloor+2)\bigr)\Bigr)^7\]
for an absolute implied constant.
\item If $\chi$ is real, then in the region \hskip .2cm
 $\ds1-\frac{c}{(\log (N(\lfloor|t|\rfloor+2)))^{9}}<\sigma\le 2$, $|t|\ge 1$,
 \[ L(s,\chi)^{-1} \ll \Bigl(\log \bigl(N(\lfloor|t|\rfloor+2)\bigr)\Bigr)^7\]
for an absolute implied constant.
\item If $\chi$ is real and $\e>0$ is given such that $N^\e \ge \log N$, then
  in the region\\
  $1 - \frac{c}{N^{9\e}}<\sigma\le 2$,\, $\frac 1{10N^{\e}} \leq |t| \leq 1$,
  we have
\[ L(s,\chi)^{-1} \ll N^{7\e}\]
for an absolute implied constant.
\item If $\chi$ is real and $\e>0$ is given, then when $N$ is sufficiently large
  (depending on $\e$), for $|s-1| \leq \frac 1{N^{\e/2}}$ we have
\[ L(s,\chi)^{-1} \ll_\e N^{\e/2}\]
for an ineffective implied constant depending on $\e$.
\end{enumerate}
\end{theorem}
\begin{proof} See equations (3), (4) and (5) on page 218 of Ramachandra's book
  \cite{Ra}.
  The fourth case requires Siegel's Theorem, which is why the
  constant in that case is not effective.
\end{proof}

\begin{corollary} \label{Lcor}
Fix $\e>0$.  For all Dirichlet characters $\chi$ of modulus $N$,
\begin{equation}\label{Lit}
   L(1+it,\chi)^{-1} \ll_\e N^{\e} (\log(|t|+3))^{7}
\end{equation}
  for an ineffective implied constant depending only on $\e$.
\end{corollary}
\begin{proof}
Note that since $\log 3>1$,
\[\log(N(\lfloor |t|\rfloor+2))\le \log N+\log(|t|+3)
  =\log(|t|+3)
  \left(\frac{\log N}{\log(|t|+3)}+1\right)\]
\[\le \log(|t|+3)(\log N +1)\ll_\e
  \log(|t|+3)N^{\e/7}.\]
Therefore by parts 1 and 2 of the theorem, \eqref{Lit} holds if
  $\chi$ is non-real, or if $\chi$ is a non-principal real character and $|t|\ge 1$.

%
  Suppose $\chi$ is real and non-principal, and let $\e' = \e/7$.
We need to establish \eqref{Lit} for $|t| \leq 1$.
  Because $\frac 1{10N^{\e'}} \leq \frac 1{N^{\e'/2}}$,  we see that
  either $\frac 1{10N^{\e'}} \leq |t| \leq 1$ or $|t| \leq \frac 1{N^{\e'/2}}$ must hold.
   Therefore as long as $N$ is sufficiently large ($N \ge C(\e)$),
  \[L(1+it,\chi)^{-1}   \ll N^{7\e'} \ll  N^{\e} (\log(3+|t|))^{7},\]
as needed.
We still have to treat the case $N < C(\e)$, $|t| \leq 1$.
   We know that $L(1+it,\chi)^{-1}$ is continuous in $t$,
  and hence bounded on $|t| \leq 1$.
  There are only finitely many characters $\chi$ with modulus $< C(\e)$, so their
  $L$-functions can be bounded uniformly on $|t|\le 1$.
   Thus $L(1+it,\chi) \ll 1$ on $|t|\le 1$ when $N<C(\e)$.

Lastly, suppose $\chi$ is the principal character modulo $N$.
 Recall the well-known  estimate $\zeta(1+it)^{-1} \ll (\log(3+|t|))^{7}$
(\cite{In}, Theorem 10, p.28).  Then
\begin{equation}\label{zeta}
\hskip-2cm L(1+it, \chi)^{-1} =  \zeta(1+it)^{-1}\prod_{p|N} (1-p^{-(1+it)})^{-1}
\end{equation}
\[ \ll (\log(3+|t|))^7
  \prod_{p|N} (1-\frac 1p)^{-1}
 \ll (\log(3+|t|))^7 \prod_{p|N} 2\]
   \[\ll_\e (\log(3+|t|))^7 N^{\e}.\qedhere\]
\end{proof}

\pagebreak
\section{Bi-$K_\infty$-invariant functions on $\GL_2(\R)$}\label{3}

Our objective is to study the cusp forms of weight $0$,
  realized as certain right $K_\infty$-invariant $L^2$-functions on $G(\R)\times G(\Af)$.
  In order to isolate the $K_\infty$-invariant subspace of $L^2$, we will use
  an operator $R(f_\infty\times f_{\fin})$, where $f_\infty$ is a bi-$K_\infty$-invariant
  function on $G(\R)$.  In this section we review the properties of
  such functions which will be useful in what follows.

\subsection{Several guises}

Let $m$ be a fixed nonnegative integer or $\infty$.  Define
   $C^m_c(G^+//K)$\index{notations}{C cinftyGK@$C^m_c(G^+//K)$} to be the space of 
  $m$-times continuously differentiable functions $f$ on
\[G(\R)^+=\{g\in G(\R)|\, \det(g)>0\},\]
 whose support is compact modulo $Z(\R)$, and which satisfy
\begin{equation}\label{Kinv}
f(zk g k')=f(g)
\end{equation}
for all $z\in Z(\R)$ and $k,k'\in K_\infty$.     
In later sections, we will view these as functions on $G(\R)$ by setting
  $f(g)=0$ if $\det(g)<0$.  When $m=0$, we sometimes 
  denote the space by $C_c(G^+//K)$.\index{notations}{C cinftyGK@$C_c(G^+//K)$}

In terms of the Cartan decomposition\index{keywords}{Cartan decomposition}\index{notations'}{y@$y$ in the Cartan decomposition}
\begin{equation}\label{Cartan}
G(\R)^+=Z(\R)K_\infty\left\{\mat{y^{1/2}}{}{}{y^{-1/2}}\right\}K_\infty,
\end{equation}
  an element $f\in C_c^m(G^+//K)$ depends only on the parameter $y$.
 As a function of $y$, it is invariant under $y\mapsto y^{-1}$,
 since $f(\smat{}1{-1}{}g\smat{}{-1}1{}) = f(g)$.
Thus we have the following isomorphism
\[C^m_c(G^+//K)\longrightarrow C_c^m(\R^+)^w,\]
 where $C_c^m(\R^+)^w$\label{Cw}\index{notations}{C cinftyR+w@$C_c^m(\R^+)^w$}
   is the space of smooth compactly supported functions on $\R^+$ (the set of positive
  real numbers)
  that are invariant under $y\mapsto y^{-1}$.  The value of such a 
  function depends only on the unordered pair $\{y,y^{-1}\}$.
  The set of such pairs is in 1-1 correspondence with the real interval
  $[0,\infty)$ via $\{y,y^{-1}\}\leftrightarrow y+y^{-1}-2$.

\begin{proposition}\label{uprop} Suppose $m\ge 0$ and $0\le 3m'\le m+1$.  Then for $y\in \R^+$, 
  the substitution
\begin{equation}\label{u}
 u =y+y^{-1} -2\index{notations'}{u@$u=y+y^{-1} -2$}
\end{equation}
defines a $\C$-linear injection
 $C_c^m(\R^+)^w\longrightarrow C_c^{m'}([0,\infty))$  whose image
  contains $C_c^m([0,\infty))$.  In particular, this map is an isomorphism in the
  two cases $m=m'=0$ and $m=m'=\infty$.
\end{proposition}

\begin{proof}
We first consider the case of smooth functions.
  Let $a(y)\in C^\infty_c(\R^+)^w$, and let $A(u)=a(y)$
  be the associated function of $u\in[0,\infty)$.
  It is easy to see that $A$ is $C^\infty$ on $(0,\infty)$, however
  the smoothness at the endpoint $u=0$ is not obvious because
  $\tfrac{dy}{du} = (1-\frac1{y^2})^{-1}$ blows up at $y=1$.
It is helpful to write $y=e^x$, and define
$h(x)=a(e^x)=A(u)$.  Then:
\begin{itemize}
\item $h(-x)=h(x)$ is even
\item $h\in C^\infty_c(\R)$
\item $u=e^x+e^{-x}-2$, so $\frac{du}{dx}=e^x-e^{-x}=2\sinh(x)$.
\end{itemize}
For $u>0$, write
\[A^{(n)}(u) = \frac{p_n(x)}{2^n(\sinh x)^{2n-1}}.\]
When $n=1$ this holds with $p_1(x)=h'(x)$, and by differentiating, we find
 in general that
\begin{equation}\label{Pn}
p_{n+1}(x) = \sinh(x)p_n'(x)-(2n-1)\cosh(x)p_n(x).
\end{equation}
We claim that $p_{n+1}(x)$ vanishes at least to order $2n+1$ at $0$.
  We prove the claim by induction, the base case being
  $p_1(0)=h'(0)=0$, $h'$ being odd since $h$ is even.
  It is clear that $p_{n+1}(x)$ vanishes to
  at least order $2n-1$ at $0$, since this is true of both terms
  on the right-hand side of \eqref{Pn} by the inductive hypothesis.
  It remains to show that
\[p_{n+1}^{(2n-1)}(0)=p_{n+1}^{(2n)}(0)=0.\]
  By differentiating \eqref{Pn}, we see that for $0\le j\le 2n-1$,
\[p_{n+1}^{(j)}(x) = \sinh(x)p_n^{(j+1)}(x) -(2n-1-j)\cosh(x)p_n^{(j)}(x)
  + L_j(x),\]
where $L_0(x)=0$ and $L_j(x)$, $j>0$, is a linear combination of derivatives of $p_n$
  of order $<j$. In particular, $L_j(0)=0$ for all $j\le 2n-1$.
  Taking $j=2n-1$ gives
\[p_{n+1}^{(2n-1)}(0)=\sinh(0)p_n^{(2n)}(0) + L_{2n-1}(0) =0.\]
  Furthermore, it follows inductively from \eqref{Pn} that
  $p_{n+1}(x)$ is an odd function, so that the even order derivative
  $p_{n+1}^{(2n)}(0)$ vanishes.
  This proves the claim.

Now we can prove by induction that $A^{(n)}(u)$ is defined and continuous at $u=0$.
  This is clear when $n=0$.  Assuming it holds for some given $n\ge 0$, we note that
  by L'Hospital's rule,
\[A^{(n+1)}(0)=\lim_{u\to0^+}\frac{A^{(n)}(u)-A^{(n)}(0)}u=\lim_{u\to 0^+}A^{(n+1)}(u)
=\lim_{x\to 0}\frac{p_{n+1}(x)}{2^{n+1}(\sinh x)^{2n+1}},\]
provided the limit exists.
%
%
The denominator vanishes exactly to order $2n+1$ at $0$, and
 we just showed that $p_{n+1}(x)$ vanishes at least to order $2n+1$ at $0$.
  Therefore we may apply L'Hospital's rule $2n+1$ times, at which
  point we get a nonvanishing denominator at $0$,
  giving a finite limit as needed.
%
%

Now suppose $a(y)$ (hence $h(x)$) is only assumed to be $m$-times continuously 
  differentiable.
In order for the above argument to run with $m'=n+1$, we need
  $p_{n+1}(x)$ (hence $h^{(n+1)}(x)$) to be $(2n+1)$-times continuously differentiable,
  i.e. $m\ge 3n+2=3m'-1$.
  It is clear that the resulting map $C^m_c(\R^+)^w\rightarrow C_c^{m'}([0,\infty))$
  is injective.  On the other hand, given $A\in C_c^m([0,\infty))$ and defining
  $a(y)=A(y+y^{-1}-2)=A(u)$, it follows from the fact that $\tfrac{du}{dy}=(1-y^{-2})$ is smooth
  on $\R^+$ that $a\in C_c^m(\R^+)^w$, so $A$ is in the image of the map.
\end{proof}

 We define, for $f\in C_c^m(G^+//K)$,
\begin{equation}\label{V}
V(u)=V(y+y^{-1}-2)=f(\mat{y^{1/2}}{}{}{y^{-1/2}}).\index{notations}{V@$V$}
\end{equation} 
By the preceding discussion, we have the following.

\begin{proposition}\label{fV}  Suppose $0\le 3m'\le m+1$.  Then the 
  assignment $f\mapsto V$ defines an injection
  $C^m_c(G^+//K)\rightarrow C^{m'}_c([0,\infty))$ whose image contains
  $C^m_c([0,\infty))$.  In particular, it is an isomorphism if $m=m'=0$ or 
  $m=m'=\infty$.
\end{proposition}

For $g=\smat abcd\in G(\R)^+$, we can recover the parameter
  $y+y^{-1}$ as follows.  Writing
  $g=(\sqrt{\det g})k_{\theta'}\smat{y^{1/2}}{}{}{y^{-1/2}}k_{\theta}$,
  we see that
  ${}^tg\,g=(\det g)k_{\theta}^{-1}\smat{y}{}{}{y^{-1}}k_\theta$,
  where $^tg$ denotes the transpose.  Therefore
\begin{equation}\label{tt}
y+y^{-1} = \frac{\tr({}^tg\,g)}{\det g}=\frac{a^2+b^2+c^2+d^2}{ad-bc}.
\end{equation}
Thus we can recover $f$ from $V$ via:
\begin{equation}\label{Vf}
f(\mat abcd)=V(\frac{a^2+b^2+c^2+d^2}{ad-bc}-2).
\end{equation}

We can also identify $f$ with a function of two variables on the
  complex upper half-plane $\mathbf{H}$\index{notations}{H@$\mathbf{H}$, complex upper half-plane}.
 Recall the correspondence
\[G(\R)^+/Z(\R)K_\infty\longleftrightarrow\mathbf{H}\]
  induced by $\smat abcd\mapsto \smat abcd(i)=\frac{ai+b}{ci+d}$. \index{notations}{17@$\smat abcd z=\frac{az+b}{cz+d}$}
  By this, the following function is well-defined:
\begin{equation}\label{kzz}
k(z_1,z_2) = f(g_1^{-1}g_2) \qquad (z_1,z_2\in\mathbf{H}),\index{notations}{kz1z2@$k(z_1,z_2)$}
\end{equation} 
  where $g_j(i)=z_j$ for $j=1,2$.
Clearly
\[k(\g z_1,\g z_2)=k(z_1,z_2)\]
  for all $\g\in G(\R)^+$, and in particular for any real scalar $c>0$,
\begin{equation}\label{kcz}
k(cz_1,cz_2) = k(z_1,z_2).
\end{equation}

\begin{proposition}\label{kV}
With notation as above, for $g_1,g_2\in G(\R)^+$ we have
\[k(z_1,z_2) = V\left(\frac{|z_1-z_2|^2}{y_1y_2}\right)=f(g_1^{-1}g_2).\]
\end{proposition}

\begin{proof}
Writing $g_1^{-1}g_2=\smat{y_1}{x_1}{0}{1}^{-1}
  \smat{y_2}{x_2}{0}{1}
  =\smat{{\frac{y_2}{y_1}}}{\frac{(x_2-x_1)}
  {{y_1}}}{0}{1}$, by \eqref{tt} we have
\[u=\frac{\scri a^2+b^2+c^2+d^2}{\scri ad-bc}-2
  =\frac{y_2}{y_1}+\frac{(x_2-x_1)^2}{y_1y_2}
  +\frac{y_1}{y_2}-\frac{2y_1y_2}{y_1y_2}
=\frac{(x_2-x_1)^2+(y_2-y_1)^2}{y_1y_2}.\qedhere\]
\end{proof}

\subsection{The Harish-Chandra transform}

Given $f \in C_c^m(G^+//K)$, its Harish-Chandra transform\index{keywords}{Harish-Chandra transform}
  is the function of $y\in \R^+$ defined by\index{notations'}{Hf@$\mathcal{H} f$}
\[ (\mathcal{H} f) (y) = y^{-1/2} \int_\R
   f(\mat{1}{x}{}{1}\mat{y^{1/2}}{}{}{y^{-1/2}} ) dx. \]
The absolute convergence follows easily from the compactness
   of the support of $f$.
  It is clear that $\mathcal{H}f$ is also compactly supported.

If we identify $f$ with $V$, and let $u=y+y^{-1}-2$ as before,
  the transform is traditionally denoted
\[Q(u)=\int_{\R} V(u +x^2)dx,\]
following Selberg.  To see the equivalence, by \eqref{tt} we have
\[ (\mathcal{H} f)(y) = y^{-1/2} \int_{\R} f
   (\mat{y^{1/2}}{xy^{-1/2}}{0}{y^{-1/2}})dx=\int_{\R}f(\mat{y^{1/2}}x{}{y^{-1/2}})dx \]
\begin{equation}\label{HCV}
=\int_\R V(y+y^{-1}+x^2-2)dx=\int_\R V(u+x^2)dx.
\end{equation}
From this, we see that $\mathcal{H}f$ belongs to the space
  $C^m_c(\R^+)^w$.

\begin{proposition}\label{HC}
Suppose $m,m'\ge0$ with $3m'\le m+1$. 
  Then the Harish-Chandra transform defines a commutative diagram
\begin{diagram}[tight,height=.65cm,width=2.0cm,balance]
 C_c^m(G^+//K)&\rTo^{f\mapsto \mathcal{H}f} &C_c^m(\R^+)^w
\\
\dTo&&\dTo\\
C^{m'}_c([0,\infty))&\rTo^{V\mapsto Q}&C^{m'}_c([0,\infty)),
\end{diagram}
where all arrows are injective.  The image of the bottom map $V\mapsto Q$
  contains $C^m([0,\infty))$.
When $m=m'=\infty$, all arrows are isomorphisms.
%
Generally, if $m'>0$ then for $u = y+y^{-1}-2$ and
\[  \mathcal{H}f (y)= Q(u) = \int_\R V(u + x^2) dx,\]
the inverse transform is given by
\begin{equation}\label{Hinv}
 V(u) = -\frac{1}{\pi} \int_\R Q'(u + w^2)dw. 
\end{equation}
\end{proposition}

\noindent{\em Remarks:}
For the smooth case, see also \cite{La1}, \S V.3, Theorem 3, p. 71.
 For more detail about the inverse transformation,
  see Propositions \ref{Vd} and \ref{VQ} below.  For example, we will show that 
the image of the bottom map contains $C^{m'+1}([0,\infty))$.  In fact, 
  given $Q$ in this space, if we {\em define} $V$ by \eqref{Hinv}, 
  then $V\in C^{m'}([0,\infty))$ and $Q(u)=\int_{\R}V(u+x^2)dx$.

\begin{proof}
The commutativity of the diagram follows from \eqref{HCV}.
  The vertical maps are injective by Propositions \ref{uprop} and \ref{fV} above.
  As described there, the image of the right-hand vertical map
   $\mathcal Hf\mapsto Q$ contains
  $C^{m}([0,\infty))$, so by commutativity of the diagram, 
  the same holds for the image of the bottom map $V\mapsto Q$.
  It follows that when $m'=\infty$, all arrows are surjective.
  The injectivity of the horizontal arrows is a consequence of the inversion 
  formula \eqref{Hinv}, so it just remains to prove the latter.
We can differentiate $Q(u)$ under the integral sign because
  $V\in C^{m'}_c([0,\infty))$ for $m'\ge1$ (cf. Proposition \ref{lebdiff}).  Thus
\[ Q'(u) = \int_\R V'(u + x^2)dx. \]
Hence
\[ -\frac1\pi\int_\R Q'(u +  w^2) dw
 = -\frac1\pi\int_\R \int_\R V'(u + w^2 + x^2) dx dw \]
\[ = -\frac1\pi\int_0^{2\pi} \int_0^\infty V'(u + r^2) r dr d\theta
 =  -2\int_0^\infty V'(u + t) \frac{dt}{2}
  = V(u). \qedhere\]
\end{proof}

\subsection{The Mellin transform}

For $\Phi \in C^{m}_c(\R^+)$, its Mellin transform\index{keywords}{Mellin transform}
  is the function of $\C$ defined 
by\index{notations}{Mphis@$\mathcal{M},\mathcal{M}_s$ (Mellin transform)}
\begin{equation}\label{mellin}
 (\mathcal{M}\Phi)(s) = \int_0^\infty  \Phi(y) y^s \frac{dy}{y}. 
\end{equation}
This is a Fourier transform on the multiplicative group $\R^+$.
We also denote the above by $\mathcal{M}_s\Phi$.
It is easily shown to be an entire function of $s$.
When $m\ge 2$, we have
\begin{equation}\label{mi}\index{keywords'}{Mellin inversion formula}
\Phi(y)=\frac1{2\pi i}\int_{\Re(s)=\sigma}(\mathcal M\Phi)(s)y^{-s}ds
\end{equation}
for any $\sigma\in\R$.
This is the Mellin inversion formula, which 
  we will prove under somewhat more general hypotheses in Propositions
  \ref{indepsigma} and \ref{Minv}.

We say that an entire function $\eta:\C\longrightarrow\C$ is {\bf Paley-Wiener of order $m$}
  if there exists a real number $C\ge 1$ depending only on $\eta$
   such that 
\begin{equation}\label{PWm}
 |\eta(\sigma + it)| \ll_{m,\eta} \frac{C^{|\sigma|}}{(1 + |t|)^m}. 
\end{equation}
  We let $PW^m(\C)$\index{notations}{P WmC@$PW^m(\C)$}
 denote the space of such functions.
If the above holds for all $m>0$ with the same $C$, then $\eta$ belongs to
 the {\bf Paley-Wiener space} 
  $PW^\infty(\C)=PW(\C)$.\index{keywords}{Paley-Wiener function}\index{notations}{P WC@$PW(\C)$, Paley-Wiener space}

\begin{proposition}
Suppose $m\ge 0$.  Then the Mellin transform defines an injection
\[ \mathcal{M} : C_c^m(\R^+) \longrightarrow PW^m(\C)\]
whose image contains $PW^{m+2}(\C)$.
  On $PW^{m+2}(\C)$, the inverse map is given by \eqref{mi}.
In particular, if $m=\infty$ the transform is an isomorphism.
\end{proposition}

\begin{proof}
(See also \cite{La1}, \S V.3, Theorem 4, p. 76.)
 First we show that the image of the Mellin transform lies in $PW^m(\C)$.
  When $m=0$ this is obvious.  Given $\Phi\in C_c^m(\R^+)$ for $m>0$,
  we may apply integration by parts to \eqref{mellin} to get
\[\mathcal{M}\Phi(s)=\left.\Phi(y)\frac{y^s}s\right|_0^\infty-
  \frac1s\int_0^\infty\Phi'(y)y^{s+1}\frac{dy}y.\]
Since $\Phi$ is compactly supported in $\R^+$, the first term on the right vanishes.
  Continuing, we find 
\[\mathcal{M}\Phi(s)= \frac{(-1)^m}{s(s+1)\cdots(s+m-1)}
  \int_0^\infty\Phi^{(m)}(y)y^{s+m}\frac{dy}y.\]
Using this, it is straightforward to see that $\mathcal{M}\Phi$ satisfies \eqref{PWm}.
The injectivity of the map is immediate from the inversion formula \eqref{mi}.





For $\eta \in PW^{2}(\C)$, we can define $\Phi(y)=\frac1{2\pi i}\int_{\Re s=\sigma}\eta(s)
  y^{-s}ds$ as in \eqref{mi}, the convergence being absolute for any $\sigma$ by \eqref{PWm}
  with $m=2$.  To see that $\Phi$ is compactly supported, consider
\[ \Phi(y) = \frac{1}{2\pi i} \int_{\Re s = \sigma} \eta(s) y^{-s}ds
  \ll  \int_{-\infty}^\infty \frac{C^{\sigma} |y|^{-\sigma}}{(1+|t|)^2} dt \ll (C|y|^{-1})^\sigma. \]
If $|y| > |C|$, then the right-hand side approaches $0$ as 
  $\sigma \rightarrow \infty$, so $\Phi(y)=0$.
Thus $\Supp \phi \subset [-C, C]$.

Lastly, if $\eta \in PW^{m+2}(\C)$ and $\Phi$ is defined as above,
  then $\eta=\mathcal{M}\Phi$ (cf. Proposition \ref{Minv}),
  and it is not hard to show that $\Phi\in C^m_c(\R^+)$.
  The idea is that after differentiating under the integral sign $m$ times, we still have
  an integrand with sufficient (quadratic) decay in $t=\Im s$ for convergence.
  See Proposition \ref{Sd} below for details.
\end{proof}


\subsection{The Selberg transform}

If we restrict the Mellin transform to the space $C_c^{m}(\R^+)^w$ defined
  on p. \pageref{Cw},
  it gives an injection to the space
  $PW^m(\C)^{\operatorname{even}}$\index{notations}{P WCeven@$PW(\C)^{\operatorname{even}}$}
  of {\em even} functions that are Paley-Wiener of order $m$.\index{keywords'}{even Paley-Wiener functions}
  The composition of the Harish-Chandra and Mellin transforms is called the
  {\bf spherical transform},\index{keywords}{spherical transform}
  which we denote by\index{notations'}{Sfs@$(\mathcal{S}f)(s)$}
\[(\mathcal{S}f)(s) = \mathcal{M}_{s}(\mathcal{H}f).\]
Because $\mathcal{M}_s$ and $\mathcal{H}$ are injective,
  we immediately see the following.

\begin{proposition}\label{ST}
For $m\ge 0$,  the spherical transform
\begin{diagram}[tight,height=.65cm,width=2.1cm,balance]
 \mathcal{S}:C_c^m(G^+//K)&\rTo^{f\mapsto \mathcal{S}f}
   &PW^m(\C)^{\operatorname{even}}
\end{diagram}
is injective, and its image contains $PW^{m+2}(\C)^{\text{even}}$.
In particular, when $m=\infty$ it is an isomorphism.
\end{proposition}

The {\bf Selberg transform} of $f\in C_c^m(G^+//K)$\index{keywords}{Selberg transform}
  is a variant of the above, defined by
\begin{equation}\label{Selberg}\index{notations}{ht@$h(t)$, Selberg transform}
h(t)=(\mathcal{S}f)(it) = \mathcal{M}_{it}\mathcal{H}f.
\end{equation}
 It is given explicitly by
\begin{equation}\label{Sel2}
h(t) = \int_{0}^\infty\int_{-\infty}^\infty f(\smat 1x01
  \smat{y^{1/2}}{}{}{y^{-1/2}})y^{\frac12+it}\frac{dx\,dy}{y^2}
=\iint_{\mathbf{H}}k(i,z)y^{\frac12+it}dz,
\end{equation}
where $dz=\frac{dx\,dy}{y^2}$\index{notations}{dz@$dz=\frac{dx\,dy}{y^2}$ measure on $\mathbf{H}$}
   is the $G(\R)^+$-invariant measure on $\mathbf{H}$.
Note that $s\mapsto h(-is)$ is Paley-Wiener of order $m$.

\begin{proposition}
Suppose $m>2$ and $h(-is)=(\mathcal{S}f)(s)\in PW^m(\C)^{\mathrm{even}}$.
  Then the inverse of the Selberg transform is given by\index{notations}{V@$V$}
\[V(u) = \frac1{4\pi}\int_{-\infty}^\infty P_{-\frac12+it}
  (1+\tfrac u2)h(t)\tanh(\pi t)\,t\,dt,\]
for the Legendre function\index{keywords}{Legendre function} 
  $P_s(z)=P_s^0(z)$.\index{notations}{P@$P_s(z)$ (Legendre function)}
   In particular, we have
\begin{equation}\label{V0} \index{notations'}{V0@$V(0)$}
f(1)=V(0) = \frac1{4\pi}\int_{-\infty}^\infty h(t) \tanh(\pi t)\,t\, dt.
\end{equation}
\end{proposition}
\begin{proof} (See also (2.24) of \cite{Za} or ($1.64'$) of \cite{Iw}.)
Beginning with the fact that $\mathcal{M}_s(\mathcal{H}f)=h(-is)$, 
  we apply Mellin inversion \eqref{mi} to get, for $y>0$,
\[ (\mathcal{H}f)(y) = \frac 1{2\pi i}\int_{\Re s = 0} h(-is) y^{-s}ds
 = \frac 1{2\pi} \int_{\R} h(r) y^{-ir} dr  = \frac 1{2\pi} \int_{\R} h(r) y^{ir} dr,\]
since $h$ is even.
Write $y=e^{v}$ and $u=y+y^{-1}-2 = e^{v}+e^{-v}-2$, and define 
\[g(v)=Q(u)= (\mathcal{H}f)(y).\]
Then $\ds g(v) = \frac 1{2\pi} \int_{\R} h(r) e^{irv} dr$, and differentiating
  (cf. Proposition \ref{lebdiff}),
\[ g'(v) = \frac {i}{2\pi} \int_{\R} r h(r) e^{irv} dr=-\frac1{2\pi}
  \int_{\R}\sin(rv)r\,h(r)dr\]
since $h$ is even.
We have used the fact that $m>2$, so in particular the above is absolutely convergent.
Now we invert the Harish-Chandra transform via \eqref{Hinv} to get
\[ V(w) = -\frac{1}{\pi} \int_\R Q'(w+x^2) dx  = -\frac{2}{\pi}\int_0^{\infty} Q'(w+x^2) dx   \]
\[ = -\frac{1}{\pi} \int_{w}^{\infty}  \frac{Q'(u) du}{\sqrt{u-w}} 
 = -\frac{1}{\pi} \int_{\cosh^{-1} (1+\frac w2)}^{\infty}
    \frac{g'(v) dv}{\sqrt{e^v+e^{-v}-2-w}} \]
\[ = \frac{1}{2\pi^2} \int_{-\infty}^\infty \int_{\cosh^{-1} (1+\frac w2)}^{\infty}
    \frac{\sin(rv) }{\sqrt{e^{v}+e^{-v}-2-w}}dv\, rh(r)dr.\]
The interchange of the integrals is justified by the absolute convergence of the 
 integral, which follows easily by the fact that $m > 2$.
%
  As observed by Zagier (\cite{Za} (2.24)), using the identities 8.715.2 and 8.736.7 of
   \cite{GR}, it is straightforward to show that the above is
\[ = \frac{1}{4\pi} \int_\R P_{-\frac 12+ir}(1+\frac w2) \tanh(\pi r) r\, h(r) dr. \qedhere\]
\end{proof}

It is sometimes desirable to extend all of these transforms to
  functions with sufficient decay rather than just the case of
  compact support.  We will discuss this in detail in Section \ref{Val}, but we mention
  here that the following conditions are equivalent:
  \begin{itemize}
\item $V(u)=O(u^{-\frac{1+A}2})$ as $u\to\infty$
\item $Q(u)= O(u^{-A/2})$ as $u\to\infty$
\item $h(t)$ is holomorphic in the horizontal strip $|\Im(t)|<A/2$.
\end{itemize}
See \cite{Za}, p. 320.

%

\subsection{The principal series of $G(\R)$}\index{keywords}{principal series!of $\GL_2(\R)$}

Here we recall the construction of the principal series of $G(\R)$ and
  prove some well-known simple properties.
  Detailed background is given, e.g., in \S 11 of \cite{KL}.
For $\e_1, \e_2\in \{0,1\}$ and $s_1,s_2\in\C$, define a character
  $\chi=\chi(\e_1,\e_2,s_1,s_2)$ of $B(\R)$ \index{notations}{chi@$\chi$!character on $B(\R)$}
  by \index{notations}{chi@$\chi$!$=\chi(\e_1,\e_2,s_1,s_2)$}
\[\chi(\mat{a}{b}{0}{d})= \sgn(a)^{\ve_1}|a|^{s_1}\sgn(d)^{\ve_2}|d|^{s_2}.\]
  Every character of $B(\R)$ has this form. \index{notations}{pichi@$\pi_\chi$}
  We let $\pi_\chi=\pi(\e_1,\e_2,s_1,s_2)$ \index{notations}{pichi@$\pi_\chi=\pi(\e_1,\e_2,s_1,s_2)$}
  denote the representation of
  $G(\R)$ unitarily induced from $\chi$.  The underlying representation space $V_\chi$
  consists of measurable functions on $G(\R)$ satisfying
\[\phi(\mat ab0d g)=\chi(\mat ab0d)\left|\frac ad\right|^{1/2}\phi(g),\]
  with inner product given by
\[\sg{\phi_1,\phi_2}=\int_{K_\infty}\phi_1(k)\ol{\phi_2(k)}dk.\]
The action of $G(\R)$ is given by right translation
\[\pi_\chi(g)\phi(x)=\phi(xg).\]
The representation $\pi_\chi$ is unitary when $\chi$ is unitary,
  i.e. when $s_1,s_2\in i\R$.
  See \S 11.3 of \cite{KL} for details.

  We say that a vector has {\bf weight $\k$} if it transforms by
 the scalar $e^{i\k\theta}$ under the action of $k_\theta\in K_\infty$.
  A natural basis for $V_\chi$ is $\{ \phi_\k |\, \k\in \e_1+\e_2+2\Z\}$,
  where $\phi_\k$ is characterized by
  \[\phi_\k(k_\theta)=e^{i\k\theta}.\index{notations'}{phik@$\phi_\k$}
\]
  This function spans the one-dimensional space of weight $\k$ vectors
  in $V_\chi$.

We define the {\bf spectral parameter} of $\pi_\chi$ by\index{keywords}{spectral parameter!of $\pi_\chi$}
\begin{equation}\label{t}\index{notations}{t@$t$ (spectral parameter)}
t=-\frac{i}2(s_1-s_2).
\end{equation}
  The representation $\pi_\chi$ is reducible if and only if $t\neq 0$ and
  $2it+\e_1+\e_2$ is an odd integer.
  Furthermore, the Casimir element $\Delta$ in the center 
  of\index{notations}{Delta@$\Delta$!Casimir element}
  the universal enveloping algebra $U(\lie_\C)$, whose
  right regular action on $C^\infty(G(\R)^+)$ is given
in the coordinates $z\smat1x01\smat{y^{1/2}}{}{}{y^{-1/2}}k_\theta$ by
\begin{equation}\label{Cas} 
\index{notations}{Delta@$\Delta$!Casimir element}\index{keywords}{Casimir element}
R(\Delta)\phi=-y^2\left(\frac{\partial^2\phi}{\partial x^2}
+\frac{\partial^2\phi}{\partial y^2}\right)
+y\frac{\partial^2\phi}{\partial x\partial \theta},
\end{equation}
  acts on the $K_\infty$-finite vectors of $V_\chi$ by the scalar
\begin{equation}\label{nu}
\pi_\chi(\Delta) =\frac{1}4+t^2.
\end{equation}
(See e.g. \cite{KL}, pp. 169, 185.)

The only irreducible finite dimensional unitary representations of $G(\R)$
  are the unitary characters.  For the infinite dimensional ones, we have the following.

\begin{proposition}\label{LK}
Let $\pi$ be an irreducible infinite dimensional unitary representation of $G(\R)$.
  Then $\pi$ contains a nonzero vector of weight $0$ (resp. weight $1$) if and only if
  $\pi\cong\pi(\e_1,\e_2,s_1,s_2)$ is an irreducible principal series
  representation with $\e_1+\e_2$ even (resp. odd), and either:

$\cdot$ $s_1, s_2\in i\R$ (unitary principal series) or



$\cdot$ $t\in i\R$, $0<|t|<\frac12$, and $s_1+s_2\in i\R$ (complementary series),

\noindent for $t$ as in \eqref{t}. The vector is unique up to scalar multiples.
\end{proposition}

\begin{proof}
Any irreducible unitary representation $\pi$ is infinitesimally equivalent
  with a subrepresentation of a principal series representation of $G(\R)$.
  A proper subrepresentation containing
  a vector of weight $0$ or $1$ is necessarily finite dimensional
  (see e.g. \cite{KL}, p. 164).  Therefore $\pi\cong \pi(\e_1,\e_2,s_1,s_2)$
  is itself a principal series representation.  Since $\pi$ is unitary,
  one of the two given scenarios must hold.
\end{proof}

Generally, if $\phi$ is any continuous function on $G(\R)$, we extend the
 right regular action of $G(\R)$ to an action of
   $f\in C_c^m(G^+//K)$ by defining \index{notations}{Rf@$R(f)$}
  \[R(f)\phi(g') = \int_{\olG(\R)}f(g)R(g)\phi(g')dg=\int_{\olG(\R)}f(g)\phi(g'g)dg.\]
  If $\phi$ is right $Z(\R)K_\infty$-invariant, it can be viewed as a function
 on $\mathbf{H}$, and after replacing $g$ by $g'^{-1} g$ in the above, we find
\begin{equation}\label{Rfk}
R(f)\phi(g') = \int_{\mathbf{H}}k(z',z)\phi(z)\frac{dx\,dy}{y^2},
\end{equation}
where $k(z',z)$ is the function attached to $f$ in \eqref{kzz}.

As shown by Selberg \cite{S}, if $\phi(z)$ is an eigenfunction of the Laplacian
  with eigenvalue $\frac14+t^2$, then in the sense of \eqref{Rfk},
  $R(f)\phi = h(t)\phi$ for the Selberg transform $h$ of $f$
  (see also Theorem 1.16 of \cite{Iw}).
  We prove this here in the special case of interest to us.  We use the setting of
  weight $\k$ functions as an example of how the results of this section immediately
  generalize from $\k=0$.

\begin{proposition}\label{finf}
  Let $\pi=\pi_\chi$ be as above.  Let $f_\infty$
  be a continuous function whose support lies in $G(\R)^+$
   and is compact modulo $Z(\R)$,
  satisfying
\[f_\infty(z k_{\theta_1}^{-1} gk_{\theta_2})
  =\chi(z)^{-1}e^{-i\k(\theta_2-\theta_1)}f_\infty(g)
  \qquad(z\in Z(\R), \,k_{\theta_j}\in K_\infty).\]
  Then the operator $\pi(f_\infty)$ preserves the one-dimensional subspace $V_\k$ of
  weight $\k$ vectors in $V_\chi$, and vanishes on its orthogonal complement.
  If this subspace is nonzero
  (i.e. $\k\equiv \e_1+\e_2\mod 2$), then
\begin{equation}\label{lambdainf}
\pi(f_\infty)\phi_\k=h(t)\phi_\k,
\end{equation}
where $t$ is the spectral parameter \eqref{t} of $\pi$, and
  $h$ is the Selberg transform of $f_\infty$, defined in \eqref{Selberg}.
\end{proposition}


\begin{proof}
We will prove the first claim in a more general context in Lemma \ref{basic}
  below, so we grant it for now.
Hence if $\phi=\phi_\k\in V_\chi$,  $\phi$ is an eigenvector of
  $\pi(f_\infty)$ since $\dim V_\k=1$.
   The eigenvalue $\lambda$ is given by
\[\lambda=\pi(f_\infty)\phi(1)=\int_{\olG(\R)} f_\infty(g)\phi(g)dg
=\int_{\SL_2(\R)}f_\infty(g)\phi(g)dg,\]
  by our normalization of Haar measure (cf. (7.27) on page 95 of \cite{KL}).
  Here we have used the fact that $f_\infty$ is supported on $G(\R)^+$.
  Now since $f_\infty$ and $\phi$ have opposite weights on the right, the
  integrand is right $K_\infty$-invariant.  Therefore we have
\[\lambda=\int_0^\infty\int_{-\infty}^\infty
  f_\infty(\smat1x01\smat{y^{1/2}}{}{}{y^{-1/2}})
  \phi(\smat1x01\smat{y^{1/2}}{}{}{y^{-1/2}})\frac{dx\,dy}{y^2}\]
\[=\int_0^\infty\int_{-\infty}^\infty
  f_\infty(\smat1x01\smat{y^{1/2}}{}{}{y^{-1/2}})
   y^{(s_1-s_2)/2} y^{1/2}y^{-2} dx\,dy\]
\[=\int_0^\infty\left[y^{-1/2}\int_{-\infty}^\infty
  f_\infty(\smat1x01\smat{y^{1/2}}{}{}{y^{-1/2}})\,
   dx\right]y^{it}\frac{dy}y=\mathcal{M}_{it}\mathcal{H}(f_\infty)=h(t),\]
  as required.
\end{proof}


\begin{lemma}\label{basic}
  Let $G$ be a locally compact group, let $K\subset G$
  be a closed subgroup, and let $\pi$ be a unitary representation of $G$
  on a Hilbert space $V$ with central character $\chi$.
  Then for any bi-$K$-invariant function $f\in L^1(G,\chi^{-1})$
  (i.e. integrable mod center, with central character $\chi^{-1}$),
  the operator $\pi(f)$ on $V$ given by \index{notations}{pif@$\pi(f)$}
\[\pi(f)v=\int_{\olG}f(g)\pi(g)v\,dg\]
  has its image in the $K$-fixed subspace $V^K$, and
  annihilates the orthogonal complement of this subspace.
\end{lemma}
\noindent{\em Remark:} If the bi-$K$-invariance of $f$ is replaced by the
  property
\[f(k^{-1}gk')=\tau(k)\tau(k')^{-1}f(g)\]
  for a character $\tau$ of $K$, then the above holds with
  $V_\tau=\{v\in V|\,\pi(k)v=\tau(k)v\}$
  in place of $V^K$, as is easily seen by adjusting the proof below.

\begin{proof}
See page 140 of \cite{KL} for a detailed discussion of $\pi(f)$.
In particular, the vector $\pi(f)v$ is characterized by the property that
\[\sg{\pi(f)v,w}=\int_{\olG}f(g)\sg{\pi(g)v,w}dg\]
 for all $w\in V$.
Since $\pi$ is unitary, for any $k\in K$ we have
\[\sg{\pi(k)\pi(f)v,w}=\sg{\pi(f)v,\pi(k^{-1})w}=
  \int_{\olG} f(g)\sg{\pi(g)v,\pi(k^{-1})w}dg\]
\[  =\int_{\olG}f(g)\sg{\pi(kg)v,w}dg=\int_{\olG} f(k^{-1}g)\sg{\pi(g)v,w}dg
   = \sg{\pi(f)v,w}\]
by the left $K$-invariance of $f$.  Thus $\pi(k)\pi(f)v=\pi(f)v\in V^K$
  as claimed.

  The adjoint of $\pi(f)$ is the operator $\pi(f^*)$,
  where \index{notations}{13@$f^*(g)=\ol{f(g^{-1})}$}
\[f^*(g)=\ol{f(g^{-1})}\in L^1(G,\chi^{-1}).\]
  The right $K$-invariance of $f$ means that $f^*$ is left
  $K$-invariant, so the operator $\pi(f)^*=\pi(f^*)$ also has its image
  in $V^K$.  If $w\in (V^K)^\perp$, then for any $v\in V^K$ we have
\[\sg{\pi(f)w,v}=\sg{w,\pi(f^*)v}=0.\]
  Hence $\pi(f)w\in V^K \cap (V^K)^\perp =\{0\}$, as needed.
\end{proof}

%

\pagebreak
\section{Maass cusp forms}\label{4}\index{keywords}{Maass cusp forms}

Here we review some well-known properties of Maass cusp forms, and
  spell out their connection with the representation theory of the adele group
  $\GL_2(\A)$.

\subsection{Cusp forms of weight $0$}

More detail on the material below can be found e.g. in Iwaniec \cite{Iw}.
Fix a level $N\in\Z^+$, and let $\w'$ be a Dirichlet character whose conductor
  divides $N$ and which satisfies\index{notations}{w'@$\w'$, nebentypus}
\begin{equation}\label{w'}
\w'(-1)=1.
\end{equation}
We view $\w'$ as a character of $\Gamma_0(N)$ via $\w'(\smat abcd)=\w'(d)$.
  Note that $\w'(\g)=1$ if $\g\in \Gamma_1(N)$.
  Let $L^2(N,\w')$ denote the space of measurable functions \index{notations}{L20Nw@$L^2(N,\w')$}
  $u:\mathbf H\longrightarrow\C$ (modulo functions that are $0$ a.e.) such that
\begin{equation}\label{hg}
u(\g z)=\ol{\w'(\g)}u(z)
\end{equation}
 for all $\g\in\Gamma_0(N)$, and whose Petersson 
 norm\index{notations}{14@$\Vert u\Vert$, Petersson norm}\index{Petersson norm}
\begin{equation}\label{Petnorm}
\|u\|^2=\frac1{\psi(N)}\int_{\Gamma_0(N)\bs\mathbf H}|u(x+iy)|^2\frac{dx\,dy}{y^2}
\end{equation}
  is finite.
  Taking $\g=\smat{-1}{}{}{-1}$ in \eqref{hg} gives $\w'(-1)=1$ if $u(z)\neq 0$,
  which is why we imposed \eqref{w'}.

  Let $\delta\in G(\Q)^+$, and write
\[\delta^{-1}\Gamma_1(N)\delta\cap N(\Q)=\{\smat{1}{tM_\delta}{0}{1}|\, t\in \Z\},\]
  where $M_\delta>0$ (see Lemma 3.7 of \cite{KL}).
  If $u$ is any continuous function satisfying \eqref{hg}, we 
 set\index{notations}{udeltaz@$u_\delta(z) = u(\delta(z))$}
\[u_\delta(z) = u(\delta(z)).\]
   Then $u_\delta(z+M_\delta)=u_\delta(z)$, so for all $y>0$, it has a
   Fourier expansion about the cusp $q=\delta(\infty)$ of the 
 form\index{notations}{amdelta@$a_{m,\delta}(u,y)$, $m$-th Fourier coefficient of $u_\delta(z)$}\index{keywords}{Fourier expansion}
\[u_\delta(z)=\sum_{m=-\infty}^\infty a_{m,\delta}(u,y)\,e(nx/M_\delta).\]
  We drop $\delta$ and just write $a_m(u,y)$ when $q=\infty$.
  An element $u\in L^2(N,\w')$ is {\bf cuspidal} if its constant 
  terms\index{keywords}{constant term} vanish:\index{keywords}{cuspidal function}
\begin{equation}\label{a0}
a_{0,\delta}(u,y)=\frac1{M_\delta}\int_{0}^{M_\delta}u(\delta(x+iy) )dx=0
\end{equation}
  for all $\delta\in G(\Q)^+$ and a.e. $y>0$.
  The subspace of cuspidal functions is denoted $L^2_0(N,\w')$.\index{notations}{L20Nw'@$L^2_0(N,\w')$}

  The hyperbolic Laplacian is defined as an operator on $C^\infty(\mathbf{H})$ 
  by\index{keywords}{Laplacian}
\begin{equation}\label{Lap}\index{notations}{Delta@$\Delta$!Laplacian}
\Delta= -y^2\left(\frac{\partial^2}{\partial x^2}
  +\frac{\partial^2}{\partial y^2}\right).
\end{equation}
  This operator commutes with the action of $G(\R)^+$:
\[(\Delta u)(g z) = \Delta(u(gz)).\]
  By this invariance, $\Delta$ descends to an operator on
   $C^\infty_c(\Gamma_1(N)\bs \mathbf{H})$,
   which is dense in $L^2(N,\w')$.
  One can show that relative to the Petersson inner product,
  this operator is symmetric and positive:
\begin{equation}\label{symm}\sg{\Delta \phi,\psi}=\sg{\phi,\Delta \psi},\end{equation}
\begin{equation}\label{nonneg}
\sg{\Delta \phi,\phi}=\tfrac1{[\Gamma(1):\Gamma_1(N)]}
  \int_{\Gamma_1(N)\bs\mathbf{H}}\|\nabla \phi(x+iy)\|^2dx\,dy\ge 0
\end{equation}
(\cite{La1}, \S XIV.4).
The operator $\Delta$ extends to an elliptic operator on the
  distribution\index{keywords}{distribution} 
  space $\mathcal{D}'(\Gamma_1(N)\bs\mathbf{H})$
  of continuous linear functionals
  on $C_c^\infty(\Gamma_1(N)\bs\mathbf{H})$.  See \cite{F}, p. 284.
  Identifying $\phi\in L^2(N,\w')$ with the functional $f\mapsto\sg{f,\phi}$ realizes
  $L^2(N,\w')$ as a subspace of $\mathcal{D}'(\Gamma_1(N)\bs\mathbf{H})$,
  although this subspace is not stable under the extended operator $\Delta$.
%

  A {\bf Maass cusp form} of level $N$ and nebentypus $\w'$ is
  an eigenfunction $u$ of $\Delta$ in the subspace $L^2_0(N,\w')$ (\cite{Ma}).
  By the elliptic regularity theorem, such an eigenfunction is necessarily smooth,
  i.e. $u\in C^\infty(\mathbf{H})$ (cf. \cite{F2} p. 214, or \cite{La1} p. 407).
  We write $\Delta u =(\frac14+t^2) u$ for the Laplace 
  eigenvalue\index{notations}{Delta@$\Delta$!z@eigenvalue $\tfrac14+t^2$} and
  call $t$ the {\bf spectral parameter} of 
$u$.\index{keywords}{spectral parameter!of $u$}\index{notations}{t@$t$ (spectral parameter)}
  It is also customary to use $s(1-s)$ for the eigenvalue, where the
  relationship is given by $s=\frac12+it$.  We will not use this notation,
  preferring instead to use $s=it$.

\begin{theorem}\label{Delta}
The cuspidal subspace $L^2_0(N,\w')$ has an orthogonal basis consisting
  of Maass cusp forms.  Each cuspidal eigenspace of $\Delta$ is finite
   dimensional, and
the eigenvalues are positive real numbers
  $\lambda_1\le \lambda_2\le \cdots$ with no finite limit point.
\end{theorem}
\noindent{\em Remarks:} (1) A famous conjecture of Selberg asserts that
  $\lambda_1\ge \tfrac14$,\index{keywords}{Selberg's conjecture}
  or equivalently, that all of the spectral parameters $t$ are real, \cite{Sa1}.
  (It is not hard to show that the set of $t\notin \R$
  is finite; see Corollary \ref{disccor} on page \pageref{disccor}.)
  Selberg proved that $\lambda_1\ge \tfrac3{16}$.
  See \S6.2 of \cite{DI}, where this is proven as
  a consequence of the Kuznetsov formula.  The best bound to date is
  $\lambda_1\ge \tfrac14-(\tfrac7{64})^2\approx 0.238037...$, due to
  Kim and Sarnak \cite{KS}.\vskip .2cm

\noindent (2) In the case of level $N=1$, Cartier conjectured that
  the eigenvalues occur with multiplicity one (\cite{Car}).  Until very recently,
  it was widely believed that the eigenvalues of $\Delta$ on the
   {\em newforms} of level $N$ should occur with multiplicity one.
  However, Str\"omberg has discovered counterexamples on $\Gamma_0(9)$
  which, despite coming from newforms, nevertheless arise out of the spectrum
  of a congruence subgroup of lower level (\cite{St}).
  Some of his examples were found independently by Farmer and Lemurell.

\begin{proof}
 (Sketch.  See also \cite{Iw}, \S 4.3 and \cite{IK}, \S 15.5.)
The existence of the basis is a consequence of
  the complete reducibility of $L^2_0(\olG(\Q)\bs\olG(\A),\w)$
 (see Proposition \ref{diag} below).
  The discreteness of the set of eigenvalues and the finite dimensionality
  of the eigenspaces both follow from \eqref{discrete} on page \pageref{discrete}.
  The fact that there are infinitely many linearly independent cusp forms
  can be seen from Weyl's Law\index{notations}{Weyl's Law} 
  (see \eqref{Weyls}).
  By \eqref{nonneg} the eigenvalues of $\Delta$ are nonnegative.  If $\Delta u = 0$
  for $u\in L^2_0(N,\w')$, then $u$ is a harmonic function on $\Gamma_1(N)\bs \mathbf{H}$.
  By the maximum principle (\cite{F2}, p. 72), the supremum of $|u(z)|$ occurs on the
  boundary, i.e.  at a cusp, where $u$ vanishes.  Hence $u=0$.  This shows that the $\lambda_j$
  are strictly positive.
\end{proof}

If $u$ is a Maass cusp form with $\Delta$-eigenvalue $\frac14+t^2$,
  its Fourier expansion at $\infty$ has the well-known form \index{keywords}{Fourier expansion}
\begin{equation}\label{fourier} \index{notations}{am@$a_m(u)$, $m$-th Fourier coefficient of $u(z)$}
u(x+iy)=\sum_{m\in \Z-\{0\}}a_m(u)\, y^{1/2} K_{it}(2\pi |m|y)e^{2\pi i mx}
\end{equation}
for constants $a_m(u)$ called the {\bf Fourier coefficients} of $u$
   (see e.g. \cite{Bu}, \S 1.9).  The $K$-Bessel function can be
  defined by\index{notations}{Ks@$K_s(z)$, Bessel function}
\index{keywords}{Bessel function!$K$-}
\begin{equation} \label{KBessel}
K_s(z)=\frac12\int_0^\infty e^{-z(w+w^{-1})/2}\,w^s\,\frac{dw}w,
\end{equation}
for $s\in\C$ and $\Re(z)>0$.

\subsection{Hecke operators} \index{keywords}{Hecke operator}

  For $u\in L^2(N,\w')$ and an integer $\n>0$, the Hecke operator $T_\n$ is
  given by \index{notations}{T n@$T_\n$, Hecke operator}
\[T_\n u(z) = \n^{-1/2} \sum_{ad=\n\atop d>0}\sum_{r=0}^{d-1}\ol{\w'(a)}u(\frac{az+r}d).\]
One shows in the usual way that $T_\n u\in L^2(N,\w')$.

We also define $T_{-1}u(x+iy)=u(-x+iy)$.  A Maass cusp form is {\em even}
  (resp. {\em odd}) if $T_{-1}u=u$ (resp. $T_{-1}u=-u$).
  If $u$ is even, then in \eqref{fourier}
  we have $a_{-n}=a_n$, while if $u$ is odd, $a_n=-a_{-n}$.

\begin{proposition} The Hecke operators for $\gcd(\n,N)=1$ are normal operators
  on $L^2(N,\w')$.  They commute with each other and with $\Delta$.
  Hence the family of operators $\{\Delta,T_{-1},T_\n|\, \gcd(\n,N)=1\}$
  is simultaneously diagonalizable on $L^2_0(N,\w')$.
\end{proposition}

\begin{proof} We can compute the adjoint of $T_\n$ as in \S 3.9 of \cite{KL}.
  The proof of diagonalizability in the holomorphic case relies crucially on the finite
  dimensionality of $S_\k(N,\w')$.  In order to get the diagonalizability of $T_\n$
  on $L^2_0(N,\w')$ we can use the fact that each $T_\n$ preserves the $\Delta$-eigenspace
  $L^2_0(N,\w',\frac14+t^2)$, which is finite-dimensional. These subspaces exhaust the
  cuspidal spectrum by Theorem \ref{Delta}.
  See also Proposition \ref{diag} below.
\end{proof}

\noindent
  A {\bf Maass eigenform} is a cusp form $u$ which is a simultaneous eigenvector
  of the operators $T_\n$ for $\n\ge 1$, $(\n,N)=1$. \index{keywords}{Maass eigenform}
  We write $T_\n u= \lambda_\n(u) u$ for the Hecke 
eigenvalue.\index{notations}{lambdan@$\lambda_\n$, Hecke eigenvalue of $T_\n$}\index{notations}{am@$a_m(u)$, $m$-th Fourier coefficient of $u(z)$}
In this situation,
\[a_\n(u) = a_1(u)\lambda_\n(u)\]
whenever $\gcd(N,\n)=1$.  This is a consequence of the fact that
  for any cusp form $u$,
 \begin{equation}\label{amTn}
a_m(T_\n u)= \sum_{\ell|\gcd(\n,m)}\ol{\w'(\ell)}\,a_{\frac{\n m}{\ell^2}}(u),
\end{equation}
which is proven in the same way as for holomorphic cusp forms.


We now define a function which serves as the adelic counterpart to $T_\n$
  (see Lemma \ref{TnA} below).
Fix integers $N,\n\in \Z^+$ with $\gcd(\n,N)=1$, and let $\w$ be a Hecke
  character of conductor dividing $N$.  Define
  $\ff: G(\Af)\rightarrow\C$ as follows.  Let\index{notations}{M@$M_1(\n,N)$}
\[M_1(\n,N)=\{g=\smat abcd\in M_2(\Zhat)|\, \det g\in \n\Zhat^*\text{ and }
  c,(d-1)\in N\Zhat\}.\]
Let\index{notations}{M@$M_1(\n,N)_p$} $M_1(\n,N)_p$ be the local component of this
  set in $G(\Q_p)$.
    Note that if $p\nmid\n$, then $M_1(\n,N)_p=K_1(N)_p=K_1(N)\cap K_p$.
  The function $\ff$ is supported on $Z(\Af)M_1(\n,N)$ and given by
\begin{equation}\label{ff}\index{notations}{fn@$\ff$}
\ff(zm)=\frac{\ol{\w(z)}}{\meas(\ol{K_1(N)})}=\frac{\psi(N)}{\w(z)}.
\end{equation}
  It is clear that $\ff$ is well-defined
  and bi-$K_1(N)$-invariant.
  For any finite prime $p$, define a local function 
  $\ff_p$\index{notations}{fnp@$\ff_p$}
  on $G(\Q_p)$, supported on $Z(\Q_p)M_1(\n,N)_p$, by
\begin{equation}\label{ffp}
\ff_p(zm)=\frac{\ol{\w_p(z)}}{\meas(\ol{K_1(N)}_p)}.
\end{equation}
  Then $\ff(g)=\prod_p\ff_p(g_p)$.


We now recall the definition of the {\bf unramified principal series} 
  of\index{keywords}{principal series!of $\GL_2(\Q_p)$} $G(\Q_p)$.
Suppose $p\nmid N$, so $\w_p$ is an unramified
   unitary character of $\Q_p^*$.   For $\nu\in\C$, let
 \begin{equation}\label{chip}\index{notations}{chi@$\chi$!character on $B(\Q_p)$}
\chi(\smat ab0d)= \chi_1(a)\chi_2(d) \left|\frac ad\right|_p^{\nu}
\end{equation}
  be an unramified quasicharacter of the Borel subgroup $B(\Q_p)$.
  Here we take $\chi_1$ and $\chi_2$ to be finite order unramified
  characters of $\Q_p^*$ with
\[\chi_1(z)\chi_2(z)=\w_p(z).\]
  Let $V_\chi$ be the space of functions $\phi: G(\Q_p)\longrightarrow\C$ with
  the following properties:
\begin{enumerate}
\item[(i)]
   For all $\smat ab0d\in B(\Q_p)$ and all $g\in G(\Q_p)$,
\[\phi(\mat ab0d g)= \chi_1(a)\chi_2(d)\left|\frac ad\right|_p^{\nu+1/2}\phi(g).\]
\item[(ii)] There exists an open compact subgroup $J\subset G(\Q_p)$ such that
  $\phi(gk)=\phi(g)$ for all $k\in J$ and all $g\in G(\Q_p)$.
\end{enumerate}\index{notations}{pichi@$\pi_\chi$}
We let $\pi_\chi$ denote the representation of $G(\Q_p)$ on $V_\chi$ by right translation.
  It is unitary when $\chi$ is unitary, i.e. when $\nu\in i\R$.
  The space $V_\chi$ has a one-dimensional subspace of $K_p$-fixed vectors,
  spanned by the function
\begin{equation}\label{phi0}
\phi_0(\smat ab0d k)= \chi_1(a)\chi_2(d)\left|\frac ad\right|_p^{\nu+1/2}
\quad (k\in K_p).
\end{equation}

\begin{proposition}\label{unram}
The representation $(\pi_\chi,V_\chi)$ defined above is irreducible.
Every irreducible admissible unramified
  representation of $G(\Q_p)$ with central character $\w_p$ is either one-dimensional
  or of the form $\pi_\chi$ for some $\chi$ as above.
If $\pi_\chi$ is unitary then either:

$\cdot$  $\nu\in i\R$ (unitary principal series),\index{keywords}{principal series!of $\GL_2(\Q_p)$}
  or

$\cdot$ $0<|\Re \nu|<\tfrac12$ (complementary series).
\end{proposition}
\begin{proof} Refer, e.g., to Theorems 4.5.1,
  4.6.4, and 4.6.7 of \cite{Bu}.
\end{proof}

  The local component $\ff_p$ of $\ff$
 acts on the unramified vector $\phi_0$ in the following way.

\begin{proposition}\label{ind} Assume $p\nmid N$, and let
  $\n_p=\ord_p(\n)\ge 0$.  With $\ff_p$ as above,
  the function $\phi_0$ of \eqref{phi0}
  is an eigenvector of the local Hecke operator
  $\pi_\chi(\ff_p)$ with eigenvalue
 \[{p^{\n_p/2}}\lambda_{p^{\n_p}}(\chi_1,\chi_2,\nu),\]
 where \index{notations}{lambdanc@$\lambda_{\n}(\chi_1,\chi_2,s)$}
\begin{equation}\label{lambdap}
\lambda_{p^{\n_p}}(\chi_1,\chi_2,\nu)=\sum_{j=0}^{\n_p}\Bigl(\frac{p^{\n_p}}{p^{2j}}\Bigr)^\nu
  {\chi_1(p)^j\chi_2(p)^{\n_p-j}}.
\end{equation}
\end{proposition}

\begin{proof} The fact that $\phi_0$ is an eigenvector is due to Proposition
 \ref{basic}, together with the fact that the space of $K_p$-fixed vectors
  is one-dimensional.
  Thus the eigenvalue is equal to $\pi_\chi(\ff_p)\phi_0(1)$, which
  can be computed using the decomposition
\begin{equation}\label{Mp}
M_1(\n,N)_p
  =\bigcup\limits_{j=0}^{\n_p}
  \bigcup\limits_{a\in \Z/p^j\Z}\mat{p^j}{a}{}{p^{\n_p-j}}K_p
 \end{equation}
  (\cite{KL}, Lemma 13.4) as follows:
\[ \int_{\olG(\Q_p)}\ff_p(g)\phi_0(g)dg=\sum_{j=0}^{\n_p}p^j
  \left|\frac{p^j}{p^{\n_p-j}}\right|_p^{\nu+\frac12}\chi_1(p^j)\chi_2(p^{\n_p-j})
  =p^{\n_p/2}\lambda_{p^{\n_p}}(\chi_1,\chi_2,\nu).\qedhere\]
\end{proof}

\subsection{Adelic Maass forms}

Let $\w$ be the Hecke character attached to $\w'$ as in \eqref{chi'}. \index{notations}{w@$\w$, central character attached to $\w'$}
  Using \eqref{x'} and \eqref{w'}, we have
\[
 \w_\infty(-1) =\w_\infty(-1)\w'(-1) =\w_\infty(-1)\prod_{p|N}\w_p(-1) =\w(-1)=1.
\]
Since $\w_\infty$ is trivial on $\R^+$, this implies that for all $x\in \R^*$,
\begin{equation}\label{winf}
\w_\infty(x)=1.
\end{equation}

Let $L^2(\w)=L^2(\olG(\Q)\bs\olG(\A),\w)$ be the space of \index{notations}{L2@$L^2(\w)$}
  measurable $\C$-valued functions $\psi$ on $G(\A)$ (modulo functions that are
  $0$ a.e.)
  satisfying $\psi(z\g g)=\w(z)\psi(g)$ for all $\g\in G(\Q)$ and
  $z\in Z(\A)\cong \A^*$, and which are square integrable over $\olG(\Q)\bs\olG(\A)$.
  A function $\psi\in L^2(\w)$ is {\bf cuspidal} if its constant 
  term\index{keywords}{constant term} $\psi_N$ vanishes for a.e. 
  $g\in G(\A)$:\index{notations}{phiN@$\phi_N(g)= \int_{N(\Q)\bs N(\A)}\phi(ng)dn$}
\[\psi_N(g)=\int_{N(\Q)\bs N(\A)}\psi(ng)dn=0.\]
Let $\L(\w)\subset L^2(\w)$\index{notations}{L20@$\L(\w)$}
 denote the subspace of cuspidal functions.
 We let $R$\index{notations}{R@$R$, right regular representation} denote 
 the right regular representation of $G(\A)$ on $L^2(\w)$, and
  let $R_0$\index{notations}{R0@$R_0$} denote its restriction to $L^2_0(\w)$,
  which is easily seen to be an invariant subspace.

  Let $L^1(\ol{\w})$\index{notations}{L1@$L^1(\ol{\w})$}
  denote the space of measurable functions $f:G(\A)\longrightarrow\C$
  satisfying $f(zg)=\ol{\w(z)}f(g)$ for all $z\in Z(\A)$ and $g\in G(\A)$, and
  which are absolutely integrable over $\olG(\A)$.  Such a function defines
  an operator $R(f)$ on $L^2(\w)$ via
\[R(f)\index{notations}{Rf@$R(f)$}\phi(x)=\int_{\olG(\A)}f(y)\phi(xy)dy,\]
the integral converging absolutely.  Recall in fact that $\|R(f)\phi\|_{L^2}
  \le \|f\|_{L^1}\|\phi\|_{L^2}$ (see e.g. \cite{KL}, p. 140).
  The restriction of $R(f)$ to $L^2_0(\w)$ is denoted $R_0(f)$.
  For $f,h\in L^1(\ol{\w})$, the convolution\index{keywords}{convolution} 
\[f*h(x)=\index{notations}{14@$a*b(x)=\int_{\olG} a(y)b(y^{-1}x)dy$}\int
  _{\olG(\A)}f(y)h(y^{-1}x)dy\]
  also belongs to $L^1(\ol{\w})$, and by a straightforward computation we have
\[R(f*h)=R(f)R(h).\]

To each $u\in L^2(N,\w')$ we associate a function $\varphi_u$ on $G(\A)$
  using strong approximation \eqref{sa} by setting
\begin{equation}\label{phiu}\index{notations}{phi@$\varphi_u$, adelic Maass form}
\varphi_u(\g(g_\infty \times k))=u(g_\infty(i))
\end{equation}
  for $\g\in G(\Q)$, $g_\infty\in G(\R)^+$, and $k\in K_1(N)$.
  Using the modularity of $u$, it is easy to check that $\varphi_u$
  is well-defined.

\begin{proposition}\label{utophi}
  The map $u\mapsto \varphi_u$ defines surjective linear isometries
\[L^2(N,\w')\longrightarrow L^2(\w)^{K_\infty\times K_1(N)}\]
and
\[L^2_0(N,\w')\longrightarrow \L(\w)^{K_\infty\times K_1(N)},\]
  where the spaces on the right denote
 those functions satisfying
  $\varphi(gk)=\varphi(g)$ for all $k\in K_\infty\times K_1(N)$.
\end{proposition}

\begin{proof}
First we check that $\varphi_u(zg)=\w(z)\varphi_u(g)$ for all $z\in Z(\A)$.
  By strong approximation, we can assume that
  $g=g_\infty\times k\in G(\R)^+\times K_1(N)$.
  Write $z=z_\Q(z_\infty\times z_{\fin})$ for $z_\infty\in \R^+$ and
  $z_{\fin}\in \Zhat^*$.
  We have $z_{\fin}\equiv a\mod N\Zhat$ for some integer $a$ relatively prime to $N$.
  Then $\w(z)=\w'(a)$ as in \eqref{chi'}.
  Choose $b,c,d$ such that $\g=\smat abcd\in \Gamma_0(N)$.
  Then $\g z_{\fin}\in K_1(N)$, and
\[\varphi_u(zg) = \varphi_u(\g zg)=\varphi_u(\g g_\infty\times \g z_{\fin}k)=u(\g g_\infty(i))\]
  \[= \ol{\w'(d)}u(g_\infty(i))=\w'(a)\varphi_u(g)=\w(z)\varphi_u(g)\]
  as needed.

For the square integrability,
let $D_N$ be a fundamental domain in $\mathbf{H}$ for $\Gamma_0(N)\bs\mathbf{H}$.
  We identify $D_N$ with a subset of $G(\R)^+$ via $x+iy\leftrightarrow \smat yx01$.
  Then by Proposition 7.43 of \cite{KL},
\[\int_{\olG(\Q)\bs\olG(\A)}|\varphi_u(g)|^2dg = \int_{D_NK_\infty\times K_0(N)}
  |\varphi_u(g)|^2dg\]
\[ = \meas(K_0(N))\iint_{D_N}|u(x+iy)|^2\frac{dx\,dy}{y^2}=\|u\|^2.\]
This proves that the map $u\mapsto \varphi_u$ is an isometry of
 $L^2(N,\w')$
  into $L^2(\w)^{K_\infty\times K_1(N)}$, since it is clear from
  the definition \eqref{phiu} that $\varphi_u$ is invariant under
  $K_\infty\times K_1(N)$.
 For the surjectivity, we note that the inverse map is given by
\[u(z):= \varphi(g_\infty),\]
where $g_\infty\in G(\R)^+$ is any element satisfying $g_\infty(i)=z$.
   The function $u$ is well-defined since $\varphi$ is $K_\infty$-invariant.
The fact that $u(z)$ satisfies \eqref{hg} can be seen as follows.  For $\g=\smat abcd\in \Gamma_0(N)$,
  we can write $\g^{-1}=\smat {a_N}{}{}{a_N}k$ for $k\in K_1(N)$ (and $a_N$ as in \eqref{dN}).
    Thus
\[u(\g z)=\varphi(\g_\infty g_\infty)=\varphi(g_\infty\times \g^{-1}_{\fin})
=\varphi(g_\infty\times \smat{a_N}{}{}{a_N}k)\]
\[=\w(a_N)\varphi(g_\infty)=\w'(a)u(z)=\ol{\w'(d)}u(z).\]

  Lastly, for any $g\in G(\R)^+\times K_1(N)$, there exists
  $\delta\in G(\Q)^+$ determined by $g_{\fin}$ such that
\[{(\varphi_u)}_{\!\!{}_N}(g)=a_{0,\delta}(u,y)\]
 for $y=\Im g_\infty(i)$.
  This is proven just as in the holomorphic case, making the obvious adjustments.
  See \cite{KL}, pp. 200-201.
Therefore $u$ is cuspidal if and only if $\varphi_u$ is cuspidal.
\end{proof}

Next, we describe some properties of the correspondence $u\mapsto\varphi_u$.

\begin{lemma}\label{TnA}  The correspondence is equivariant for both $\Delta$ and
  the Hecke operators, in the following sense:  For all $u\in L^2(N,\w')$,
\begin{equation}\label{TnAeq}
R(\ff)\varphi_u= \int_{\olG(\Af)}\ff(g)R(g)\varphi_u\,dg= 
\sqrt{\n}\, \varphi_{T_\n u}
\end{equation}
 for $\ff$ defined in \eqref{ff},
  and if $u$ is smooth,
\begin{equation}\label{DnA}
R(\Delta)\varphi_u=\varphi_{\Delta u}.
\end{equation}
\end{lemma}

\begin{proof}
In both cases it suffices by strong approximation \eqref{sa} to show that the two functions
  agree on elements of the form $\smat yx01\in G(\R^+)$ when $u$ is smooth.
  In \eqref{DnA}, the symbol $\Delta$ is used in two different ways.
  On the right, $\Delta$ is the Laplace operator \eqref{Lap}, and on the left
  it is the Casimir element whose effect on $C^\infty(G(\R))$ is given by \eqref{Cas}.
  But because $\phi_u$ is $K_\infty$-invariant, $\frac{\partial}{\partial\theta}\phi_u=0$,
  so we can drop the second term of \eqref{Cas} to conclude
\begin{align}\label{LapCas}
\notag   R(\Delta) \varphi_u(\smat{y}x01)
&=-y^2\left(\frac{\partial^2}{\partial x^2}
  +\frac{\partial^2}{\partial y^2}\right)
 \varphi_u(\smat{y}x01)\\
&=\Delta u(x+iy) = \varphi_{\Delta u}(\smat yx01),
\end{align}
as needed.

 The proof of \eqref{TnAeq} is the same as that of the
  version for holomorphic cusp forms given in \cite{KL}, Proposition 13.6.
\end{proof}

By a theorem of Gelfand and Piatetski-Shapiro, the right regular representation
  $R_0$ of $G(\A)$ on $\L(\w)$ decomposes
  into a direct sum of irreducible unitary representations $\pi$.  Each
  such cuspidal representation $\pi$ is a restricted tensor product of
  local representations: 
$\pi= \pi_\infty\otimes \pi_{\fin}=\pi_\infty\otimes\bigotimes_p\pi_p$
  (cf. \cite{Bu}, \S 3.4).

\begin{proposition}\label{specprop}
  We have the following decomposition:
\begin{equation}\label{spec}
\L(\w)^{K_\infty\times K_1(N)}=\bigoplus_{\pi} \C v_\infty\otimes \pi_{\fin}^{K_1(N)},
\end{equation}
where $\pi$ runs through the irreducible cuspidal representations
  with infinity type of the form $\pi_\infty=\pi(\e,\e,s,-s)$, where either $s\in i\R$ or
  $-\frac12<s<\frac12$,
 and $v_\infty$ is a nonzero vector of weight $0$ (unique up to multiples).
Equivalently, $\pi$ runs through the constituents of $L^2_0(\w)$
  which contain a nonzero $K_\infty\times K_1(N)$-fixed vector.
\end{proposition}

\noindent{\em Remarks:} (1) Selberg's conjecture\index{keywords}{Selberg's conjecture}
  asserts that the
  complementary series of Proposition \ref{LK} do not actually show up here,
  i.e. that $s\in i\R$. 
   Likewise, according to the Ramanujan conjecture,\index{keywords}{Ramanujan conjecture}
 the unramified
  local factors of $\pi_{\fin}$ are unitary principal series
  rather than complementary series (cf. Proposition \ref{unram}).\vskip .2cm

\noindent (2) Caution about notation:  
  In \S11 of \cite{KL}, when discussing $\pi(\e_1,\e_2,s_1,s_2)$
  we used the notation $s=s_1-s_2$.
  In the present document, we take $s=2(s_1-s_2)$.

\begin{proof}  For any irreducible cuspidal representation $\pi$, the orthogonal
  projection map $L^2_0(\w) \rightarrow \pi$ commutes with the right regular action $R(g)$.
  As an easy consequence, we have
\[L^2_0(\w)^{K_\infty\times K_1(N)} = \bigoplus_\pi \pi^{K_\infty\times K_1(N)}.\]
  The Casimir element $\Delta$ acts on the smooth vectors of an
  irreducible finite dimensional representation
  of $G(\R)$ by a scalar which is $\le 0$ (cf. Theorem 11.15 and Proposition 11.22
  of \cite{KL}).  We conclude from the fact that $R_0(\Delta)$ is positive definite
  that $\pi_\infty$ is infinite dimensional.
  The proposition now follows immediately from Proposition \ref{LK}.
  Note that $\w_\infty$ is the trivial character,
  so $s_1+s_2=0$ in the notation of that proposition,
  and here we have set $s_1=-s_2=s$.
\end{proof}

  An {\bf adelic Hecke operator} of weight $0$ is a function on $G(\A)$
   of the form
\begin{equation}\label{Hop}
f=f_\infty\times \ff\in L^1(G(\A),\ol{\w})
\end{equation}
 with $f_\infty\in C_c(G^+//K_\infty)$ and $\ff$ as in \eqref{ff}.
We now show that for such $f$, the operator $R_0(f)$ on $L^2_0(\w)$
  is diagonalizable.
By Lemma \ref{basic}, it suffices to consider its restriction to
  $L^2_0(\w)^{K_\infty\times K_1(N)}$:

\begin{proposition}\label{diag}
For each cuspidal $\pi=\pi(\e,\e,s,-s)\otimes\pi_{\fin}$
  contributing to \eqref{spec}, choose an orthogonal basis
  $\mathcal{F}_\pi$ for the finite dimensional subspace $\C v_\infty\otimes \pi_{\fin}^{K_1(N)}$.
  Let $\mathcal{F}_\A=\bigcup_\pi\mathcal{F}_\pi$ be the resulting orthogonal basis
  for $\L(\w)^{K_\infty\times K_1(N)}$.\index{notations}{FA@$\mathcal{F}_\A$}
  Then each $\varphi\in \mathcal{F}_\A$
  is an eigenfunction of $R_0(f)$ with eigenvalue
  of the form
\begin{equation}\label{Rfeigen}
h(t)\sqrt{\n}\,\lambda_\n(\varphi),
\end{equation}
 where $t=-is$ is the spectral parameter of $\pi_\infty$,\index{notations}{ht@$h(t)$, Selberg transform}
  $h$ is the Selberg transform of $f_\infty$,
and $\lambda_\n(\varphi)=\prod_{p|\n}\lambda_{p^{\n_p}}$ is 
  determined\index{notations}{lambdanv@$\lambda_\n(\varphi)$}
  from $\pi_p$ for $p|\n$ by \eqref{lambdap}.
  Furthermore, if $u\in L^2_0(N,\w')$ is the
  function on $\mathbf{H}$ corresponding to $\varphi$,
  then $u$ is a Maass eigenform with
  $\Delta$-eigenvalue $\frac14+t^2$, and
 Hecke eigenvalue $\lambda_\n(u)=\lambda_\n(\varphi)$.\index{notations}{lambdanu@$\lambda_\n(u)$}
\end{proposition}

\begin{proof}
Let $\pi$ be one of the given cuspidal representations.
When $p|\n$, $\pi_p^{K_1(N)_p}=\pi_p^{K_p}$ is nonzero, so $\pi_p$ is an
  unramified unitary principal series representation.  Write $\pi_p^{K_p}=\C v_p$.
  By Proposition \ref{ind}, we have
\[\pi_p(\ff_p)v_p = p^{\n_p/2}\lambda_{p^{\n_p}}v_p.\]
At the archimedean place, Proposition \ref{finf} gives
\[\pi_\infty(f_\infty)v_\infty=h(t)v_\infty.\]

 Consider any $v\in \C v_\infty\otimes \pi_{\fin}^{K_1(N)}$.  For any object
  defined as a product of local objects, let us for the moment
  use a prime $'$ to denote the product over just the finite primes $p\nmid \n$.  Then
  \[v=v_\infty\otimes v'\otimes \bigotimes_{p|\n}v_p\]
  for $v_p$ as above and
\[v'\in (\pi')^{K_1(N)'}= \bigotimes_{p\nmid \n}\pi_p^{K_1(N)_p}.\]
By the definition \eqref{ff} of $\ff$, we see that
\[\pi'({\ff}')v' = {\meas(\ol{K_1(N)'})}^{-1}\int_{\ol{K_1(N)'}}
  \pi'(k)v'dk = v'.\]
Letting $\varphi\in L^2_0(\w)$ denote the function corresponding to $v$,
  it follows (e.g. by Proposition 13.17 of \cite{KL}) that
\[R(f)\varphi=\pi_\infty(f_\infty)v_\infty \otimes \pi'({\ff}')v'\otimes
   \bigotimes_{p|\n} \pi_p(\ff_p)v_p=\sqrt{\n}\lambda_{\n}(\varphi)h(t)\varphi.\]

Now let $u$ be the element of $L^2_0(N,\w')$ attached to $v$.
  We need to show that $\Delta u=(\frac14+t^2)u$.
  For any $X\in \lie_\R$, we have
\[\pi(X)v\eqdef \ddt \pi(\exp(tX)\times 1_{\fin})v
  =\ddt\pi_\infty(\exp(tX))v_\infty\otimes v_{\fin}.\]
Therefore
\[\pi(\Delta)v=\pi_\infty(\Delta)v_\infty\otimes v_{\fin}
  =(\tfrac14+t^2)v\]
 by \eqref{nu}.  Equivalently,
   $R(\Delta)\varphi=(\frac14+t^2)\varphi$, so by \eqref{DnA},
  $\Delta u=(\frac14+t^2)u$.
Lastly, by Lemma \ref{TnA} we also have $\lambda_\n(u)=\lambda_\n(\varphi)$.
\end{proof}

With $\mathcal{F}_\A$ as in the above proposition, we let
\begin{equation}\label{Fbasis} \index{notations}{F@$\mathcal{F}$}
\mathcal{F}\subset L^2_0(N,\w')
\end{equation}
  be the corresponding orthogonal basis.  It consists of
  Maass eigenforms as shown above.  Using Proposition \ref{diag}, we can
  arrange further for each $u\in\mathcal{F}$ to be
  an eigenvector of $T_{-1}$.


\pagebreak
\section{Eisenstein series}\label{5}

The continuous part of $L^2(\w)$
  is explicitly describable in terms of Eisenstein series
  (see Sec. \ref{Spec}).
  Because we are interested in automorphic
  forms of weight $\k=0$ and level $N$, we will concentrate on
  $K_\infty\times K_1(N)$-invariant Eisenstein series.

\subsection{Induced representations of $G(\A)$}\label{Ind}

We begin by constructing certain principal series representations of $G(\A)$.  
  These\index{keywords}{principal series!of $\GL_2(\R)$}
  are representations induced from characters of the Borel subgroup
  $B(\A)=M(\A)N(\A)$.
  Any character of $B(\A)$ is trivial on the commutator subgroup $N(\A)$, and
  hence is really defined on the diagonal group $M(\A)\cong \A^*\times \A^*$.
  We are only interested in $G(\Q)$-invariant functions, so we want a
  character of $B(\Q)\bs B(\A)$, which by the above is nothing more than a pair
  of Hecke characters, say $\chi_1\otimes|\cdot|^{s_1}$ and $\chi_2\otimes|\cdot|^{s_2}$, where
  $\chi_1,\chi_2$ have finite order.  Furthermore, we need the product of these two
  characters to equal our fixed central character $\w$, which has finite order.
  This means in particular that $s_2=-s_1$.

Thus for finite order Hecke characters $\chi_1$ and $\chi_2$
  with $\chi_1\chi_2=\w$, and $s\in\C$, we consider the character of $B(\A)$ defined by
\[\mat ab0d\mapsto \chi_1(a)\chi_2(d)\left|\frac ad\right|^s.\] 
 We let $(\pi_s, H(\chi_1,\chi_2,s))$ \index{notations}{pis@$\pi_s$} denote the representation of $G(\A)$ unitarily induced from
  this character. \index{notations}{Hchis@$H(\chi_1,\chi_2,s)$}
This Hilbert space has a dense subspace spanned by the continuous
  functions $\phi:G(\A)\longrightarrow\C$ satisfying
\begin{equation}\label{ind2}
\phi(\mat ab0d g)=\chi_1(a)\chi_2(d)\left|\frac ad\right|^{s+1/2}\phi(g).
\end{equation}
  The inner product is defined by
\[\sg{\phi,\psi}=\int_{K}\phi(k)\ol{\psi(k)}dk.\]
  This is nondegenerate since, by the decomposition $G=BK$,
   any $\phi\in H(\chi_1,\chi_2,s)$ is determined by its restriction to $K$.
  The right regular representation $\pi_s=\pi_s(\chi_1,\chi_2)$
  of $G(\A)$ on $H(\chi_1,\chi_2,s)$ is unitary
  if $s\in i\R$ (for the idea, see e.g. \cite{KL} Proposition 11.8).
As explained in \S4B-4C of \cite{GJ}, we have
\begin{equation}\label{Hisom}
H(\chi_1,\chi_2,s)\cong H(\chi_2,\chi_1,-s)\end{equation}
  as representations of $G(\A)$.

  Restriction to $K$ identifies $H(\chi_1,\chi_2,s)$ with the subspace
  of $L^2(K)$ consisting of functions satisfying
\[ f(\smat ab0d k)=\chi_1(a)\chi_2(d)f(k)\qquad (\smat ab0d\in B(\A)\cap K).\]
In this way, the spaces $H(\chi_1,\chi_2,s)$ form a trivial vector bundle over the above
  subspace of $L^2(K)$.
Given $\phi\in H(\chi_1,\chi_2,0)$ and $s\in\C$, we define 
  $\phi_s\in H(\chi_1,\chi_2,s)$ by
\[\phi_s(\mat ab0d k)=\chi_1(a)\chi_2(d)\left|\frac ad\right|^{s+1/2}
  \phi(k).\]
  Equivalently, if  $H:G(\A)\rightarrow \R^+$ is the 
  {\bf height function}\index{keywords}{height function}
  defined by
\begin{equation}\label{H}\index{notations}{H@$H$, height function}
H(g)=H(\mat a{}{}d \mat{1}{x}{}{1} k)=\log\left|\frac ad\right|,
\end{equation}
  then \index{notations}{phis@$\phi_s(g) = e^{s H(g)}\phi(g)$}
\[\phi_s(g) = e^{s H(g)}\phi(g).\]
  The map from $H(\chi_1,\chi_2,0)$ to $H(\chi_1,\chi_2,s)$ taking
  $\phi\mapsto \phi_s$ is an isomorphism of Hilbert spaces.
  We set \index{notations}{Hchi0@$H(\chi_1,\chi_2)=H(\chi_1,\chi_2,0)$}
\[H(\chi_1,\chi_2)\eqdef H(\chi_1,\chi_2,0).\]


\begin{lemma}\label{phiinf}
 Suppose $\phi\in H(\chi_1,\chi_2,s)$ is a right
  $K_\infty\times K_1(N)$-invariant function, i.e. $\phi\in H(\chi_1,\chi_2,s)
  ^{K_\infty\times K_1(N)}$.
Define $\phi_\infty:G(\R)^+\longrightarrow \C$ by
\begin{equation}\label{phiinfeq}
\phi_\infty(g_\infty)=\Im(z)^{s+1/2}\qquad(z=g_\infty(i)\in\mathbf{H}),
\end{equation}
and for $p\nmid N$, define $\phi_p:G(\Q_p)\longrightarrow\C$ by
\[\phi_p(\mat ab0dk)=\chi_{1p}(a)\chi_{2p}(d)\left|\frac ad\right|_p^{s+1/2}.\]
Also set $\phi'=\prod_{p\nmid N}\phi_p$, and
define $\phi_N:\prod_{p|N}G(\Q_p)\longrightarrow\C$ by
\[\phi_N(g_N)=\phi(1_\infty\times g_N\times 1').\]
 (When $N=1$, the above is just the constant $\phi(1)$.) 
  Then $\phi$ is factorizable as
\[\phi(g_\infty\times g_N\times g')=\phi_\infty(g_\infty)\phi_N(g_N)\phi'(g').\]
\end{lemma}

\begin{proof}
Write
\[g_\infty=\mat u00u \mat1x01\mat {y^{1/2}}{}{}{y^{-1/2}} k_\infty.\]
  Then since $\phi$ is $K_\infty$-invariant,
\[\phi(g_\infty\times g_{\fin}) =
\w_\infty(u)\chi_{1\infty}(y^{1/2})\chi_{2\infty}(y^{-1/2})
  y^{s+\frac12}\phi(1_\infty\times g_{\fin})\]
\[  =y^{s+\frac12}\phi(1_\infty\times g_{\fin})
=\phi_\infty(g_\infty)\phi(1_\infty\times g_{\fin})\]
since $\w_\infty$ is trivial by \eqref{winf},
   and $\chi_{1\infty},\chi_{2\infty}$ are trivial on $\R^+$.
Now according to the Iwasawa decomposition, write $g'=b'k'$ for 
  $k'\in K'=\prod_{p\nmid N}K_p$ and $b\in \prod_{p\nmid N}B(\Q_p)$.  Then
  by the same argument, using the fact that $\phi$ is right invariant under
  $K'$, we have
\[\phi(1_\infty\times g_N\times g')= \phi'(g')\phi(1_\infty\times g_N\times 1')
  =\phi'(g')\phi_N(g_N).\]
  The lemma follows.
\end{proof}

\begin{proposition}\label{phieigen} If $f=f_\infty\times \ff$ with
  $f_\infty\in C_c(G^+//K_\infty)$, the operator $\pi_s(f)$
  acts by the scalar\index{notations}{lambdanc@$\lambda_{\n}(\chi_1,\chi_2,s)$}
\[h(t)\sqrt{\n}\,\lambda_{\n}(\chi_1,\chi_2,it)\]
  on $H(\chi_1,\chi_2,s)^{K_\infty\times K_1(N)}$, and vanishes on the
  orthogonal complement of this finite dimensional subspace.
Here, $t=-is$, $h$ is the Selberg transform of $f_\infty$, and
\begin{equation}\label{lam}
  \lambda_\n(\chi_1,\chi_2,s)= \prod_{p|\n}
   \lambda_{p^{\n_p}}(\chi_{1p},\chi_{2p},s)=\n^s \sum_{d|\n}
  \frac{\ol{\chi_1(d_N) \chi_2((\tfrac \n d)_N)}}{d^{2s}},\end{equation}
for $\lambda_{p^{\n_p}}(\chi_{1p},\chi_{2p},s)$ as in Proposition \ref{ind}.
\end{proposition}

\begin{proof}
  The dimension of $H(\chi_1,\chi_2,s)^{K_\infty \times K_1(N)}$
  is computed in Section \ref{Orth} below.
By Lemma \ref{basic}, $\pi_s(f)$ vanishes on its orthogonal complement.
  The second equality in \eqref{lam} comes from the fact that
  for any finite order Hecke character $\chi$ of conductor dividing $N$,
  and any positive integer $d|\n$,
\[1=\chi(d)=\prod_{p|\n}\chi_p(d)\prod_{p|N}\chi_p(d),\]
so 
\[\prod_{p|\n}\chi_p(p^{d_p})=\prod_{p|\n}\chi_p(d)=\prod_{p|N}\ol{\chi_p(d)}
  =\ol{\chi(d_N)}.\]

  For $\phi\in H(\chi_1,\chi_2,s)^{K_\infty\times K_1(N)}$,
  write
\[\phi=\phi_\infty\otimes \phi_N\otimes \bigotimes_{p\nmid N}\phi_p\]
  as in the above lemma.  Note that $\phi_\infty$ is a weight $0$ vector
  in an induced representation $\pi_\infty=\pi(\e_1,\e_2,s,-s)$ with spectral parameter
  $t=-is$.  Likewise if $p\nmid N$, then $\phi_p$ is the $K_p$-fixed vector
  in the unramified representation $\pi_p$ of $G(\Q_p)$ induced from the character
  $(\chi_{1p},\chi_{2p})$ of $B(\Q_p)$ (as in \eqref{chip}, taking $\nu=s$).  Setting
  $f_N=\prod_{p|N}f_p$, we have
\[\pi_s(f)\phi=\pi_\infty(f_\infty)\phi_\infty\otimes R(f_N)\phi_N\otimes
   \bigotimes_{p\nmid N}\pi_p(\ff_p)\phi_p.\]
  Here
  \[\pi_\infty(f_\infty)\phi_\infty=h(t)\phi_\infty\]
  by Proposition \ref{finf}, and if $p\nmid N$,
\[\pi_p(\ff_p)\phi_p = p^{\n_p/2}\lambda_{p^{\n_p}}(\chi_{1p},\chi_{2p},s)\phi_p\]
  by Proposition \ref{ind}, while by \eqref{ffp}
\[R(f_N)\phi_N(x)= \frac1{\meas(\ol{K_1(N)})}\int_{\prod_{p|N}\ol{K_1(N)}_p}\phi_N(xk)dk
  =\phi_N(x).\]
  Therefore $\pi_s(f)\phi=h(t)\sqrt{\n}\,\lambda_{\n}(\chi_1,\chi_2,s)\phi$ as claimed.
\end{proof}

\subsection{Definition of Eisenstein series} \index{keywords}{Eisenstein series}

The elements of $H(\chi_1,\chi_2,s)$ are $B(\Q)$-invariant by construction.
   We use them to define
  $G(\Q)$-invariant functions (automorphic forms) on $G(\A)$ by averaging: \index{notations}{E@$E(\phi,s,g)= \sum_{\sss \g\in B(\Q)\bs G(\Q)}\phi_s(\g g)$}\index{notations}{E@$E(\phi_s,g)= E(\phi,s,g)$}
\[E(\phi,s,g)=E(\phi_s,g)= \sum_{\g\in B(\Q)\bs G(\Q)}\phi_s(\g g)
  \qquad (\phi\in H(\chi_1,\chi_2)).\]
This sum converges absolutely when $\Re(s)>1/2$ (see Proposition \ref{convergence} below).
   For now we will assume that $s$ belongs to this domain.
  However, the Eisenstein series $E$ has a meromorphic continuation to
  the complex plane.  We will prove this well-known result below in the case where $\phi$ is
  $K_\infty\times K_1(N)$-invariant by writing down the Fourier expansion of $E$,
  which is seen to be meromorphic on $\C$
  (cf. Theorem \ref{merom}).

  For a fixed level $N$,
  we will only be interested in the case where $\phi$ is a nonzero (right)
  $K_\infty\times K_1(N)$-invariant vector.
  Such $\phi$ exists if and only if the product 
  $\c_{\chi_1}\c_{\chi_2}$ of the conductors divides $N$
  (see Corollary \ref{nonzero} below).
  In this case, $E(\phi,s,g)$ is left $G(\Q)$-invariant and right $K_1(N)$-invariant.
  By strong approximation we have $G(\A)=G(\Q)(G(\R)^+\times K_1(N))$.
  Therefore it suffices to investigate $E(\phi,s,g_\infty\times 1_{\fin})$, where
  $g_\infty\in G(\R)^+=Z(\R)B(\R)K_\infty$.  Furthermore, because
\begin{enumerate}
\item $E(\phi,s,g_\infty)$ is right $K_\infty$-invariant
\item the central character $\w=\chi_1\chi_2$ is trivial on $Z(\R)$
  (see \eqref{winf}),
\end{enumerate}
  the value of $E(\phi,s,g_\infty)$ depends only on $z=g_\infty(i)\in \mathbf{H}$.
  So for $z=x+iy\in\mathbf{H}$, we define \index{notations}{E@$E_\phi(s,x+iy)=E(\phi,s, \smat yx01\times 1_{\fin})$}
\[E_\phi(s,z) = E(\phi,s, g_\infty\times 1_{\fin}),\]
  for any $g_\infty \in Z(\R)\smat1x01\smat{y^{1/2}}{}{}{y^{-1/2}}K_\infty$.
  Thus
\[E_\phi(s,z)=\sum_{\g\in B(\Q)\bs G(\Q)}\phi_s(\g g_\infty)
=\sum_{\g\in B(\Q)\bs G(\Q)}e^{s H(\g_\infty g_\infty\times \g_{\fin})}
 \phi(\g_\infty g_\infty\times \g_{\fin}),\]
where $H$ is the height function defined in \eqref{H}.

\begin{lemma}\label{reps} A set of representatives for $B(\Q)\bs G(\Q)$ is given by
  $\pm N(\Z)\bs \SL_2(\Z)$.  The latter set is in one-to-one correspondence
  with ordered pairs $(c,d)$ of relatively prime integers
 with $c> 0$, together with $(0,1)$, via
  $\pm N(\Z)\smat abcd\leftrightarrow (c,d)$.
\end{lemma}

\begin{proof}
The first assertion follows from the decomposition
\begin{equation}\label{BG}G(\Q)=B(\Q)\SL_2(\Z)\end{equation}
 since $B\bs B\Gamma\cong (B\cap\Gamma)\bs \Gamma$.
  The decomposition \eqref{BG} can easily be proven directly as follows.
Let $g = \smat abcd \in G(\Q)$.
Write $(c\,\,\,\,d)=t(c'\,\,\,\, d')$, where $t \in \Q$, $c', d' \in \Z$
   and $\gcd(c',d') = 1$.
There exist integers $x$ and $y$ such that $c' x - d' y=1$.
Then $\smat {d'}{x}{-c'}{-y}\in\SL_2(\Z)$
and $\smat abcd \smat {d'}{x}{-c'}{-y} \in B(\Q)$.

For the representatives, view $\Z\times\Z$ as a set of row vectors, and
  consider the right action of $\SL_2(\Z)$ on this set.
   The stabilizer of $(0\,\,\,1)$ is $N(\Z)$.  Therefore $N(\Z)\bs \SL_2(\Z)$
  is in one-to-one correspondence with the orbit of $(0\,\,\,1)$. It is easy to
  see that this orbit is
  the set of ordered pairs of relatively prime integers:
\[(0\,\,\,1)\mat abcd = (c\,\,\,d).\]
  Considering instead the set $\pm\bs (\Z\times \Z)$, the stabilizer of $\pm(0\,\,\,1)$
  is $\pm N(\Z)$ and we can take $c\ge 0$, obtaining the given set of pairs $(c,d)$.
\end{proof}

By the above and Lemma \ref{phiinf}, we have
\[E_\phi(s,z)
 =\sum_{\g\in \pm N(\Z)\bs\SL_2(\Z)}e^{s H(\g_\infty g_\infty\times 1_{\fin})}\phi_\infty(\g g_\infty)
  \phi_{\fin}(\g).\]
This holds since $\g_{\fin}\in \SL_2(\Z)\subset K_{\fin}$ and
 the height function is $K_{\fin}$-invariant.  Now using
\[e^{H(\g_\infty g_\infty\times 1_{\fin})}=
  |\Im(\g z)|=\frac{y}{|cz+d|^2} \qquad(\g=\smat abcd),\]
 together with \eqref{phiinfeq} (with $s=0$), we have
\begin{equation}\label{Eisen}
E_\phi(s,z) = y^{1/2+s}\phi_{\fin}(0,1)+
  \sum_{c> 0}\sum_{d\in\Z\atop{\gcd(c,d)=1}}
   \frac{y^{1/2+s}}{|cz+d|^{1+2s}}\phi_{\fin}(c,d).
\end{equation}
  Here we have written $\phi_{\fin}(c,d)$ to denote $\phi_{\fin}(\g)$.
  This notation is apt because $\phi_{\fin}$ is left $B(\Q)$-invariant, so that
 for $\g=\smat abcd\in\SL_2(\Z)$, $\phi_{\fin}(\g)$ depends only on
  $(c,d)$ by Lemma \ref{reps}.

\subsection{The finite part of $\phi$}

We eventually need to compute the Fourier coefficients of $E_\phi(s,z)$ for
  $\phi$ in an orthonormal basis for $H(\chi_1,\chi_2)^{K_\infty\times K_1(N)}$.
  Since we can take $\phi_\infty$ to be the function determined in
  Lemma \ref{phiinf}, in order to find such a basis we just need
  to write down the possibilities for $\phi_{\fin}$.

\begin{lemma}\label{K1N} $K_{\fin}=\SL_2(\Z)K_1(N)$.
\end{lemma}

\begin{proof}
Let $S=\SL_2(\Zhat)\cap K_1(N)$ denote the set of determinant $1$
  elements of $K_1(N)$.
  Note that $S$ is an open subgroup of $\SL_2(\Zhat)$.  Hence
 \[\SL_2(\Zhat)=\SL_2(\Z)\cdot S\]
since $\SL_2(\Z)$ is dense in $\SL_2(\Zhat)$ (see e.g. Proposition 6.6 of \cite{KL}).
  From this we obtain the following decomposition:
\[K_{\fin}=\SL_2(\Zhat)\mat{\Zhat^*}{}{}{1}=\SL_2(\Z)\left[S\mat{\Zhat^*}{}{}1\right].\]
  The lemma follows since the expression in the brackets is exactly $K_1(N)$.
\end{proof}

\begin{lemma}\label{phifin}
For a nonzero vector $\phi\in H(\chi_1,\chi_2)^{K_\infty\times K_1(N)}$,
  with $\phi_\infty$ given as in Lemma \ref{phiinf},
  the finite part $\phi_{\fin}$ is determined by its restriction to $\SL_2(\Z)$.
  For $\g=\smat abcd\in\SL_2(\Z)$, $\phi_{\fin}(\g)$ depends only on $(c,d) \mod N$.
\end{lemma}

\begin{proof}
By strong approximation\index{keywords}{strong approximation!for $B(\A)$}
   for $B(\A)$ (\cite{KL}, Prop. 6.5),
\[G(\A)=B(\A)K = B(\Q)(B(\R)^+K_\infty\times K_{\fin}).\]
  Hence by $B(\Q)$-invariance, $\phi$ is determined by its restriction to
  $G(\R)^+\times K_{\fin}$.  We assume that $\phi_\infty$ is the function given in
  Lemma \ref{phiinf}.
  By $K_1(N)$-invariance and Lemma \ref{K1N}, $\phi_{\fin}$ is determined by its values on $\SL_2(\Z)$.
  Lastly, because $K(N)\subset K_1(N)$, $\phi_{\fin}$ determines a function on
  $K_{\fin}/K(N)\cong G(\Z/N\Z)$.   Therefore $\phi_{\fin}(\g)$ depends only on the
  entries modulo $N$.
\end{proof}

\begin{proposition}\label{convergence}
For any $\phi\in H(\chi_1,\chi_2)^{K_\infty\times K_1(N)}$ and all
  $z\in \mathbf{H}$, the Eisenstein series
$E_\phi(s,z)$ is absolutely convergent on $\Re(s)>1/2$.  For any $\delta>0$ and
  compact set $C\subset\mathbf{H}$, the convergence is uniform on the set
  $\Re(s)\ge \frac12+\delta$ and $z\in C$.
\end{proposition}
\noindent{\em Remark:} A similar proof applies to the case of arbitrary $\phi\in H(\chi_1,\chi_2)$
  (cf. \cite{Bu}, Proposition 3.7.2).

\begin{proof}
  It follows from Lemma \ref{phifin} that $\phi_{\fin}$ is a bounded function.
  So up to a constant multiple, $E_\phi(s,z)$ is majorized by the classical series
\[E(\Re(s),z)=\sum_{(c,d)\neq (0,0)} \frac{y^{\Re(s)+1/2}}{|cz+d|^{2\Re(s)+1}},\]
  which is easily seen to converge when $\Re(s)>1/2$.
\end{proof}

%

As indicated in the proof of Lemma \ref{phifin}, $\phi_{\fin}$ can be viewed as a function
  on $G(\Z/N\Z)$. \index{notations}{D@$D(\chi_1,\chi_2,N)$}
Let $D(\chi_1,\chi_2,N)$ denote the space of all functions $\phi$ on $G(\Z/N\Z)$ satisfying
\[  \phi(\mat{a}{b}{}{d} k \mat{a'}{b'}{}{1}) = \chi_{1}(a) \chi_{2}(d)  \phi(k)\]
for all $k\in G(\Z/N\Z)$ and $a,d,a'\in (\Z/N\Z)^*$ and $b,b'\in\Z/N\Z$,
or equivalently,
\[  \phi(\mat{a}{b}{}{d} k \mat{a'}{b'}{}{d'}) = \chi_{1}(a) \chi_{2}(d) \w(d') \phi(k).\]
Here $d'\in (\Z/N\Z)^*$, and we view $\chi_1$ and $\chi_2$ as characters of $(\Z/N\Z)^*$ as in \eqref{chi'},
i.e. $\chi_j(a)=\chi_j(a_N)$.
We make $D(\chi_1,\chi_2,N)$ into a finite dimensional Hilbert space by
  defining
\begin{equation}\label{Dip}
 \sg{\phi_1, \phi_2} = |G(\Z/N\Z)|^{-1} \sum_{k \in G(\Z/N\Z)}
   \phi_1(k) \ol{\phi_2(k)}.
\end{equation}

Notice that if $\phi,\psi\in H(\chi_1,\chi_2)^{K_\infty\times K_1(N)}$, then
  with notation as in Lemma \ref{phiinf}, it is easy to see that
\[
\sg{\phi,\psi}= 
  \int_{K_\infty}\phi_\infty(k)\ol{\psi_\infty(k)}dk \int_{K_N}
  \phi_N(k)\ol{\psi_N(k)}dk\int_{K'}\phi'(k)\ol{\psi'(k)}dk\]
\begin{equation}\label{sgN} =\int_{K_N}\phi_N(k)\ol{\psi_N(k)}dk,
\end{equation}
where $K_N=\prod_{p|N}K_p$.
  Letting $K_N(N)=\{k\in K_N|\, k\equiv 1\mod N\}$,
  the $K_1(N)$-invariance then gives
\begin{align}\label{ipN}
\sg{\phi,\psi}&=[K_N:K_N(N)]^{-1}\sum_{k\in K_N/K_N(N)}\phi_N(k)\ol{\psi_N(k)}\\
\notag  &=|G(\Z/N\Z)|^{-1}\sum_{k\in G(\Z/N\Z)}\phi_N(k)\ol{\psi_N(k)}.
\end{align}
(When $N=1$, this is just $\phi(1)\ol{\psi(1)}$.)
  In view of \eqref{Dip}, this proves the following.

\begin{lemma}
The identification of $\phi_{\fin}$ with a function on $G(\Z/N\Z)$ induces
an isometry of $H(\chi_1,\chi_2)^{K_\infty\times K_1(N)}$ with $D(\chi_1,\chi_2,N)$.
\end{lemma}

The space $D(\chi_1,\chi_2,N)$ can be analyzed locally since
\[ G(\Z/N\Z) \cong \prod_{p|N} G(\Z_p/N\Z_p).\]
We let $D_p(\chi_1,\chi_2,N)$ denote the space of functions on $G(\Z_p/N\Z_p)$
  satisfying\footnote{Here and henceforth, for $a\in\Q_p^*$ we evaluate $\chi_1(a)$
  by embedding $a$ as an idele which is 1 outside $p$.
  This is equivalent to $\chi_{1p}(a)$
  but we sometimes wish to avoid the extra subscript when the context is
  completely local.}
\begin{equation}\label{Dp2}
   \phi(\mat{a}{b}{}{d} k \mat{a'}{b'}{}{d'})
  = \chi_{1}(a) \chi_{2}(d) \chi_{1}(d') \chi_{2}(d') \phi(k)
\end{equation}
for $k\in G(\Z_p/N\Z_p)$, $a,d,a',d'\in(\Z_p/N\Z_p)^*$, and
 $b,b'\in \Z_p/N\Z_p$.
This is a Hilbert space with inner product given by the local analog of \eqref{Dip}:
\begin{equation}\label{Diplocal}
 \sg{\phi_1, \phi_2} = [K_p:K_p(p^{N_p})]^{-1} \sum_{k \in G(\Z_p/N\Z_p)}
   \phi_1(k) \ol{\phi_2(k)},
\end{equation}
  and we have isometries
\begin{equation}\label{Dlocal} \index{notations}{D2@$D_p(\chi_1,\chi_2,N)$}
H(\chi_1,\chi_2)^{K_\infty\times K_1(N)}\cong
  D(\chi_1,\chi_2,N) \cong \bigotimes_{p|N}D_p(\chi_1,\chi_2,N).
\end{equation}
When $N=1$, the empty tensor product on the right is to be
  interpreted as $\C$.

\subsection{An orthogonal basis for $H(\chi_1,\chi_2)^{K_\infty\times K_1(N)}$}\label{Orth}

Most of the material in this section is drawn from pages 305-306 of
   Casselman's article \cite{Ca}.

In order to construct an orthogonal basis for
  $H(\chi_1,\chi_2)^{K_\infty\times K_1(N)}$,
  we see from \eqref{Dlocal} that it suffices to do so for $D_p(\chi_1,\chi_2,N)$.
Define
\[B(\Z_p/N\Z_p) = \{ \smat{a}{b}{0}{d} \,|\, a, d \in (\Z_p/N\Z_p)^*, b \in \Z_p / N\Z_p\}.\]

\begin{proposition}
For a prime $p|N$, we have the following disjoint union:
   \[ G(\Z_p/N\Z_p) = \bigcup_{i=0}^{N_p} B(\Z_p/N\Z_p) \mat{1}{0}{p^i}{1}
    B(\Z_p/N\Z_p). \]
\end{proposition}

\begin{proof}
    Let $g = \smat abcd\in G(\Z_p/N\Z_p)$. Setting $i=\min(\ord_p(c),N_p)$, it is elementary
  to show that $g$ belongs only to the double coset of $\smat{1}0{p^i}1$.
%
  For future reference, we give the decomposition explicitly.
  There are three cases.  If $c=0$, then $p^{N_p}\equiv 0\mod N\Z_p$
  and
\begin{equation}\label{decomp1}
\mat ab0d\in B(\Z_p/N\Z_p)= B(\Z_p/N\Z_p) \mat{1}{0}{p^{N_p}}{1}
    B(\Z_p/N\Z_p). \end{equation}
  Second, suppose that $0<i<N_p$.  Then $a$ is a unit mod $p$ and
\begin{equation}\label{decomp2}
\mat abcd = \mat{\frac{p^ia}c}{}{}1\mat10{p^i}1\mat{\frac{c}{p^i}}{\frac{bc}{ap^i}}{}
  {\frac{ad-bc}a} \qquad (i=\ord_p(c)).
\end{equation}
If $i=0$, then $c$ is a unit, and we have
\begin{equation}\label{decomp3}
\mat abcd = \mat{1}{\frac ac-1}{}1\mat1011
\mat{c}{d-\frac{ad-bc}c}{}{\frac{ad-bc}c}.
\end{equation}
%
%
%
%
\end{proof}

By equation \eqref{Dp2} and the above proposition, we see that
  a function $\phi\in D_p(\chi_{1}, \chi_{2}, N)$ is determined by
  its values on the matrices $\smat{1}{0}{p^i}{1}$, for $i=0, \ldots, N_p$.
  Therefore if $D_p(\chi_1,\chi_2,N)$ is nonzero, it is spanned by
  functions $\vpa$ satisfying
\begin{equation}\label{vpadef} \index{notations}{phii@$\phi_i=\phi_{p,i,N_p}^{\chi_{1p}, \chi_{2p}}$}
 \phi_{p,i,N_p}^{\chi_{1p}, \chi_{2p}} (\mat{1}{0}{p^j}{1}) = \delta_{ij}.
\end{equation}
   Often we denote the above by $\phi_{i}$ when the other parameters are clear from the context.
   Because the decomposition of $g$ into the form $\smat ab0d\smat10{p^i}1\smat{a'}{b'}{0}{d'}$
   is not unique, for some values of $i$
   it may not be possible to start with \eqref{vpadef} and extend
  to $G(\Z_p/N\Z_p)$ via \eqref{Dp2}.
   We give here the conditions on $i$ under which such a function exists:

\begin{proposition}\label{phii} \index{notations'}{phii@$\phi_i=\phi_{p,i,N_p}^{\chi_{1p}, \chi_{2p}}$}
The function $\phi_i=\vpa$ is well-defined on $G(\Z_p/N\Z_p)$ if and only if
    \begin{equation} \label{chi}
    \ord_p (\c_{\chi_{2}}) \leq i \leq N_p - \ord_p(\c_{\chi_{1}}),
    \end{equation}
where $N_p=\ord_p(N)$.
\end{proposition}

    \begin{proof}
  First, we suppose that \eqref{chi} holds, and we check that $\phi_i$ is well-defined.
  It suffices to show that if
\begin{equation}\label{mats}
\mat ab0d\mat10{p^i}1=\mat10{p^i}1\mat rs0t,
\end{equation}
  then the two values produced by $\phi_i$ using \eqref{Dp2} coincide:
\begin{equation}\label{toshow}
\chi_1(a)\chi_2(d)=\chi_1(t)\chi_2(t).
\end{equation}
The equality \eqref{mats} gives
\[\mat rs0t=\mat{a+bp^i}{b}{dp^i-ap^i-bp^{2i}}{d-bp^i}.\]
From the lower left corner, we see that $(d-bp^i)p^i\equiv ap^i\mod p^{N_p}$,
  so
\begin{equation}\label{abd}
d-bp^i\equiv a\mod p^{N_p-i}.
\end{equation}
  Because $\ord_p(\c_{\chi_1})\le N_p-i$, this implies
\[\chi_1(a)=\chi_1(d-bp^i)=\chi_1(t).\]
Similarly, because $\ord_p(\c_{\chi_2})\le i$, we have
\[\chi_2(d)=\chi_2(d-bp^i)=\chi_2(t).\]
  This proves \eqref{toshow}, so $\phi_i$ is well-defined.

Conversely, assume $\phi_i$ is well-defined.
  Thus we suppose that whenever \eqref{abd} holds, we have the equality
\[\chi_1(a)\chi_2(d) = \chi_1(d-bp^i)\chi_2(d-bp^i).\]
Using this we must deduce \eqref{chi}.
  Set $a=1$, $b=0$, and $d=1+up^{N_p-i}$ for any $u\in\Z_p$.
  Then \eqref{abd} holds, so by our hypothesis we get
\[\chi_2(1+up^{N_p-i}) = \chi_1(1+up^{N_p-i})\chi_2(1+up^{N_p-i}).\]
This implies $\chi_1(1+up^{N_p-i})=1$, so $\ord_p(\c_{\chi_1})\le N_p-i$ as needed.
Now set $a=1$ and $d=1+bp^i$ for any $b\in\Z_p$.  Then \eqref{abd} holds, so we have
\[\chi_2(1+bp^i) = \chi_1(1)\chi_2(1)=1.\]
Thus $\ord_p(\c_{\chi_2})\le i$ as needed.
\end{proof}


\begin{corollary}\label{localnonzero}
Given that $(\chi_1\chi_2)_p=\w_p$, the space $D_p(\chi_1,\chi_2,N)$
  is nonzero if and only if
  \[\ord_p(\c_{\chi_1})+\ord_p(\c_{\chi_2})\le N_p,\]
i.e. if and only if $N\Z_p\subset \c_{\chi_1}\c_{\chi_2}\Z_p$.
  If nonzero, its dimension is equal to
 $1+\ord_p(\frac{N}{\c_{\chi_1}\c_{\chi_2}})$, with an orthogonal basis given by
\[\mathcal{B}_p=\mathcal{B}_p(\chi_1,\chi_2) \index{notations}{Bcp@$\mathcal{B}_p$}
  =\{\phi_{i}|\,\ord_p(\c_{\chi_2})\le i\le {N_p-\ord_p(\c_{\chi_1})}\}.\]
\end{corollary}

\begin{proof} The only point remaining is the orthogonality of $\{\phi_i\}$,
  which follows immediately from the definition of the inner product
  \eqref{Diplocal} since these functions have disjoint support.
\end{proof}

Tensoring the local spaces together, we have:

\begin{corollary} \label{nonzero}
Given that $\chi_1\chi_2=\w$, the space $H(\chi_1,\chi_2)^{K_\infty\times K_1(N)}$
  is nonzero if and only if $\c_{\chi_1}\c_{\chi_2}|N$.
 If nonzero,  its dimension is $\tau(\frac N{\c_{\chi_1}\c_{\chi_2}})$\index{notations}{tau@$\tau$, divisor function}
  for the divisor function $\tau$, with an orthogonal basis given by
\[\mathcal{B}=\mathcal{B}(\chi_1,\chi_2)\index{notations}{Bc@$\mathcal{B}$} \index{notations}{phiip@$\phi_{(i_p)}=\prod_{p"|N}\phi_{i_p}$}
  =\{\phi_{(i_p)}=\prod_{p|N}\phi_{i_p}|\, \phi_{i_p}\in \mathcal{B}_p\}.\]
Here we implicitly use the natural identification \eqref{Dlocal}.
The norm of $\phi_{(i_p)}\in\mathcal{B}$ is given by\index{notations}{ip@$(i_p), i_p$}
\begin{equation}\label{norm}
\|\phi_{(i_p)}\|^2 =
  \prod_{p|N\atop i_p=0}\frac p{(p+1)}
  \prod_{p|N\atop 0<i_p<N_p}\frac{p-1}{p^{i_p}(p+1)}
 \prod_{p|N\atop i_p=N_p}\frac1{p^{N_p-1}(p+1)}.
\end{equation}
\end{corollary}

\begin{proof}
The claim about the dimension follows from the fact that by \eqref{chi}
  the number of tuples $(i_p)$ is
\[\prod_{p|N} (N_p-\ord_p(\c_{\chi_1})-\ord_p(\c_{\chi_2})+1)
  = \tau(\frac N{\c_{\chi_1}\c_{\chi_2}}).\]
For the norm, by \eqref{sgN} we have
\[\|\phi_{(i_p)}\|^2  = \prod_{p|N}\int_{K_p}|\phi_{i_p}(k)|^2dk
=\prod_{p|N}\meas\left\{\smat abcd\in K_p|\, \min(\ord_p(c),N_p)=i_p\right\}.\]
When $i_p=N_p$, the corresponding set is just $K_0(N)_p$, which
  has measure
\[\frac1{\psi_p(N)}=\frac1{p^{N_p-1}(p+1)}\]
(cf. \cite{KL}, pp. 206-207).  When $0\le i_p<N_p$, the corresponding
  set is equal to $K_0(p^{i_p})_p - K_0(p^{i_p+1})_p$, which has measure
$\frac1{\psi_p(p^{i_p})}-\frac1{\psi_p(p^{i_p+1})}$.  This works out to
  $\frac{p-1}{p^{i_p}(p+1)}$ if $0<i_p<N_p$ and, using
 $\psi_p(1)=1$, $\frac{p}{p+1}$ if $i_p=0$.
\end{proof}

\subsection{Evaluation of the basis elements}

Given a basis element $\phi=\phi_{(i_p)}$ of $H(\chi_1,\chi_2)^{K_\infty\times K_1(N)}$,
  we will need to compute the Fourier expansion of the associated
  Eisenstein series.  From our expression \eqref{Eisen} for $E_\phi(s,z)$, we see that
  we need to be able to evaluate $\phi_{\fin}(\smat abcd)$ for
  $\smat abcd\in \SL_2(\Z)$.

\begin{proposition}\label{philocal} \index{notations}{phii@$\phi_i=\phi_{p,i,N_p}^{\chi_{1p}, \chi_{2p}}$}
   Let $p|N$.  For a local element $\phi_i\in \mathcal{B}_p$,
  and $k=\smat abcd\in K_p$, $\phi_i(k)=0$ unless $i=\min(\ord_p(c),N_p)$.
  If this condition is met, then
\[\phi_i(\mat abcd) = \left\{\begin{array}{ll}\chi_2(d)&i=N_p,\\\\ 
 \chi_1(ad-bc)\ol{\chi_1(\frac c{p^i})}\,\chi_2(d)&0<i<N_p,\\\\
  \chi_1(ad-bc)\ol{\chi_1(c)}&i=0.\end{array}\right.\]
\end{proposition}

\begin{proof}
This follows from the definition \eqref{Dp2} of $D_p(\chi_1,\chi_2,N)$
  and the decompositions \eqref{decomp1}-\eqref{decomp3}.    If $i=N_p$, then by \eqref{chi}
  $\ord_p(\c_{\chi_1})\le N_p-i=0$, so $\chi_1$ is unramified at $p$.  Since $N_p>0$, we
  see that $a$ must be a unit, and by \eqref{Dp2},
\[\phi_i(\mat abcd)=\chi_1(a)\chi_2(d)=\chi_2(d).\]
  If $0<i<N_p$, then $a$ is a unit and by \eqref{decomp2} we have
\[\phi_i(\mat abcd)=\chi_1(\frac{p^ia}{c})\chi_1(\frac{ad-bc}a)\chi_2(\frac{ad-bc}a)\]
  \[=\ol{\chi_1(\tfrac c{p^i})}\chi_1(ad-bc)\chi_2(d-\tfrac{bc}a)
  =\ol{\chi_1(\tfrac c{p^i})}\chi_1(ad-bc)\chi_2(d)\]
  since $\frac{bc}a\in c\Z_p=p^i\Z_p\subset \c_{\chi_2}\Z_p$.
  When $i=0$, we have $\ord_p(\c_{\chi_2})= 0$, so $\chi_2$ is unramified at $p$.
  Then \eqref{decomp3} gives
\[\phi_i(\mat abcd)=
\w(\frac{ad-bc}c)=\chi_1(ad-bc)\ol{\chi_1(c)}.\qedhere\]
\end{proof}

Multiplying these local results together, we have, for $\smat abcd\in\SL_2(\Zhat)$,
\begin{equation}\label{phi1}
\phi_{(i_p)}(\smat abcd_{\fin}) = \prod_{p|N,\atop i_p=0}\ol{\chi_{1p}(c)}
  \prod_{p|N,\atop 0<i_p<N_p}\ol{\chi_{1p}(\tfrac c{p^{i_p}})}\chi_{2p}(d)
  \prod_{p|N,\atop i_p=N_p}\chi_{2p}(d)
\end{equation}
under the assumption that $\min(\ord_p(c),N_p)=i_p$ for all $p$ (otherwise the value is $0$).
We can express this as a product of two Dirichlet characters as follows.
Let\index{notations}{N1@$\ds N_1=\prod_{p"|N,\atop i_p<N_p}p^{N_p}$}
\[N_1\eqdef\prod_{p|N,\atop i_p<N_p}p^{N_p}.\]
  Note that $\c_{\chi_1}|N_1$.
  Attach to $\chi_1$ a Dirichlet character modulo $N_1$ by \index{notations}{chi1'@$\chi'_1$, Dirichlet character modulo $N_1$}
\[\chi_1'(x)\eqdef\prod_{p|N_1}\chi_{1p}(x)\qquad (x,N_1)=1\]
  as in \eqref{chi'}-\eqref{x'} with $N_1$ in place of $N$.
  We extend $\chi_1'$ to $\Z$ in the usual way by taking it to be $0$
  if $(x,N_1)>1$.  For convenience later,
   we also set $\chi_1'(x)=0$ if $x$ is not an integer.
  Let \index{notations}{M@$M=\prod_{p"|N}p^{i_p}$}
\[M\eqdef \prod_{p|N}p^{i_p}.\]
  Then assuming $\phi_{(i_p)}(\smat abcd)\neq 0$, 
  we have $M|c$ since $i_p\le \ord_p(c)$ for all $p$, and
\[\chi_1'(\frac cM)=\prod_{p|N_1}\chi_{1p}(\frac{c}{p^{i_p}\frac{M}{p^{i_p}}})
  =\prod_{p|N_1}\chi_{1p}(\tfrac c{p^{i_p}})\,\ol{\chi_{1p}(\tfrac M{p^{i_p}})}.\]
Therefore defining the constant\index{notations}{C i@$C_{(i_p)}=\prod_{p"|N_1}\ol{\chi_{1p}(\tfrac M{p^{i_p}})}$}
\[C_{(i_p)}=\prod_{p|N_1}\ol{\chi_{1p}(\tfrac M{p^{i_p}})},\]
  we have
\[\prod_{p|N_1}\ol{\chi_{1p}(\tfrac c{p^{i_p}})}=C_{(i_p)}\ol{\chi_1'(\tfrac cM)}.\]
Similarly, we set \index{notations}{N2@$\ds N_2=\prod_{p"|N,\atop i_p>0}p^{N_p}$}
\[N_2\eqdef\prod_{p|N,\atop i_p>0}p^{N_p}\]
  (the lexical ambiguity between the above
 definition and $N_2=\ord_2(N)$ should not cause confusion).
Observing that $\c_{\chi_2}|\,M|\,N_2$, we define a Dirichlet character
  modulo $N_2$ by \index{notations}{chi2'@$\chi'_2$, Dirichlet character modulo $M$}
\[\chi_2'(x)\eqdef\prod_{p|N_2}\chi_{2p}(x)\qquad (x,N_2)=1,\]
extending to all of $\Z$ by $\chi_2'(x)=0$ if $(x,N_2)>1$.
Note that because $M$ and $N_2$ have the same set of prime divisors
  and $\c_{\chi_2}|M$, we have
\begin{equation}\label{modM}
\chi_2'(d+Mx)=\chi_2'(d)
\end{equation}
for all $x\in\Z$.
With the above notation, \eqref{phi1} becomes\index{notations}{phiip@$\phi_{(i_p)}=\prod_{p"|N}\phi_{i_p}$}
\begin{equation}\label{phi}
\phi_{(i_p)}(\mat abcd_{\fin})=C_{(i_p)}\,\ol{\chi_1'(\tfrac c M)}\,\chi_2'(d).
\end{equation}

In the preceding discussion, \eqref{phi} was established under the assumption
  that $\min(\ord_p(c),N_p)=i_p$ for all $p$.  However, it actually holds
  in general:

\begin{proposition}
Equation \eqref{phi} is valid for all $\smat abcd\in \SL_2(\Z)$.
\end{proposition}
\begin{proof}
When $\min(\ord_p(c),N_p)\neq i_p$ for some $p$,
  the left-hand side of \eqref{phi} is equal to $0$.
  Thus it suffices to show that the same is true of $\chi_1'(\tfrac c M)$.
 By definition, $\chi'_1(\tfrac c M)$ is nonzero if and only if
$M|c$ and $\gcd(\tfrac c M, N_1) = 1$.  This is equivalent to
 $\ord_p(c)\ge i_p$ for all $p$ and $\ord_p(c)=i_p$ when $i_p<N_p$.
These conditions occur precisely when $i_p = \min(\ord_p(c), N_p)$.
%
\end{proof}

\subsection{Fourier expansion of Eisenstein series}\label{Fourier}

For any $\phi\in H(\chi_1,\chi_2)^{K_\infty\times K_1(N)}$, the Eisenstein series $E_\phi(s,z)$ has
  period one as a function of $z\in\mathbf H$.  Indeed, writing
   $z=g_\infty(i)$,
\[E_\phi(s,z+1)=E(\phi,s,\smat 1101g_\infty\times 1_{\fin})
  =E(\phi,s,g_\infty\times \smat 1{-1}01_{\fin})=E_\phi(s,z),\]
  the second equality holding by the left $G(\Q)$-invariance of $E(\phi,s,g)$,
  and the third equality holding by the right $K_1(N)$-invariance of $\phi$.
It follows that $E_\phi(s,z)$ has a Fourier expansion \index{notations}{ams@$a_m(s,y)=a_m^{\phi}(s,y)$, $m$-th Fourier coefficient of $E_\phi(s,z)$} \index{keywords}{Fourier expansion!of Eisenstein series}
\begin{equation}\label{ams}
 E_\phi(s,z) = \sum_{m \in \Z} a_m(s,y) e(mx), 
\end{equation}
valid when $\Re(s)>1/2$ by Proposition \ref{convergence}.
  It turns out that the right-hand side also converges
  for other $s$, to a meromorphic function continuing $E_\phi(s,z)$.
  This will be described in the next section.
  Here we will compute the Fourier coefficients when
  $\phi=\phi_{(i_p)}$.
%

Henceforth we fix the tuple $(i_p)_{p|N}$, setting $\phi=\phi_{(i_p)}$ and
   $M=\prod_{p|N}p^{i_p}$ as before.  Assuming $\Re(s)>1/2$,
by \eqref{Eisen} and \eqref{phi} we have
  \[ E_\phi(s,z)= y^{1/2+s}C_{(i_p)} \chi'_1(0) +  y^{1/2+s}C_{(i_p)}\sum_{c>0}
  \sum_{d\in\Z\atop{(d,c)=1}}
     \frac{\ol{\chi'_1(\frac{c}{M})} \chi'_2(d)}{|cz+d|^{1+2s}}.\]
Recall that\index{notations}{chip@$\chi_1'(0)$}
\[\chi'_1(0)=\left\{\begin{array}{ll} 1& \text{if }N_1=1, \text{ i.e. }i_p=N_p\text{ for all }
   p|N,\\
  0& \text{if }N_1>1,\text{ i.e. } i_p<N_p\text{ for some }p|N,\end{array}\right.\]
and $\ol{\chi_1'(c/M)}=0$ unless $M|c$.

It will be convenient to sum over all $d\in\Z$ rather than the
  restricted set $(d,c)=1$.  We need the following lemma.

\begin{lemma}
Suppose $\gcd(c,d)=n$. Write $c = nc'$ and  $d=nd'$ for
   integers $c',d'$.
Then
\begin{equation} \label{chieq} 
\ol{\chi'_1(c/M)} \chi'_2(d) =
\ol{\chi'_1(n)} \chi'_2(n) \ol{\chi'_1({c'}/M)} \chi'_2(d').
\end{equation}
\end{lemma}

\begin{proof}
If $\gcd (n,N_2)>1$, then $\chi'_2(d) = 0 = \chi'_2(n)$.
So equation \eqref{chieq} is valid in this case.
On the other hand, suppose $(n,N_2)=1$. Then $(n,M)=1$ because $M|N_2$.
If $M\nmid c$, then both sides of \eqref{chieq} vanish.
If $M|c$, then $M|c'$,  and \eqref{chieq} follows
  by the multiplicativity of Dirichlet characters.
\end{proof}

\noindent Using the above lemma, we have
\[\sum_{c>0}\sum_{d\in\Z} \frac{\ol{\chi'_1(\frac{c}{M})} \chi'_2(d)}{|cz+d|^{1+2s}}
=\sum_{n>0}\sum_{c\in n\Z^+}\sum_{d\in\Z\atop{(d,c)=n}}
  \frac{\ol{\chi_1'(\frac cM)}\chi_2'(d)}{|cz+d|^{1+2s}}\]
\[=\sum_{n>0}\frac{\ol{\chi_1'(n)}\chi_2'(n)}{n^{1+2s}}
  \sum_{c>0}\sum_{d\in\Z\atop{(d,c)=1}}
  \frac{\ol{\chi_1'(\frac cM)}\chi_2'(d)}{|cz+d|^{1+2s}}\]
\[=L_N(1+2s,{\chi_1}\ol{\chi_2}) \sum_{c>0}\sum_{d\in\Z\atop{(d,c)=1}}
  \frac{\ol{\chi_1'(\frac cM)}\chi_2'(d)}{|cz+d|^{1+2s}}.\]
Here we have applied \eqref{L}, using the fact that $\ol{\chi_1'}\chi_2'$
  has modulus $\lcm(N_1,N_2)=N$.
The above has period one as a function of $z$.
   This can be seen from the the fact that the Eisenstein series has period one,
  or it can be seen directly as follows:
   \[ \sum_{c > 0}\sum_{d\in \Z} \frac{\ol{\chi'_1(\frac{c}{M})} \chi'_2(d)}{|cz+c+d|^{1+2s}}
   = \sum_{c > 0} \sum_{d \in \Z} \frac{\ol{\chi'_1(\frac{c}{M})} \chi'_2(d-c)}{|cz+d|^{1+2s}}. \]
The summand vanishes unless $M|c$.
   Therefore by \eqref{modM}, $\chi_2'(d-c)=\chi_2'(d)$ in all nonzero terms,
   as needed.
  By this periodicity, the double sum has a Fourier expansion
\[ \sum_{c > 0} \sum_{d\in \Z} \frac{\ol{\chi'_1(\frac{c}{M})} \chi'_2(d)}{|cz+d|^{1+2s}}
   =  \sum_{m \in \Z} b_m(s,y) e(mx). \]

The coefficient $b_m(s,y)$ is related to $a_m(s,y)$ of \eqref{ams}
 since\index{notations}{bms@$b_m(s,y)$}
  \begin{equation}\label{ab}
 E_\phi(s,z)= y^{1/2+s}C_{(i_p)} \chi'_1(0) +  \frac{y^{1/2+s}C_{(i_p)}}
  {L_N(1+2s,\chi_1\ol{\chi_2})} \sum_{c>0} \sum_{d\in\Z}
     \frac{\ol{\chi'_1(\frac{c}{M})} \chi'_2(d)}{|cz+d|^{1+2s}}.
\end{equation}
Explicitly, for $\Re(s)>1/2$ we have \index{notations}{ams@$a_m(s,y)=a_m^{\phi}(s,y)$, $m$-th Fourier coefficient of $E_\phi(s,z)$}
\[\hskip -.1cm a_m(s,y) =
  \begin{cases}
     y^{1/2+s}C_{(i_p)} \chi'_1(0) + y^{1/2+s}C_{(i_p)}
   L_N(1+2s,\chi_1 \ol{\chi_2})^{-1} b_0(s,y)
   & \text{if $m = 0$,}
 \\
y^{1/2+s} C_{(i_p)} L_N(1+2s,{\chi_1} \ol{\chi_2})^{-1} b_m(s,y) &
     \text{if $m \neq 0$.}
  \end{cases}
  \]

We now compute the coefficients $b_m(s,y)$.
We have
   \[ b_m(s,y) =  \sum_{c >  0} \sum_{d\in\Z} \int_0^{1}
  \frac{\ol{\chi_1'(\frac cM)}\chi_2'(d)}{|cz+d|^{1+2s}} \,e(-mx) dx \]
   \[ = \sum_{c > 0} \sum_{d \in \Z/c\Z} \int_0^{1}
    \sum_{t \in \Z} \frac{\ol{\chi_1'(\frac cM)}\chi_2'(d+ct)}{|cz+ d + c t|^{1+2s}} \,e(-mx) dx.\]
As before, the integrand is nonzero only if $M|c$, and under this
  assumption
  $\chi'_2(d+ct) = \chi'_2(d)$ by \eqref{modM}.  Therefore the above is
   \[   = \sum_{c > 0} \sum_{d \in \Z/c\Z} \int_{-\infty}^{\infty}
   \frac{\ol{\chi_1'(\frac cM)}\chi_2'(d)}{|cz+ d|^{1+2s}} \,e(-mx) dx\]
    \[ = \sum_{c > 0} \frac{ \ol{\chi_1'(\frac cM)}\sum_{d \in \Z/c\Z}\chi_2'(d)}{c^{1+2s}}
     \int_{-\infty}^{\infty} \frac{e(-mx)}{|z+\frac dc|^{1+2s}} dx \]
    \[ = \sum_{c > 0} \frac{\ol{\chi_1'(\frac cM)} \sum_{d \in \Z/c\Z}\chi_2'(d)}{c^{1+2s}}
     \int_{-\infty}^{\infty} \frac{e(-m(x-\frac dc))}{(x^2+y^2)^{1/2+s}} dx \]
    \begin{equation}\label{Sdef}
 = \sum_{c \in M\Z^+} \frac{\ol{\chi'_1(\frac{c}{M})}}{c^{1+2s}}
          \Biggl(\sum_{d \in \Z/c\Z} \chi'_2(d) e(\frac{dm}{c})\Biggr)
          \int_{-\infty}^{\infty} \frac{e(-mx)}{(x^2+y^2)^{1/2+s}} dx.
\end{equation}
Now apply the well-known formula:
 \[ \int_{-\infty}^\infty \frac{e(-mx) }{(x^2+y^2)^{1/2+s}}dx =
\begin{cases}
    \frac{\ds 2 \pi^{1/2+s} |y|^{-s} |m|^s}{\ds \Gamma({\scriptstyle \frac{1}{2}}+s)}
   K_{s}(2\pi |m||y|) &  m\neq 0,\\\\
    \frac{\ds \sqrt{\pi} y^{-2s} \Gamma(s)}{\ds \Gamma({\scriptstyle \frac{1}{2}}+s)} &  m = 0
    \end{cases}\]
(\cite{Bu}, p. 67).
   By \eqref{modM}, the character sum $S$ in parentheses in \eqref{Sdef} satisfies
\begin{equation}\label{S0}\hskip -.4cm
 S = \sum_{d \in \Z/c\Z} \chi_2'(d-M) e(\frac{dm}{c}) =
   \sum_{d \in \Z/c\Z} \chi_2'(d)\, e\Bigl(\frac{(d+M)m}{c}\Bigr)
   = e\bigl(\frac{mM}{c}\bigr) S.
 \end{equation}
Hence if $e(\frac{mM}{c})\neq 1$
 (or equivalently $c\nmid mM$), then $S=0$.
  Therefore if $\Re(s)>1/2$, the Fourier coefficient is given by \index{notations}{bms@$b_m(s,y)$}
\[ b_m(s,y) = \begin{cases}
  \ds \frac{2 \pi^{1/2+s} y^{-s} |m|^s}{\Gamma(\frac{1}{2}+s)}
  \sigma_s(\chi_1',\chi_2',m) K_{s}(2\pi |m|y)
    & m \neq 0, \\\\
  \ds \frac{\sqrt{\pi} y^{-2s} \Gamma(s)}{\Gamma(\frac{1}{2}+s)}
\sigma_s(\chi_1',\chi_2',0)
 & m= 0, \end{cases} \]
for the sum (see also \S\ref{charsum})
\begin{equation}\label{sigma}
\sigma_s(\chi_1',\chi_2',m)=
      \sum\limits_{c\in M\Z^+\atop{c|mM}}
   \frac{\ol{\chi'_1(\frac{c}{M})}}{{c}^{1+2s}}
  \sum_{d \in \Z/c\Z} \chi'_2(d)\, e(\frac{dm}{c})
\end{equation}
\[ =\sum_{c|m}\frac{\ol{\chi_1'(c)}}{(Mc)^{1+2s}}
  \sum_{d\in\Z/Mc\Z}\chi_2'(d)e(\frac{dm}{Mc}).\]
In the second sum, each summand is defined for $d\mod M\Z$, since $M$ is a modulus for
  $\chi_2'$ and $e(\tfrac{(d+M)m}{Mc})=e(\tfrac{dm}{Mc})$ 
  since $c|m$.  Thus
\begin{equation}\label{sigma2}\index{notations}{sigma@$\sigma_s(\chi_1',\chi_2',m)$}
\sigma_{s}(\chi_1',\chi_2',m)=\frac1{M^{1+2s}}\sum_{c|m}\frac{\ol{\chi_1'(c)}}{c^{2s}}
  \sum_{d\in\Z/M\Z}\chi_2'(d)\,e(\frac{dm}{Mc}).
\end{equation}
  We emphasize that even though $m<0$ is allowed, the sum is extended only
  over the {\em positive} divisors $c$ of $m$.
  When $m\neq 0$, the sum is finite.
  However, when $m=0$, the sum is extended over all $c\in \Z^+$,
  and only converges absolutely for $\Re(s)>1/2$. Indeed we have
  the following.

\begin{proposition}\label{sigma0}
When $m=0$,
\[\sigma_s(\chi_1',\chi_2',0)=\begin{cases} \frac{\varphi(M)}{M^{1+2s}}
  L_{N_1}(2s,\w)&\text{if $\chi_2$ is trivial},\\
  0&\text{otherwise,}\end{cases}\]
where $\varphi$ is the Euler $\varphi$-function.
\end{proposition}
\begin{proof}
By \eqref{sigma2},
\[\sigma_s(\chi_1',\chi_2',0)=\frac1{M^{1+2s}}
      \sum\limits_{c=1}^\infty
   \frac{\ol{\chi'_1(c)}}{{c}^{2s}}
  \sum_{d \in \Z/M\Z} \chi'_2(d).\]
The sum over $d$ vanishes unless $\chi_2'$ is
  the principal character modulo $M$.
Indeed,
\[\sum_{d=1}^M\chi_2'(d) =\sum_{d\in(\Z/M\Z)^*}\chi_2'(d)
=\begin{cases} {\varphi(M)}&\text{if $\chi_2'$ is principal}\\
  0&\text{otherwise.}\end{cases}\]
Therefore if $\chi_2'$ is principal (in which case
  $\chi_2$ is trivial by \eqref{chi'} with $N_2$ in place of $N$), we find
\[\sigma_s(\chi_1',\chi_2',0)=
  \frac{\varphi(M)}{M^{2s+1}}L(2s,\ol{\chi_1'}).\]
Applying \eqref{L}, the proposition follows since $\chi_1=\w$ in this case.
\end{proof}

\subsection{Meromorphic continuation}

To summarize the previous section, for the scaled basis element
\[\phi=\tfrac1{C_{(i_p)}}\phi_{(i_p)}\in
  H(\chi_1,\chi_2)^{K_\infty\times K_1(N)},\]
  we have, for $\Re(s)>1/2$,\index{notations}{E@$E_\phi(s,x+iy)=E(\phi,s, \smat yx01\times 1_{\fin})$}
\begin{align} 
\notag E_\phi(s,z)&=y^{1/2+s}\chi_1'(0) + y^{1/2+s}\sum_{c>0}
  \sum_{(d,c)=1}\frac{\ol{\chi_1'(\frac cM)}\chi_2'(d)}{|cz+d|^{1+2s}}\\
\label{EN}
&=y^{1/2+s}\chi_1'(0) +y^{1/2-s}\delta_{\chi_2}
  \frac{\varphi(M)\sqrt{\pi}\,\Gamma(s)\,L_{N_1}(2s,\w)}
  {M^{1+2s}\, \Gamma(\frac12+s) L_N(1+2s,\w)}
\end{align}
\begin{equation}\label{four}
+\frac{2y^{1/2}\pi^{1/2+s}}{\Gamma(\frac12+s)L_N(1+2s,{\chi_1}\ol{\chi_2})}
  \sum_{m\neq 0}|m|^s\sigma_s(\chi_1',\chi_2',m)K_s(2\pi|m|y)e(mx).
\end{equation}
Here $\delta_{\chi_2}\in\{0,1\}$ is nonzero if and only if $\chi_2$
  is the trivial character.

\begin{theorem}\label{merom} The Fourier expansion \eqref{EN}-\eqref{four}
defines a meromorphic function on $\C$
  which continues $E_\phi(s,z)$.  It is holomorphic in the half-plane
  $\Re(s)\ge 0$,
  except possibly for a simple pole at $s=1/2$ which occurs precisely when
  $\chi_1$ and $\chi_2$ are both trivial.  In the event of a pole, its
  residue is
\[\frac{3\varphi(M)}{\pi M^2}\prod_{p|N\atop{i_p=N_p}}(1-p^{-2})^{-1}
  \prod_{p|N\atop{i_p<N_p}}(1+p^{-1})^{-1}\]
for the Euler $\varphi$-function.
\end{theorem}

\begin{proof}
From the meromorphic continuation of Dirichlet $L$-functions, 
  we see that the constant term\index{keywords}{constant term!of Eisenstein series}
   \eqref{EN} is meromorphic.  Since $\w'(-1)=1$, the
  completed $L$-function of $\w$ has the form
\[\Lambda(2s,\w)=\pi^{-s}\,\Gamma(s)L(2s,\w)\]
  and is entire unless $\w=triv$ is the trivial character, in which case
  it has simple poles at $s=0,\frac12$ (\cite{Bu} Theorem 1.1.1).
  Therefore
\[\Gamma(s)L_{N_1}(2s,\w)=\pi^{s}
  \Lambda(2s,\w)\prod_{p|N_1\atop p\nmid \c_\w} (1-\w_p(p)p^{-2s})\]
is entire unless $\w=triv$, in which case it has a simple pole
  at $s=1/2$ and possibly (if $N_1=1$) a simple pole at $s=0$.
  This possible pole at $s=0$ is cancelled by the simple pole
  of $L_N(1+2s,triv)$ at $s=0$ occurring in the denominator when $\w$ is trivial.
  Recall also that in general $\Gamma(\tfrac12+s)L_N(1+2s,\w)$ is nonzero when $\Re(s)\ge 0$.
  This shows that \eqref{EN} has the desired properties, as does the
  first factor of \eqref{four}.

  It remains to consider the sum in \eqref{four}.  From \eqref{sigma2},
  for $m\neq0$ we have
\[|\sigma_s(\chi_1',\chi_2',m)| \le \frac1{M^{1+2\Re(s)}}
   \sum_{c|m}\frac1{c^{2\Re(s)}}
  \Bigl(\sum_{d\in \Z/M\Z}1\Bigr)=\frac1{M^{2\Re(s)}}\sum_{c|m}\frac 1{c^{2\Re(s)}}.\]
When $\Re(s)\ge 0$, this is
\[
\le M^{-2\Re(s)}\tau(m)  \ll |m|^{\e},
\]
while if $\Re(s)\le 0$ it is
\[
\le (|m|M)^{-2\Re(s)}\tau(m)  \ll |m|^{2|\!\Re(s)|+\e}.
\]
Here as usual $\tau(d)$ denotes the number of positive divisors of $d$,
  and is well known to be $\ll |d|^\e$.
  Furthermore, the Bessel function decays exponentially.  In fact,
  for real $x>1+|s|^2$,\index{notations}{Ks@$K_s(z)$, Bessel function}
\index{keywords}{Bessel function!$K$-}
\[K_s(x) =\sqrt{\frac{\pi}{2x}}e^{-x}(1+O\left(\frac{|s|^2+1}{x}\right))\]
for an absolute implied constant (\cite{Wa}, p. 219, \cite{Iw}, p. 204).
Now suppose $s$ and $y$ are restricted to fixed compact subsets of $\C$ and
  $\R^+$ respectively.  Then by the above, there exists a constant $C$, depending
  only on the two compact sets, such that
  $|K_s(2\pi|m|y)|\le Ce^{-2\pi|m|y}$ for all $m$.
  It follows that the sum in
   \eqref{four} converges uniformly on compact sets, so
  the sum is entire.

In the event of a pole at $s=1/2$, the singular part of the
   Eisenstein series is the term
\[y^{1/2-s}\frac{\varphi(M)}{M^{1+2s}}\frac{\Lambda(2s,triv)\prod_{p|N_1}(1-p^{-2s})}
  {\Lambda(1+2s,triv)\prod_{p|N}(1-p^{-1-2s})}.\]
The formula for the residue follows since
 $\Lambda(2s,triv)$ has residue $\tfrac12$ at $s=1/2$,
  while in the denominator
  $\Lambda(2,triv)=\pi^{-1}\zeta(2)=\tfrac{\pi}6$.
\end{proof}

\subsection{Character sums}\label{charsum}

In order to prove Theorem \ref{dist}, we will need good
  bounds for the Fourier coefficients of normalized Eisenstein series.
  For this purpose we now examine more closely the character sums occurring there.

Let $\chi$ be a Dirichlet character mod $M$ of conductor $\c_{\chi}|M$.
  For a prime $p|M$, we define a Dirichlet character $\chi_p$
  modulo $p^{M_p}$ by\index{notations}{chip@$\chi_p$, local component of a Dirichlet character}
\begin{equation}\label{chi_p}
\hskip -.2cm \chi_p(d) = \begin{cases}0&\text{if }p|d\\
  \chi(x)& \text{if } p\nmid d,\text{ where }x\equiv d\mod p^{M_p},
  x\equiv 1\mod q^{M_q}\, (q\neq p).
  \end{cases}
\end{equation}
The value $\chi_p(d)$ is independent of both the choice of $x$ and the choice of
  modulus $M\in \c_{\chi}\Z^+\cap p\Z$.  With the above definition, we have
   $\chi=\prod_{p|M}\chi_p$.  If we take 
  $\chi_p=\1$\index{notations}{1@$\mathbf 1$ (constant function $1$)}
  to be the constant function $1$ on $\Z$ when $p\nmid M$, then the product can be extended over
  all primes $p$.

We review some well-known facts about Gauss sums.  For $\chi$ as above, 
  define\index{keywords}{Gauss sum}\index{notations}{G1@$G_\chi(m)$, Gauss sum}
\[G_\chi(m)=\sum_{d\in \Z/M\Z}\chi(d)e(\tfrac{dm}M).\]
 Assuming that either $(m,M)=1$ or $\chi$ is primitive, we have
\begin{equation}\label{Gtau}
G_\chi(m)=\tau(\chi)\ol{\chi(m)},
\end{equation}
where $\tau(\chi)=G_\chi(1)$ 
  (\cite{IK}, \S 3.4).\index{notations}{tau@$\tau(\chi)$, Gauss sum}
In general, suppose $\chi^0$ is the primitive character inducing $\chi$,
  and write $M=\ell\c_\chi$.  Then (\cite{Mi}, Lemma 3.1.3)
\begin{equation}\label{Gchi}
G_{\chi}(m)=\tau(\chi^0)\sum_{a|(\ell,m)}a\,\mu(\ell/a)
  \chi^0(\ell/a)\ol{\chi^0(m/a)}
\end{equation}
for the M\"obius function\index{notations}{mu@$\mu$, M\"obius function}\index{keywords}{Mob@M\"obius function}
\[\mu(n)=\begin{cases}1&\text{if }n=1\\
  (-1)^r&\text{if $n=p_1\cdots p_r$ for distinct primes
  $p_1,\ldots,p_r$,}\\
0&\text{if $n$ has a square factor $>1$.}\end{cases}\]
It is well-known that 
\begin{equation}\label{abstau}
|\tau(\chi^0)|= \c_\chi^{1/2}.
\end{equation}
Therefore \eqref{Gchi} gives
\begin{equation}\label{Gbound}
|G_\chi(m)|\le \c_\chi^{1/2}\sigma(|m|),
\end{equation}
where, for $k>0$, $\sigma(k)=\sum_{d|k,d>0}d$.\index{notations}{sigma@$\sigma(m)=\sum_{d"|m}d$}

%


\begin{proposition}\label{Eisbound}
Let $m_1,m_2$ be nonzero integers.  Then
\[\|\phi_{(i_p)}\|^{-2}\left|\sigma_{it}(\chi_1',\chi_2',m_1)
  \ol{\sigma_{it}(\chi_1',\chi_2',m_2)}\right|= O(N^{\e}),\]
where the implied constant depends only on $m_1, m_2$ and $\e$.
\end{proposition}

\begin{proof}
Write $m=m_1$ or $m_2$.  From \eqref{sigma2}, 
\[\sigma_{s}(\chi_1',\chi_2',m)=\frac1{M^{1+2s}}\sum_{c|m}\frac{\ol{\chi_1'(c)}}{c^{2s}}
  G_{\chi_2'}(m/c),\]
where $M=\prod p^{i_p}$.
Applying \eqref{Gbound}, this gives
\begin{equation}\label{sigbound}
|\sigma_{it}(\chi_1',\chi_2',m)|\le \frac{\c_{\chi_2}^{1/2}}M\sum_{c|m}\sigma(c)
\le \frac{\c_{\chi_2}^{1/2}}M \tau(|m|)\sigma(|m|).
\end{equation}
By \eqref{norm}, we have
\[ \|\phi_{(i_p)}\|^{-2}=\prod_{p|N\atop {i_p=0}}(1+\frac1p)
  \prod_{p|N\atop{0<i_p<N_p}}p^{i_p}(1+\frac2{p-1})
  \prod_{p|N\atop{i_p=N_p}}p^{i_p}(1+\frac1p).
\]
Therefore
\begin{equation}\label{phibound} \|\phi_{(i_p)}\|^{-2}\le M
  \prod_{p|N}(1+\frac2{p-1}).\end{equation}
Together, these bounds give
\[\frac{\left|\sigma_{it}(\chi_1',\chi_2',m)\right|}
{\|\phi_{(i_{p})}\|}\le
 \frac{\c_{\chi_2}^{1/2}M^{1/2}}{M}\tau(|m|)\sigma(|m|) 
 \prod_{p|N}(1+\frac{2}{p-1})^{1/2}\]
\begin{equation}
\ll_\e \tau(|m|)\sigma(|m|)N^{\e/2}.
\end{equation}
The proposition follows.
\end{proof}

\pagebreak
\section{The kernel of $R(f)$}\label{6}

In this section we give the spectral formula for the kernel function of $R(f)$.
  We refer to \cite{Ar1}, \cite{GJ} and \cite{Kn} for further discussion
  and theoretical background.
  Our purpose is to show that these spectral terms converge
  absolutely in a strong sense (Theorems \ref{bc1} and \ref{bc2}).
  This provides the justification for their use in the relative trace
  formula.  Our treatment is based on the methods of Arthur
  \cite{Ar1}, \cite{Ar2}, especially Lemma 4.4 of \cite{Ar2}.
  His result holds for any connected reductive algebraic group over $\Q$.
   In the setting of $\GL(2)$ it gives
  e.g. that for an orthonormal basis $\{\phi\}\subset H(0)$
   (cf. \eqref{H(s)}),
\[\int_{-\infty}^\infty\Biggl|\sum_{\phi}E(\pi_{it}(f)\phi_{it},x)
  \ol{E(\phi_{it},y)}\Biggr|dt<\infty.\]
  We give a detailed discussion here partly to avoid referring the
  reader to a paper on general groups just for a result about $\GL(2)$, but also
  because we need to show that the absolute values can be
  brought inside the sum, at least for the class of functions $f$
  considered in this paper.

\subsection{The spectral decomposition}\label{Spec}

The right regular representation of $G(\A)$ on $L^2(\w)$
  decomposes in terms of cuspidal representations on $\GL(m)$ for $m\le 2$.
  The continuous part of $L^2(\w)$ is indexed by certain cuspidal representations of $\GL(1)$,
  i.e. Hecke characters, and the discrete part consists of irreducible cuspidal
  representations and, in some situations, one-dimensional representations.

Suppose $\chi$ is a Hecke character satisfying $\chi^2=\w$.
Then defining \index{notations}{phichi@$\phi_\chi=\chi\circ\det$}
\[\phi_\chi(g)=\chi(\det(g))\qquad(g\in G(\A)),\]
  we see that $\phi_\chi$ is square integrable modulo $Z(\A)$, with
\[\|\phi_\chi\|^2 = \meas(\olG(\Q)\bs\olG(\A))={\pi}/3.\]
  Note that
$\phi_\chi(zg)=\w(z)\phi_\chi(g)$ for all $z\in Z(\A)$.
  Therefore $\phi_\chi$ spans a one-dimensional
  subrepresentation of $L^2(\w)$, which we denote by $\C_\chi$.
  Conversely, any one-dimensional subrepresentation of
  $L^2(\w)$ arises in this way from a character satisfying $\chi^2=\w$.

\begin{proposition}
  The spaces $\C_\chi$ are mutually orthogonal, and also orthogonal
  to $L^2_0(\w)$.
\end{proposition}

\begin{proof}  Suppose $V$ is a unitary representation of a group $G$.
  Then for any closed $G$-stable
  subspace $S$, the action of $G$ preserves the decomposition
  $V=S\oplus S^\perp$.  If $W$ is any other closed $G$-stable subspace of $V$,
  then it is easy to show that $W=(W\cap S)\oplus (W\cap S^\perp)$.
  In particular if $W$ is one-dimensional,
  then $W\subset S$ or $W\subset S^\perp$.
  Applying this with $S=\C_{\chi_1}$ and $W=\C_{\chi_2}$ shows the first claim,
  and taking $S=L^2_0(\w)$ gives the second.  Note that $\phi_\chi$ is not cuspidal
  because its constant term is
\[\int_{N(\Q)\bs N(\A)}\phi_\chi(ng)dn=\phi_\chi(g)\neq 0.\qedhere\]
\end{proof}

  We denote by $L^2_{\res}(\w)$ the Hilbert direct sum \index{notations}{L2res@$L^2_{\res}(\w)$}
\[L^2_{\res}(\w)=\bigoplus_{\chi^2=\w}\C_{\chi}.\]
  (These characters arise from the
  residues of certain Eisenstein series at $s=1/2$).
  If $L^2_{\res}(\w)$ is nonzero, then it is infinite dimensional.
  To see this, note that if there exists $\chi$ with $\chi^2=\w$, then
  $L^2_{\res}(\w) =\bigoplus_{\eta^2=1} \C_{\chi\eta}$.  There
  are infinitely many quadratic Hecke characters $\eta$.

  The direct sum \index{notations}{L2disc@$L^2_{\disc}(\w)$}
\[L^2_{\disc}(\w)\eqdef L^2_0(\w)\oplus L^2_{\res}(\w)\]
  is the discrete part of the spectrum of $L^2(\w)$.
We next describe its orthogonal complement $L^2_{\cont}(\w)$.\index{notations}{L2cont@$L^2_{\cont}(\w)$}
%
\comment{  for a pair $(\chi_1,\chi_2)$
  of finite order Hecke characters, we let
  $\mathcal{L}(\chi_1,\chi_2)$ denote the direct integral of the representations
  $H(\chi_1,\chi_2,it)$ for $t\in \R$.  This Hilbert space consists of the
  measurable functions
\[A:i\R\longrightarrow H(\chi_1,\chi_2,0)\]
  such that
\[\|A\|^2 = \frac1{\pi}\int_{-\infty}^\infty \|A(it)\|^2 dt<\infty.\]
  The associated inner product is
\begin{equation}\label{ipcont0}
\sg{A,B}=\frac1\pi\int_{-\infty}^\infty \sg{A(it),B(it)}dt.
\end{equation}
  The representation of $G(\A)$ on $\mathcal L(\chi_1,\chi_2)$ is defined by
\[ (gA)(it) = \pi_{it}(g)(A(it)),\]
  where we identify $A(it)$ with its correspondent in
  $H(\chi_1,\chi_2,it)$.
  The above representation of $G(\A)$ is unitary because
\[\|gA\|^2 = \frac1\pi\int_{-\infty}^\infty\|\pi_{it}(g)A(it)\|^2 dt = \|A\|^2\]
  since each $\pi_{it}$ is unitary.
  There is a $G(\A)$-equivariant isometry
\[L^2_{\cont}(\w)\longrightarrow \bigoplus_{\chi_1\chi_2=\w}
  \mathcal{L}(\chi_1,\chi_2),\]
where this Hilbert space direct sum is taken over unordered
   pairs $\{\chi_1,\chi_2\}$ of unitary Hecke characters
  satisfying $\chi_1\chi_2=\w$.
}
%
For $s\in\C$ define
\begin{equation}\label{H(s)} \index{notations}{Hs@$H(s)=\bigoplus_{\chi_1\chi_2=\w}H(\chi_1,\chi_2,s)$}
H(s)=\bigoplus_{\chi_1\chi_2=\w}H(\chi_1,\chi_2,s),
\end{equation}
where $H(\chi_1,\chi_2,s)$ is defined in \S \ref{Ind}, and
this Hilbert space direct sum is taken over all ordered pairs of finite order Hecke
  characters whose product is $\w$.\\

\noindent{\em Remark:}
 Define a character $(\w,s)$ of $B'\eqdef Z(\A)N(\A)M(\Q)M(\R^+)$
  by
\[(\w,s):\smat z{}{}z\smat yx01\mapsto \w(z)|y|^{s+1/2}.\]
The right regular representation $\pi_s$ of $G(\A)$ on $H(s)$
  is equivalent to the induced representation $\Ind_{B'}^{G(\A)}(\w,s)$.
  This latter viewpoint is the one taken in \cite{GJ}.
  To see the equivalence, note by transitivity of induction that
\begin{equation}\label{B'}
\Ind_{B'}^{G(\A)}(\w,s)=\Ind_{B(\A)}^{G(\A)}\Ind_{B'}^{B(\A)}(\w,s).
\end{equation}
By restriction, $\Ind_{B'}^{B(\A)}(\w,s)$ can be identified with
  $\Ind_{Z(\A)M(\Q)M(\R^+)}^{M(\A)}(\w,s)$.  The quotient $(Z(\A)M(\Q)M(\R^+))\bs M(\A)$
  is compact, so by the Peter-Weyl theorem,
 the functions $(\chi_1,\chi_2,s)=\chi_1(a)\chi_2(d)|\tfrac ad|^{s+\frac12}$
  with $\chi_1\chi_2=\w$ form a basis for the induced space.
  Thus the right-hand side of \eqref{B'} is equal to
  $\Ind_{B(\A)}^{G(\A)}\bigoplus(\chi_1,\chi_2,s)=H(s)$.
  \\

Let
\[H=\int_{-\infty}^\infty H(it) dt \]
 be the direct integral.
  This Hilbert space consists of all functions
\[A:i\R\longrightarrow \bigcup_{t\in \R}H(it)\qquad(\text{disjoint union})\]
(identifying functions that are equal a.e.)
satisfying:
\begin{itemize}
\item $A(it)\in H(it)$ for all $t$,
\item  the composition
  $i\R{A\atop{\xrightarrow{\hspace*{.5cm}}\atop{}}} \bigcup H(it)\longrightarrow H(0)$,
   obtained by identifying $H(it)$ with $H(0)$ as in \S\ref{Ind}, is measurable,
\item $\ds\|A\|^2 \eqdef \frac1{\pi}\int_{-\infty}^\infty \|A(it)\|^2 dt<\infty$.
\end{itemize}
The associated inner product is given by
\begin{equation}\label{ipcont}
\sg{A,B} = \frac1\pi\int_{-\infty}^\infty
  \sg{A(it),B(it)}dt.
\end{equation}
   Define an action of $G(\A)$ on $H$ by
\[(gA)(it) =\pi_{it}(g)A(it).\]
  This representation is unitary since each $\pi_{it}$ is unitary:
\[\|gA\|^2 = \frac1\pi\int_{-\infty}^\infty\|\pi_{it}(g)A(it)\|^2 dt = \|A\|^2.\]
 By \S 4 of \cite{GJ}, there is a $G(\A)$-equivariant Hilbert space 
  isomorphism\index{keywords}{Hilbert space isomorphism}\footnote{By an
  isomorphism of Hilbert spaces, we mean a bijective linear isometry.}
\[M(it): H(it)\longrightarrow H(-it).\]
Thus the subspace
\begin{equation} \label{mathcalL} \index{notations}{L@$\mathcal{L}$}
\mathcal{L}=\{A\in H|\, A(-it)=M(it)A(it)\text{ for all }t\in \R\}
\end{equation}
 is stable under the action of $G(\A)$.  It is this representation which is
  isomorphic to $L^2_{\cont}(\w)$:

\begin{theorem}\label{L2}
  Consider the orthogonal decomposition
\[L^2(\w) = L^2_{\disc}(\w)\oplus L^2_{\cont}(\w).\]
  There is a $G(\A)$-equivariant isomorphism of Hilbert spaces \index{notations}{S@$S$}
\[S: L^2_{\cont}(\w)\longrightarrow \mathcal{L},\]
which we extend to the full space $L^2(\w)$ by taking $S=0$ on $L^2_{\disc}(\w)$,
characterized by the property that for any $\psi\in L^2(\w)$ and $\phi\in H(0)$,
\begin{equation}\label{516}
 \sg{S\psi(it),\phi_{it}} = \frac 12 \sg{\psi,E(\phi,it, \cdot)}
=\frac 12 \int_{\olG(\Q)\bs \olG(\A)} \psi(g)\ol{E(\phi,it, g)} \, dg
 \end{equation}
for almost all $t$.
The following Parseval identity\index{keywords}{Parseval's identity}
   holds for $\psi,\eta\in L^2(\w)$:
\begin{align}\label{parseval}
\notag \sg{\psi,\eta}=&\sum_{\varphi}\sg{\psi,\varphi}\ol{\sg{\eta,\varphi}}+
  \frac{3}{\pi}\sum_{\chi^2=\w}\sg{\psi,\phi_\chi}\ol{\sg{\eta,\phi_\chi}}\\
  &+\frac{1}{4\pi}\sum_{\phi}\int_{-\infty}^\infty \sg{\psi,E(\phi,it,\cdot)}
  \ol{\sg{\eta,E(\phi,it,\cdot)}}dt.
\end{align}
Here, $\varphi$ (resp. $\phi$) runs through an orthonormal basis
 for $L^2_0(\w)$ (resp. $H(0)$).
\end{theorem}
\noindent{\em Remarks:}
 The fact that $S$ is an intertwining operator can be seen from \eqref{516}.  Indeed,
  for any $\phi\in H(0)$,
\[\sg{SR(g)\psi(it),\phi_{it}}=\frac12\sg{R(g)\psi,E(\phi_{it},\cdot)}
  =\frac12\sg{\psi,R(g^{-1})E(\phi_{it},\cdot)}\]
\[=\frac12\sg{\psi,E(\pi_{it}(g^{-1})\phi_{it},\cdot)}=\sg{S\psi,\pi_{it}(g^{-1})\phi_{it}}
  =\sg{\pi_{it}(g)S\psi(it),\phi_{it}}\]
as claimed.  Passing to the second line, we used $R(g)E(\phi_s,x)=E(\pi_s(g)\phi_s,x)$,
  which is clear when $\Re(s)>1/2$ and holds for $\Re(s)=0$ by analytic continuation.
  In a similar fashion,
     we can derive the useful identity
\begin{equation}\label{move}
\sg{\psi,E(\pi_{it}(f)\phi_{it},\cdot)}
  =\sg{R(f)^*\psi,E(\phi_{it},\cdot)}
\end{equation}
for $\psi\in L^2(\w)$ and $f\in L^1(\ol{\w})$.


\begin{proof}  See \S4 of \cite{GJ} for an explicit construction of $S$.  The identity
  \eqref{516} is their (5.16).  We just explain how to
  derive the Parseval identity from their discussion.
  Let $P_{\operatorname{disc}}$\index{notations}{P@$P_{\disc}$, projection onto
  $L^2_{\disc}(\w)$}
   (resp. $P_{\cont}$)\index{notations}{P@$P_{\cont}$, projection onto $L^2_{\cont}(\w)$}
   be the orthogonal projection of $L^2(\w)$ onto $L^2_{\disc}(\w)$
  (resp. $L^2_{\cont}(\w)$).  Then
\[\sg{\psi,\eta} = \sg{P_{\operatorname{disc}}\psi, P_{\operatorname{disc}}\eta}
   + \sg{P_{\cont}\psi, P_{\cont}\eta}.\]
We apply the usual Parseval identity in $L^2_{\disc}(\w)$ to obtain the discrete part of
  \eqref{parseval}.  For the continuous part, by \eqref{ipcont} we have
\[
 \sg{P_{\cont}\psi, P_{\cont}\eta} = \sg{S\psi, S\eta}
=\frac1\pi \int_{-\infty}^\infty
  \sg{S\psi(it),S\eta(it)}dt
\]
\[ =\frac1{\pi}\int_{-\infty}^\infty \sum_{\phi}
   \sg{S\psi(it), \phi_{it}} \ol{\sg{S\eta(it),\phi_{it}}}\, dt.\]
Here $\phi$ runs through an orthonormal basis for $H(0)$,
 and we have applied Parseval's identity in $H(it)$.
 We pull the sum out (justification given below) and apply \eqref{516} to get
\begin{equation}\label{Pcont}
 \sg{P_{\cont}\psi, P_{\cont}\eta} = \frac 1{4\pi} \sum_{\phi}
\int_{-\infty}^\infty \sg{\psi, E(\phi,it, \cdot)}
  \ol{\sg{\eta,E(\phi,it, \cdot)}}\, dt,
\end{equation}
which gives \eqref{parseval}.  To justify pulling out the sum, we need to show
  convergence of
\[\int_{-\infty}^\infty \sum_\phi|\sg{S\psi(it), \phi_{it}}
  \ol{\sg{S\eta(it),\phi_{it}}}|\, dt.\]
Applying Cauchy-Schwarz to the sum, the above is
\[\le \int_{-\infty}^\infty \Bigl(\sum_\phi|\sg{S\psi(it), \phi_{it}}|^2\Bigr)^{1/2}
  \Bigl(\sum_\phi|\sg{S\eta(it),\phi_{it}}|^2\Bigr)^{1/2}\, dt\]
\[= \int_{-\infty}^\infty \|S\psi(it)\|\,\|S\eta(it)\|\,dt\quad(\text{Parseval's})\]
\[\le \left[\int_{-\infty}^\infty \|S\psi(it)\|^2dt\right]^{1/2}
 \left[\int_{-\infty}^\infty \|S\eta(it)\|^2dt\right]^{1/2}<\infty
  \quad(\text{Cauchy-Schwarz}).\qedhere\]
\end{proof}

\subsection{Kernel functions}

Suppose $X$ is a Radon measure space, and $T$ is a bounded \index{keywords}{kernel function}
  linear operator on $L^2(X)$.  We say that a measurable function $K(x,y)$
  on $X\times X$ is a {\bf kernel function} for $T$ if $T=T_K$, where \index{notations}{T K@$T_K$}
\[T_K\psi(x)\eqdef\int_{X} K(x,y)\,\psi(y)\,dy.\]
  If the equality $T\psi=T_K\psi$ is only known to hold for all
  $\psi$ which are bounded and of compact support,
  then we say that $K(x,y)$ is a
  \label{weakker}{\bf weak kernel} for $T$. \index{keywords}{weak kernel function}
  We shall repeatedly use the fact that $K$ is a weak kernel for $T$ if and
  only if $\sg{T_K\psi_1,\psi_2}=\sg{T\psi_1,\psi_2}$ for all
  $\psi_1,\psi_2$ which are bounded and compactly supported.

\begin{lemma}\label{weak} If $K(x,y)$ and $K'(x,y)$ are weak kernel functions for
  $T$, then $K(x,y)=K'(x,y)$ for almost all $(x,y)\in X\times X$.
\end{lemma}
\begin{proof}
This is straightforward; see the proof of Proposition 15.1 of \cite{KL}.
\end{proof}

Given an operator $T$ on a Hilbert space,
  its {\bf Hilbert-Schmidt norm}\index{keywords}{Hilbert-Schmidt norm} is
  defined by
\[\|T\|_{HS}^2=\sum \|Te_i\|^2,\]
  where $\{e_i\}$ is an orthonormal basis for the space.
  If the norm is finite, then it is independent of the choice of basis,
  and we say that $T$ 
   is a {\bf Hilbert-Schmidt operator}\index{keywords}{Hilbert-Schmidt operator}.
It is well-known that an operator $T$ on $L^2(X)$
  is Hilbert-Schmidt if and only if
  it has a square integrable kernel $K(x,y)\in L^2(X\times X)$
  (\cite{RS}, Theorem VI.23).  In this situation,
  if we let $\{\psi\}$ and $\{\phi\}$ be orthonormal bases for $L^2(X)$,
  then $\{\psi\otimes \ol{\phi}\}$ is an orthonormal basis for $L^2(X\times X)$
  (\cite{RS}, p. 51) and for almost all $(x,y)$ we have\index{notations}{Ker@$K(x,y)$}
\[K(x,y)=\sum_{\psi,\phi}\sg{K,\psi\otimes\ol{\phi}}\psi(x)\ol{\phi(y)}
  =\sum_{\phi}\Bigl(\sum_{\psi}\sg{T\phi,\psi}\psi(x)\Bigr)\ol{\phi(y)}\]
\begin{equation}\label{HS}
=\sum_{\varphi}T\phi(x)\ol{\phi(y)}.
\end{equation}

For an integer $m\ge 0$, let $C^m_c(G(\A),\ol{\w})$ denote the 
  space\index{notations}{C cm@$C^m_c(G(\A),\ol{\w})$}
  of factorizable functions $f=f_\infty \prod_p f_p$ on $G(\A)$ 
  with the following properties:
\begin{itemize}
\item $f$ has compact support mod $Z(\A)$
\item $f$ transforms under $Z(\A)$ by $\ol{\w}$
\item $f_\infty$ is $m$-times continuously differentiable on $G(\R)$
\item Each $f_p$ is locally constant, and for almost all $p$, $f_p$ is 
  the function supported on $Z(\Q_p)K_p$ defined by $f_p(zk)=\ol{\w_p(z)}$.
\end{itemize}

\begin{theorem}\label{tc}
  Suppose $m\ge 3$.  Then for any $f\in C_c^m(G(\A),\ol{\w})$,
  the operator $R_0(f)$ on $L^2_0(\w)$ is Hilbert-Schmidt.
  When $f_\infty$ is bi-$K_\infty$-invariant, $m\ge 2$ suffices.
\end{theorem}
\begin{proof}
  In the case of interest to us here, where $f_\infty$ is bi-$K_\infty$-invariant,
  we will prove that $m\ge 2$ suffices
   in Corollary \ref{tcpf} later on, as a consequence of a more general
  result where we allow $f_\infty$ to have noncompact support.
  For the general case of $f\in C^m_c(G(\A),\ol\w)$, see Theorem 2.1 of
  \cite{GJ} for a sketch over the adeles, and \cite{Bu} or \S 3 of \cite{Kn}
  for proofs over $G(\R)^+$.  As can be seen from the proof in \cite{Kn}, $m=3$
  suffices.
\end{proof}

Let $f\in L^1(\ol\w)$.
  Then for all $\psi\in L^2(\w)$, we have
\[R(f)\psi(x)=\int_{\olG(\A)}f(y)\psi(xy)dy=\int_{\olG(\A)}f(x^{-1}y)\psi(y)dy\]
 \[ =\int_{\olG(\Q)\bs\olG(\A)}K(x,y)\psi(y)dy\]
    for the kernel function
    \begin{equation}\label{K} \index{notations}{Kerf@$K_f(x,y)$, kernel function of $R(f)$}
 K(x,y)=K_f(x, y) = \sum_{\gamma \in \olG(\Q)} f(x^{-1} \gamma y).
 \end{equation}
   If $f$ is continuous and compactly supported modulo the center, then because
   $\olG(\Q)$ is a discrete subset of $\olG(\A)$, the sum is locally finite, so
  $K(x, y)$ is a continuous function on $G(\A)\times G(\A)$.

The expression \eqref{K} is the geometric form of the kernel.
When $f\in C^m_c(G(\A),\ol\w)$ for $m$ sufficiently large (we will prove in Corollary
 \ref{c} that $m=8$ suffices), the kernel
  also has a spectral expansion of the following form,
  valid almost everywhere in $G(\A)\times 
G(\A)$:\index{notations}{Kercont@$ K_{\cont}(x, y)$}
  \index{notations}{Kercusp@$K_{\cusp}(x,y)$}\index{notations}{Kerres@$K_{\res}(x,y)$}
\begin{equation}\label{Kdec}
 K(x, y) = K_{\cont}(x, y) + K_{\cusp}(x,y) + K_{\res}(x,y).
\end{equation}
Here
\[K_{\cont}(x,y)=\frac1{4\pi}\sum_\phi\int_{-\infty}^\infty E(\pi_{it}(f)\phi_{it},x)
  \ol{E(\phi_{it},y)}dt\]
for an orthonormal basis $\{\phi\}$ for $H(0)$,
\begin{equation}\label{Kcusp}
K_{\cusp}(x,y) =\sum_{\varphi}R(f)\varphi(x)\ol{\varphi(y)}
\end{equation}
 as in \eqref{HS}
for an orthonormal basis $\{\varphi\}$ for $L^2_0(\w)$, and
\begin{align}\notag K_{\res}(x,y)&
=\frac3\pi\sum_{\chi^2=\w} R(f)\phi_\chi(x)\ol{\phi_\chi(y)}
\\
\notag  &=\frac{3}{\pi}\sum_{\chi^2=\w}\chi(\det x)
  \ol{\chi(\det y)}\int_{\olG(\A)}f(g)\chi(\det g)dg\\
  \label{Pres}
  &=\begin{cases}\frac{3}{\pi}\int_{\scriptscriptstyle \olG(\A)}f(g)dg
  &\text{if }\w\text{ is trivial,}\\
  0&\text{otherwise.}\end{cases}
\end{align}
To see \eqref{Pres}, notice that in the integral on the previous line,
  if we replace $g$ by $gk$ for $k\in K_1(N)$, a factor of $\chi(\det k)$ comes out.
  So the integral vanishes unless $\chi$ is trivial on $\Zhat^*$.
  Since $\chi$ has finite order, this is possible only if $\chi$ is trivial
  (since $\Q$ has class number $1$),
  which means that $\w=\chi^2$ is also trivial.

  If $f=f_\infty\ff$ is a weight $0$ Hecke operator as in \eqref{Hop},
  then using \eqref{Sel2} and \eqref{Mp}, \eqref{Pres} becomes
\begin{equation}\label{residue}
\frac3\pi\int_{\olG(\R)}f_\infty(g)dg
  \int_{\ol{M_1(\n,N)}}\ff(m)dm=\frac3\pi h(\tfrac i2)\prod_{p|\n}\sum_{j=0}^{\n_p}p^j
=
  \frac3 \pi h(\tfrac i2) \sum_{d|\n}d
\end{equation}
 for the Selberg transform $h$ of $f_\infty$.

For such a Hecke operator,
  we will derive \eqref{Kdec} from the spectral
  decomposition in Theorem \ref{L2} using a nice choice of basis,
  and show that for this choice it is in fact valid for {\em all} $(x,y)$.
  This is a special case of a result of Arthur (\cite{Ar2} \S4, culminating on p. 935).
  We need this fact because our principal objective is to derive a
  relative trace formula by integrating $K(x,y)$ over
\[(x,y)\in (N(\Q)\bs N(\A))^2,\]
  a space which has {\em measure zero} in $(\olG(\Q)\bs\olG(\A))^2$, so an
  almost-everywhere spectral expression 
  for $K(x,y)$ is not adequate.

\subsection{A spectral lower bound for $K_{h*h^*}(x,x)$}

In this section we will take $f=h*h^*$ for a suitable function
 $h$,    
  and give a spectral lower bound for $K_f(x,x)$ in Proposition \ref{hhcase}.
We begin with the following lemma.

\begin{lemma}[\cite{GGK}, Lemma 5.2.1]\label{ggk}
Let $X$ be a Radon measure space,  and let $T$ be an operator
  on $L^2(X)$.  Suppose there is a continuous weak kernel function
  $K(x,y)$ for $T$.  Suppose further that
  for all bounded compactly supported $\psi$,
  \[\sg{T\psi,\psi}\ge 0.\]
  Then $K(x,x)\ge 0$ for all $x\in X$.
\end{lemma}

\begin{proof}
Suppose for some $x$ that $\Re K(x,x)<0$.  By continuity there
  exists a compact neighborhood $U\subset X$ of $x$
  such that $\Re K(a,b)<0$ for all
  $(a,b)\in U\times U$.  Let $\psi$ be the characteristic function of $U$.
  Then
\[0\le \sg{T\psi,\psi}=\int_X T\psi(a)\ol{\psi(a)}da
  =\int_U\int_UK(a,b)db\,da.\]
The right-hand side has negative real part, which is a contradiction.
  Therefore $\Re K(x,x)\ge 0$.
  By a similar argument, we find also that $\Im K(x,x)=0$.
\end{proof}

\begin{proposition}\label{hhcase}
Let $h\in C_c^m(G(\A),\ol{\w})$ be a bi-$K_\infty\times K_1(N)$-invariant
  function for $m\ge 2$.  Choose orthonormal bases $\{\varphi\}$ and
  $\{\phi\}$ for $L^2_{\disc}(\w)$ and $H(0)^{K_\infty\times K_1(N)}$
   respectively, consisting of
  continuous functions.  Then for all $x\in \olG(\A)$,
\begin{equation}\label{Kxx}
\sum_{\varphi} |R(h)\varphi(x)|^2+\frac1{4\pi}\sum_{\phi}\int_{-\infty}^\infty
  |E(\pi_{it}(h)\phi_{it},x)|^2dt\le K_{h*h^*}(x,x).
\end{equation}
Here $K_{h*h^*}(x,y)$ is the geometric kernel defined in \eqref{K}.
\end{proposition}

\noindent{\em Remark:}
  The set $\{\phi\}$ can be extended to an orthonormal basis for all of $H(0)$
  in \eqref{Kxx}.  Indeed, because $h$ is $K_\infty\times K_1(N)$-invariant,
  by Lemma \ref{basic} $\pi_{it}(h)\phi_{it}$ vanishes when $\phi$ belongs to
  the orthogonal complement of the finite dimensional subspace
  \[H(0)^{K_\infty\times K_1(N)}=\bigoplus_{\chi_1\chi_2=\w\atop{\c_{\chi_1}\c_{\chi_2}|N}}
  H(\chi_1,\chi_2)^{K_\infty\times K_1(N)}.\]


We will prove the proposition in stages.  It is an application
  of Lemma \ref{ggk}, but complicated by the fact that we do not
  know {\em a priori} that the left-hand side of \eqref{Kxx} is
  continuous.  Thus we will approximate it by a
  partial sum, defined as follows.
  Fix an orthonormal subset $Q\subset L^2_{\disc}(\w)$,
 and let $J$ be a symmetric compact subset of $\R$.
  To these we attach the following function
\begin{equation}\label{K'}
K'(x,y)= K'_{\disc}(x,y)+K'_{\cont}(x,y),
\end{equation}
where
\[K'_{\disc}(x,y) = \sum_{\varphi\in Q}R(h)\varphi(x)
  \ol{R(h)\varphi(y)}\]
and
\[K'_{\cont}(x,y)= \frac1{4\pi}\sum_{\phi}\int_J
  E(\pi_{it}(h)\phi_{it},x)\ol{E(\pi_{it}(h)\phi_{it},y)}dt.\]
Here $\phi$ runs through an orthonormal basis
 for $H(0)^{K_\infty\times K_1(N)}$.

%

\begin{lemma}\label{T'}
There exists a bounded linear operator $T'_{\cont}$ on $L^2(\w)$ for
  which $K_{\cont}'$ is a weak kernel:   for all bounded $\psi\in L^2(\w)$
  with compact support modulo $Z(\A)G(\Q)$,
\begin{equation}\label{TK}
T_{\cont}'\psi(x)=\int_{\olG(\Q)\bs\olG(\A)}
  K_{\cont}'(x,y)\psi(y)dy
\end{equation}
for almost all $x$.  The analogous statement for $K_{\disc}'$ also holds.
\end{lemma}

\begin{proof}
For any measurable symmetric subset $J\subset \R$, define
\[\mathcal{L}_J = \int_J H(it)dt \cap \mathcal{L},\]
  where $\mathcal{L}$ was defined in \eqref{mathcalL}.
  Here we regard each element of the direct integral
  as a function on all of $\R$, taking the value $0$ at points outside $J$.
  It is easy to see that $\mathcal{L}_J$ is a closed $G(\A)$-invariant
  subspace of $\mathcal{L}$, and we have the orthogonal decomposition
\[\mathcal{L}=\mathcal{L}_J\oplus \mathcal{L}_{\R-J}.\]
  We denote the analogous decomposition in $L^2_{\cont}(\w)\cong\mathcal{L}$ by
\[L^2_{\cont}(\w)=L_J\oplus L_{\R-J}.\]
  Define a $G(\A)$-equivariant map $S_J:L^2(\w)\longrightarrow \mathcal{L}_J$ by
\[S_J\psi =\begin{cases} S\psi& \text{if }\psi\in L_J
  \\0&\text{if }\psi\in (L_J)^\perp,\end{cases}\qquad\text{i.e.}\qquad
S_J\psi(it) =\begin{cases} S\psi(it)& \text{for a.e. }t\in J
  \\0&\text{for a.e. } t\notin J.\end{cases}\]
  Its restriction to $L_J$ is an isomorphism of Hilbert spaces.  The map
  \[P_J\eqdef (S_J)^*\,S_J\]
  is the orthogonal projection of $L^2(\w)$ onto $L_J$, so $S_J=S\circ P_J$.

Now let $J$ be the given compact set.
  Define $T_{\cont}'= P_J\, R(h*h^*)\,P_J$.  
  It is a bounded operator because $\|T_{\cont}'\|\le \|R(h*h^*)\|\le \|h*h^*\|_{L^1}$
  (cf. \cite{KL}, p. 140).
  For bounded compactly supported $\psi_1,\psi_2\in L^2(\w)$,
\[\sg{T'_{\cont} \psi_1,\psi_2} = \sg{P_JR(h*h^*)P_J\psi_1,\psi_2}
  =\sg{R(h^*)P_J\psi_1,R(h^*)P_J\psi_2}\]
\[ =\sg{P_JR(h^*)\psi_1,P_JR(h^*)\psi_2} 
 =\sg{SP_JR(h^*)\psi_1,SP_JR(h^*)\psi_2} \]
\[=\sg{S_JR(h^*)\psi_1,S_JR(h^*)\psi_2}
=\frac1\pi\int_{J}\sg{SR(h^*)\psi_1(it),SR(h^*)\psi_2(it)}dt\]
\[=\frac1\pi\int_J\sg{\pi_{it}(h^*)S\psi_1(it),\pi_{it}(h^*)S\psi_2(it)}dt
\]
\[=\frac1\pi\int_J\sum_{\phi}\sg{S\psi_1(it),\pi_{it}(h)\phi_{it}}
 \ol{\sg{S\psi_2(it),\pi_{it}(h)\phi_{it}}}\,dt\]
\begin{equation}\label{Tpsipsi}
=\frac1{4\pi}\int_J\sum_{\phi}\sg{\psi_1,E(\pi_{it}(h)\phi_{it},\cdot)}
  \ol{\sg{\psi_2,E(\pi_{it}(h)\phi_{it},\cdot)}}\,dt
\end{equation}
%
\[\hskip -.2cm =\hskip -.3cm\IL{\olG(\Q)\bs\olG(\A)}\hskip -.1cm
  \left[\,\,\IL{\olG(\Q)\bs\olG(\A)}
  \hskip -.2cm\left\{\frac1{4\pi}\sum_{\phi}
  \int_{J} E(\pi_{it}(h)\phi_{it},x)\ol{E(\pi_{it}(h)\phi_{it},y)}dt
  \right\}\psi_1(y)dy\right]\ol{\psi_2(x)}dx.\]
  The interchange of the sum and integrals is justified by Fubini's theorem,
  since the Eisenstein series are continuous, $J$ is compact,
  the sum over $\phi$ is finite, and since $\psi_1,\psi_2$ are bounded with compact
  support modulo $Z(\A)G(\Q)$.
  This proves \eqref{TK}.

For $K'_{\disc}$, the statement is much easier because
$K'_{\disc}(x,y)$ is square integrable over $(\olG(\Q)\bs\olG(\A))^2$,
  so for almost all $x$ the expression $\int K'_{\disc}(x,y)\psi(y)dy$
 is meaningful for {\em all} $\psi\in L^2(\w)$,
  and serves to define $T'_{\disc}\psi(x)$.
To see the square integrability, note that by the Cauchy-Schwarz inequality,
\[|K'_{\operatorname{disc}}(x,y)|^2 \leq (\sum_{\varphi \in Q}
 |R(h)\varphi(x)|^2) (\sum_{\varphi \in Q}
|R(h)\varphi(y)|^2).\]
Therefore
\[ \iint_{(\olG(\Q) \bs \olG(\A))^2}
 |K'_{\disc}(x,y)|^2 dx \,dy \leq
   (\sum_{\varphi \in Q} \|R(h)\varphi\|^2) (\sum_{\varphi \in Q} \|
   R(h)\varphi\|^2),\]
which is finite since $R(h)$ is a Hilbert-Schmidt operator on $L^2_{\disc}(\w)$
  by Theorem \ref{tc}.
(On $L^2_{\res}(\w)$ it actually has finite rank as shown in \eqref{Pres}.)
\end{proof}

\begin{proposition} Let $T'=T'_{\disc}+T'_{\cont}$ with notation as
  in the above lemma.   Suppose that the orthonormal set
  $Q\subset L^2_{\disc}(\w)$ is finite,
  and that $J\subset \R$ is compact and symmetric.  Then
   \[ \sg{T' \psi, \psi} \leq \sg{R(h*h^*) \psi, \psi} \]
for all bounded $\psi\in L^2(\w)$ of compact support modulo $Z(\A)G(\Q)$.
\end{proposition}

\begin{proof}
Extend $Q$ to an orthonormal basis $\widetilde{Q}$ of $L^2_{\disc}(\w)$.
  We have
\[ \sg{T'_{\disc} \psi, \psi} =
 \IL{\olG(\Q) \bs \olG(\A)}\Bigl( \IL{\,\olG(\Q) \bs \olG(\A)}
\sum_{\varphi\in Q}R(h)\varphi(x)
  \ol{R(h)\varphi(y)}\psi(y)dy\Bigr)\ol{\psi(x)}dx.\]
The sum can be pulled out because of the conditions placed on $\psi$ and
  since $Q$ is finite.  So the above is
\[= \sum_{\varphi \in Q}  |\sg{R(h)\varphi,\psi}|^2 \leq
  \sum_{\varphi\in \widetilde{Q}} |\sg{R(h)\varphi,\psi}|^2
  =\sum_{\varphi\in\widetilde{Q}} |\sg{\varphi,R(h)^*\psi}|^2 \]
\[   =\sg{P_{\disc}R(h)^*\psi,P_{\disc}R(h)^*\psi}
  =\sg{P_{\disc}R(h*h^*)\psi,\psi}.
\]
Passing to the last line, we applied Parseval's identity \eqref{parseval},
  while the last equality follows easily by the fact that $R(h)$ commutes with the
  orthogonal projection $P_{\disc}$.

  Likewise,
  by \eqref{Tpsipsi},
\[  \sg{T'_{\cont} \psi, \psi}
=\frac{1}{4\pi}  \sum_{\phi}
\int_{J} |\sg{\psi, E(\pi_{it}(h) \phi_{it}, \cdot)}|^2 dt \]
\[ \leq \frac{1}{4\pi}  \sum_{\phi} \int_{-\infty}^\infty
   |\sg{\psi, E(\pi_{it}(h) \phi_{it}, \cdot)}|^2 dt
= \frac{1}{4\pi}  \sum_\phi \int_{-\infty}^\infty |\sg{R(h)^*\psi,
  E(\phi_{it}, \cdot)}|^2 dt \]
\[= \sg{P_{\cont}R(h)^*\psi,P_{\cont}R(h)^*\psi}=\sg{P_{\cont}R(h*h^*)\psi,\psi}.\]
Again we used Parseval's identity \eqref{Pcont} in passing to the
  last line.
We have also used \eqref{move}.
\end{proof}

\begin{proof}[Proof of Proposition \ref{hhcase}]
Let $Q$ be a finite subset of
  the given orthonormal basis $\{\varphi\}$ of $L^2_{\disc}(\w)$, and let $J$
  be a symmetric compact subset of $\R$.
Let $K'(x,y)$ be the associated partial kernel function as above, and set $f=h*h^*$.
  Then $K'(x,y)$ is continuous since all $\varphi$ and $\phi$ are continuous
  by hypothesis.  On the other hand,
  we saw in \eqref{K} that $K_f(x,y)$ is also continuous.
  By the above proposition,
  $\sg{(R(f)-T')\psi,\psi}\ge 0$ for all bounded
  $\psi\in L^2(\w)$ of compact support modulo the center.
  Hence by Lemma \ref{ggk}, $K_f(x,x)-K'(x,x)\ge 0$
  for all $x\in G(\A)$.  It follows that
\[\sup_{Q,J} K'(x,x)\le K_f(x,x).\]
The proposition now follows, since the supremum
  is precisely the left-hand side of \eqref{Kxx}.
\end{proof}


\subsection{The spectral form of the kernel of $R(f)$}\label{specsec}

The following lemma, which follows from a result of Duflo and Labesse,
  will enable us to reduce to the special situation $f=h*h^*$ discussed above.

\begin{lemma} \label{con} 
   Let $r\ge 1$, and suppose $f \in C_c^{4r}(G(\A),\ol{\w})$ is
  bi-invariant under $K_\infty\times K_1(N)$.
  Then there exist functions $h_1, h_2, k_1, k_2 \in C_c^{2r-2}(G(\A),\ol{\w})$
  which are also bi-invariant under $K_\infty\times K_1(N)$ such that
   \[ f = h_1*h_2 + k_1*k_2. \]
   \end{lemma}
\begin{proof}
   Write $K'=K_\infty\times K_1(N)$.  In this proof only, we normalize so that $\meas(K')=1$.
   A function $a(g)$ on $G(\A)$ is said to be $K'$-central if
   $a(kg) = a(gk)$ for all $k\in K'$.
  For any function $a(g)$ we define
   \[\tilde{a}(g) = \int_{K'} a(kg) dk.\]
   Obviously $\tilde a$ is left $K'$-invariant.  If $a$ is $K'$-central, then
   $\tilde{a}(g)$ is also right $K'$-invariant.

   Define an action of $\lie_\R=\text{Lie}(G(\R))$\index{notations}{g@$\lie_\R$ (Lie algebra of $\GL_2(\R)$)}
   on the smooth functions by
\begin{equation}\label{Xg}\index{notations}{15@$X*f=\ddt f(\exp(-tX)\,\cdot\,)$}
 X*f(g) = \ddt f(\exp(-tX)g).
\end{equation}
   This extends naturally to an action of the universal enveloping algebra $U(\lie_\C)$.
   By \cite{DL} (I.1.11),
 there exist
   $K'$-central functions $a\in C^{2r-2}_c(G(\A),\ol{\w})$, $b\in C^\infty_c(G(\A),\ol{\w})$,
   and a differential operator $D\in U(\lie_\C)$ of order 2, such that
   \[ f = a*(D^{r+1}*f)+b*f.\]
   Let $c = D^{r+1}*f$.  Because $f$ is $C^{4r}$, $c$ is $C^{4r-(2r+2)}=C^{2r-2}$.
  It follows from \eqref{Xg} that $c$ is right $K'$-invariant.
  By the left $K'$-invariance of $f$,
   \[  f(x) = \int_{K'} f(kx) dk = \int_{K'} \int_{\olG} a(g) c(g^{-1} k x) dg\,dk +
  \int_{K'}\int_{\olG} b(g) f(g^{-1} k x) dg\,dk \]
   \[=\int_{\olG} \int_{K'} a(kg) c(g^{-1} x)dk\, dg + \int_{\olG} \int_{K'} b(kg) f(g^{-1} x)dk\, dg
    = (\tilde{a}*c)(x) + (\tilde{b}*f)(x). \]
   Because $\tilde{a}$ is bi-$K'$-invariant, it is easy to verify that
  $\tilde{a}*{c} = \tilde{a}*\tilde{c}$.
   Therefore we can take $h_1 = \tilde{a}$, $h_2 = \tilde{c}$, $k_1 = \tilde{b}$ and $k_2=f$.
   \end{proof}

\begin{theorem} \label{bc1} 
 Let $f = f_\infty f^{\n}$, where $f_\infty\in C_c^{m}(G(\R)^+//K_\infty)$ for $m\ge 8$.
  Let $\mathcal{F}_\A$ be an orthonormal basis for $L^2_0(\w)^{K_\infty\times K_1(N)}$,
   chosen as in Proposition \ref{diag}.
  Then  both
  \[\sum_{\varphi \in \mathcal{F}_\A} R(f)\varphi(x) \ol{\varphi(y)}
  \quad\text{ and } \quad
  \sum_{\varphi \in \mathcal{F}_\A} |R(f)\varphi(x) \ol{\varphi(y)}| \]
  are bounded on any compact subset of $G(\A)\times G(\A)$ and continuous in $x$ and $y$ separately.
\end{theorem}

\begin{proof} It suffices to prove the assertion for the expression with
   the absolute values.  Because $m\ge 8$, by Lemma \ref{con} there exist
$h_1, h_2,k_1,k_2\in C^{2}_c(G(\A),\ol{\w})$ such that $f=h_1*h_2+k_1*k_2$.
  By linearity and the triangle inequality, it suffices to prove the
  theorem for $f = h_1*h_2$.

 By Proposition \ref{diag}, for $\varphi \in \mathcal{F}_\A$
   we can write
\[R(f)\varphi=\lambda\varphi,\quad R(h_1)\varphi=\lambda_1\varphi,\quad R(h_2)\varphi=\lambda_2\varphi.\]
    Note that $\varphi$ is also an eigenvector of $R(h_1^*)$.
   The eigenvalue is $\ol{\lambda_1}$ since
    \[\sg{R(h_1^*)\varphi, \varphi} = \sg{\varphi, R(h_1)\varphi} =
  \ol{\lambda_1} \sg{\varphi,\varphi}.\]
Furthermore, $\lambda = \lambda_1 \lambda_2$ since
  \[\lambda\sg{\varphi,\varphi} = \sg{R(f)\varphi, \varphi} = \sg{R(h_2) \varphi, R(h_1^*) \varphi} =
    \lambda_1 \lambda_2 \sg{\varphi,\varphi}.\]
    This implies that
    \[ R(f) \varphi(x) \ol{\varphi(y)} = \lambda_1 \lambda_2 \varphi(x) \ol{\varphi(y)}
  = R(h_1) \varphi(x) \ol{R(h_2^*)\varphi(y)}. \]
By Cauchy-Schwarz, for any subset $S$ of $\mathcal{F}_\A$,
    \[ \sum_{\varphi \in S} |R(f) \varphi (x) \ol{\varphi(y)}|
 = \sum_{\varphi \in S} |R(h_1) \varphi (x) \ol{R(h_2^*)\varphi(y)}| \]
    \begin{equation}\label{Kineq}
 \leq \Bigl( \sum_{\varphi \in S} |R(h_1) \varphi (x)|^2 \Bigr)^{1/2}
   \Bigl( \sum_{\varphi \in S} |R(h_2^*) \varphi (y)|^2 \Bigr)^{1/2}
\end{equation}
    \[ \leq K_{h_1*h_1^*}(x,x)^{1/2} K_{h_2^**h_2}(y,y)^{1/2}. \]
 The last inequality holds by Proposition \ref{hhcase}.
  Because the two kernels are continuous,
  the above is bounded on any compact set.

Now we show that $\sum_\varphi|R(f)\varphi(x)\ol{\varphi(y)}|$
   is continuous in $y$ for fixed $x$.
  Let $U$ be any compact subset of $G(\A)$.
  Fix $x\in G(\A)$.  It suffices to show that the series converges
  uniformly as a function of $y\in U$.
   Let $C$ be an upper bound for $K_{h_2^**h_2}(y,y)^{1/2}$ on $U$.
  Fix $\e>0$.
  We know that
  $\sum_\varphi |R(h_1)\varphi(x)|^2<\infty$.
Hence for any ordering $\varphi_1,\varphi_2,\ldots$ of $\{\varphi\}$,
  there exists $N>0$ such that
\[\sum_{n\ge N}|R(h_1)\varphi_n(x)|^2 <\frac {\e^2}{C^2}.\]
  Therefore by \eqref{Kineq},
\[\sum_{n\ge N} |R(f)\varphi(x)\ol{\varphi(y)}|\le
  C\Bigl(\sum_{n\ge N}|R(h_1)\varphi_n(x)|^2\Bigr)^{1/2}<\e.\]
  Hence the series converges uniformly for $y\in U$, as needed.
  Similarly for fixed $y$, the sum is continuous in $x$.
\end{proof}

\begin{theorem} \label{bc2}
Let $f = f_\infty f^{\n}$ be as in the previous theorem.
  Then  both
  \begin{equation}\label{Kcont1}
 K_{\cont}(x,y)=\frac{1}{4\pi} \sum_{\phi}  \int_{-\infty}^\infty
         E(\pi_{it}(f)\phi_{it},x) \ol{E(\phi_{it},y)} dt
 \end{equation}
  and
  \begin{equation}\label{abscont}
      \frac{1}{4\pi} \sum_{\phi}  \int_{-\infty}^\infty
         |E(\pi_{it}(f)\phi_{it},x) \ol{E(\phi_{it},y)}| dt
\end{equation}
  are bounded on any compact subset of $G(\A)\times G(\A)$ and continuous in $x$ and $y$ separately.
  Here $\phi$ runs through an orthonormal basis for
   $H(0)^{K_\infty \times K_1(N)}$.
\end{theorem}

\begin{proof}
  The proof is similar to that of the previous theorem.  We can assume
  $f=h_1*h_2$.  For $j\ge 1$, let $R_j=[-j,-j+1]\cup [j-1,j]$, and define
\[G_j(x,y)= \sum_{\phi} \frac1{4\pi}\int_{R_j}
  \left|E(\pi_{it}(f)\phi_{it},x)\ol{E(\phi_{it},y)}\right|dt.\]
  It is a continuous function of $x$ and $y$.
  Note that \eqref{abscont} is equal to $\sum_{j}G_j(x,y)$.
  It suffices to show that for fixed $x$ this series converges uniformly for
  $y$ in a compact set.
  Write
\[\pi_{it}(h_1)\phi_{it}=\lambda^{\phi}_1(t)\phi_{it},\qquad
   \pi_{it}(h_2)\phi_{it}=\lambda^{\phi}_2(t)\phi_{it}.\]
For any set $S$ of natural numbers,
$\sum_{j\in S}G_j(x,y)$ is
\[\le \frac1{4\pi}\sum_{j\in S}\sum_\phi
  \Bigl(\int_{R_j}
  |\lambda^\phi_1(t)E(\phi,it,x)|^2dt\Bigr)^{1/2}
\Bigl(\int_{R_j}
  |\lambda^\phi_2(t)E(\phi,it,y)|^2dt\Bigr)^{1/2}\]
\[\le\Bigl(\frac1{4\pi}\sum_{j\in S}
  \sum_\phi \int_{R_j}|\lambda_1^\phi(t) E(\phi,it,x)|^2dt\Bigr)^{1/2}
 \Bigl(\frac1{4\pi}\sum_{j\in S}
  \sum_\phi \int_{R_j}|\lambda_2^\phi(t)
  E(\phi,it,y)|^2dt\Bigr)^{1/2}\]
\[\le K_{h_1*h_1^*}(x,x)^{1/2}K_{h_2^**h_2}(y,y)^{1/2}\]
  by Proposition \ref{hhcase}.  The proof now proceeds as before.
\end{proof}

Now we derive the spectral formula for the kernel $K(x,y)$ of $R(f)$.
Because
\[R(f) = R(f)P_{\disc} + R(f)P_{\cont},\]
it suffices to give kernel functions for each of the operators on the
  right-hand side.
  The operator $R(f)P_{\disc}$ is Hilbert-Schmidt, so its kernel is
  given by \index{notations}{Kerdisc@$K_{\disc}(x,y)$}
\[K_{\disc}(x,y)=\sum_{\{\varphi\}\subset L^2(\w)\atop{\text{o.n.b.}}}
  R(f)P_{\disc}\varphi(x)\ol{\varphi(y)}.\]
Because $R(f)P_{\disc}$ annihilates all $\varphi\in L^2_{\cont}(\w)$,
  the above is equal to the sum $K_{\cusp}(x,y)+K_{\res}(x,y)$
  as in \eqref{Kcusp} and \eqref{Pres}.
For $R(f)P_{\cont}$, suppose $\psi_1,\psi_2\in L^2(\w)$ are
   bounded and compactly supported modulo $Z(\A)G(\Q)$.   Then
\begin{equation}\label{cont}
\sg{R(f)P_{\cont}\psi_1,\psi_2}=\sg{P_{\cont}\psi_1,P_{\cont}R(f^*)\psi_2}
\end{equation}
\[  =\frac1{4\pi}\sum_\phi\int_{-\infty}^\infty \sg{\psi_1,E(\phi_{it},\cdot)}
  \sg{E(\phi_{it},\cdot),R(f^*)\psi_2}dt\]
\[=\int_{\olG}\left[\int_{\olG}\left\{\frac1{4\pi}\sum_\phi
  \int_{-\infty}^\infty E(\pi_{it}(f)\phi_{it},x)\ol{E(\phi_{it},y)}dt\right\}
  \psi_1(y)dy\right]\ol{\psi_2(x)}dx.\]
We used Parseval's identity \eqref{Pcont} when passing to the second line,
  and \eqref{move} when passing to the third line.
  The convergence is absolute
  by Theorem \ref{bc2} and the conditions on $\psi_1,\psi_2$, so the
  rearrangement of the sum and integrals is justified.
  It follows that the expression in the braces, which coincides with
  \eqref{Kcont1}, is a weak kernel function for $R(f)P_{\cont}$.

\begin{corollary} \label{c}  Suppose $f_\infty\in C^m_c(G(\R)^+//K_\infty)$
  for $m\ge 8$, and let $f=f_\infty\ff$.  Then for all $x,y\in G(\A)$,
\[ K(x,y) = K_{\cusp}(x,y) + K_{\res}(x,y) + K_{\cont}(x,y), \]
where we choose bases as in Theorems \ref{bc1} and \ref{bc2}.
  Each function on the right is separately continuous in each
 variable.
\end{corollary}

\begin{proof}
Denote the right-hand side by $\Psi(x,y)$.  As we have just shown, $\Psi$ is a weak
  kernel function for $R(f)$.  By Lemma \ref{weak} we conclude that
  $K(x,y)=\Psi(x,y)$ almost everywhere in $G(\A) \times G(\A)$.
  We know that $K(x,y)$ is continuous.
By the above theorems, $\Psi(x,y)$ is continuous in $x$
  and $y$ separately.  By Lemma \ref{equal} below, it follows that
  $\Psi(x,y)=K(x,y)$ for all $x$ and $y$.
\end{proof}

\begin{lemma}\label{equal}
 Let $X$ and $Y$ be two positive Borel measure spaces.
  Let $D$ be a measurable function on $X\times Y$ such that
  $D(x,y) = 0$ almost everywhere and $D(x,y)$ is a continuous function of x and y separately.
  Then $D(x,y) = 0$ for all $x$ and $y$.
\end{lemma}
\begin{proof}
 Because $\int_X \int_Y |D(x,y)| dy dx = 0$,
 the set $\{x\in X |\, \int_Y |D(x,y)|dy > 0\}$ has measure zero.
 Let $S\subset X$ denote its complement.
 For fixed $x' \in S$,  $D(x',y)=0$ for almost all $y \in Y$.
By the continuity of $y \mapsto D(x',y)$,  $D(x',y) = 0$ for all $y \in Y$.
Therefore $S\times Y\subset \{(x,y) | D(x,y) = 0\}$.
This means that for any $y \in Y$, $D(x,y)=0$ for all $x \in S$,
  i.e. for almost all $x\in X$.
Now by the continuity of $x \mapsto D(x,y)$, it follows that $D(x,y)=0$
   for all $x \in X$ and all $y\in Y$.
\end{proof}

%

\pagebreak
\section{A Fourier trace formula for $\GL(2)$}\label{ftf}

For integers $m_1,m_2, \n>0$, we will compute a variant of
  the Kuznetsov/Brug\-geman trace formula, involving
  Fourier coefficients at $m_1,m_2$,
  the eigenvalues of $T_\n$, and Kloosterman sums.

Let $f=f_\infty\times \ff$,  with $f_\infty\in C_c^m(G^+//K_\infty)$ for $m\ge 8$
  to allow for use of Corollary \ref{c}
  (though for the convergence of the cuspidal term we will take $m\ge 10$).
  For real numbers $y_1,y_2>0$ and $K(x,y)$ as in \eqref{K},
  consider the expression
\begin{align}\label{I} \index{notations}{I@$I$}
I=
\frac{1}{\sqrt{y_1y_2}}\hskip -.1cm\IIL{(N(\Q)\bs N(\A))^2}\hskip -.3cm
  K(n_1\smat {y_1}{}{}1, n_2\smat{y_2}{}{}1)\,
  \ol{\theta_{m_1}(n_1)}\,\theta_{m_2}(n_2)\, dn_1 dn_2,
\end{align}
where
\[\theta_m(\smat{1}x01)=\theta_m(x)=\theta(-mx)\]
for the standard character $\theta$ defined by \eqref{theta},
  and $dn_j$ is the Haar measure of total volume $1$.
  We will compute the relative trace formula obtained by evaluating
  the above in two ways, using the geometric and spectral
  expressions for the kernel.
The result is a primitive Kuznetsov formula
   given as Theorem \ref{main1}.
  The variables $y_1,y_2$ give us some extra flexibility.
To obtain a more refined formula, we will set
\begin{equation}\label{w}
y_1m_1=y_2m_2=w
\end{equation}
in the primitive formula, and then integrate $w$ from $0$ to $\infty$.
  The result is Theorem \ref{main}, which is a generalized
  Kuznetsov formula.

\subsection{Convergence of the spectral side}\label{conv}

According to Corollary \ref{c},
\[K(x,y)= K_{\cusp}(x,y) + K_{\cont}(x,y)+K_{\res}(x,y),\]
  where each term on the right is separately continuous in each
  variable.  Each term is also bounded
  on the compact set $(N(\Q)\bs N(\A))^2$
  by Theorems \ref{bc1} and \ref{bc2}, and hence integrable there.
  Furthermore, the sums defining $K_{\cusp}$ and $K_{\cont}$ can
  be pulled out of the double integral for the same reason.

  The justification for integrating over $w$ will be handled later.

\subsection{Cuspidal contribution}

Here we will compute the cuspidal term \index{notations}{Icusp@$I_{\cusp}$}
\[I_{\cusp}=
\frac1{\sqrt{y_1y_2}}\IIL{(N(\Q)\bs N(\A))^2}
  K_{\cusp}(n_1\mat {y_1}{}{}1, n_2 \mat{y_2}{}{}1)\,
  \ol{\theta_{m_1}(n_1)}\theta_{m_2}(n_2) dn_1 dn_2.\]
  By Lemma \ref{basic}, $R_0(f)$ annihilates the orthogonal complement
  of $L^2_0(\w)^{K_\infty\times K_1(N)}$.
  Let $\mathcal{F}_\A$ be the eigenbasis of
  $L^2_0(\w)^{K_\infty\times K_1(N)}$ defined in Proposition \ref{diag},
  so that for $\varphi\in\mathcal{F}_\A$
  we have $R(f)\varphi(x)=\sqrt{\n}\,h(t)\lambda_\n(\varphi)\phi(x)$.
  Then $K_{\cusp}(x,y)$ equals
\[\sqrt{\n}\sum_{\varphi_j\in \mathcal{F}_\A}
 \frac{h(t_j)\,\lambda_\n(\varphi_j)\,\varphi_j(x)\ol{\varphi_j(y)}}
  {\|\varphi_j\|^2}\index{notations}{Kercusp@$K_{\cusp}(x,y)$}
=\sqrt{\n}\sum_{u_j\in\mathcal F} \frac{h(t_j)
   \lambda_\n(u_j)\, \varphi_{u_j}(x)\ol{\varphi_{u_j}(y)}}{\|u_j\|^2}\]
for $\mathcal{F}$ as in \eqref{Fbasis}.
%
As explained in Section \ref{conv} above, $I_{\cusp}$ is absolutely convergent, and
  by Fubini's Theorem
\[\hskip -.3cm I_{\cusp} =\sqrt{\frac{\n}{y_1y_2}}\sum_j \frac{h(t_j)
   \lambda_\n(u_j)}{\|u_j\|^2} \IL{\Q\bs\A}\varphi_{u_j}\left(\smat {y_1}x01
  \right)
  \ol{\theta_{m_1}(x)}\,dx
\IL{\Q\bs\A}\ol{\varphi_{u_j}\left(\smat {y_2}x01\right)}
  {\theta_{m_2}(x)}dx.\]

\begin{lemma}\label{flem} Let $u$ be a Maass cusp form with Fourier expansion
  as in \eqref{fourier}.  Then for $r\in\Q$ and $y\in \R^*$, \index{notations}{am@$a_m(u)$, $m$-th Fourier coefficient of $u(z)$}
\[ \int_{\Q\bs\A}\varphi_u(\smat yx01)
  \ol{\theta_r(x)}dx =
  \begin{cases} a_r(u)y^{1/2}K_{it}(2\pi |r|y) &\text{if } r\in \Z-\{0\}
  \\0&\text{otherwise.}\end{cases}\]
\end{lemma}

\begin{proof}
Using the fundamental domain $[0,1]\times \Zhat$ for $\Q\bs\A$ and \eqref{phiu},
  we have
\[ \int_{\Q\bs\A}\varphi_u(\smat yx01)
  \ol{\theta_r(x)}dx =
  \int_0^1 u(x+iy) \theta_\infty(rx)dx\int_{\Zhat}\theta_{\fin}(rx)dx.\]
The second integral on the right vanishes unless $r\in \Z$, in which case
  it is equal to $1$.  Assuming $r\in\Z$,
  this becomes $\int_0^1 u(x+iy)e^{-2\pi i rx}dx$,
 and the assertion then follows by substituting the Fourier
  expansion \eqref{fourier} of $u$.
\end{proof}

\begin{lemma}\label{disc} Let $\{u_j\}$ be an orthogonal basis for $L^2_0(N,\w')$
  consisting of cusp forms.  Let $t_j$ be the spectral parameter of $u_j$.
  Then for any $M>0$,
\begin{equation}\label{discrete}
\Bigl|\{j:\,|t_j|\le M\}\Bigr|<\infty.
\end{equation}
\end{lemma}
\noindent{\em Remark:} Much more is known.  According to Weyl's 
  Law\index{keywords}{Weyl's Law} (which in this context
  follows from the Selberg trace formula),
\begin{equation}\label{Weyls}
\Bigl|\{j:\,|t_j|\le M\}\Bigr|=
\frac{\operatorname{vol}(\Gamma_0(N)\bs\mathbf{H})}{4\pi}M^2+O(N^{1/2}M\log(NM))
\end{equation}
  (\cite{IK} p. 391, \cite{Sel3} p. 668).
\begin{proof}
  Let $h(iz)\in PW^4(\C)^{\text{even}}$.
  By Proposition \ref{ST}, there exists a function
  $f_\infty\in C_c^{2}(G^+//K_\infty)$ whose Selberg transform is $h(t)$.
  Let $f'=f_\infty\times f^1$, where $f^1$ is the
  Hecke operator on $G(\Af)$ with $\n=1$.  By Proposition \ref{diag},
  the operator $R_0(f')$ is diagonalizable with eigenvalues $h(t_j)$.
  By Theorem \ref{tc}, this operator is Hilbert-Schmidt.
  Therefore
  \[\sum_{j}|h(t_j)|^2<\infty.\]
  On the other hand, if \eqref{discrete} fails to hold,  the
  set $\{t_j:\, |t_j|\le M\}$ has a limit point $P\in\C$.
  Choosing $h$ so that $h(P)\neq 0$ would then contradict the above summability.
  It remains to show that such $h$ exists.  If $P=0$, we can let $h$ be the 
  Mellin transform of a nonzero element $\Phi\in C^\infty_c(\R^+)^w$ that assumes
  only nonnegative real values.  Then $h(0)=\int_0^\infty\Phi(y)\frac{dy}y>0$,
  as needed.  Now suppose $P\neq 0$.  Let $h_1\in PW(\C)^{\text{even}}$ be nonzero,
   with $h_1(Q)\neq 0$, say.  By continuity, we may assume
  that $Q\neq 0$.  Then we can take $h(z)={h_1(\frac{Q}{P}z)}$.
\end{proof}

\begin{corollary}\label{disccor}
The set of exceptional\index{keywords}{exceptional parameter}\index{keywords}{spectral parameter!exceptional}
 spectral parameters $t_j\notin \R$ is finite.
\end{corollary}
\begin{proof} If $t_j$ is exceptional, then by Proposition \ref{specprop},
  $t_j=-is$ for some real $s\in (-\tfrac12,\tfrac12)$.  In particular,
  $|t_j|<\tfrac12$, and the above lemma shows that the set of such $t_j$ is
  finite.
\end{proof}

\begin{lemma}\label{htc}
With $t_j$ as in Lemma \ref{disc}, we have
\[  \sum_j|h(t_j)|<\infty\]
 for any function $h(iz)\in PW^m(\C)^{\text{even}}$ with $m\ge 10$.
\end{lemma}
\noindent{\em Remark:} More is known.
Using Weyl's Law \eqref{Weyls}, it is straightforward to show that $1+|t_j|\gg {j}^{1/2}$.
Therefore $\sum |h(t_j)|\ll \sum (1+|t_j|)^{-m}\ll\sum j^{-m/2}<\infty$ if $m>2$.

\begin{proof}
Let $f'=f_\infty\times f^1$ be the global function attached to $h$ as
  in the proof of Lemma \ref{disc}.
Let $\varphi_j=\frac{\varphi_{u_j}}{\|\varphi_{u_j}\|}\in L^2_0(\w)$ be the unit
  vector attached to $u_j$.  Noting that $f_\infty\in C_c^{8}(G^+//K_\infty)$
  by Proposition \ref{ST}, we can write $f'=a*b+c*d$ for
  bi-$K_\infty$-invariant functions $a,b,c,d\in C^{2}_c(G(\A),\ol\w)$ by Lemma \ref{con}.
  Then by \eqref{Rfeigen},
\[\sum_j|h(t_j)|=\sum_j|\sg{R_0(f')\varphi_j,\varphi_j}|=\sum_j|\sg{R_0(a)R_0(b)\varphi_j,\varphi_j}
  +\sg{R_0(c)R_0(d)\varphi_j,\varphi_j}|\]
\[ \le \sum_j|\sg{R_0(b)\varphi_j,R_0(a^*)\varphi_j}|+\sum_j|\sg{R_0(d)\varphi_j,R_0(c^*)\varphi_j}|\]
\[ \le \sum_j\|R_0(b)\varphi_j\|\|R_0(a^*)\varphi_j\|+\sum_j\|R_0(d)\varphi_j\|\|R_0(c^*)\varphi_j\|\]
\vskip -.3cm
\begin{align*}
\le
 \Bigl(\sum_j\|R_0(b)\varphi_j\|^2\Bigr)^{1/2}&\Bigl(\sum_j\|R_0(a^*)\varphi_j\|^2\Bigr)^{1/2}\\
 &+\Bigl(\sum_j\|R_0(d)\varphi_j\|^2\Bigr)^{1/2}\Bigl(\sum_j\|R_0(c^*)\varphi_j\|^2\Bigr)^{1/2}.
\end{align*}
The above is finite since all four operators are Hilbert-Schmidt by Theorem \ref{tc}.
\end{proof}

\begin{proposition}\label{ktfcusp}
Given $h(iz)\in PW^{m}(\C)^{\text{even}}$ with $m\ge 10$,
let $f_\infty$ be its inverse spherical transform in $C^{m-2}(G^+//K_\infty)$
  as in Proposition \ref{ST}, and let $f=f_\infty\times \ff$.
Then the integral $I_{\cusp}$
 is absolutely convergent.
  It vanishes unless $m_1,m_2\in\Z-\{0\}$.  Let $\mathcal{F}$ be
  an orthogonal eigenbasis for $L^2_0(N,\w')$ as in \eqref{Fbasis}.
    Then for $m_1,m_2\in\Z^+$, we have
\begin{equation}\label{cuspys} \index{notations}{Icusp@$I_{\cusp}$}
\hskip -.2cm I_{\cusp}=\sqrt{\n} \sum_{u_j\in \mathcal{F}}
 \frac{\lambda_\n(u_j)\, a_{m_1}(u_j) \ol{a_{m_2}(u_j)}}
 {\|u_j\|^2}
  {h(t_j) K_{it_j}(2\pi m_1y_1)\,{K_{it_j}(2\pi m_2y_2)}},
\end{equation}
the sum converging absolutely.  Now suppose $m\ge 12$.  Then 
 letting $I_{\cusp}(w)$ denote the above when $w=y_1m_1=y_2m_2$,
we have
\begin{equation}\label{Icusp}
\int_0^\infty I_{\cusp}(w)dw=
  \frac{\pi\sqrt{\n}}8 \sum_{u_j\in \mathcal{F}}
 \frac{\lambda_\n(u_j)\, a_{m_1}(u_j) \ol{a_{m_2}(u_j)}}
 {\|u_j\|^2}
  \frac{h(t_j)}{\cosh(\pi t_j)},
\end{equation}
the sum and integral converging absolutely.
\end{proposition}

\noindent{\em Remark:} 
In fact, \eqref{Icusp} converges absolutely for any function $h$ (holomorphic or not)
  satisfying a bound of the form $h(t)\ll\frac1{(1+|t|)^m}$ for $m>3$.
 Indeed, granting Weyl's Law, one finds as
  in the proof below that \eqref{Icusp} is $\ll \sum |h(t_j)|(1+|t_j|)
  \ll \sum (1+|t_j|)^{-m+1}\ll\sum j^{(-m+1)/2}<\infty$ if $m>3$.
  By the fact that Weyl's Law is an exact asymptotic, any improvement allowing smaller
  $m$ (say $m>2$) must come from a strengthening of the estimate \eqref{fcbound}.

\begin{proof}
The absolute convergence of $I_{\cusp}$ (with absolute values inside the sum
  defining $K_{\cusp}$) is a consequence of the continuity result of Theorem
  \ref{bc1} and the compactness of $N(\Q)\bs N(\A)$.  The equality \eqref{cuspys}
  and fact that $m_1,m_2$ must be integers follow immediately from Lemma \ref{flem}
  and the discussion preceding it.  The second
  Bessel factor does not need the complex conjugate because
  $t_j$ is purely imaginary or purely real (cf. Proposition
  \ref{LK}), so that $K_{it_j}(x)$ is real for real $x$ by \eqref{KBessel}
  (if $t$ is real, consider $w\mapsto w^{-1}$ in that equation).

Equation \eqref{Icusp} follows formally from \eqref{cuspys} by the
   identity\index{notations}{Ks@$K_s(z)$, Bessel function}\index{keywords}{Bessel function!$K$-}
\begin{equation}\label{Kint}
\int_0^\infty K_{it}(2\pi w)^2 dw = \frac{\Gamma(\tfrac12+it)\Gamma(\tfrac12-it)}8
  =\frac{\pi}{8\cosh(\pi t)}
\end{equation}
(\cite{GR}, 6.576), which is valid whenever $\Im(t)=\Re(it)<\tfrac12$, which holds
  here by Proposition \ref{LK}.
  As justification, we have to prove that the integral
\[
\int_0^\infty \sqrt{\n} \sum_{u_j\in \mathcal{F}}
 \frac{|\lambda_\n(u_j)\, a_{m_1}(u_j) \ol{a_{m_2}(u_j)}|}
 {\|u_j\|^2}
  {|h(t_j) K_{it_j}(2\pi w)^2|dw}
  \]
converges.  By \eqref{Kint} and the fact that $|K_{it}(2\pi w)|^2=K_{it}(2\pi w)^2$,
this amounts to showing that the right-hand side of \eqref{Icusp} 
  is absolutely convergent.

Recall that $|\lambda_\n(u_j)|$ is bounded by a constant
  depending only on $\n$.
  (For an elementary proof, see \cite{appendix}, Proposition 2.9.
  \label{KS}Currently the best known bound is
  $\tau(\n)\n^{7/64}$ due to Kim and Sarnak \cite{KS}.  According to
   the Ramanujan conjecture,\index{keywords}{Ramanujan conjecture}
 $|\lambda_{\n}(u_j)|\le \tau(\n)$.)

For any Maass cusp form $u$ with spectral parameter $t$, and
   $m\neq 0$, we have the well-known elementary bound
\[\frac{|a_n(u)|^2}{\|u\|^2}\ll \psi(N)(|t|+\frac{|n|}N)\,e^{\pi|t|},\]
where the implied constant is absolute (see Theorem 3.2 of \cite{Iw}).
This gives
\begin{equation}\label{fcbound}
\frac{|a_{m_1}(u)\ol{a_{m_2}(u)}|}{\|u\|^2} \ll_{N,m_1,m_2} (1+|t|)e^{\pi|t|}.
\end{equation}
When $t$ is real, the exponential factor is negated by
  $\cosh(\pi t)\gg e^{\pi|t|}$ in the denominator of \eqref{Icusp}.
  For the finitely many non-real $t_j$, we have $|t_j|\le \tfrac12$.  Hence
\begin{equation}\label{cuspsum}
\sum_{u_j\in\mathcal{F}}  \frac{| \lambda_\n(u_j)\, a_{m_1}(u_j)
  \ol{a_{m_2}(u_j)}h(t_j)|} {\|u_j\|^2\cosh(\pi t_j)}\ll \sum_{j}|h(t_j)|(1+|t_j|).
\end{equation}
Let $\tilde{h}(t)=t^2h(t)$.  Assuming $h(iz)$ is Paley-Wiener of order $m\ge 12$, 
  it is easy to show that
  $\tilde{h}(iz)\in PW^{10}(\C)^{\text{even}}$.  Note that
\[\sum_{|t_j|\ge2} |h(t_j)|(1+|t_j|)< \sum_{|t_j|\ge 2}|t_j|^2|h(t_j)|=
  \sum_{|t_j|\ge 2}|\tilde{h}(t_j)|.\]
By Lemma \ref{htc} above, the latter expression is finite.
In view of \eqref{discrete} with $M=2$,
  this implies that the right-hand side of \eqref{cuspsum} is finite.
\end{proof}

\subsection{Residual contribution}

By \eqref{Pres}, $K_{\res}$ vanishes when $\w$ is nontrivial.
  Otherwise by \eqref{residue} we have \index{notations}{Ires@$I_{\res}$}
\[I_{\res}(f,m_1,m_2,w)= \frac1{\sqrt{y_1y_2}}\IIL{(\Q\bs\A)^2}
  K_{\res}(\smat {y_1}{x_1}01,\smat {y_2}{x_2}01)
  \theta_{m_1}(x_1)\,\ol{\theta_{m_2}(x_2)}\, dx_1\,dx_2\]
\[
= \frac{3h(\frac i2)}{\pi\sqrt{y_1y_2}}\Bigl(\sum_{d|\n}d\Bigr)
\int_{\Q\bs\A}\ol{\theta(m_1x)}dx
  \int_{\Q\bs\A}{\theta(m_2 x)}dx.
\]
Both integrals vanish since $m_1,m_2\neq 0$.
Thus the residual spectrum makes no contribution to this trace formula.

\subsection{Continuous contribution}

We continue to assume that $f=f_\infty\times \ff$ satisfies the hypotheses of
  \eqref{I}.
Using the decomposition $H(0)=\bigoplus H(\chi_1,\chi_2)$, the continuous kernel
 is given by\index{notations}{Kercont@$ K_{\cont}(x, y)$}
 \begin{equation}\label{Kcont2}
K_{\cont}(g_1,g_2)=\frac{1}{4\pi} \sum_{\chi_1,\chi_2}\sum_{\phi}
   \int_{-\infty}^\infty
   E(\pi_{it}(f)\phi_{it}, g_1) \ol{E(\phi_{it}, g_2)} dt.
 \end{equation}
Here $\chi_1,\chi_2$ ranges through all (ordered) pairs of finite order Hecke characters
  satisfying $\chi_1\chi_2=\w$, and $\phi$ ranges through an orthonormal basis
  for $H(\chi_1,\chi_2)^{K_\infty\times K_1(N)}$.
By Corollary \ref{nonzero}, the latter space is
  nonzero if and only if $\c_{\chi_1}\c_{\chi_2}|N$.
  In particular, both sums above are finite.

  For $\phi\in H(\chi_1,\chi_2)^{K_\infty\times K_1(N)}$ we know that
  \[\pi_{it}(f)\phi_{it}=\sqrt{\n}\,\lambda_{\n}(\chi_1,\chi_2,it)h(t)\phi_{it}\]
 by Proposition \ref{phieigen}.
  Hence for the orthogonal basis $\mathcal{B}(\chi_1,\chi_2)$
 given in Corollary \ref{nonzero},
\[K_{\cont}(g_1,g_2)=\]
\begin{equation}\label{Kcont}
\frac{\sqrt{\n}}{4\pi} \sum_{\chi_1,\chi_2}
  \sum_{\phi\in \mathcal{B}(\chi_1,\chi_2)}\frac1{\|\phi\|^2}
   \int_{-\infty}^\infty h(t)
\lambda_{\n}(\chi_1,\chi_2,it)
   E(\phi, it, g_1) \ol{E(\phi, it,g_2)} dt.
\end{equation}

We need to integrate the above over $(N(\Q)\bs N(\A))^2$.
   For $m\in\Q$ and real $y>0$, let \index{notations}{ams@$a_m(s,y)=a_m^{\phi}(s,y)$, $m$-th Fourier coefficient of $E_\phi(s,z)$}
\[a_m^\phi(s,y) = \int_{N(\Q)\bs N(\A)} E(\phi,s,n\smat y{}{}1)\,
   \ol{\theta_m(n)}dn\]
 be the adelic Fourier coefficient of $E(\phi,s,n\smat y{}{}1)$.
  The above coincides with
  the $m^{\text{th}}$ Fourier coefficient of $E_\phi(s,z)$,
  which was denoted $a_m(s,y)$ earlier.
  Indeed, using the fundamental domain $[0,1]\times\Zhat$ for $\Q\bs\A$,
\begin{align*} a_m^\phi(s,y)&=
\int_0^1\int_{\Zhat} E(\phi,s,\smat yx01\times \smat 1u01)e(-mx)\theta_{\fin}(mu)\,du\,dx\\
&=
\int_0^1 E(\phi,s,\smat yx01\times 1_{\fin})e(-mx)dx
  \int_{\Zhat}\theta_{\fin}(mu)\,du.
\end{align*}
  Because $\Zhat=\ker\theta_{\fin}$,  this is
\begin{align*}
 &=\left\{\begin{array}{cl}\ds
\int_0^1 E_\phi(s,x+iy)e(-mx)\,dx&m\in\Z\\
0&\text{otherwise}\end{array}\right.
\end{align*}
$= a_m(s,y)$, as claimed.

 Therefore by the formula for Fourier coefficients given
  in \S \ref{Fourier}, the continuous contribution to the trace formula 
  is\index{notations}{Icont@$I_{\cont}$}
\[I_{\cont}=\frac1{\sqrt{y_1y_2}}\IIL{(\Q\bs \A)^2}
  K_{\cont}(\smat{y_1}{x_1}{}{1},\smat{y_2}{x_2}{}1)\ol{\theta_{m_1}(x_1)}
  \theta_{m_2}(x_2) dx_1dx_2\]
  \[= \frac{\sqrt{\n}}{4\pi\sqrt{y_1y_2}} \sum_{\chi_1,\chi_2}
  \sum_{\phi}
  \int_{-\infty}^{\infty}
 \frac{\lambda_{\n}(\chi_1,\chi_2,it)}{\|\phi\|^2}
 h(t) a_{m_1}^\phi(it,y_1)
   \, \ol{a_{m_2}^\phi(it,y_2)}\, dt\]
\begin{equation}\label{contys}
\hskip -.3cm =  {\scri\sqrt{\n}}{\sum\limits_{\chi_1,\chi_2}
  \sum\limits_{(i_p)}}
 \IL{\R} \frac{\scri \lambda_{\n}(\chi_1,\chi_2,it)\sigma_{it}(\chi_1',\chi_2',m_1)
  \ol{\sigma_{it}(\chi_1',\chi_2',m_2)} (\frac{m_1}{m_2})^{it}
 K_{it}(2\pi m_1 y_1) {K_{it}(2\pi m_2 y_2)}h(t)}
   {\scri \|\phi_{(i_p)}\|^2|\Gamma(\frac{1}{2}+it)|^2|L_N(1+2it,\chi_1 \ol{\chi_2})|^2}dt.
\end{equation}
The notation is as in \S \ref{Fourier}, and is also recalled in Theorem
  \ref{main} below.  We have used the fact that $|C_{(i_p)}|=1$
  and that $K_{it}(x)$ is real.

We explained in \S\ref{conv} why the above expression is absolutely convergent.
  This can also be seen directly, using the following lemma.

\begin{lemma}\label{besselbound} For $y>0$ and 
  $s=\sigma+it$ with $\sigma\ge 0$,\index{notations}{Ks@$K_s(z)$, Bessel function}
\index{keywords}{Bessel function!$K$-}
\[\frac{|K_{s}(y)|}{|\Gamma(\tfrac12+s)|}\le 
  \left(\frac2y\right)^{\!\sigma}\frac{(1+2|s|)}{y\sqrt{\pi}}.\]
\end{lemma}

\begin{proof}
By Basset's formula for $K_\nu(z)$ (eq. (1) on p. 172 of \cite{Wa}),
\[\frac{K_s(y)}{\Gamma(\tfrac12+s)}=
\Bigl(\frac 2y\Bigr)^{\!s}\,\frac1{\sqrt{\pi}}
  \int_0^\infty \frac{\cos(yx)}{(x^2+1)^{\frac12+s}}dx,\]
valid for $\sigma>0$.
Integrating by parts, we have
\[\frac{K_s(y)}{\Gamma(\tfrac12+s)}=
\Bigl(\frac 2y\Bigr)^{\!s}\,\frac{(1+2s)}{y\sqrt{\pi}}
  \int_0^\infty \frac{x\sin(yx)}{(x^2+1)^{\frac32+s}}dx.\]
This equality is valid on $\sigma>-\tfrac12$ since the right-hand side is convergent
  for such $s$.
  Therefore if $\sigma\ge 0$,
\[\frac{|K_{s}(y)|}{|\Gamma(\tfrac12+s)|}\le
\Bigl(\frac 2y\Bigr)^{\!\sigma}\,\frac{(1+2|s|)}{y\sqrt{\pi}}
\int_0^\infty \frac{x}{(x^2+1)^{\frac32}}dx=
\Bigl(\frac 2y\Bigr)^{\!\sigma}\,\frac{(1+2|s|)}{y\sqrt{\pi}}.\qedhere
\]
\end{proof}

%

By the lemma, the combined contribution of the Bessel and Gamma functions in the
  integrand of \eqref{contys} is $\ll (1+2|t|)^2$.
By Corollary \ref{Lcor},
\begin{equation}\label{Lbound}
L_N(1+2it,\chi_1\ol{\chi_2})^{-1}=L(1+2it,\ol{\chi_1'}\chi_2')^{-1}
\ll_\e N^\e(\log(3+2|t|))^7.
\end{equation}
 Thus, the absolute convergence of the integral \eqref{contys} follows from
\begin{equation}\label{PW}
h(t)\ll (1+|t|)^{-4}
\end{equation}
(which holds since $h(iz)$ is Paley-Wiener of order $m\ge 10>4$).

In fact we can prove the following asymptotic bound for $I_{\cont}$.

\begin{proposition}\label{bound1}
For any $\e>0$, the quantity \eqref{contys} is $\ll N^{\frac12+\e}$.
The implied constant is ineffective, but depends only on $\e$, $\n$, $m_1$,
  $m_2$, $f_\infty$, and $y_1,y_2$.
\end{proposition}

\begin{proof}
By Proposition \ref{Eisbound}, we have
\begin{equation}\label{sigitbound}
\frac{\left|\sigma_{it}(\chi_1',\chi_2',m_1)
  \ol{\sigma_{it}(\chi_1',\chi_2',m_2)}\right|} {\|\phi_{(i_p)}\|^2}
\ll_{m_1,m_2,\e} N^{\e}.
   \end{equation}
  It is clear that
\[|\lambda_{\n}(\chi_1,\chi_2,it)|=\Bigl|\sum_{d|\n}(\tfrac\n{d^2})^{it}
  \ol{\chi_1(d_N)\chi_2((\tfrac\n d)_N)}\Bigr| \le \tau(\n) \ll_{\n} 1.\]
  Therefore by 
  Lemma \ref{besselbound} and \eqref{Lbound}-\eqref{sigitbound},
   the integral occurring in \eqref{contys} is 
\begin{equation}\label{sb}
\ll N^{3\e}\int_{-\infty}^\infty{h(t)}(1+2|t|)^2(\log(3+2|t|))^{14}dt
  =O(N^{3\e}),
\end{equation}
 with the implied constant depending on $f_\infty, \e, y_1, y_2, m_1, m_2,$ and $\n$.
Thus, recalling Corollaries \ref{localnonzero} and \ref{nonzero},  \eqref{contys} is
\[\ll N^{3\e} \sum_{(i_p)\atop{0\le i_p\le N_p}} \sum_{(\nu_{1p},\nu_{2p})
  \atop{\nu_{2p}\le i_p\le N_p-\nu_{1p}}} \sum_{\chi_1\chi_2=\w\atop{\c_{\chi_j}=
  \prod_{p} p^{\nu_{jp}}\atop{(j=1,2)}}} 1.\]
The set of tuples $(i_p)$ is in 1-1 correspondence with the set of positive divisors
  $M=\prod p^{i_p}$ of $N$, and likewise $\{(\nu_{1p})\}\leftrightarrow
   \{\nu_1|\frac NM\}$ and $\{(\nu_{2p})\}\leftrightarrow \{\nu_2|M\}$.
   Hence the above triple sum can be rewritten
\[\sum_{M|N}\,\sum_{\nu_2|M}\,\sum_{\nu_1|\frac{N}M}\,\sum_{\chi_1\chi_2=\w\atop
  {\c_{\chi_1}=\nu_1\atop{\c_{\chi_2}=\nu_2}}}1.\]
The number of terms $M$ is $\ll N^\e$, and the same is true for $\nu_1$ and $\nu_2$.
Thus for fixed $M,\nu_1,\nu_2$, it remains to count the number of pairs $(\chi_1,\chi_2)$.
Because $\nu_1\nu_2|N$, there exists $j\in\{1,2\}$ such that $\nu_j \le N^{1/2}$.
   The number of possibilities for the character $\chi_j$ of conductor $\nu_j$
  is $|(\Z/\nu_j\Z)^*|\le N^{1/2}$.  Once $\chi_j$ is chosen, its
 counterpart is determined by $\chi_1\chi_2=\w$.
  Thus the number of pairs $(\chi_1,\chi_2)$ is $\le N^{1/2}$.
  This proves that \eqref{contys} is $\ll N^{\frac12+6\e}$.
\end{proof}

Lastly, let $I_{\cont}(w)$ denote the quantity in \eqref{contys} with
 $w=m_1y_1=m_2y_2$.
Using \eqref{Kint},
we see that $\ds\int_0^\infty I_{\cont}(w)dw$ is equal to
\begin{equation}\label{Icont}
 \frac{\sqrt{\n}}8 \sum_{\chi_1\chi_2=\w} \sum_{(i_p)}
\int_{-\infty}^{\infty} \frac{\lambda_\n(\chi_1,\chi_2,it)
  \sigma_{it}(\chi_1',\chi_2',m_1)\ol{\sigma_{it}(\chi_1',\chi_2',m_2)}
(\frac{m_1}{m_2})^{it}h(t)}
   {\|\phi_{(i_p)}\|^2|L_N(1+2it,{\chi_1} \ol{\chi_2})|^2} dt.
\end{equation}
The exchange of the $dt$ and $dw$ integrals is justified by
  the absolute convergence of \eqref{Icont}, which is proven 
  in the same way as Proposition \ref{bound1},
  giving the following.
\begin{proposition}\label{bound}
For any $\e>0$, the quantity \eqref{Icont} is $\ll_\e N^{\frac12+\e}$.
\end{proposition}

\subsection{Geometric side}

In this section, we take $f$ as in \eqref{I}, although we can relax the requirement
  on $m$.  In Proposition \ref{cell1}, we will require $m\ge 5$, and in Propositions
  \ref{cell2} and \ref{Klbound} we need $m>2$.

  For positive integers $m_1,m_2$, we need to evaluate the integral \eqref{I}
\[I=I_{f,m_1,m_2,y_1,y_2}=\frac{1}{\sqrt{y_1y_2}}\int_{\Q\bs\A}
\int_{\Q\bs \A} K(\smat{y_1}{t_1}01,
  \smat {y_2}{t_2}01)\theta(m_1t_1-m_2t_2)dt_1\,dt_2,\]
using $K(g_1,g_2)=\sum_{\g\in\olG(\Q)}f(g_1^{-1}\g g_2)$.
 This was carried out in detail with a different choice of $f_\infty$
  in \cite{KL1}, and we follow the same procedure here.

Let $H=N\times N$, and endow it with an action on $\olG$ by
  \[(n_1,n_2)\cdot \g = n_1^{-1}\g n_2.\]
We break the sum over $\g\in\olG(\Q)$ into $H(\Q)$-orbits to get
\[
I=\hskip-.4cm\sum_{\delta\in N(\Q)\bs \olG(\Q)/N(\Q)}
  \int_{H_\delta(\Q)\bs H(\A)}
  \hskip -.2cm f(\smat{y_1}{t_1}{}{1}^{-1}\delta \smat{y_2}{t_2}{}{1})
  \frac{\ol{\theta_{m_1}(t_1)}\theta_{m_2}(t_2)}{\sqrt{y_1y_2}} \,d(t_1,t_2).\]
Here $H_\delta$ denotes the stabilizer of $\delta$, \index{notations}{Hdelta@$H_\delta$}
  and $d(t_1,t_2)$ denotes the quotient measure coming from the
  Haar measure $dt_1dt_2$ on $H(\A)\cong \A\times \A$, the latter
  being normalized as in Section \ref{notation}.
The interchange of the sum and the integral is justified because
the function $\sum_\g |f(x^{-1}\g y)|$ is continuous and hence
  integrable over the compact set $H(\Q)\bs H(\A)$.

We let $I_\delta(f)$ denote the integral attached to $\delta$ as above. \index{notations}{Idelta@$I_\delta$}
The following set of representatives $\delta$ is obtained from the Bruhat decomposition:
   \[ \{ \left.\smat{\gamma}{0}{0}{1} \right|\, \gamma \in \Q^* \}
   \bigcup \{ \left.\smat{0}{-\mu}{1}{0} \right|\, \mu \in \Q^* \}. \]
An orbit $\delta$ is {\bf relevant} if the character
  $\ol{\theta_{m_1}}\theta_{m_2}$ is trivial on $H_\delta(\A)$.
  The orbital integral $I_\delta(f)$ vanishes if $\delta$ is not relevant.
  It is straightforward to show that
  the relevant orbits are
   \[ \{ \left.\smat{m_2/m_1}{0}{0}{1} \right.\}
   \bigcup \{ \left.\smat{0}{-\mu}{1}{0} \right|\, \mu \in \Q^* \}. \]
  See \cite{KL1} for details.

\subsubsection{First cell term}

\begin{proposition}\label{cell1}
  Let $\delta=\smat{m_2/m_1}{}{}1$ for $m_1,m_2>0$.  Then $I_\delta(f)$
  is nonzero only if $m_1m_2= b^2\n$ for some positive integer
  $b|\gcd(m_1,m_2)$.
  If this condition is met, then
\begin{equation}\label{cellys}
 I_\delta(f)= \frac{\psi(N)\ol{\w'(m_1/b)}}{b\,\sqrt{y_1y_2}}
\int_{-\infty}^\infty V\left(\frac{t^2+m_1^2y_1^2 +{m_2^2y_2^2}}{m_1y_1m_2y_2}
-2\right)e^{2\pi it}dt
\end{equation}
for $V$ as in \eqref{V}.  Letting $I_\delta(w)$ denote the above quantity
  for $w=y_1m_1=y_2m_2$, we have
\begin{equation}\label{Icell}
 \int_0^\infty I_\delta(w)dw =
\frac{\psi(N)\sqrt{\n}}{2\,\w'(\sqrt{{m_1\n}/{m_2}})}V(0).
\end{equation}
\end{proposition}

\begin{proof}
For this choice of $\delta$, we find as in \cite{KL1} that
\[H_\delta(\Q)=\left\{(\mat{1}{\frac{m_2}{m_1}t}{}1,\mat1t{}1)
  |\, t\in\Q\right\}.\]
Now note that
\[  \mat{y_1}{t_1}{}{1}^{-1}\mat{\frac{m_2}{m_1}}{}{}1 \mat{y_2}{t_2}{}{1}
  = \mat{\frac{m_2y_2}{m_1y_1}}{\frac1{y_1}(\frac{m_2}{m_1}t_2-t_1)}{}1\]
\[ =\mat{m_1y_1}{}{}{m_1y_1}^{-1}\mat{m_2y_2}{m_2t_2-m_1t_1}{}{m_1y_1}.\]
Here we view $m_j\in \Q^*\subset\A^*$ and $y_j\in\R^+$.
Thus because $f$ is invariant under $Z(\R)^+Z(\Q)$,
\[I_\delta(f)=\IIL{\left\{\left(\frac{m_2}{m_1}t_2,t_2\right)\in \Q^2\right\}\bs\A^2}
  f(\mat{m_2y_2}{m_2t_2-m_1t_1}{0}{m_1y_1})\frac{\theta(m_1t_1-m_2t_2)}{\sqrt{y_1y_2}}
  d(t_1,t_2).\]
Here $d(t_1,t_2)$ is the quotient measure coming from $dt_1dt_2$.
In $\A^2$, let $t=m_2t_2-m_1t_1$.  Then the map $\A^2\rightarrow \A^2$ defined
  by $(t_1,t_2)\mapsto (t,t_2)$ induces a homeomorphism between the quotient spaces
\[\left\{\left(\tfrac{m_2}{m_1}t_2,t_2\right)| t_2\in \Q\right\}\bs\A^2\longrightarrow
\{(0,t_2)|t_2\in \Q\}\bs \A^2=\A\times(\Q\bs\A).\]
 Noting that $dt\,dt_2=|m_1|_{\A}dt_1dt_2=dt_1dt_2$, we see that the 
  quotient measure is $d(t_1,t_2)=dt\,dt_2$, where we use $dt_2$ now to 
  represent the quotient measure on $\Q\bs\A$.  Thus
\[I_\delta(f)=\int_{\A}\int_{\Q\bs\A}
 f(\smat{m_2y_2}{t}{}{m_1y_1})
  \frac{\theta(-t)dt_2\,dt}{\sqrt{y_1y_2}}
=\int_\A f(\smat{m_2y_2}{t}{}{m_1y_1}) \frac{\theta(-t)}
 {\sqrt{y_1y_2}}dt.\]
  The integral factors as $I_\delta(f)_{\fin}I_\delta(f)_\infty$.
  As shown in Proposition 3.3 of \cite{KL1}, the finite part vanishes unless
  $m_1m_2=b^2\n$ for some positive integer $b|\gcd(m_1,m_2)$, in which case
\[I_\delta(f)_{\fin} = \frac{\psi(N)}{b\w'(m_1/b)}.\]
The archimedean part is
\[I_\delta(f)_\infty=\frac1{\sqrt{y_1y_2}}\int_{-\infty}^\infty
  f_\infty(\mat{m_2y_2}{t}{}{m_1y_1})e^{2\pi it}dt,\]
and \eqref{cellys} follows upon using \eqref{Vf}.


Set $w=y_1m_1=y_2m_2$ in \eqref{cellys}, so
 $\frac{1}{\sqrt{y_1y_2}}=\frac{\sqrt{m_1m_2}}w$.
  Then
\[I_\delta(w) =\frac{\psi(N)\sqrt{m_1m_2}}{b\w'(m_1/b)}\frac1w\int_{-\infty}^\infty
  V(\frac{t^2}{w^2})e^{2\pi i t}dt.\]
Replacing $t$ by $wt$ and $dt$ by $\frac{dt}w$, and using
   $\frac{\sqrt{m_1m_2}}b=\sqrt{\n}$, we have
\[\int_0^\infty I_\delta(w)dw=
 \frac{\psi(N)\sqrt{\n}}{\w'(\sqrt{m_1\n/m_2})}
\int_0^\infty\int_{-\infty}^\infty V(t^2)e^{2\pi i wt}dt\,dw.\]
Let $r(t)=V(t^2)$, a compactly supported continuous even function.  Note that
  \[\hat{r}(w)=\int_{-\infty}^\infty V(t^2)e^{-2\pi i wt}dt
  =\int_{-\infty}^\infty V(t^2)e^{2\pi i wt}dt
\]
 is also an even function.  By Fourier inversion,
\[\frac12V(0)=\frac12r(0)=\frac12\int_{-\infty}^\infty \hat{r}(w)dw
  =\int_0^\infty \hat{r}(w) dw,\]
which proves \eqref{Icell}.  We recall that Fourier inversion is valid so long as
  $r$ and $\hat{r}$ are both integrable.  Because $f_\infty\in C_c^m(G^+//K_\infty)$
  for $m\ge 5$, $V$ and $r$ are compactly supported and
  twice continuously differentiable 
  by Proposition \ref{fV}.  It follows by a standard argument (see Proposition
  \ref{1911} below) that $\hat{r}(w)
  \ll (1+w^2)^{-1}$ and hence is integrable.
\end{proof}

\subsubsection{Second cell terms}

We will need a few facts and definitions.
For the function $k(z_1,z_2)$
  attached to $f_\infty$ in \eqref{kzz}, \index{keywords}{Zagier transform}
  define its {\bf Zagier transform} by \index{notations}{Z@$\mathcal{Z}$}\index{notations}{Z@$Z(t)=\mathcal{Z}k(1,t)$}
\[\mathcal{Z}k(s,t)=\iint_{\mathbf H} k(z+t,\frac{-1}z) y^s dz,\]
  where $dz=\frac{dx\,dy}{y^2}$.
Using Proposition \ref{kV}, we see that
\[\mathcal{Z}k(s,t)= \iint_{\mathbf H}V\Bigl(\frac{|z^2+tz+1|^2}{y^2}\Bigr)y^sdz
=\iint_{\mathbf H}V\Bigl(\frac{|z^2+1-\frac{t^2}4|^2}{y^2}\Bigr)y^sdz,\]
where the second expression comes from completing the square and
  replacing $z$ by $z-\frac t2$.  We refer to Proposition 4 of \cite{Za} (where the
  above is denoted $V(s,t)$) for the absolute convergence and other information.
  In Section 5 of \cite{Za}, it is computed in terms of the Selberg transform $h(t)$.

We will only be interested in the case $s=1$, so we set
\begin{equation}\label{Zt}Z(t)=\mathcal{Z}k(1,t)
=\iint_{\mathbf H}V\left(\frac{|z^2+1-\frac{t^2}4|^2}{y^2}\right)\frac{dy}ydx.
\end{equation}
  This is expressed in terms of the Selberg transform $h(t)$ in
  (4.12) of \cite{Za}.
Since $V$ is compactly supported, $Z(t)$ is also compactly supported as
   a function of $t\in \R$. Indeed, writing
  $u=x^2$, $v=y^2$ and $w=-(1-\tfrac{t^2}4)$, we have
\[\frac{|z^2+1-\tfrac{t^2}4|^2}{y^2}=\frac{(u-v-w)^2+4uv}v=\frac{(-u-v+w)^2+4vw}v
  \ge 4w=t^2-4.\]
Thus, if $|t|$ is sufficiently large, $t^2-4$ exceeds the supremum of $\Supp(V)$,
  and the integrand of \eqref{Zt} is $0$.

  The orbital integral attached to $\delta =\smat 0{-\mu}{1}{0}$
involves the Fourier transform
\[\widehat{Z}(a)=\int_{-\infty}^\infty Z(t)e^{-2\pi i at}dt
=\int_{-\infty}^\infty Z(t)e^{2\pi i at}dt.\]

\begin{proposition}\label{Z} For $a\neq 0$, we have
\[\widehat{Z}(a) = \frac{i}{4a}\int_{-\infty}^\infty
   J_{2it}(4\pi a)\frac{h(t)\,t} {\cosh(\pi t)}\,dt\]
for the $J$-Bessel function, and the Selberg transform $h(t)$ of
  $f_\infty$.\index{notations}{J@$J_s$, $J$-Bessel function}
\index{keywords}{Bessel function!$J$-}
\end{proposition}

\begin{proof} This is due to Zagier.  The proof is explained
  in \S2.1 of \cite{Joy}.  Another account is given in \cite{LiX},
  Lemma 3.4.  The absolute convergence of the integral holds by the fact that
  $h\in PW^m(\C)^{\text{even}}$ for $m>2$ (see the proof of Proposition
  \ref{Klbound} below).
\end{proof}

Also, for any $c\in \c_{\w'}\Z^+$, we will need the following
  generalized Kloosterman sum: \index{keywords}{Kloosterman sum!generalized}
\begin{equation}\label{Kloos} \index{notations}{S k@$S_{\chi}(a,b;\n;c)$}
S_{\w'}(a,b;\n;c)=\sum_{\scri d,d'\in \Z/c\Z,\atop{\scri dd'=\n}}
  \ol{\w'(d)}e(\frac{ad+bd'}{c}).
\end{equation}
We will describe some basic properties of these sums in Section \ref{Klsec}.

\begin{proposition}\label{cell2}  Let $\delta=\smat{0}{-\mu}10$
  for $\mu\in\Q^*$.
  Then $I_\delta(f)$ is nonzero only if $\mu=\frac{\n}{c^2}$ for some positive
  integer $c\in N\Z$.  Under this condition,
\begin{equation}\label{cell2ys}
I_\delta(f) = {\psi(N)}\frac{S_{\w'}(m_2,m_1;\n;c)}
  {\sqrt{y_1y_2}} \IIL{\R\times\R}k(z_1,\frac{-\n}{c^2z_2})
  e^{2\pi i(m_2x_2-m_1x_1)}dx_1dx_2,
\end{equation}
where $k(z_1,z_2)=f_\infty(g_1^{-1}g_2)$ as in \eqref{kzz}.
Taking $I_\delta(w)$ to be the above quantity when $w=m_1y_1=m_2y_2$,
the integral
$\int_0^\infty I_\delta(w)dw$ equals
\begin{equation}\label{Icell2}
  \frac{i\sqrt{\n}\psi(N)}{4}\frac{S_{\w'}(m_2,m_1;\n;c)}{c}
  \int_{-\infty}^\infty J_{2it}\left(\frac{4\pi\sqrt{\n m_1m_2}}c\right)
  \frac{h(t)\,t}{\cosh(\pi t)}dt.
\end{equation}
\end{proposition}

\begin{proof}
For $\delta=\smat{}{-\mu}{1}{}$, we find that $H_\delta(\A)=\{(1,1)\}$.
  Given $y_1,y_2>0$, we need to integrate
\begin{equation}\label{fcell2}
f(\smat{y_1}{x_1}{}{1}^{-1}\smat{}{-\mu}{1}{}
  \smat{y_2}{x_2}{}{1})\theta(m_1x_1-m_2x_2).
\end{equation}
 Again the integral factors as $I_\delta(f)_{\fin} I_\delta(f)_\infty$.
Because $f_\infty$ is supported on $G(\R)^+$, the archimedean integral vanishes
  unless $\mu>0$.  Under this assumption, the finite part was shown
  in Proposition 3.7 of \cite{KL1} to vanish unless $\mu=\frac{\n}{c^2}$
  for some $c\in N\Z^+$, in which case
\begin{equation}\label{Jfin}
I_\delta(f)_{\fin} = \frac{\psi(N)}{\w'(-1)}S_{\w'}(m_2,m_1;\n;c)
  =\psi(N)S_{\w'}(m_2,m_1;\n,c).
\end{equation}
From \eqref{kzz},
  $f_\infty(\smat{y_1}{x_1}{}{1}^{-1}\smat{}{-\mu}{1}{}
  \smat{y_2}{x_2}{}{1})=k(z_1,\frac{-\mu}{z_2})$, so
 the archi\-medean part can be written as
\[I_\delta(f)_\infty = \frac{1}{\sqrt{y_1y_2}}\iint_{\R\times\R}
  k(z_1,\frac{-\mu}{z_2}) e(m_2x_2-m_1x_1)dx_1dx_2.\]
Setting $\mu=\frac{\n}{c^2}$, \eqref{cell2ys} follows immediately.

  Now write $x_1=\sqrt{\mu m_2/m_1}\, t_1$ and
  $x_2=\sqrt{\mu m_1/m_2}\, t_2$, so that $dx_1dx_2=\mu dt_1dt_2
  =\frac{\n}{c^2}dt_1dt_2$.
  Then
\[\left(x_1+iy_1,\frac{-\mu}{x_2+iy_2}\right)
  =\sqrt{\frac{\mu m_2}{m_1}}
  \left(t_1+\frac{iy_1m_1}{\sqrt{\mu m_1m_2}},\,\frac{-1}{t_2+\frac{im_2y_2}
  {\sqrt{\mu m_1m_2}}}\right).\]
 The scalar in front does not affect the value of $k$ by \eqref{kcz}.
Set $w=m_1y_1=m_2y_2$, so that
 $\frac1{\sqrt{y_1y_2}}= \frac{\sqrt{m_1m_2}}w$.  The archimedean part of
  the integral of $I_\delta(w)$ is the product of $\frac{\n\sqrt{m_1m_2}}{c^2}$ with
\[
\int_0^\infty \iint_{\R\times\R}
 k\Bigl(t_1+\frac{\scri iwc}{\scri \sqrt{\n m_1m_2}}, \frac{-1}{t_2+\frac{iwc}
 {\sqrt{\n m_1m_2}}}\Bigr)
e^{2\pi i\frac{\scri\sqrt{\n m_1m_2}}{\scri c}
  (t_2-t_1)}dt_1dt_2\frac{dw}w.\]
Arguing formally for the moment, we exchange the order of integration.
Substitute $t=t_1-t_2$ for $t_1$, $x$ for $t_2$, and $y$ for
 $\frac{wc}{\sqrt{\n m_1m_2}}$.  Then
  because $\frac{dw}w$ is a multiplicative Haar measure,
  the above integral is
\begin{equation}\label{triple}
=\int_{-\infty}^\infty\int_{-\infty}^\infty\int_0^\infty
  k\left(x+iy+t,\frac{-1}{x+iy}\right)\frac1ydy\,dx\,
  e^{-2\pi it\frac{\scri\sqrt{\n m_1m_2}}{\scri c}} dt
\end{equation}
\[=  \widehat{Z}(\frac{\sqrt{\n m_1m_2}}c)
=\frac{ic}{4\sqrt{\n m_1m_2}}
  \int_{-\infty}^\infty J_{2it}(\frac{4\pi\sqrt{\n m_1m_2}}c)\frac{h(t)\,t}
  {\cosh(\pi t)}dt\]
by Proposition \ref{Z}.
  Formula \eqref{Icell2} now follows upon multiplying by
  $\frac{\n\sqrt{m_1m_2}}{c^2}$ and the finite part \eqref{Jfin}.
The exchange of the order of integration is justified by Fubini's Theorem.
  Indeed, as explained after \eqref{Zt}, $\mathcal{Z}|k|(1,t)$ is compactly 
  supported as a function of $t$.  Therefore the triple integral \eqref{triple}
  is absolutely convergent.
\end{proof}

The geometric side is equal to the main term from Proposition \ref{cell1} plus
  the sum over $c\in N\Z^+$ of the term in Proposition \ref{cell2}.

\begin{proposition}\label{Klbound}  We have the following bound for the sum of
  the Kloosterman terms on the refined geometric side:
\begin{equation}\label{klas}
\sum_{c\in N\Z^+}
  \frac{i\sqrt{\n}\psi(N)}{4}\frac{S_{\w'}(m_2,m_1;\n;c)}{c}
  \int_{-\infty}^\infty J_{2it}\left(\frac{4\pi\sqrt{\n m_1m_2}}c\right)
  \frac{h(t)\,t}{\cosh(\pi t)}dt
\end{equation}
\[=O(N^{\e}),
\]
where the implied constant depends on $\n$, $h$, $m_1$, $m_2$, and $0<\e<1$.
\end{proposition}

\begin{proof}
(See also Theorem 16.8 on page 414 of \cite{IK} for the case $\n=1$, $\w=1$.)
We will show below that
\begin{equation}\label{Jb}
 \int_{-\infty }^{\infty} J_{2it}\left( \frac{4\pi \sqrt{\n m_1 m_2}}{c}\right) \frac{h(t)\,t}{\cosh(\pi t)} dt
   \ll \left( \frac{2\pi \sqrt{\n m_1 m_2}}{c}\right)^{1-\e}.
\end{equation}
In Theorem \ref{kloos} we will prove the Weil-type bound
\[|S_{\w'}(m_2,m_1;\n;c)|\le \tau(\n)\tau(c)(m_2\n,m_1\n,c)^{1/2}c^{1/2}\c_{\w'}^{1/2}.\]
Together these statements imply that \eqref{klas} is
\begin{equation}\label{KJb}
\ll\psi(N)\c_{\w'}^{1/2}\sum_{c\in N\Z^+}
  \frac{\tau(c)}{c^{3/2-\e}}\le
  \frac{\psi(N)N^{1/2}\tau(N)}{N^{3/2-\e}}\sum_{c\in \Z^+}
  \frac{\tau(c)}{c^{3/2-\e}}.
\end{equation}
Using $\tau(N)\ll N^\e$ and
  $\psi(N)=N\prod_{p|N}(1+\frac1p)\ll N^{1+\e}$,
this gives
\[\eqref{klas}\ll \frac{N^{1+\e}N^\e}{N^{1-\e}}=N^{3\e},\]
as needed.


It remains to establish \eqref{Jb}.
  We let $s=\sigma+it$ be a complex variable.
  Let $\sigma_0 < \tfrac12$ be a fixed positive number.
   The restriction on $\sigma_0$ is to ensure that $\cosh(-i\pi s)$ is nonzero on
   the strip $0\le \sigma \le \sigma_0$.
From the integral representation\index{notations}{J@$J_s$, $J$-Bessel function}
\index{keywords}{Bessel function!$J$-}
\[J_s(x)=\frac{(x/2)^s}{\Gamma(s+1/2)\sqrt{\pi}}\int_0^\pi\cos(x\cos\theta)\sin^{2s}\theta\,d\theta
  \qquad(\Re s >-\tfrac12)\]
(\cite{AAR}, Corollary 4.11.2),  we see that for $\sigma\ge 0$,
   \[ J_{s}(x)\ll \left( \frac x2 \right)^{\sigma} \frac{1}{|\Gamma(s+\frac 12)|} \]
for an absolute implied constant.
By the hypotheses on $h(t)$, there exists a positive constant $C$ such that 
   \[ h(-is) \ll \frac{C^{\sigma}}{(1+|t|)^M}, \]
for any $m\ge M\ge 2$.
By these asymptotics, the integrand of \eqref{Jb} is
\begin{equation} \label{geoJ}
   J_{2s}\left( \frac{4\pi \sqrt{\n m_1 m_2}}{c}\right) \frac{h(-is)(-is)}{\cosh(-is\pi)}
\end{equation}
  \[ \ll\left( \frac{2\pi \sqrt{\n m_1 m_2}}{c}\right)^{2\sigma}
  \frac{|s|C^{\sigma}}{(1+|t|)^M} \frac{1}{|\Gamma(2s+\frac 12)\cosh(-is\pi)|}.\]
By \cite{AAR}, Corollary 1.4.4,
    for $0 \le \sigma \le \tfrac12$ and $|t| \ge 1$,
   \[ |\Gamma(s)| = \sqrt{2\pi} |t|^{\sigma-1/2} e^{-\pi|t|/2} (1+O(1/|t|))\]
for an absolute implied constant.
Thus for $0 \le \sigma \le \sigma_0< \tfrac12$ and $|t| \ge 1$,
    \[ \frac{1}{|\Gamma(2s+\frac 12)\cosh(-is\pi)|} \ll |t|^{-2\sigma}
    \frac{e^{\pi |t|}}{|e^{(t-i\sigma)\pi} + e^{(-t+i\sigma)\pi}|}
  = O_{\sigma_0}(1). \]
    The left hand side is continuous and hence bounded on the compact set
  $0 \le \sigma \le \sigma_0$, $|t| \le 1$.
    Thus the expression is bounded on the whole strip $0 \le \sigma \le \sigma_0$.
  (We have imposed $\sigma_0<\tfrac12$ in order to avoid the zero 
  of $\cosh(-is\pi)=\cos(s\pi)$ at $s=\tfrac12$.)
    Hence for such $\sigma$,
\begin{equation} \label{Jbound}
   J_{2s}\left( \frac{4\pi \sqrt{\n m_1 m_2}}{c}\right) \frac{h(-is)(-is)}{\cosh(-is\pi)}
   \ll_{\sigma_0}\left( \frac{2\pi \sqrt{\n m_1 m_2}}{c}\right)^{2\sigma}
  \frac{|s|C^{\sigma}}{(1+|t|)^M}.
\end{equation}

Let $T$ be an arbitrary large real number, and let
   $\mathcal{R}_T$ be the contour defined by the rectangle with vertices
   $A=-iT$, $B=\sigma_0-iT$, $C=\sigma_0+iT$ and $D=iT$,
   with counter-clockwise orientation.
By the Cauchy residue theorem,
  \[ \int_{\mathcal{R}_T} J_{2s}\left( \frac{4\pi \sqrt{\n m_1 m_2}}{c}\right)
   \frac{h(-is)(-is)}{\cosh(-is\pi)} dt  = 0. \]
  By the estimate \eqref{Jbound} with $M\ge 2$,  we see that
 \[\lim_{T\to \infty} \int_{\ol{AB}} \eqref{geoJ}ds =
  \lim_{T\to \infty}\int_{\ol{CD}} \eqref{geoJ}ds=0,\]
 and that \eqref{geoJ} is absolutely integrable over $i\R$ and $\sigma_0+i\R$.
  Taking $T \rightarrow \infty$,
  \[\hskip -.2cm \int_{-\infty }^{\infty} J_{2it}\left( \frac{4\pi \sqrt{\n m_1 m_2}}{c}\right)
   \frac{h(t)\,t}{\cosh(\pi t)} dt
  = \int_{\Re s = \sigma_0} J_{2s}\left( \frac{4\pi \sqrt{\n m_1 m_2}}{c}\right)
   \frac{h(-is)(-is)}{\cosh(-is\pi)} dt\]
  \[ \ll_{\sigma_0} \left( \frac{2\pi \sqrt{\n m_1 m_2}}{c}\right)^{2\sigma_0} \]
  by \eqref{Jbound} with $M>2$.  Taking $\sigma_0=\tfrac12-\frac{\e}2$, we obtain
  \eqref{Jb}.
\end{proof}

\subsection{Final formulas}\label{final}

The formulas given below follow upon
  equating the geometric side with the spectral side in the two
  cases (primitive and refined) computed above.

\begin{theorem}[Pre-KTF]\label{main1}
Let $\mathcal{F}$ be an orthogonal eigenbasis of $T_\n$ for
  the space
  $L^2_0(N,\w')$ of cusp forms of weight $0$, chosen as in \eqref{Fbasis}.
  Let $h(iz)\in PW^{12}(\C)^{\text{even}}$, and let
  $f_\infty\in C_c^{10}(G^+//K_\infty)$, $V$, and $k$ be the associated
  functions as in \eqref{Selberg}, \eqref{V} and \eqref{kzz}.
  Then for any positive integers $m_1,m_2$, and real $y_1,y_2>0$, we have
\[\sqrt{\n} \sum_{u_j\in \mathcal{F}}
 \frac{\lambda_\n(u_j)\, a_{m_1}(u_j) \ol{a_{m_2}(u_j)}}
 {\|u_j\|^2}
  {h(t_j) K_{it_j}(2\pi m_1y_1)\,{K_{it_j}(2\pi m_2y_2)}}\]
  \[ +  \sqrt{\n}\sum_{\chi_1,\chi_2} \sum_{(i_p)}
  \int_{-\infty}^{\infty} \frac{\lambda_{\n}(\chi_1,\chi_2,it) h(t)
(\frac{m_1}{m_2})^{it} K_{it}(2\pi m_1 y_1) {K_{it}(2\pi m_2 y_2)}}
   {\|\phi_{(i_p)}\|^2\,|\Gamma(\frac{1}{2}+it)L_N(1+2it,{\chi_1}\ol{\chi_2})|^2} \]
\[\hskip 3cm\times\, \sigma_{it}(\chi_1',\chi_2',m_1)
  \ol{\sigma_{it}(\chi_1',\chi_2',m_2)} \,dt
\]
\[ =T(m_1,m_2,\n)
 \frac{\sqrt{\n}\,\psi(N)\ol{\w'(\sqrt{\tfrac{\n m_1}{m_2}})}}{\sqrt{{m_1m_2y_1y_2}}}
\int_{-\infty}^\infty V\left(\frac{t^2+m_1^2y_1^2 +{m_2^2y_2^2}}{m_1y_1m_2y_2}
-2\right)e^{2\pi it}dt
\]
\[
+ {\psi(N)}\sum_{c\in N\Z^+}\frac{S_{\w'}(m_2,m_1;\n;c)}
  {\sqrt{y_1y_2}} \IIL{\R\times\R}k(z_1,\frac{-\n}{c^2z_2})
  e^{2\pi i(m_2x_2-m_1x_1)}dx_1dx_2,
\]
where:
\begin{itemize}
\item $\chi_1,\chi_2$ range through all ordered pairs of finite order Hecke characters
  with $\chi_1\chi_2=\w$ and whose conductors satisfy $\c_{\chi_1}\c_{\chi_2}|N$.
\item $L_N(s,\chi_1\ol{\chi_2})$ is the partial $L$-function defined in \eqref{partialL}.
\item $\ds \lambda_{\n}(\chi_1,\chi_2,it)=\sum_{d|\n}\left(\tfrac{\n}{d^2}\right)^{it}
  \ol{\chi_1(d_N)\chi_2((\tfrac \n d)_N)}$,
  where $d_N$ is the idele which agrees with $d$ at all places $p|N$ and is $1$
  at all other places.
\item $(i_p)$ runs through all sequences $(i_p)_{p|N}$ with\index{notations}{ip@$(i_p), i_p$}
 \[\ord_p(\c_{\chi_2})\le i_p\le N_p-\ord_p(\c_{\chi_1}).\]
\item $\chi_1'$ is the Dirichlet character of modulus
  $N_1=\prod\limits_{p|N\atop{i_p<N_p}}p^{N_p}$ attached to $\chi_1$ as in \eqref{chi'}.
\item $\chi_2'$ is the Dirichlet character of modulus
  $N_2=\prod\limits_{p|N\atop{i_p>0}}p^{N_p}$ attached to $\chi_2$ as in \eqref{chi'}.
\item $z_1=x_1+iy_1$, $z_2=x_2+iy_2$.
\item All other notation is given in Theorem \ref{main} below.
\end{itemize}
\end{theorem}
\vskip .2cm
\noindent{\em Remark:} The hypothesis that $h$ be Paley-Wiener of order 12 
  arises from the following places.
  We need the inverse Selberg transform $f_\infty$ to be in $C_c^8$
  in order to apply Corollary \ref{c},
  whose hypothesis stems from the restrictions in Lemma \ref{con}.
  By Proposition \ref{ST}, we are only able to guarantee this if $h\in PW^{10}$.
  Furthermore, we needed $z^2h(iz)\in PW^{10}$ to prove the convergence of the
  cuspidal term.  As remarked there, assuming Weyl's Law would render this step
  unnecessary.  In computing the main geometric term, we required $V$ 
  to be twice differentiable 
  to justify using Fourier inversion.  For this it would be enough for
  $f_\infty$ to be in $C^5_c$ (see Proposition \ref{fV}), or for
  $h$ to be Paley-Wiener of order $m>4$ (see Proposition \ref{Vd} below).
  We will discuss weakening the hypotheses in Section \ref{Val}.
  \\

For the refined version of the KTF given below, we have multiplied
  each term by $\frac{8}{\pi\sqrt{\n}}$, and we have used formula \eqref{V0} for $V(0)$.
We have also expressed everything in purely classical (non-adelic) terms, 
  replacing the sum over pairs $\chi_1,\chi_2$ of Hecke characters
  of conductor dividing $N$ by a sum
  over pairs $\tc_1,\tc_2$ of Dirichlet characters of modulus $N$.
  Indeed, the two correspond bijectively by \eqref{chi'}.
  Furthermore, $L_N(s,\chi_1\ol{\chi_2})=L(s,\tc_1^{-1}\tc_2)$ by \eqref{L}.
  Lastly, we point out that by that fact that $\chi_{1p}$ is unramified
  when $i_p=N_p$, $\tc_1$ is induced (in the sense of \eqref{olddc})
  from the Dirichlet character $\chi_1'$ of modulus $N_1$
  attached to $\chi_1$ in the above theorem, and likewise $\tc_2$ 
  is induced from $\chi_2'$.

\begin{theorem}[KTF]\label{main}
  Let $h(iz)\in PW^{12}(\C)^{\operatorname{even}}$ (see the remark above).
Let $\mathcal{F}$ be an orthogonal eigenbasis of $T_\n$ for the space
  $L^2_0(N,\w')$ of cusp forms of weight $0$, chosen as in \eqref{Fbasis}.
  Then for any positive integers $m_1,m_2$, we have
\[
  \sum_{u_j\in \mathcal{F}}
 \frac{\lambda_\n(u_j)\, a_{m_1}(u_j) \ol{a_{m_2}(u_j)}}
 {\|u_j\|^2} \frac{h(t_j)}{\cosh(\pi t_j)}
\]
\[  +\frac1\pi\sum_{\tc_1,\tc_2} \sum_{(i_p)}
\int_{-\infty}^{\infty} \frac{\lambda_\n(\chi_1',\chi_2',it)
  \sigma_{it}(\chi_1',\chi_2',m_1)
  \ol{\sigma_{it}(\chi_1',\chi_2',m_2)} (\frac{m_1}{m_2})^{it} h(t)}
   {\|\phi_{(i_p)}\|^2\,|L(1+2it,{\tc_1}^{-1} {\tc_2})|^2} dt
\]
\[=T(m_1,m_2,\n)
 {\psi(N)}\ol{\w'\bigl(\sqrt{\tfrac{m_1\n}{m_2}}\bigr)}
 \frac1{\pi^2}\int_{-\infty}^\infty h(t) \tanh(\pi t)\,t\, dt
\]
\[ + \frac{2i\psi(N)}{\pi}\sum_{c\in N\Z^+}\frac{S_{\w'}(m_2,m_1;\n;c)}{c}
  \int_{-\infty}^\infty J_{2it}\left(\frac{4\pi\sqrt{\n m_1m_2}}c\right)
  \frac{h(t)\, t}{\cosh(\pi t)}dt,
\]
where:
\begin{itemize}
\item $\psi(N)=[\SL_2(\Z):\Gamma_0(N)] = N\prod_{p|N}(1+\tfrac1p)$.
\item The Petersson norm is given by
$\ds\|u_j\|^2=\frac1{\psi(N)}\int_{\Gamma_0(N)\bs\mathbf H}|u_j(x+iy)|^2\frac{dx\,dy}{y^2}$.
\item For $u_j\in\mathcal{F}$, $\Delta u_j=(\frac14+t_j^2)u_j$
  and $T_\n u_j=\lambda_\n(u_j)u_j$.
\item $T(m_1,m_2,\n)=\left\{\begin{array}{cl}1&\text{if }m_1m_2=b^2\n
   \text{ for some integer }b|\gcd(m_1,m_2)\\
0&\text{otherwise.}\end{array}\right.$\\
    Equivalently, $T(a_1,a_2,a_3)\in \{0,1\}$ \index{notations}{T a@$T(a_1,a_2,a_3)$}
  is nonzero if and only if $a_ia_j/a_k$ is a perfect square for all distinct
  $i,j,k\in\{1,2,3\}$.
\item $\tc_1,\tc_2$ range through all ordered pairs of Dirichlet characters modulo $N$
  for which $\tc_1\tc_2=\w'$ and whose conductors satisfy $\c_{\tc_1}\c_{\tc_2}|N$.
\item $(i_p)$ runs through all sequences $(i_p)_{p|N}$ with\index{notations}{ip@$(i_p), i_p$}
 \[\ord_p(\c_{\tc_2})\le i_p\le \ord_p(N)-\ord_p(\c_{\tc_1}).\]
\item $\ds\|\phi_{(i_p)}\|^2=
  \prod_{p|N\atop i_p=0}\frac p{(p+1)}
  \prod_{p|N\atop 0<i_p<N_p}\frac{p-1}{p^{i_p}(p+1)}
 \prod_{p|N\atop i_p=N_p}\frac1{p^{N_p-1}(p+1)}.$
\item $\chi_1'$ is the Dirichlet character mod
  $N_1=\prod\limits_{p|N\atop{i_p<N_p}}p^{N_p}$ inducing $\tc_1$ as in \eqref{olddc}.
\item $\chi_2'$ is the Dirichlet character mod
  $N_2=\prod\limits_{p|N\atop{i_p>0}}p^{N_p}$ inducing $\tc_2$ as in \eqref{olddc}.
\item $\ds \lambda_{\n}(\chi_1',\chi_2',it)=\sum_{d|\n}\left(\tfrac{\n}{d^2}\right)^{it}
  \ol{\chi_1'(d)\chi_2'(\tfrac \n d)}$.  (This is the same as $\lambda_{\n}(\tc_1,\tc_2,it)$
  since $\tc_1,\tc_2$ are induced from $\chi_1',\chi_2'$ and $(\n,N)=1$.
  It is also the same as $\lambda_{\n}(\chi_1,\chi_2,it)$ from the previous theorem,
  by \eqref{x'}.)
\item $M=\prod_{p|N}p^{i_p}$ is also a modulus for $\chi_2'$.
\item $\ds \sigma_{it}(\chi_1',\chi_2',m)=\frac1{M^{1+2it}}
  \sum_{c|m}
   \frac{\scri\ol{\chi'_1(c)}}{{c}^{2it}}
 \sum_{d\in\Z/M\Z}\chi_2'(d)e(\frac{dm}{Mc})$.  The sum over $d$ is
  also expressed in terms of the primitive character inducing $\chi_2'$ in 
  \eqref{Gchi}.
\item $\ds S_{\w'}(m_2,m_1;\n;c)=\sum_{d,d'\in\Z/c\Z\atop{dd'=\n}}\ol{\w'(d)}
  e(\frac{m_2d+m_1d'}c).$
\end{itemize}
\end{theorem}

\subsection{Classical derivation}\label{classder}

When we take $\n=1$ in the above theorem, we obtain
  the ``classical" Kuznetsov formula
\[
  \sum_{u_j\in \mathcal{F}}
 \frac{a_{m_1}(u_j) \ol{a_{m_2}(u_j)}}
 {\|u_j\|^2} \frac{h(t_j)}{\cosh(\pi t_j)}
\]
\[  +\frac1\pi\sum_{\tc_1,\tc_2} \sum_{(i_p)}
\int_{-\infty}^{\infty}
 \frac{ \sigma_{it}(\chi_1',\chi_2',m_1)
  \ol{\sigma_{it}(\chi_1',\chi_2',m_2)} (\frac{m_1}{m_2})^{it} h(t)}
   {\|\phi_{(i_p)}\|^2\,|L(1+2it,{\tc_1}^{-1} {\tc_2})|^2} dt
\]
\begin{equation}\label{CKTF}=\delta(m_1,m_2)
 {\psi(N)}\ol{\w'\bigl(\sqrt{{m_1}/{m_2}}\bigr)}
 \frac1{\pi^2}\int_{-\infty}^\infty h(t) \tanh(\pi t)\,t\, dt
\end{equation}
\[ + \frac{2i\psi(N)}{\pi}\sum_{c\in N\Z^+}\frac{S_{\w'}(m_2,m_1;c)}{c}
  \int_{-\infty}^\infty J_{2it}\left(\frac{4\pi\sqrt{m_1m_2}}c\right)
  \frac{h(t)\, t}{\cosh(\pi t)}dt.\]
  Conversely, Theorem \ref{main} can also be derived from \eqref{CKTF}.
  To see this, start by choosing the orthogonal basis $\mathcal{F}$
  to consist of Hecke eigenvectors with $a_1(u_j)=1$.  (Such a basis is easily constructed
  by a Gram-Schmidt procedure; cf. \cite{KL1}, Lemma 3.10.)
  With this normalization, by \eqref{amTn}, for all $u_j\in \mathcal F$ we have
\begin{equation}\label{nm}
  \lambda_{\n}(u_j)a_{m_1}(u_j)
  =\sum_{\ell|\gcd(\n,m_1)}\ol{\w'(\ell)}\,a_{\tfrac{\n m_1}{\ell^2}}(u_j).
\end{equation}
If we denote the classical formula \eqref{CKTF} by $\CK(m_1,m_2)$, 
  then the sum
\begin{equation}\label{nsum}
\sum_{\ell|\gcd(\n,m_1)}\ol{\w'(\ell)}\CK(\tfrac{\n m_1}{\ell^2},m_2)
\end{equation}
 is precisely Theorem \ref{main}.
 The proof of this assertion involves proving four identities, one for each
  of the four terms (cuspidal, Eisenstein, main, Kloosterman) of \eqref{CKTF}.
 Indeed each term can be summed individually over $\ell$ as in \eqref{nsum}
  to recover the corresponding term in Theorem \ref{main}.
  For the cuspidal term, this is immediate from \eqref{nm}.
For the Kloosterman term, after summing over $\ell$,
  one applies a generalization of Selberg's identity for Kloosterman sums,
  given in \eqref{BKVid} below, to obtain the corresponding term in Theorem \ref{main}.
For the main term, the desired identity follows from
\[\sum_{\ell|(\n,m_1)}\ol{\w'(\ell)}\,\delta(\tfrac{\n m_1}{\ell^2},m_2)
  \ol{\w'(\sqrt{\tfrac{\n m_1}{\ell^2 m_2}})}=\ol{\w'(\sqrt{\tfrac{\n m_1}{m_2}})}
  \sum_{\ell|(\n,m_1)}\delta(\tfrac{\n m_1}{\ell^2},m_2)\]
\[=T(m_1,m_2,\n)
  \ol{\w'(\sqrt{\tfrac{\n m_1}{m_2}})}.\]
The manipulations required for the Eisenstein term are a bit more involved.
  The goal is to prove that for any integers $n,m$ prime to $N$,
\begin{equation}
  \lambda_n(\chi_1',\chi_2',it)\sigma_{it}(\chi_1',\chi_2',m)m^{it}
  = \sum_{\ell|(n,m)}\ol{\w'(\ell)}\sigma_{it}(\chi_1',\chi_2',\tfrac{mn}{\ell^2})
  (\tfrac{nm}{\ell^2})^{it}.
\end{equation}
Dividing both sides by $\tfrac{(nm)^{it}}{M^{1+2it}}$, 
  using $\ol{\w'(\ell)}=\ol{\chi_1'(\ell)\chi_2'(\ell)}$,
  and simplifying what remains,
 one reduces the problem to showing that
\begin{align*}
\sum_{d|n}\sum_{c|m}\frac{\ol{\chi_1'(dc)\chi_2'(\tfrac nd)}}{(dc)^{2it}}
  &\sum_{b\mod M}\chi_2'(b)e(\frac{bm}{Mc})\\
&=
  \sum_{\ell|(n,m)} \sum_{r|\tfrac{nm}{\ell^2}}
  \frac{\ol{ \chi_1'(\ell r)\chi_2'(\ell)}}{(\ell r)^{2it}}
  \sum_{b\mod M}\chi_2'(b)e(\frac{bmn}{M\ell^2r}).
\end{align*}
Setting $dc=\ell r=a$, it then suffices to show that for each divisor $a|nm$, 
\begin{equation}\label{comp}
\hskip -.3cm\sum_{d|n,d|a,\atop{\frac ad|m}}
  \ol{\chi_2'(\tfrac nd)}
  \sum_{b\mod M}\chi_2'(b)e(\frac{bmd}{Ma})
=\hskip -.1cm\sum_{\ell|(m,n),\ell|a,\atop{a|\frac{nm}{\ell}}}\hskip -.1cm
  \ol{\chi_2'(\ell)}
  \sum_{b\mod M}\chi_2'(b)e(\frac{bmn}{M\ell a}).
\end{equation}

\begin{proposition}\label{prop}
Given $a|mn$ as above, define
\[D(a)=\{(d,c)|\, dc=a, d|n, c|m\}\]
and
\[D'(a)=\{(\ell,r)|\, \ell r=a, \ell|(n,m), r|\tfrac{nm}{\ell^2}\}.\]
Then the map
\[(d,c)\mapsto \left(\tfrac{(n,a)c}a,\tfrac{a^2}{(n,a)c}\right)
  =\left(\tfrac{(n,a)}d,\tfrac{ad}{(n,a)}\right)\]
defines a bijection from $D(a)$ to $D'(a)$, 
with inverse
\[(\ell,r)\mapsto \left(\tfrac{(n,a)r}a,\tfrac{a^2}{(n,a)r}\right)
  =\left(\tfrac{(n,a)}\ell,\tfrac{a\ell}{(n,a)}\right).\]
\end{proposition}

\noindent The proof of the proposition is left to the reader.  Using it, we see that the
  left-hand side of \eqref{comp} is equal to
\[\sum_{\ell|(m,n),\ell|a,\atop{a|\frac{nm}{\ell}}}
  \ol{\chi_2'(\frac {\ell n}{(n,a)})}\sum_{b\mod M}\chi_2'(b)
e(\frac{bm(n,a)}{M\ell a}).\]
Replacing $b$ by $b\tfrac n{(n,a)}$ (which is valid since $\tfrac n{(n,a)}$ is prime to $M$),
we obtain the right-hand side of \eqref{comp}, as needed.

\pagebreak
\section{Validity of the KTF for a broader class of $h$}\label{Val}

  We have shown that the Kuznetsov trace formula is valid for 
  the restricted class of
  functions $h$ with $h(iz)\in PW^{12}(\C)^{even}$.
  For certain applications it is useful to allow for a wider
  class of functions.  For example, the Gaussian $h(t)=e^{-t^2/T^2}$ has
  rapid decay for real $t$, but $h(iz)$ is not Paley-Wiener of any order,
  due to faster than exponential growth when $z$ is real.
  Here we consider functions $h$ satisfying:
   \begin{equation} \label{ht}
  \begin{cases}
  h \text{ is even,}\\
    \text{$h(t)$ is holomorphic in the region $|\Im t| < A$,}  \\
   h(t) \ll (1+|t|)^{-B} \text{ in the region }|\Im t|< A,
   \end{cases}
  \end{equation}
  for positive real constants $A$ and $B$.  

  \begin{theorem} \label{hmain}
  If $A,B$ are sufficiently large, then
  the KTF (Theorem \ref{main}) is valid for all functions $h$ satisfying \eqref{ht}.
  \end{theorem}
\noindent{\em Remarks:}  
  (1) We will not obtain the optimum values for $A,B$.
  Kuznetsov's original paper established the formula in the case $N=1$
  for any $A>\tfrac12$ and $B>2$, \cite{Ku}.
  According to \cite{IK} Theorem 16.3, these parameters work for any level $N$.
  This is the range used by
  Selberg in his original work on the trace formula \cite{S}.
  Since, as proven by Selberg, we have
   $|\Im(t)|< \tfrac14$ for the cuspidal spectral parameters $t$,
  it is plausible that $A>\tfrac14$ would suffice.
  This has been proven to be the case when $N=1$ by Yoshida, \cite{Y}.
  However, allowing $A<\tfrac12$ results in poorer control over the size of the
  Kloosterman term.  See Proposition \ref{absconv} below and the remarks following it.
\vskip .2cm

\noindent (2) Given the above theorem, one can use the following idea of Kuznetsov to show
  that in fact $B>2$ suffices.  Briefly, suppose $h$ satisfies \eqref{ht} for some
  $A$ sufficiently large as in the theorem, and some $B>2$.  Choose $\alpha>0$
  very small, but still large enough that $\frac1{2\alpha}<A$.  Define, for
  $r\in\R$,
\[h_r(t)=-\frac12\left(h\Bigl(\frac{r+\frac i2}\alpha\Bigr)+
  h\Bigl(\frac{r-\frac i2}\alpha\Bigr)\right)
  \frac{\cosh(\pi r)\cosh(\pi\alpha t)}{\cosh(\pi(r-\alpha t))\cosh(\pi(r+\alpha t))}.\]
  Then $h_r(t)$ satisfies \eqref{ht} for any $B$ (it has exponential decay as $|\!\Re(t)|\to
  \infty$), and for $A=\frac1{2\alpha}$. Therefore if $\alpha$ is chosen suitably,
  the KTF is valid for $h_r(t)$ by the theorem.
  In each term of this KTF, we integrate $r$ over $\R$ and use the identity
\[\int_{-\infty}^\infty h_r(t)dr=h(t),\]
  which is valid for $h$ analytic on $|\Im t|\le \frac1{2\alpha}$
  (\cite{IK} Lemma 16.4, \cite{Ku} (6.1)),
  to conclude that the KTF is valid for $h$.  It is not too hard to justify this
  process by showing that everything is absolutely convergent, using the fact that
  $B>2$.\\

We prove Theorem \ref{hmain} at the end of \S\ref{hcomp}.
We will make use of the KTF already established for
  Paley-Wiener functions of sufficiently high order,
  and a limiting procedure.
  Given $h$ as in Theorem \ref{hmain}, let $f$ be the corresponding function on $G(\R)^+$,
  i.e. the inverse Selberg transform of $h$.
  It might not be smooth or compactly supported modulo $Z(\R)$.
  In \S \ref{fT}, we will define a family of compactly supported $C^m$ 
  functions $f_T$ on $G(\R)^+$
   for $T > 1$ and some $m>0$, such that $f_T \rightarrow f$ pointwise as $T\to\infty$.
  We let $h_T\in PW^m(\C)^{\text{even}}$ be the Selberg transform of $f_T$.
The KTF holds for $h_T$ if $m$ is sufficiently large, and we show that
  \[ \lim_{T\rightarrow \infty} (\text{Spec. side of KTF for $h_T$}) =
  \text{Spec. side of KTF for $h$} \]
  and
    \[ \lim_{T\rightarrow \infty} (\text{Geo. side of KTF for $h_T$}) =
  \text{Geo. side of KTF for $h$},\]
  thus establishing the KTF for $h$.

  We note that Finis, Lapid and M\"uller have used a different 
  limiting method for $\GL(2)$, and (in large part) beyond,
 to extend Arthur's trace formula to a space of smooth functions
   allowing non-compact support even at nonarchi\-medean places
  (\cite{FL}, \cite{FLM}).

 In \S \ref{oe}, we extend the basic integral transforms of \S \ref{3}
  to allow for non-compactly supported
   functions.  We then discuss the relationships between the various functions
  $f,V,Q,h$, and establish bounds for certain of their derivatives.
    In \S\ref{fT} - \S \ref{hcomp}, we define $h_T$ as in the above 
  discussion, and apply a limiting process to the KTF for $h_T$.
 In the final two sections, we prove a needed auxiliary result, namely
  that for a test function $f=f_\infty\times f_{\fin}$
  with $f_{\fin}$ Schwartz-Bruhat and $f_\infty$ bi-$K_\infty$-invariant, twice differentiable,
 and of mild polynomial decay, the operator
  $R_0(f)$ is Hilbert-Schmidt.\\  

\noindent{\bf Notation.}
Throughout this section, all the constants implicit in $\ll$ may depend on $A$,
  $B$ and $h$ (and hence $V$, $f$, $Q$ etc.) unless otherwise stated.  
  The notation $C_\ell$ will denote a constant depending on $\ell$, $A$, 
  $B$ and $h$, and may have different values in different places.

\subsection{Preliminaries} \label{oe}

We start by setting out some necessary trivialities.
\begin{proposition} \label{lebc}
     Let $I$ be an interval on the real line.
     Suppose $f$ is a measurable function on $\R \times I$ with $f(t,y)$ continuous
    in $y$ for a.e. $t\in\R$.  Suppose $|f(t,y)| \le F(t)$ for some function $F \in L^1(\R)$.
     Then $\int_\R f(t,y) dt$ is a continuous function of $y \in I$.
\end{proposition}
\begin{proof} For any $\e\neq 0$ with $y+\e\in I$,
 \[  \int_\R f(t,y+\e)dt - \int_\R f(t,y) dt = \int_\R (f(t,y+\e)-f(t,y)) dt. \]
  The integrand is bounded by $2F(t)$. 
  By \dct, the integral goes to $0$ as $\e \rightarrow 0$.
  Thus $\int_\R f(t,y)dt$ is a continuous function of  $y \in I$.
\end{proof}

\begin{proposition} \label{lebdiff}
     Let $I$ be an interval.
     Suppose $f(t,y)$ is a measurable function on $\R \times I$ such that for a.e. $t\in\R$
  the partial derivative $f_y(t,y)$ exists and is continuous in $y$.  Suppose further that
     $|f(t,y)| \le F_0(t)$ a.e. and $|f_y(t,y)| \le F_1(t)$ a.e. for some $F_0, F_1 \in L^1(\R)$.
   Then $\int_\R f(t,y) dy$ and $\int_\R f_y(t,y) dt$ are continuous functions of $y \in I$ and
    \[ \frac{d}{dy} \int_\R f(t,y) dt = \int_\R f_y(t,y) dt. \]
(Here, we may view $f_y(t,y)$ as a function on $\R$ by prescribing
  arbitrary values on the measure $0$ set of $t$ for which the derivative is undefined.)
\end{proposition}
\begin{proof} The continuity of the integrals follows from the previous proposition.
  Let $y_0\in I$ be fixed.
   Because the following double integral is absolutely convergent, we can apply
   Fubini's theorem:
  \[ \int_{y_0}^y \int_\R f_y(t,x) dt\,dx = \int_\R \int_{y_0}^y f_y(t,x) dx\, dt\]
  \[ = \int_\R (f(t,y)-f(t,y_0)) dt = \int_\R f(t,y) dt - \int_\R f(t,y_0) dt. \]
  Differentiating with respect to $y$, the assertion follows by the 
  fundamental theorem of calculus.
\end{proof}

By induction, we have the following.
\begin{corollary} \label{lebdiff2}
    Let $I$ be an interval.
     Suppose $f(t,y)$ is a measurable function on $\R \times I$ such that
     for $k=0, 1, \ldots, \ell$:\\
     (i)  $\frac{\partial^k f(t,y)}{\partial^k y} $ exists and is 
  continuous in $y$ for a.e. $t\in\R$,\\
     (ii) there exists $F_k \in L^1(\R)$ such that 
   $|\frac{\partial^k f(t,y)}{\partial y^k}| \le F_k(t)$ a.e.\\
     Then
     $\frac{d^\ell}{d y^\ell} \int_\R f(t,y) dt$ is a continuous function of $y \in I$, and
    \[ \frac{d^\ell}{d y^\ell} \int_\R f(t,y) dt = \int_\R \frac{\partial^\ell f(t,y)}
  {\partial y^\ell} dt. \]
(We may view the integrand $\frac{\partial^\ell f(t,y)}{\partial y^\ell}$ 
  as a function on $\R$ by assigning artibarary values
  on the measure $0$ set for which the derivative is undefined.)
\end{corollary}

\begin{proposition} \label{ll}
   Let $a, b$ and $c$ be positive real numbers.
   Suppose $f$ is a continuous function on $\R$ satisfying 
   $f(x) \ll_{a,b} |x|^{-a}$ for $|x| > b$.
   Then $f(x) \ll_{a,b,c} (c+|x|)^{-a}$ for all $x$.
\end{proposition}
\begin{proof}
   It is easy to show that $|x|^{-a}\ll (c+|x|)^{-a}$ for $|x|>b$.
   By the continuity of $f$,  $f(x) \ll 1 \le (b+c)^a (c+|x|)^{-a}$ for $|x| \le b$.
   The proposition follows.
\end{proof}

\begin{proposition} \label{da}
    Suppose $f$ is a continuous function on an interval $[a,b)$
    with a continuous derivative on $(a,b)$.
    Suppose $\lim_{x \rightarrow a^+} f'(x) = A$.
    Then $f$ has a continuous derivative on $[a,b)$ with $f'(a)=A$.
\end{proposition}
\begin{proof} By definition, $f'(a)=\lim\limits_{x\to a^+}\frac{f(x)-f(a)}{x-a}$.
  Since $f$ is continuous at $a$, we can apply L'Hospital's rule, giving
  $f'(a)=\lim\limits_{x\to a^+}f'(x)$, as needed.
  \end{proof}

\begin{proposition} \label{diffcomp} 
   Let $\ell \ge 1$ be an integer, let $I$ be an interval, and let $g \in C^{\ell}(I)$
  be real-valued.
   Let $J$ be an interval containing the image of $g$.
   Let $f \in C^{\ell}(J)$.
   Then
   \[ \frac{d^{\ell}}{dt^{\ell}} f(g(t)) =
     \sum_{r=1}^\ell \sum_{a_1 + a_2 + \cdots + a_r = \ell \atop \ell \ge a_1 \ge \cdots \ge a_r \ge 1}
      A_{a_1, \ldots, a_r} f^{(r)}(g(t)) g^{(a_1)}(t) \cdots g^{(a_r)}(t), \]
    where $A_{a_1, \ldots, a_r}$ are nonnegative integers independent of $f,g$.
   \end{proposition}
\begin{proof} Induction. \end{proof}

\begin{proposition} \label{1911} 
    Suppose $\phi$ is a function on $\R$ which is $\ell$-times continuously 
   differentiable, with $\phi^{(k)}(\pm \infty)=0$ and
    $\phi^{(k)} \in L^1(\R)$ for $k=1, \ldots, \ell$.
    Then for such $k$ and real $t \neq 0$,
    \[ |\hat{\phi}(t)| \le \frac{1}{|2\pi t|^{k}} \int_\R |\phi^{(k)}(x)| dx, \]
    where $\hat{\phi}(t)=\int_{\R}\phi(y)e^{-2\pi i ty}dy$ 
  is the Fourier transform\index{keywords}{Fourier transform} of $\phi$.
   \end{proposition}
\begin{proof}
    See Lemma 19.11 and Proposition 8.15 of \cite{KL}.
\end{proof}

\subsubsection*{$V$ revisited} \label{hV}

We now re-examine the integral transforms of Section \ref{3}, without the 
  hypothesis of compact support.
Let $C^{m}(G^+//K_\infty)$\index{notations}{C cinfty@$C^m(G^+//K_\infty)$}
   denote the set of bi-$K_{\infty}$-invariant complex-valued functions 
  with continuous $m$-th derivative.
   Let $C^{m}(\R^+)^w$ be the set of $a:\R^+\longrightarrow\C$ with continuous $m$-th 
  derivative, satisfying $a(y)=a(y^{-1})$.

For $f \in C^{m}(G^+//K_\infty)$ and $u\ge 0$, we define\index{notations}{V@$V$}
   \begin{equation} \label{Vf1}
    V(u) = V(y+y^{-1}-2) = f( \mat{y^{1/2}}{}{}{y^{-1/2}}).
    \end{equation}
In the other direction,
   \begin{equation} \label{Vf2}
    f(\mat abcd) = V(\frac{a^2+b^2+c^2+d^2}{ad-bc}-2).
    \end{equation}

\begin{proposition} \label{uydiff}
   For $y \in \R^+$, the substitution
   \[ u = y+y^{-1}-2 \]
   defines a linear injection: $C^m(\R^+)^w \longrightarrow C^{m'}([0,\infty))$ when
  $3m'\le m+1$.  Any function in the image of the map is $C^m$ on $(0,\infty)$.
\end{proposition}
   \begin{proof}
   See Proposition \ref{uprop}.  The proof given there does not actually use
  the hypothesis of compact support.
   \end{proof}


\subsubsection*{The Harish-Chandra transform revisited} \label{hHS}

Given $f\in C^m(G^+//K_\infty)$, its Harish-Chandra transform\index{keywords}{Harish-Chandra transform}
 is the function on $\R^+$ defined by
\[ (\mathcal{H} f)(y) = y^{-1/2} \int_\R f( \mat 1x{}1 \mat{y^{1/2}}{}{}{y^{-1/2}} ) dx, \]
   provided the integral is absolutely convergent.  If $V\in C^{m'}([0,\infty))$ 
  is the function associated to $f$ as above, then
\[ \mathcal{H} f(y) = \int_\R V(y+y^{-1}-2+x^2) dx.  \]

\subsubsection*{The Mellin transform revisited} \label{Mellin}
  Let $\Phi$ be a measurable complex-valued function on $\R^+$.
  Its Mellin transform is the function of $\C$ defined\index{keywords}{Mellin transform} 
by\index{notations}{Mphis@$\mathcal{M},\mathcal{M}_s$ (Mellin transform)}
\[ (\mathcal{M}\Phi)(s) =\mathcal{M}_s\Phi= \int_0^\infty \Phi(y) y^s \tfrac{dy}y, \]
    provided the integral is absolutely convergent.
For example, starting with $f\in C^m(G^+//K_\infty)$ with compact support or just
  sufficient decay, one can define
  $\Phi=\mathcal{H}f$ and $h(t)=\mathcal{M}_{it}\Phi$.  However, our interest here
  is to go in the other direction, starting from $h$.
  Thus we shall need to consider conditions under which 
  the inverse Mellin transform exists.
Throughout this section, $\eta$ denotes a complex-valued function satisfying:
  \begin{equation} \label{eta}
  \begin{cases}
   \text{$\eta(s)$ is a holomorphic function in $A_1 < \Re s <  A_2$,}  \\
   \eta(s) \ll (1+|s|)^{-B} \text{ in the same strip},
   \end{cases}
   \end{equation}
     for some real numbers $A_1<A_2$ and $B>0$.

\begin{proposition} \label{indepsigma}
   Suppose $B>1$ and $\sigma$ is a real number satisfying $A_1 < \sigma < A_2$.
   For $y > 0$, define
  \begin{equation} \label{phisigma}
   \Phi_{\sigma}(y) = \frac{1}{2\pi i} \int_{\Re s = \sigma} \eta(s) y^{-s} ds.
   \end{equation}
  The integral is absolutely convergent and independent of $\sigma$. Therefore we can define
   \begin{equation} \label{Mphi}
    \Phi(y) = \frac{1}{2\pi i} \int_{\Re s = \sigma} \eta(s) y^{-s} ds.
    \end{equation}
  Furthermore if $A_1=-A_2$ and $\eta$ is an even function, then $\Phi(y)=\Phi(y^{-1})$.
\end{proposition}

\begin{proof} The absolute convergence of $\eqref{phisigma}$ follows from  $B > 1$ and \eqref{eta}.
  Let $\sigma_0<\sigma_1$ be two real numbers in the open interval $(A_1, A_2)$.
  For $\alpha>0$, let $\Gamma_\alpha$ be the rectangle with vertices
   $\sigma_0 \pm \alpha i$, $\sigma_1 \pm \alpha i$, and counterclockwise orientation.
  By Cauchy's theorem, $\int_{\sss \Gamma_\alpha} \eta(s) y^{-s} ds=0$.
  By \eqref{eta}, $\int_{\sigma_0}^{\sigma_1} \eta(\sigma\pm i \alpha ) y^{-(\sigma\pm i \alpha)} 
  d\sigma \rightarrow 0$ as $\alpha \rightarrow \infty$.  It follows that $\Phi_\sigma$ is
  independent of $\sigma$.

Finally, suppose $\eta$ is an even function.  Letting $\sigma=0$ in \eqref{Mphi},
 \[ \Phi(y^{-1}) =  \frac{1}{2\pi i} \int_{i\R} \eta(s) y^{s} ds
 = \frac{1}{2\pi i} \int_{i\R} \eta(-s) y^{-s} ds
  = \frac{1}{2\pi i} \int_{i\R } \eta(s) y^{-s} ds = \Phi(y).\qedhere\]
\end{proof}

\begin{proposition}   \label{Minv}
   Suppose $B>1$ and fix $s=\sigma+i\tau$ with $A_1 < \sigma < A_2$.
   Let $\Phi$ be the function defined by \eqref{Mphi}.
   Then $\mathcal{M}\Phi(s)$ is absolutely convergent and equal to $\eta(s)$.
\end{proposition}
\begin{proof}
   Write $y=e^{2 \pi v}$.  Then  by \eqref{Mphi},
   \begin{equation}\label{etaPhi}
 \Phi(y) =  \frac 1{2\pi} \int_{\R} \eta(\sigma+it) e^{-2\pi v \sigma} 
  e^{-2\pi i v t} dt=\frac1{2\pi e^{2\pi v\sigma}}\widehat{\eta_\sigma}(v), 
\end{equation}
where $\eta_\sigma(t)=\eta(\sigma+it)$.
   Because $B>1$, \eqref{eta} shows that $\eta_\sigma \in L^1(\R)$.
   Let $0<r < \min(A_2 - \sigma, \sigma-A_1)$.  Then by Cauchy's integral formula,
  \[ \eta^{(k)}(s)
    = \frac{k!}{2\pi i}\int_{|z-s|=r} \frac{\eta(z)dz}{(z-s)^{k+1}}
   \ll \int_{|z-s|=r}  \frac{|dz|}{{r^{k+1}(1+|z|)^B}}
   \ll_r \frac  1{(1+|s|)^B}. \]
  In order to remove the dependence on $r$ in the estimates that follow, we take
   $r=\tfrac12\min(A_2 - \sigma, \sigma-A_1)$.

   From the above, we see that $\eta_\sigma^{(k)}(t) = \eta^{(k)}(\sigma+i t) \in L^1(\R)$ and
   $\eta_\sigma^{(k)}(\pm \infty)=0$.  By Proposition \ref{1911},
 \begin{equation} \label{Phisigma}
  \widehat{\eta_\sigma}(v) \ll |v|^{-2} \int_\R |\eta''_\sigma(t)| dt \ll_\sigma v^{-2} 
  \quad(v\neq 0).
  \end{equation}
 Thus $\widehat{\eta_\sigma}\in L^1(\R)$, so given $s=\sigma+i\tau$ with $\sigma\in (A_1,A_2)$,
  \eqref{etaPhi} gives
   \[ \int_0^\infty |\Phi(y) y^s| \frac {dy}y
    = \int_{\R} |\widehat{\eta_\sigma}(v)| dv < \infty. \]
   This shows that $\mathcal M\Phi(s)$ is absolutely convergent.

   Because $\eta_\sigma$ is continuous and integrable, and 
  $\widehat{\eta_\sigma} \in L^1(\R)$, we may apply Fourier inversion, giving:
   \[ \eta(s)=\eta_\sigma(\tau) = \int_\R \widehat{\eta_\sigma}(v) e^{2\pi i v \tau} dv 
   = 2 \pi \int_\R \Phi(e^{2 \pi v}) e^{2\pi  v \sigma} e^{2\pi i v \tau} dv\]
   \[ = 2 \pi \int_\R \Phi(e^{2 \pi v}) e^{2\pi v (\sigma+i\tau)} dv
   = \int_0^\infty \Phi(y) y^{s} \tfrac{dy}y. \qedhere\]
\end{proof}

\subsubsection*{Relationship between $h$ and $V$} \label{handV}

   Throughout this section we assume that $h$ satisfies $\eqref{ht}$.
  We take 
\[\eta(s)=h(-is),\]
  which satisfies \eqref{eta} with $A_1=-A$ and $A_2=A$.

\begin{proposition} \label{Sy}
   Suppose $B>1$, and $\sigma$ is a real number with $|\sigma| <A$.
   For $y > 0$, define
   \begin{equation} \label{Sye}
      \Phi(y) = \frac 1{2\pi i} \int_{\Re s = \sigma} h(-is) y^{-s} ds.
    \end{equation}
  Then $\Phi$ belongs to $C(\R^+)^w$, is independent of $\sigma$, and
     \[ \mathcal{M}_{it}\Phi=h(t) \]
for all complex numbers $t$ with $|\Im(t)|<A$.
    If we also define $\Phi(0)=0$, then $\Phi$ is continuous on $[0,\infty)$.
\end{proposition}
\begin{proof} 
    In view of Proposition \ref{Minv}, it only remains to verify the continuity of $\Phi$
  at $y=0$.  This will be done in greater generality in the next proposition.
\end{proof}

\begin{proposition} \label{Sd}
    Suppose $0\le \ell < \min(B-1,A)$ is an integer, and $\sigma$ is a real number with 
   $|\sigma|<A$.  Then the function $\Phi$ defined in \eqref{Sye} has a continuous  
  $\ell$-th derivative on $[0,\infty)$.
    In fact, for $y>0$,
   \begin{equation} \label{Sdiff}
    \Phi^{(\ell)}(y)
    =  \frac 1{2\pi i} \int_{\Re s = \sigma}  \left( \prod_{k=0}^{\ell-1}(-s-k) \right) 
  h(-is) y^{-s-\ell} ds.
    \end{equation}
   The above integral is absolutely convergent and independent of $\sigma$. 
   For $y=0$, $\Phi^{(\ell)}(0)=0$. Lastly,
  \begin{equation} \label{Sdbound}
  \Phi^{(\ell)}(y) \ll_\ell (1+y)^{-A-\ell}.
  \end{equation}
\end{proposition}

\begin{proof} Suppose $0 \le j \le \ell < B-1$,  and write $s = \sigma+it$. We have
\begin{equation}\label{xi}
 \left| \left( \prod_{k=0}^{j-1}(-s-k) \right) h(-is)\right|
\le C_j \frac{\prod_{k=0}^{j-1}(|s|+k)}{(1+|s|)^B}
  \le C_j \frac{\prod_{k=0}^{j-1}(A+|t|+k)}{(1+|t|)^B}.
\end{equation}
Letting  $\xi_j(s)=\left( \prod_{k=0}^{j-1}(-s-k) \right) h(-is)$, we see from
 the first inequality in \eqref{xi} that $\xi_j(s) \ll_j \frac 1{(1+|s|)^{B-j}}$,
  so it satisfies \eqref{ht} 
  with $B$ replaced by $B-j>1$.
   By Proposition \ref{indepsigma}, the right-hand side of \eqref{Sdiff} is 
  absolutely convergent and independent of $\sigma$.

Given $y>0$, let $y_0=y/2$, $y_1=2y$. Then $y \in I=[y_0,y_1]$.
  Define $Y=y_0$ if $\sigma+j > 0$, $Y=y_1$ if $\sigma + j \le 0$.
  Then by \eqref{xi}, $|\xi_j(s)y^{s-j}|\le F_j(t)$, where $F_j(t) = 
  C_j \frac{\prod_{k=0}^{j-1}(A+t+k)}{(1+|t|)^B}{Y}^{-\sigma-j} \ll (1+|t|)^{-(B-j)}$ is
  integrable since $B-j>1$.
  Thus by Corollary \ref{lebdiff2},
  \[ \frac{d^{\ell}}{dy^{\ell}} \Phi(y) = \frac 1{2\pi i} \int_{\Re s = \sigma} 
  h(-is) \frac{d^{\ell}y^{-s}}{dy^{\ell}} ds, \]
  where the integral is continuous in $y>0$.
  This proves \eqref{Sdiff}.
  Furthermore, taking $y=e^{2\pi v}$,
\[  \frac 1{2\pi i} \int_{\Re s = \sigma}  \xi_\ell(s) y^{-s-\ell} ds
   = \frac{e^{-2\pi(\sigma+\ell) v}}{2\pi} \int_\R \xi_\ell(\sigma+it) e^{-2\pi i tv} dt. \]
Choose $\sigma$ such that $-A < \sigma < -\ell$.  Then
  \[ \lim_{y\rightarrow 0^+} \frac{d^{\ell}}{dy^{\ell}} \Phi(y)
    = \lim_{v \rightarrow -\infty} \frac{e^{-2\pi(\sigma+\ell) v}}{2\pi} 
  \int_\R \xi_\ell(\sigma+it) e^{-2\pi i tv} dt = 0. \]
The $\ell$-differentiability of $\Phi(y)$ and the continuity of the $\ell$-th derivative 
  at $y=0$ now follow by Proposition \ref{da}.


 To obtain the bound \eqref{Sdbound}, first suppose $y>0$.  Then
  \[ \int_{\Re s = \sigma} \left| \left( \prod_{k=0}^{\ell-1}(-s-k) \right) h(-is) y^{-s-\ell} \right| |ds|
     \ll_\ell y^{-\sigma-\ell} \int_\R \frac {dt}{(1+|t|)^{B-\ell}} \ll_\ell y^{-\sigma-\ell},\]
 since $B-\ell >1$. The implied constant is independent of $\sigma$, so we can let $\sigma \to A^-$ 
  to obtain $\Phi^{(\ell)}(y) \ll y^{-A-\ell}$.
  The desired bound then follows by Proposition \ref{ll}.
\end{proof}

 For $u\ge 0$, define 
\begin{equation}\label{Qdef}
Q(u) = \Phi(y),
\end{equation}
 where $y=y(u)= \frac{2+u+ \sqrt{4u+u^2}}{2}>0$.
  Note that in the other direction, $u = y+y^{-1}-2$.

\begin{proposition} \label{ydiff}
For $y(u)$ as above, and any nonnegative integer $\ell$,
    \[ y^{(\ell)}(u) \ll u^{-\ell+1}\quad \text{for $u > 1$}.\]
\end{proposition}
\begin{proof} We have
    \[ y'(u) = \frac{1}{2} \left(1+\frac{2+u}{\sqrt{u (4+u)}}\right) \ll 1, \]
    \[ y''(u) = -\frac{2}{(u (4+u))^{3/2}} \ll u^{-3} \ll u^{-1}. \]
    For $\ell \ge 3$,
    \[ y^{(\ell)}(u) = -2 \sum_{i=0}^{\ell-2} {{\ell-2}\choose i}\frac{d^iu^{-3/2}}{du^i}  
   \frac{d^{\ell-2-i}(u+4)^{-3/2}}{du^{\ell-2-i}}
      \ll u^{-\ell-1} \ll u^{-\ell+1},\]
      where $n\choose i$ is the binomial coefficient.
\end{proof}

\begin{proposition}  \label{Qub}
  The function $Q(u)$ is continuous on $[0,\infty)$.
     Suppose $0\le \ell < \min(B-1,A)$.
   Then $Q(u)$ is
   $\ell$-times continuously differentiable on the open interval $(0,\infty)$, where
  it satisfies
    \begin{equation} \label{Qubound}
    Q^{(\ell)}(u) \ll_{\ell} (1+u)^{-A-\ell}.
    \end{equation}
\end{proposition}
\noindent{\em Remark:} In Corollary \ref{Q0} below, we will show that
  if $\ell<\min(B-2,A-1)$, then the above assertions also hold at the endpoint 
  $u=0$.

   \begin{proof}
The continuity of $Q(u)=\Phi(y)$ is immediate from that of $\Phi$ and $y(u)$.
Because $\ell < \min(B-1 ,A)$,  $\Phi$ has a continuous 
  $\ell$-th derivative by Proposition \ref{Sd}.
   By Proposition \ref{uydiff}, $Q$ has a continuous $\ell$-th derivative
  on $(0,\infty)$.

When $\ell=0$, the bound \eqref{Qubound} is immediate from \eqref{Sdbound} and the fact
  that $y(u)\sim u$.  Suppose $\ell>0$.
  By Proposition \ref{diffcomp}, \eqref{Sdbound}, and Proposition \ref{ydiff},
   for $u > 1$ we have
   \[ \frac{d^\ell}{du^{\ell}} Q(u)
   = \frac{d^\ell}{du^{\ell}} \Phi(y(u))
   \ll_\ell  \sum_{r=1}^\ell \sum_{a_1 + a_2 + \cdots + a_r = \ell \atop \ell \ge a_1 \ge \cdots \ge a_r \ge 1}
       \Phi^{(r)}(y(u)) y^{(a_1)}(u) \cdots y^{(a_r)}(u)  \]
\[  \ll_\ell \sum_{r=1}^\ell \sum_{a_1 + a_2 + \cdots + a_r = \ell \atop \ell \ge a_1 \ge \cdots \ge a_r \ge 1}
    (1+y(u))^{-A-r} u^{-a_1+1} \cdots u^{-a_r+1}
    \ll_\ell u^{-A-\ell}\]
   since $y(u)\sim u$.
  The bound \eqref{Qubound} follows for all $u> 0$
  by Proposition \ref{ll}.
   \end{proof}

\begin{proposition} \label{Vd}
 Suppose $B>2$ and $A>1$.  Then the function
\begin{equation} \label{Vu}
    V(u) =  -\frac 1{\pi} \int_\R Q'(u+w^2) dw
\end{equation}
is absolutely convergent and continuous for $u \ge 0$.
In fact, for any nonnegative integer $\ell < \min(B-2,A-1)$,
   $V(u)$ has a continuous $\ell$-th derivative given by
  \begin{equation} \label{Vdiff}
   V^{(\ell)}(u) =  -\frac 1{\pi} \int_\R Q^{(\ell+1)}(u+w^2) dw,
  \end{equation}
  the integral converging absolutely.   Furthermore, for all $u\ge 0$,
 \begin{equation} \label{Vdbound}
  V^{(\ell)} (u) \ll_\ell (1+u)^{-A-\ell-\frac 12}.
  \end{equation}
\end{proposition}
\noindent{\em Remark:} When $u=0$, the integrands of \eqref{Vu} and \eqref{Vdiff}
  may be undefined at $w=0$, but the integrals still make sense.

\begin{proof}
%
  Suppose $0\le k< \min(B-2,A-1)$.  Then
 $k+1<\min(B-1,A)$, so by Proposition \ref{Qub},
  $Q$ is $(k+1)$-times continuously differentiable on $(0,\infty)$, and
   \begin{equation} \label{VdQ}
    |Q^{(k+1)}(u+w^2)| \le C_k (1+u+w^2)^{-A-k-1} \le C_k(1+w^2)^{-A-k-1}
    \end{equation}
for $u>0$, where $C_k$ is a positive constant.
  Now apply Corollary \ref{lebdiff2} with $y=u$, $t=w$, $f(t,y)=Q'(u+w^2)$, and
  $F_k(w)$ equal to the
  right-hand side of \eqref{VdQ}.  The equality \eqref{Vdiff} and its continuity and 
  absolute convergence follow.

To obtain the bound \eqref{Vdbound}, we observe that
  \[ |V^{(\ell)}(u)|\le \frac1\pi\int_\R |Q^{(\ell+1)}(u+w^2)| dw  \ll_\ell \int_\R (1+u+w^2)^{-A-\ell-1} dw \]
  \[ =  (1+u)^{-A-\ell-1} \int_\R (1+((1+u)^{-\frac 12}w)^2)^{-A-\ell-1} dw\]
\[ =   (1+u)^{-A-\ell-\frac 12}\int_\R (1+w^2)^{-A-\ell-1} dw\ll (1+u)^{-A-\ell-\frac12}.\qedhere \]
   \end{proof}

\begin{proposition} \label{VQ}
   Suppose $B>2$ and $A>1$. Then for all $u \ge 0$,
  \begin{equation}\label{VQeq}
 \int_\R V(u+x^2) dx = Q(u), 
\end{equation}
the integral converging absolutely.
\end{proposition}
\begin{proof}
   Under the given hypothesis, we can take $\ell=1$ in \eqref{Qubound} to give
  \[
  \int_\R \int_\R  |Q'(u+x^2+w^2)| dwdx
  \ll_\ell \int_\R \int_\R (1+u+x^2+w^2)^{-A-1} dwdx \]
  \[ =  \int_{0}^\infty \int_{0}^{2\pi} (1+u+r^2)^{-A-1} r\,d\theta dr = 
  2\pi \int_{0}^{\infty} (1+u+r^2)^{-A-1} rdr< \infty. \]
(The bound for the integrand we applied is valid whenever $u+x^2+w^2>0$, i.e. 
  for almost all $x, w$.)
  Therefore the integral in \eqref{VQeq} is absolutely convergent.
  It defines a continuous function of $u\ge 0$ by Proposition \ref{lebc},
  since $V(u+x^2)\le (1+x^2)^{-A-\frac12}$ by \eqref{Vdbound},
   the latter function being integrable.  Furthermore, assuming $u>0$,
  \[ \int_\R V(u+x^2) dx
  = -\frac 1{\pi} \int_\R \int_\R Q'(u+x^2+w^2) dwdx \]
  \begin{equation}\label{dr} = -2 \int_0^\infty Q'(u+r^2) r dr
  = \left.-Q(u+r^2)\right|_{r=0}^{r=\infty} = Q(u). 
\end{equation}
  In the last step we used \eqref{Qubound} with $\ell=0$.
 This proves \eqref{VQeq} for $u>0$.
  Our use of the fundamental theorem of calculus in \eqref{dr}
  may not be valid when $u=0$,
  due to a possible discontinuity of the integrand at $r=0$ in that case.
  However, because both sides of the proposed equality \eqref{VQeq} are 
  continuous functions of $u\ge 0$
  which agree for all $u>0$, they are equal when $u=0$ as well.
\end{proof}

\begin{corollary}\label{Q0} If $\ell<\min(B-2,A-1)$, then $Q\in C^\ell([0,\infty))$,
  and \eqref{Qubound} holds for all $u\ge 0$.
\end{corollary}
\begin{proof} By \eqref{Vdbound}, $V^{(\ell)}(1+u+x^2)
  \ll (1+x^2)^{-A-\ell-\frac12}$.  Since the latter is integrable over $\R$,
  we can differentiate \eqref{VQeq} under the integral sign
  (cf. Corollary \ref{lebdiff2}) to obtain the result.
\end{proof}

\begin{proposition} \label{Sinv}
 Suppose $B>2$ and $A>1$.\index{keywords}{Selberg transform}
Let $f$ be the function on $G(\R^+)$ corresponding to $V$ as in \eqref{Vf2}.
   Then for $|\Im t| < A$, the Selberg transform of $f$ is absolutely convergent and equal to $h$:
     \[ (\mathcal{S} f)(it) = \mathcal{M}_{it}\mathcal{H}f = h(t).\]
\end{proposition}
\begin{proof}
   By Proposition \ref{VQ}, \eqref{HCV}, and \eqref{Qdef},
   $(\mathcal{H}f)(y)=Q(u)=\Phi(y)$. 
  By Proposition \ref{Sy}, $\mathcal{M}_{it}\Phi = h(t)$.
 \end{proof}


\subsection{Smooth truncation} \label{fT}

In this section, we suppose $h(t)$ satisfies \eqref{ht} for some $B>2$ and $A>1$, and
  continue with the same notation from the previous section.
  We will need to truncate $V$ in a way that preserves its differentiability.  
This requires a smooth bump function.

Let $\rho:\R\longrightarrow [0,1]$ be a smooth function such that: 
\begin{enumerate}
\item[(i)] $\rho(x)=0$ for $x \le 0$,
  \item[(ii)] $\rho(x) = 1$ for $x \ge 1$.
\end{enumerate}
  For $T > 0$, define
  \[\rho_T(x) =
  \begin{cases}
  1 & \text{if $|x| < T$},\\
  \rho(T+1-|x|) & \text{if $T \le |x| \le T+1$}, \\
  0 & \text{if $|x| > T+1$}.
  \end{cases} \]
Then $\rho_T$ is a smooth bump function with support in $[-(T+1),T+1]$.
  Letting $\widetilde{\rho}_T=1-\rho_T$, the graphs of $\rho_T$ and $\widetilde{\rho}_T$
  are given below:
\vskip -.0cm
\begin{figure}[htbp]
\begin{center}
\input{rhoT.pstex_t}
\end{center}
\end{figure}
\vskip -.3cm
\noindent  For $j\ge 1$,
$\rho_T^{(j)}(x)=0$ unless $T\le |x|\le T+1$.
 Thus $\rho_T^{(j)} \ll_j \chi_{[T,T+1]}$ on $\R_{\ge 0}$, where 
  $\chi_I$ denotes the characteristic function of the set $I$, 
  and by construction the implied constant is independent of $T$.

For $u \ge 0$, define
  \[ V_T(u) = V(u) \rho_T(\log(1+u)). \]
  Define
  \[ \widetilde{V}_T(u) = V(u) - V_T(u) = V(u) \widetilde{\rho}_T(\log(1+u)). \]
  Let $f_T$ (resp. $\widetilde{f}_T$) be the bi-$K_\infty$-invariant function on
   $G(\R)^+$ corresponding to $V_T$ (resp. $\widetilde{V}_T$) as in \eqref{Vf2}.
  Because $V_T$ is compactly supported, the support of $f_T$ is compact 
  modulo the center.

Given the functions $\Phi(y)=Q(u)$ attached to $h$ as in the previous section,
  we let $\Phi_T(y)$, (resp. $\widetilde{\Phi}_T(y)$) be the Harish-Chandra 
  transform of $f_T$ (resp. $\widetilde{f}_T$), and set
  $Q_T(u)=\Phi_T(y)$ and $\widetilde{Q}_T(u)=\widetilde{\Phi}_T(y)$,
  where $u=y+y^{-1}-2$.
  Lastly, we define $h_T(t)$ (resp. $\widetilde{h}_T(t)$) to be the Selberg 
  transform of $f_T$ (resp. $\widetilde{f}_T$) as in Proposition \ref{Sinv}.
  By the linearity of the various integral transforms, in each case we have the relation
$\widetilde{\Box}_T=\Box-\Box_T$.
  
  Suppose $\ell<\min(B-2,A-1)$, so that by Proposition \ref{Vd}, 
  $V\in C^\ell([0,\infty))$.  Then $V_T\in C^\ell_c([0,\infty))$, 
  so by Proposition \ref{fV}, $f_T\in C^\ell_c(G^+//K_\infty)$, 
  and by Proposition \ref{ST}, $h_T\in PW^\ell(\C)^{\text{even}}$.

\begin{proposition}\label{VTprop}
 Suppose $B>2$, $A>1$, and $0\le \ell <\min(B-2,A-1)$.
    Then for $u \ge 0$,
   \begin{equation} \label{VTell}
    \widetilde{V}_T^{(\ell)}(u) \ll_{\ell} (1+u)^{-A-\ell-\frac 12} 
  \chi_{[T,\infty)}(\log(1+u)).
   \end{equation}
\end{proposition}
\begin{proof}
  For $j \geq 1$, Proposition \ref{diffcomp} gives
  \[ \frac{d^j}{du^j} \widetilde{\rho}_T (\log(1+u)) \ll_j
      \sum_{r=1}^j  \sum_{a_1 + a_2 + \cdots + a_r = j \atop r \ge a_1 \ge a_2 \cdots \ge a_r \ge 1}
     \rho_T^{(r)}(\log(1+u)) (1+u)^{-a_1-a_2 -\cdots -a_r} \]
  \[ \ll (1+u)^{-j} \chi_{[T,T+1]}(\log(1+u)). \]
  By the bound \eqref{Vdbound}, for $u \ge 0$ we have
  \[ \frac{d^\ell}{du^{\ell}} \widetilde{V}_T(u)
  =  \frac{d^\ell}{du^{\ell}} V(u)\widetilde{\rho}_T(\log(1+u))
  = \sum_{j=0}^{\ell} {\ell \choose j} V^{(\ell-j)}(u) \widetilde{\rho}^{(j)}_T(\log(1+u)) \]
   \[\hskip -.3cm \ll_{\ell} (1+u)^{-A-\ell-\frac 12}\chi_{[T,\infty)}(\log(1+u))
  + \sum_{j=1}^\ell (1+u)^{-A-\ell + j -\frac 12 } (1+u)^{-j} \chi_{[T,T+1]}(\log(1+u)) \]
   \[ \ll_{\ell} (1+u)^{-A-\ell-\frac 12} \chi_{[T,\infty)}(\log(1+u)).\qedhere\]
\end{proof}

\begin{proposition} 
 Suppose $B>2$, $A>1$.
    Then for $u \ge 0$,
  \[ \widetilde{Q}_T(u) = \int_\R \widetilde{V}_T(u+w^2) dw.\]
 In fact, if $0\le \ell <\min(B-2,A-1)$, then
  $\widetilde{Q}_T$ has a continuous $\ell$-th derivative on $[0,\infty)$ given by
  \begin{equation} \label{QTdiff}
      \widetilde{Q}^{(\ell)}_T(u) = \int_\R \widetilde{V}^{(\ell)}_T(u+w^2) dw,
  \end{equation}
  the integral being absolutely convergent and continuous. Further,
  \begin{equation} \label{QTbound}
  |\widetilde{Q}^{(\ell)}_T(u)| \le  \frac{E_{\ell,T}(u)}{(1+u)^{A+\ell}},
  \end{equation}
  where $E_{\ell,T}(u) \ll_\ell 1$ is a nonzero measurable function 
  with $\lim\limits_{T\rightarrow 0} E_{\ell,T}(u)=0$.
\end{proposition}

\begin{proof} Let $C_\ell>0$ be the implied constant in \eqref{VTell}.
    Then for $0\le k\le\ell$, 
\[\widetilde{V}_T^{(k)}(u+w^2) \le C_k (1+u+w^2)^{-A-k-\frac 12}
    \le C_k (1+w^2)^{-A-k-\frac 12}.\]  
 Letting $F_k(w)$ denote the latter expression,
  we apply Corollary \ref{lebdiff2} to conclude that \eqref{QTdiff} holds and is absolutely
  convergent and continuous.

It remains to establish the bound \eqref{QTbound}.
   By the previous proposition,
  \[ \left| \int_\R \widetilde{V}^{(\ell)}_T(u+w^2) dw \right|
  \le \int_\R C_\ell \chi_{[T,\infty)}(\log(1+u+w^2))(1+u+w^2)^{-A-\ell-\frac 12} dw\]
  \[= (1+u)^{-A-\ell-\frac 12} C_\ell \int_\R
    \chi_{[T,\infty)}(\log(1+u+w^2))
    (1+ ((1+u)^{-\frac 12}w)^2)^{-A-\ell-\frac 12} dw \]
  \[ =(1+u)^{-A-\ell} C_\ell\int_\R \chi_{[T,\infty)}(\log(1+u+(1+u)w^2))
   (1+ w^2)^{-A-\ell-\frac 12} dw.\]
  Let
  \[ E_{\ell,T}(u) = C_\ell\int_\R \chi_{[T,\infty)}
  (\log(1+u+(1+u)w^2))(1+ w^2)^{-A-\ell-\frac 12} dw.\]
  Note that
   \[ |E_{\ell,T}(u)| \leq C_\ell \int_\R (1+ w^2)^{-A-\ell-\frac 12} dw < \infty. \]
  By \dct, $\lim\limits_{T\rightarrow \infty} E_{\ell,T}(u)=0$. This completes the proof.
\end{proof}

\begin{corollary}
 Suppose $B>2$, $A>1$, and $0\le \ell <\min(B-2,A-1)$.
  Then for $v\in \R$ we have
  \[ \left| \frac{d^\ell}{dv^{\ell}} \widetilde{Q}_T(e^{2\pi v} + e^{-2\pi v} - 2) \right|
  \le \frac{\hat{E}_{\ell,T}(|v|)}{(e^{2\pi v}+e^{-2\pi v})^{A}}, \]
  where $\hat{E}_{\ell,T}(|v|) \ll_\ell 1$ is a nonzero measurable function with 
  $\lim\limits_{T\rightarrow \infty} \hat{E}_{\ell,T}(|v|)= 0$.
\end{corollary}
\begin{proof}
  When $\ell=0$, the assertion is immediate from \eqref{QTbound}, taking 
  $\hat{E}_{0,T}(|v|)=C_0E_{0,T}(e^{2\pi v}+e^{-2\pi v}-2)$ for a sufficiently large
  constant $C_0$.  Suppose now that $\ell>0$.
   Using Proposition \ref{diffcomp} and the fact that
$\frac{d^i}{dv^{i}} (e^{2\pi v} + e^{-2\pi v} - 2) 
  \ll_i e^{2\pi v}+e^{-2\pi v}$, we have
\[ \frac{d^\ell}{dv^{\ell}} \widetilde{Q}_T(e^{2\pi v} + e^{-2\pi v} - 2)
    \ll_\ell \sum_{r=1}^{\ell} \widetilde{Q}^{(r)}_T(e^{2\pi v} + e^{-2\pi v} - 2)
    \sum_{a_1 + a_2 + \cdots + a_r = \ell \atop \ell \ge a_1 \ge \cdots \ge a_r \ge 1}
    (e^{2\pi v} + e^{-2\pi v})^r.  \]
   By the bound \eqref{QTbound}, this is
 \[ \ll_\ell \sum_{r=1}^\ell
     \frac{E_{r,T}(e^{2\pi v} + e^{-2\pi v} - 2) (e^{2\pi v} + e^{-2\pi v})^r}
  {(e^{2\pi v} + e^{-2\pi v} - 1)^{A+r}}
    \ll   \frac{\sum_{r=1}^\ell E_{r,T}(e^{2\pi v} + e^{-2\pi v} - 2)}
  {(e^{2\pi v} + e^{-2\pi v})^{A}}. \]
  Thus we can take $\hat{E}_{\ell,T}(|v|) = 
  C_\ell \sum_{r=1}^\ell E_{r,T}(e^{2\pi v} + e^{-2\pi v} - 2)$ for a sufficiently
  large constant $C_\ell$.
\end{proof}

\begin{proposition} \label{hTbound} 
 Suppose $B>2$, $A>1$, and $0\le \ell <\min(B-2,A-1)$.
  Let $0<A'<A$.
   Then there exists a positive real number $\mathcal{E}_{\ell,T}$ such that
     for $|\Im t| \le A'$,
    \[ |\widetilde{h}_T(t)| \le \frac{\mathcal{E}_{\ell,T}}{(1+|t|)^{\ell}}
    \text{ and }
     \lim_{T\rightarrow \infty} \mathcal{E}_{\ell,T}=0. \]
\end{proposition}

\begin{proof}
 Write $t = x + i \beta$ with $|\beta|\le A'$.  Then
 \[
   \widetilde{h}_T(t) = \mathcal{M}_{it}\widetilde{\Phi}_T =
  2 \pi \int_{\R} \widetilde{Q}_T(e^{2\pi v} + e^{-2\pi v}-2) e^{2 \pi i t v} dv
   \]
  \[
       = 2 \pi \int_{\R} \widetilde{Q}_T(e^{2\pi v} + e^{-2\pi v}-2) e^{-2 \pi v \beta} e^{2 \pi i v x} dv.
   \]
Since this is a Fourier transform, we can bound it using Proposition \ref{1911}.
First, by the above Corollary,
   \[ \frac{d^{\ell}}{dv^{\ell}} (\widetilde{Q}_T(e^{2\pi v} + e^{-2\pi v}-2)e^{-2\pi v \beta}) \]
    \[ = \sum_{i=0}^{\ell} {\ell\choose i}  (-2\pi \beta)^{\ell-i} e^{-2\pi v \beta}
     \frac{d^i}{dv^{i}}\widetilde{Q}_T(e^{2\pi v} + e^{-2\pi v}-2) \]
   \[ \ll_\ell e^{2\pi A' v} \sum_{i=0}^{\ell} \frac{\hat{E}_{i,T}(|v|) dv}
  {(e^{2\pi v}+e^{-2\pi v})^{A}}. \]
   Let $C_\ell$ be the implied constant in the above inequality.

  By Proposition \ref{1911}, for $|x|=|\Re(t)| \ge 1$,
  \[
   |\widetilde{h}_T(t)| \le
   \frac 1{|2\pi x|^{\ell}}  \int_{\R}\left|
   \frac{d^{\ell}}{dv^{\ell}} (\widetilde{Q}_T(e^{2\pi v} + e^{-2\pi v}-2) e^{-2\pi v \beta})
   \right| dv.
   \]
Since $\frac{1+|t|}{|x|} \le \frac{1+|x|+|\beta|}{|x|} = |x|^{-1} +1+ |x|^{-1} |\beta| 
  \le 2 + A$, the above is
  \[  \le
  \frac {(2+A)^{\ell}}{(1+|t|)^{\ell}}  \frac{C_\ell}{(2\pi)^{\ell} }  \sum_{i=0}^{\ell}
  \int_{\R} \frac{  e^{2\pi A' v}\hat{E}_{i,T}(|v|)  dv}{(e^{2\pi v}+e^{-2\pi v})^{A}},
  \]
which converges since $A'<A$.
On the other hand, if $|x|=|\Re(t)| \le 1$, then $1+|t|\le 1+|x|+|\beta|\le 2+A$, so 
  $1\le \tfrac{(2+A)^\ell}{(1+|t|)^\ell}$, and
\[ |\widetilde{h}_T(t)|
   \le 2 \pi \int_{\R} |\widetilde{Q}_T(e^{2\pi v} + e^{-2\pi v}-2)| 
  e^{-2 \pi v \beta} dv \]
 \[  \le \frac { 2 \pi (2+A)^{\ell} }{(1+|t|)^{\ell}}
  \int_{\R}  \frac{e^{2\pi A' v} \hat{E}_{0,T}(|v|) dv}{(e^{2\pi v}+e^{-2\pi v})^{A}}. \]
Hence if we define
  \[ \mathcal{E}_{\ell,T} =
    2 \pi (2+A)^{\ell}
   \int_{\R}  \frac{e^{2\pi A' v} \hat{E}_{0,T}(|v|) dv}{(e^{2\pi v}+e^{-2\pi v})^{A}}
   +   \frac{C_\ell (2+A)^{\ell}}{(2\pi)^{\ell} }   \sum_{i=0}^{\ell}
  \int_{\R} \frac{  e^{2\pi A' v}\hat{E}_{i,T}(|v|) dv}{(e^{2\pi v}+e^{-2\pi v})^{A}},\]
then $|\widetilde{h}_T(t)|\le \frac{\mathcal{E}_{\ell,T}}{(1+|t|)^\ell}$
  for all $t$ in the strip $|\Im(t)|\le A'$, as needed.

   Using $\hat{E}_{i,T}(|v|) \ll_\ell 1$, by the dominated convergence theorem
   we have
  \[ \lim_{T\rightarrow \infty}  \mathcal{E}_{\ell,T}
  =   2 \pi (2+A)^{\ell}
   \int_{\R}  \frac{e^{2\pi A' v} \lim_{T \rightarrow \infty} \hat{E}_{0,T}(|v|) dv}{(e^{2\pi v}+e^{-2\pi v})^{A}} \]
   \[ +  \frac{C_\ell (2+A)^{\ell}}{(2\pi)^{\ell} }   \sum_{i=0}^{\ell}
  \int_{\R} \frac{  e^{2\pi A' v}\lim_{T \rightarrow \infty} 
  \hat{E}_{i,T}(|v|) }{(e^{2\pi v}+e^{-2\pi v})^{A}} dv = 0.\qedhere\]
\end{proof}

\subsection{Comparing the KTF for $h$ and $h_T$} \label{hcomp}

We set the following notation for the various terms in the KTF, together with their
  absolute value counterparts (see Theorem \ref{main} for notation):
\[ \operatorname{Spec}_1(h) =  \sum_{u_j\in \mathcal{F}}
  \frac{\lambda_\n(u_j)\, a_{m_1}(u_j) \ol{a_{m_2}(u_j)}}
 {\|u_j\|^2} \frac{h(t_j)}{\cosh(\pi t_j)}\]
 \[ \operatorname{Spec}^a_1(h) =  \sum_{u_j\in \mathcal{F}}
  \left|\frac{\lambda_\n(u_j)\, a_{m_1}(u_j) \ol{a_{m_2}(u_j)}}
 {\|u_j\|^2} \frac{h(t_j)}{\cosh(\pi t_j)}\right|\]
\[ \operatorname{Spec}_2(h) = \frac 1{\pi}  \sum_{\tc_1,\tc_2} \sum_{(i_p)}
    \int_{-\infty}^{\infty}  \tfrac{\lambda_\n(\chi_1',\chi_2',it)
    \sigma_{it}(\chi_1',\chi_2',m_1)
    \ol{\sigma_{it}(\chi_1',\chi_2',m_2)} (\frac{m_1}{m_2})^{it} h(t)}
    {\|\phi_{(i_p)}\|^2\,|L(1+2it,{\tc_1} {\tc_2}^{-1})|^2}  dt, \]
\[ \operatorname{Spec}^a_2(h) = \frac 1{\pi} \sum_{\tc_1,\tc_2} \sum_{(i_p)}
    \int_{-\infty}^{\infty}   \tfrac{|\lambda_\n(\chi_1',\chi_2',it)
    \sigma_{it}(\chi_1',\chi_2',m_1)
    \ol{\sigma_{it}(\chi_1',\chi_2',m_2)} (\frac{m_1}{m_2})^{it} h(t)|}
    {\|\phi_{(i_p)}\|^2\,|L(1+2it,{\tc_1} {\tc_2}^{-1})|^2}   dt, \]
\[ \operatorname{Geo}_1(h)=T(m_1,m_2,\n) {\psi(N)}\ol{\w'\bigl(\sqrt{\tfrac{m_1\n}{m_2}}\bigr)} \frac1{\pi^2}
      \int_{-\infty}^\infty  h(t) \tanh(\pi t)\,t\, dt, \]
\[ \operatorname{Geo}^a_1(h)=T(m_1,m_2,\n) {\psi(N)}\left|\w'\bigl(\sqrt{\tfrac{m_1\n}{m_2}}\bigr)\right| \frac1{\pi^2}
\int_{-\infty}^\infty | h(t) \tanh(\pi t)\,t|\, dt, \]
\[   \operatorname{Geo}_2(h) = \frac{2i\psi(N)}{\pi} \sum_{c\in N\Z^+}\frac{S_{\w'}(m_2,m_1;\n;c)}{c}
    \int_{-\infty}^\infty  J_{2it}\left(\tfrac{4\pi\sqrt{\n m_1m_2}}c\right)
    \frac{h(t)\, t}{\cosh(\pi t)}  dt, \]
\[  \operatorname{Geo}^a_2(h) = \frac{2\psi(N)}{\pi} \sum_{c\in N\Z^+}\frac{|S_{\w'}(m_2,m_1;\n;c)|}{c}
    \int_{-\infty}^\infty  \left| J_{2it}\left(\tfrac{4\pi\sqrt{\n m_1m_2}}c\right)
    \frac{h(t)\, t}{\cosh(\pi t)}  \right|dt.\]

\begin{proposition} \label{absconv}
  Suppose $h$ satisfies \eqref{ht} for some $A>\tfrac 14$ and $B>2$. 
  Then: (i) $\operatorname{Spec}^a_2(h) < \infty$,
   (ii) $\operatorname{Geo}^a_1(h) < \infty$,
   and (iii) $\operatorname{Geo}^a_2(h) < \infty$.
\end{proposition}
\noindent{\em Remarks:} (1)
  Allowing $\tfrac14<A< \tfrac12$ rather than $A\ge \tfrac12$ comes
  at a price, namely, the Kloosterman term may no longer be $O(N^\e)$ (which 
  holds when $A\ge \tfrac12$), but instead only $O(N^{\frac12-\delta})$ for any 
  $0<\delta<2(A-\tfrac14)$.\vskip .2cm

\noindent (2) It is known that $\operatorname{Spec}_1^a(h)<\infty$ under the above conditions as well.
See the remark after Proposition \ref{ktfcusp}, where we explain why $B>3$ suffices.

\begin{proof}
The proof of assertion (i) follows the same outline as that of Prop 7.6,
   although Lemma 7.5 is not needed since $\operatorname{Spec}_2^a(h)$ does not
  involve Bessel functions.  One obtains an estimate like
   \eqref{sb}, but without the factor of $(1+2|t|)^2$, so that $B>1$ suffices.
  The holomorphy of $h$ is not needed here, so any value of $A$ is allowable.
   Assertion (ii) is trivial since $h(t)$ is integrable and $|\tanh(\pi t)|\le 1$.
   Assertion (iii) follows from the proof of Prop 7.11.  It requires
  $A>\tfrac14$ and $B>2$.  At the end of the proof, one can take $\sigma_0=
  \tfrac14+{\e}<\min(A,\tfrac12)$ to obtain an exponent of $\tfrac12+2\e$ in place of
  $1-\e$ in \eqref{Jb}.  Then in place of \eqref{KJb}, for any $\e'>0$
   we can obtain the bound
 \[  \ll \psi(N)\c_{\w'}^{\frac12}\sum_{c\in N\Z^+} \frac{\tau(c)}{c^{1+2\e}}
  \ll \frac{N^{1+\frac{\e'}2}N^{\frac12}N^{\frac{\e'}2}}{N^{1+2\e}}\sum_{c\in\Z^+}
  \frac{\tau(c)}{c^{1+2\e}}=O(N^{\frac12+\e'-2\e}).
  \]
This proves the proposition.  The bound asserted in the remark follows upon
  observing that $\delta=2\e-\e'$ can assume any positive number
  less than $2(A-\tfrac14)$ when $A<\tfrac12$.
\end{proof}

\begin{proposition} \label{hmain2} Suppose 
  $\ell\ge12$ is an integer for which
   \begin{equation} \label{hmain2e} 
\operatorname{Spec}_1^a(r_\ell)<\infty,
\end{equation}
where $r_\ell(t)=\frac1{(1+|t|)^{\ell}}$.
   Let $B>\ell+2$ and $A>\ell+1$ be real constants.
   Then for any function $h$ satisfying \eqref{ht} with these values, the KTF is valid:
   \[ \operatorname{Spec}_1(h) + \operatorname{Spec}_2(h) = 
  \operatorname{Geo}_1(h) + \operatorname{Geo}_2(h).\]
\end{proposition}

\begin{proof} 
Using the fact (Proposition \ref{specprop}) that all of the spectral parameters of 
  $L^2_0(N,\w')$ satisfy $|\Im(t_j)|< \tfrac12$, 
   we apply Proposition \ref{hTbound} with $\tfrac12\le A'<A$, giving
   \begin{equation} \label{spec1limit}
    \operatorname{Spec}_1^a(h-h_T) 
    =\operatorname{Spec}_1^a(\widetilde{h}_T) 
  \le\mathcal{E}_{\ell,T} \operatorname{Spec}_1^a(r_\ell)
   \end{equation}
for $r_\ell$ as in \eqref{hmain2e}.
  Noting that $h_{T}(iz)\in PW^\ell(\C)^{\text{even}}$ with $\ell\ge 12$
  (see the discussion just before Proposition \ref{VTprop}),
   Proposition \ref{ktfcusp} gives 
  $\operatorname{Spec}_1^a(h_T)<\infty$.
  Thus by \eqref{spec1limit},
   \[ \operatorname{Spec}_1^a(h) = \operatorname{Spec}_1^a(h_T+\widetilde{h}_T)
   \le \operatorname{Spec}_1^a(h_T) +\operatorname{Spec}_1^a(\widetilde{h}_T)
  < \infty. \]
    Hence $\operatorname{Spec}_1(h)$ exists.
 Because  $\lim\limits_{T\rightarrow \infty} \mathcal{E}_{\ell,T} \rightarrow 0$,
    using \eqref{spec1limit} we have
    \[ \lim_{T\rightarrow \infty}  |\operatorname{Spec}_1(h)-\operatorname{Spec}_1(h_T)|
  = \lim_{T\rightarrow \infty}  |\operatorname{Spec}_1(\widetilde{h}_T)|
    \le \lim_{T\rightarrow \infty}  \operatorname{Spec}_1^a(\widetilde{h}_T)=0. \]
    Hence
   \[ \lim_{T \rightarrow \infty} \operatorname{Spec}_1(h_T) = \operatorname{Spec}_1(h). \]

Now let $X$ denote $\operatorname{Spec}_2^a$, $\operatorname{Geo}_1^a$, or 
      $\operatorname{Geo}_2^a$.  Then by Proposition \ref{hTbound},
\[X(\widetilde{h}_T) \le \mathcal{E}_{\ell,T}X(r_\ell).\]
  Noting that $r_\ell\ll h_\ell$, where the function 
  $h_\ell=\frac{1}{(1+t^2)^{\ell/2}}$ satisfies \eqref{ht} with 
  $A=1$ and $B = \ell>2$, we conclude from Proposition \ref{absconv}
  that the above expression is finite.
   By the same reasoning as for $\operatorname{Spec}_1$, we see that 
  $X(h)$ exists, and is equal to $\lim\limits_{T\to\infty}X(h_T)$.
    Because the KTF is valid for $h_T$, it follows that
 \[ \operatorname{Geo}_1(h) + \operatorname{Geo}_2(h)
 = \lim_{T \rightarrow \infty} (\operatorname{Geo}_1(h_T) + \operatorname{Geo}_2(h_T))\]
 \[  = \lim_{T \rightarrow \infty} (\operatorname{Spec}_1(h_T) + \operatorname{Spec}_2(h_T))
   = \operatorname{Spec}_1(h) + \operatorname{Spec}_2(h).\]
      This completes the proof.
\end{proof}
\begin{proof}[Proof of Theorem \ref{hmain}]
The theorem follows immediately by the fact that one can take $\ell=12$ in 
  Proposition \ref{hmain2}.  
  See the remark after Proposition \ref{ktfcusp}, or for a self-contained proof
  of this fact, see Proposition \ref{ell=12} below.
\end{proof}

\subsection{$R_0(f)$ for $f$ not smooth or compactly supported}

In this section and the next, $f$ will denote a function on $G(\A)$, rather than
  on $G(\R)^+$.  
The purpose of these sections is to prove that
%
%
  $R_0(f)$ is a Hilbert-Schmidt operator
  under certain mild assumptions on $f$.  
  See Proposition \ref{Tcusp}.
  The discussion that follows is independent of the material in the previous
  sections.  In particular,
  we do not assume \eqref{ht} or equivalent bounds unless explicitly 
  stated.

Throughout this section let $f=f_\infty\times f_{\fin}$ be a complex-valued 
  function on $G(\A)$,
  with $f_\infty$ bi-$K_\infty$-invariant, supported on $G(\R)^+$,
  and satisfying
  \begin{equation} \label{fbound}
   f_\infty (\mat abcd)
     \ll \frac{(ad-bc)^{\alpha/2}}{(a^2+b^2+c^2+d^2+2(ad-bc))^{\alpha/2}}
     \end{equation}
 for some $\alpha> 2$.  Here $\alpha$ plays the role of the weight $\k$ in \S18-19 of \cite{KL}.
The above is equivalent to
  \begin{equation}\label{Vbound}
 V(u) \ll \frac1{(u+4)^{\alpha/2}},
 \end{equation}
where $V$ is the function attached to $f_\infty$ in \eqref{Vf1}.
  We do not assume that $f_\infty$ is smooth, although eventually we will require it 
  to be twice differentiable.
  We assume that $f_{\fin}$ is locally constant and compactly supported
  modulo $Z(\Af)$, and that $f(zg)=\ol{\w(z)}f(g)$ for all $z\in Z(\A)$ and $g\in G(\A)$.
  In fact, we shall assume that
\begin{equation}\label{ffsup}
\Supp(f_{\fin}) = Z(\Af)K'\delta K'
\end{equation}
  for some $\delta\in M_2(\Zhat)$
  and some open compact subgroup $K'\subset K_{\fin}$ under which $f_{\fin}$ is bi-invariant.
This entails no loss of generality, since any function $f_{\fin}$ as described
  above \eqref{ffsup} is a finite
  linear combination of functions as in \eqref{ffsup}.

\begin{proposition} \label{fL2}
  For $f$ as above, $f \in L^{q}(G(\A),\ol{\w})$ for all $q\ge 1$.
\end{proposition}
\begin{proof}
Let $f_\alpha(\smat abcd)$ denote the right-hand side of \eqref{fbound}. 
  As a function of $G(\R)$, it is bi-$K_\infty$-invariant.
  Indeed, the right or left action of $K_\infty$ on the space of matrices
  $M_2(\R)\cong \R^4$ is unitary, so it preserves the norm $a^2+b^2+c^2+d^2$.
  Therefore it is convenient to integrate using the Cartan decomposition 
  \eqref{Cartan}.
  For 
  $p=\tfrac q2\ge \tfrac12$, we have
\[  \int_{\olG(\A)} |f(g)|^{2p} dg\ll\int_{\olG(\R)}|f_\infty(g)|^{2p}dg
\ll\int_1^\infty f_\alpha(\smat{t^{1/2}}{}{}{t^{-1/2}})^{2p}(1-t^{-2})dt\]
(see e.g. the integration formulas (7.27) and (7.23) of \cite{KL}).  The latter
  integral is
\[=\int_1^\infty\frac{(1-t^{-2})}{(t+t^{-1}+2)^{p\alpha}}dt
  =\int_4^\infty\frac1{u^{p\alpha}}du <\infty,\]
since $p\ge \tfrac12$ and $\alpha>2$.
\end{proof}

By the above proposition, $f\in L^1(G(\A),\ol{\w})$.  Therefore it defines an operator $R(f)$ 
  on $L^2(\w)$,
  given by the kernel
    \[ K(g_1, g_2) = \sum_{\gamma \in \olG(\Q)} f(g_1^{-1} \gamma g_2). \]
  We will work with Arthur's truncated\index{keywords}{truncation}
  kernel $K^T(g_1,g_2)$, defined as follows.
 For $T>0$, let $\tT:G(\A)\longrightarrow \{0,1\}$ be the characteristic function of
  the set of $g\in G(\A)$ with height $H(g)>T$.  Then\index{notations}{KT@$K^T(x,y)$}\index{keywords}{kernel function!truncated}
\begin{align} \label{KT}
K^T(g_1,g_2)
= K(g_1,g_2) - \sum_{\sss \delta \in B(\Q) \bs G(\Q)} \sum_{\sss \mu \in \ol{M}(\Q)}
      \left( \int_{\sss N(\A)} f(g_1^{-1} \mu n \delta g_2) dn \right) \tT(\delta g_2)\\
\notag = \sum_{\sss \gamma \in \olG(\Q)} f(g_1^{-1} \gamma g_2) - \sum_{\sss \delta \in B(\Q) \bs G(\Q)} \sum_{\sss \mu \in \ol{M}(\Q)}
      \left( \int_{\sss N(\A)} f(g_1^{-1} \mu n \delta g_2) dn \right) \tT(\delta g_2).
\end{align}
This is a function on $G(\A)\times G(\A)$, but it is not hard to see that
  it is well-defined on $(B(\Q)\bs G(\A)) \times (G(\Q) \bs G(\A))$.


\begin{proposition} \label{KTabscon}
  For all $g_1,g_2\in G(\A)$, $K^T(g_1,g_2)$ is absolutely convergent, i.e.
\[  \sum_{\sss \gamma \in \olG(\Q)} |f(g_1^{-1} \gamma g_2)| + \sum_{\sss \delta \in B(\Q) \bs G(\Q)} \sum_{\sss \mu \in \ol{M}(\Q)}
      \int_{\sss N(\A)} |f(g_1^{-1} \mu n \delta g_2)| dn \tT(\delta g_2) < \infty. \]
Furthermore, the above is bounded on compact subsets of $\olG(\A)\times\olG(\A)$.
\end{proposition}
\begin{proof} 
  By Proposition 18.4 of \cite{KL} and 
  the discussion following its proof, the sum over $\g$ is convergent, and in
  fact continuous as a function of $(g_1,g_2)$, so the assertions hold
  for this piece of the function.  
  For the same reason, the sum 
  $\sum_{\mu}\sum_{n'\in N(\Q)}|f(g_1^{-1}\mu n' n \delta g_2)|$ converges to a 
  continuous function of $n\in N(\Q)\bs N(\A)$.  Therefore it is integrable over
  the compact set $N(\Q)\bs N(\A)$, i.e.
\[\int_{\sss N(\Q)\bs N(\A)} \sum_{\mu}\sum_{n'\in N(\Q)}|f(g_1^{-1}\mu n' n \delta g_2)|dn=
  \sum_{\mu}\int_{\sss N(\A)}|f(g_1^{-1}\mu n \delta g_2)|dn<\infty.\]
By Lemma 17.1 of \cite{KL}, $\tT(\delta g_2)\neq 0$ for at most 
  one $\delta\in B(\Q)\bs G(\Q)$.  In fact, since $K^T$ is left $G(\Q)$-invariant
  as a function of $g_2$, we can assume that $g_2$ lies in a fixed
  fundamental domain for $\olG(\Q)\bs\olG(\A)$,  so the set of $\delta$ that
  contribute to the sum is finite and independent of $g_2$
  (\cite{KL}, Proposition 17.2).  The first assertion of the proposition
  now follows immediately.  From the fact that the expression is a finite
  sum of functions of $(g_1,g_2)\in \olG(\Q)\times\olG(\A)$, 
  each of which is a product of a 
  continuous function with a characteristic function, we see that it
  is bounded on compact subsets.
 \end{proof}


Let $K_{[0,\pi)}$ denote the set of matrices $\smat{\cos\theta}{\sin\theta}{-\sin\theta}
  {\cos\theta}$ with $\theta\in [0,\pi)$.  Then 
\[F\eqdef\label{Fdef}\left\{\mat1x01\mat{y^{1/2}}{}{}{y^{-1/2}}k|\, x\in[-\tfrac12,\tfrac12],
  y>0, x^2+y^2\ge 1, k\in K_{[0,\pi)}\right\}\]
 is a fundamental domain in $\SL_2(\R)$ for the quotient $\SL_2(\Z) \bs \SL_2(\R)$.
  This means that the projection $F\rightarrow \SL_2(\Z)\bs\SL_2(\R)$ is 
 surjective, and injective except on a set of measure zero.
    The set $Z(\R)^+F$ is then a fundamental domain for $\SL_2(\Z)\bs G(\R)^+$,
  and it follows (from strong approximation for $G(\A)$ and the ``divorce theorem" 
  on page 101 of \cite{KL})
 that the set
\[\mathfrak{F}\index{notations}{F@$\mathfrak{F},\ol{\mathfrak{F}}$} = Z(\R)^+ {F}\times K_{\fin}\] 
 is a fundamental domain in $G(\A)$ for $G(\Q)\bs G(\A)$.
The subset \[\ol{\mathfrak F}=F\times K_{\fin}\] 
contains a fundamental domain for $Z(\A)G(\Q)\bs G(\A)$, and can be used as a
  domain of integration for the latter quotient (\cite{KL}, Corollary 7.44).
    Let $L^2(\mathfrak{F}, \w)$ be the Hilbert space of measurable functions 
  $\phi: \mathfrak{F} \longrightarrow \C$ such that
\begin{itemize}
\item $\phi(z g) = \w(z) \phi(g)$  for all $g\in\mathfrak{F}$ and
   $z \in Z(\A)\cap \mathfrak{F}=Z(\R)^+\times \Zhat^*$, 
    \item $\ds \|\phi\|_{\mathfrak{F}}^2=\int_{\sss \ol{\mathfrak{F}}} 
  |\phi(g)|^2 dg < \infty$.
\end{itemize}


\begin{lemma}\label{mu}  Let $\delta\in M_2(\Zhat)$ be as in \eqref{ffsup}, and define 
  $D\in\Z^+$ by $D\Zhat = (\det\delta)\Zhat$.  Suppose 
\[f(g_1^{-1}\mu\smat 1t01g_2)\neq 0\]
  for some $g_1,g_2\in \mathfrak{F}$, $\mu\in M(\Q)$, and $t\in\A$.
  Then $t_{\fin}\in \frac1D\Zhat$ and $\mu\in Z(\Q)\smat a{}{}d$ for integers $a,d>0$
  with $ad=D$.
\end{lemma}
\begin{proof}
Consider the finite part $f_{\fin}(k_1^{-1}\mu\smat 1{t_{\fin}}01 k_2)\neq 0$.
  By \eqref{ffsup}, there exists $\beta\in \Af^*$ such that 
\[\det\mu\in \beta^2D\Zhat^*.\]
By strong approximation for the ideles (\cite{KL}, Prop. 
  5.10),\index{keywords}{strong approximation!for $\Af^*$}
   $\Af^*=\Q^*\Zhat^*$, so $\beta=r\beta'$ for some
  $r\in\Q^*$ and $\beta'\in\Zhat^*$.  Therefore 
\[r^{-2}\det\mu \in D\Zhat^*\cap \Q^*=\{D,-D\}.\]
  Writing $r^{-1}\mu=\smat a{}{}d\in M(\Q)$, we have
   $ad=\pm D$.  Replacing $r$ by $-r$ if necessary,
  we can assume that $a>0$.
  Now $k_1^{-1}\smat a{}{}d \smat1{t_{\fin}}01k_2\in \Supp(f_{\fin})$, and since its determinant
  belongs to $D\Zhat^*$, we see that its $Z(\Af)$ component as in \eqref{ffsup} 
  must belong to $\Zhat^*$. It follows that $k_1^{-1}\smat a{at_{\fin}}0d k_2\in M_2(\Zhat)$, and 
  hence $\smat a{at_{\fin}}0d\in M_2(\Zhat)$.
  This means that $a,d\in\Z$ and $t_{\fin}\in \tfrac1a\Zhat\subset\tfrac1D\Zhat$.
  Finally, the fact that $f_\infty$ is supported in $G(\R)^+$ 
  implies $ad>0$, so $ad=D$.
\end{proof}

\begin{proposition} \label{KTsi} 
  Let $f$ be as described at the beginning of this section, and
  suppose in addition that $f_\infty$ is twice continuously 
  differentiable.  
  Let $V$ be the function on $[0,\infty)$
  attached to $f_\infty$ as in \eqref{Vf1}.
  Suppose there exists $\e>0$ such that for all $u>0$,
  \begin{equation} \label{Vsi}
  \begin{cases}
   V(u), V'(u) \ll (1+u)^{-1-\e} \\
   V''(u) \ll (1+u)^{-3/2-\e}.
   \end{cases}
  \end{equation}
(The bound on $V(u)$ is already a consequence of \eqref{Vbound}.) Then
  \begin{equation} \label{KTHS}
  \|K^T\|^2_{\sss {\mathfrak{F}} \times G(\Q) \bs G(\A)} = \int_{\sss \ol{\mathfrak{F}}} 
  \int_{\sss \olG(\Q) \bs \olG(\A)} |K^T(g_1, g_2)|^2dg_2 dg_1 < \infty,
  \end{equation}
  or equivalently,
  \[
  \|K^T\|^2_{\sss {\mathfrak{F}} \times \mathfrak{F}} =  \int_{\sss \ol{\mathfrak{F}}} 
 \int_{\sss \ol{\mathfrak{F}}} |K^T(g_1, g_2)|^2dg_2 dg_1 < \infty.
  \]
\end{proposition}
\noindent{\em Remark:} We do not assume that $V$ is differentiable
  at the endpoint $u=0$.

\begin{proof}
 The proof is somewhat involved and will be given in the next subsection.
 It basically follows \S 19 of \cite{KL}.
\end{proof}

Under the hypotheses of the above proposition, we can define a map 
  $T_{K^T}: L^2(\w) \rightarrow L^2(\mathfrak{F},\w)$ by
\[T_{K^T}\phi(g_1) = \int_{\sss \olG(\Q) \bs \olG(\A)} K^T(g_1, g_2) \phi(g_2) dg_2.\]
Let $r: L^2(\mathfrak F,\w) \rightarrow L^2(\w)$ be the map defined by 
  $r\phi(G(\Q)g)=\phi(g)$ for a.e. $g\in \mathfrak{F}$.  (The set of points $g\in \mathfrak{F}$
  for which $\phi(g)$ is not uniquely determined by $G(\Q)g$ has measure $0$.)
  Because $\mathfrak{F}$ is a fundamental domain for 
  $G(\Q)\bs G(\A)$, the map $r$ is an isomorphism.
  By identifying the two spaces in this way, we can abuse terminology and refer
  to $T_{K^T}$ as an operator on $L^2(\w)$.
  For future reference, we let $L^2_0(\mathfrak{F}, \w)$ be the 
  preimage of $L^2_0(\w)$ under $r$.

\begin{corollary}\label{HScor}
     The map $T_{K^T}: L^2(\w) \longrightarrow L^2(\mathfrak{F}, \w)\cong L^2(\w)$
   is a Hilbert-Schmidt operator. 
 The Hilbert-Schmidt norm $\|T_{K^T}\|^2_{HS}$ is equal to \eqref{KTHS}.
\end{corollary}
\begin{proof} This is a consequence of \eqref{KTHS}. (See \cite{RS} Theorem VI.23).
 \end{proof}

The next proposition shows that $T_{K^T}$ coincides with $R(f)$ on the cuspidal
  subspace, and it then follows from Corollary \ref{HScor} that $R_0(f)$ is
  Hilbert-Schmidt.

\begin{proposition}\label{Tcusp}\index{notations}{R0@$R_0$}
  Suppose the hypotheses of Proposition \ref{KTsi} are satisfied. Then
  $T_{K^T}|_{L_0^2(\w)} = R(f)|_{L_0^2(\w)} = R_0(f)$.
As a result, the operator $R_0(f)$ is Hilbert-Schmidt.
\end{proposition}
\begin{proof}
    Let $\phi \in L_0^2(\w)$ and $g_1\in \mathfrak{F}$. 
Then 
\[R(f)\phi(g_1)=\int_{\olG(\Q)\bs\olG(\A)}K(g_1,g_2)\phi(g_2)dg_2,\]
 where the integral converges absolutely since $f\in L^1(\ol\w)$
  (cf. (10.7) of \cite{KL}).  Thus
  by the linearity of integration, $T_{K^T} \phi(g_1)$ is equal to
    \[ 
  R(f) \phi (g_1) - \IL{\sss \olG(\Q) \bs \olG(\A)} \sum_{\sss \delta \in B(\Q) \bs G(\Q)} 
\sum_{\sss \mu \in \ol{M}(\Q)}\left( \int_{\sss N(\A)} f(g_1^{-1} \mu n \delta g_2) 
  \phi(g_2) dn \right) \tT(\delta g_2) dg_2.\]
  It suffices to show that the second term vanishes.
  At the end of the proof, we will verify that it is absolutely convergent,
  so we can rearrange the sums and integrals.  Granting this for the moment,
  by the left $G(\Q)$-invariance of $\phi$, the second term is 
\begin{equation}\label{2abs}  \int_{\sss \olG(\Q) \bs \olG(\A)}
    \sum_{\sss \delta \in B(\Q) \bs G(\Q)}
    \sum_{\sss \mu \in \ol{M}(\Q)}
   \int_{\sss N(\A)} f(g_1^{-1} \mu n \delta g_2) dn\,\phi(\delta g_2) \tT(\delta g_2) dg_2
   \end{equation}
\[ =\int_{\sss \ol{B}(\Q) \bs \olG(\A)}
    \sum_{\sss \mu \in \ol{M}(\Q)}
   \int_{\sss N(\A)}  f(g_1^{-1} \mu n g_2)dn\,  \phi(g_2) \tT(g_2) dg_2 \]
\[ =\int_{\sss \ol{B}(\Q) N(\A) \bs \olG(\A)}
\int_{\sss N(\Q) \bs N(\A)}
    \sum_{\sss \mu \in \ol{M}(\Q)}
   \int_{\sss N(\A)}  f(g_1^{-1} \mu n n' g_2)dn\,  \phi(n' g_2) \tT(n' g_2) dn' dg_2 \]
\[ =\int_{\sss \ol{B}(\Q) N(\A) \bs \olG(\A)}
\int_{\sss N(\Q) \bs N(\A)}
    \sum_{\sss \mu \in \ol{M}(\Q)}
   \int_{\sss N(\A)}  f(g_1^{-1} \mu n  g_2)dn\,  \phi(n' g_2) \tT(g_2) dn' dg_2 \]
\[=\int_{\sss \ol{B}(\Q) N(\A) \bs \olG(\A)}
    \sum_{\sss \mu \in \ol{M}(\Q)}
   \int_{\sss N(\A)} f(g_1^{-1} \mu n  g_2)dn\,
   \left(\int_{\sss N(\Q) \bs N(\A)} \phi(n' g_2)dn' \right)\tT(g_2)  dg_2.\]
    This vanishes because $\phi$ is cuspidal, and hence $T_{K^T} \phi = R(f) \phi$
  as needed.

It remains to prove the absolute convergence.
By Proposition \ref{KTabscon},
\[  \Phi_{g_1}(g_2)\eqdef \sum_{\sss \delta \in B(\Q) \bs G(\Q)} \sum_{\sss \mu \in \ol{M}(\Q)}
      \int_{\sss N(\A)} | f(g_1^{-1} \mu n \delta g_2)| dn \tT(\delta g_2) \]
  is convergent, and bounded on compact sets.   We will show that it is
  square-integrable (and in fact bounded) as a function of $g_2\in \olG(\Q)\bs\olG(\A)$.
  Partition the fundamental domain $\mathfrak F$ as
\begin{equation}\label{FT}
\mathfrak{F}=\mathfrak{F}_T\cup \widetilde{\mathfrak{F}}_T,
\end{equation}
where $\mathfrak{F}_T\index{notations}{F@$\mathfrak{F}_T,\ol{\mathfrak{F}}_T$} = \{g \in \mathfrak{F} \,|\, H(g) > T\}$, and 
$\widetilde{\mathfrak{F}}_T $ is its complement.
  Correspondingly, we set $\ol{\mathfrak{F}}=\ol{\mathfrak{F}}_T\cup
  \ol{\widetilde{\mathfrak{F}}}_T$.
  Clearly $\ol{\widetilde{\mathfrak{F}}}_T$ is compact.
  In particular, $\Phi_{g_1}$ is square integrable over $\ol{\widetilde{\mathfrak F}}_T$.

For $g_2 \in \mathfrak{F}_T$,
 $\tT(\delta g_2)\neq 0$ only for $\delta=1$ (see e.g. \cite{KL}, Lemma 17.1), so
\[ \Phi_{g_1}(g_2) = \sum_{\sss \mu \in \ol{M}(\Q)}
      \int_{\sss N(\A)} |f(g_1^{-1} \mu n  g_2)| dn\qquad(g_2\in \mathfrak{F}_T).\]
For $i=1,2$, write $g_i = \smat{1}{x_i}{}{1} \smat{y_i^{1/2}}{}{}{y_i^{-1/2}} r_{i}
  \times k_{i}$
  for $x_i \in \R, y_i > 0, r_{i} \in K_\infty$, and  $k_{i} \in K_{\fin}$. 
By Lemma \ref{mu},
\[\Phi_{g_1}(g_2) = \sum_{a|D, a>0,\atop{ad=D}} \int_{\sss N(\R)}|f_\infty(g_{1\infty}^{-1}
  \smat a{}{}dng_{2\infty})|dn
\int_{\sss N(\frac1D\Zhat)} |f_{\fin}(k_1^{-1} \smat a{}{}d n  k_2)| dn,\]
where $N(\frac1D\Zhat)=\{\smat 1t01|\, t\in \tfrac1D\Zhat\}$.
  The finite part is obviously bounded by $\meas(\tfrac1D\Zhat)=D$.
  For the infinite part, we refer ahead to the bound \eqref{Fhat0} 
  in the next section (the proof there for $f$ works just as well for 
  $|f|$), by which for any given $\e>0$,
\[ \Phi_{g_1}(g_2) \ll_{\e} \sum_{ad=D} \frac{(\frac{dy_1 y_2}a)^{\frac12}}
  {(\frac{ay_2}{dy_1}+\frac{dy_1}{ay_2}-1)^{\frac12+\e}}\ll\frac{(y_1y_2)^{\frac12}}
  {(\frac{y_2}{y_1}+\frac{y_1}{y_2}-1)^{\frac12+\e}} \ll_{g_1} 1
   \qquad(y_2>T).\]
It follows that $\Phi_{g_1}$ is square-integrable on the finite measure space
  $\ol{\mathfrak{F}}_T$, and hence
\[ \Phi_{g_1} \in L^2(\ol{\mathfrak{F}}),\text{ or equivalently, }\,\Phi_{g_1} \in L^2(\olG(\Q) \bs \olG(\A)). \]
 Therefore by Cauchy-Schwarz, for any $\phi \in L^2(\w)$,
\[ \IL{\sss\olG(\Q) \bs \olG(\A)} \Phi_{g_1}(g)|\phi(g)| dg
  \le \Bigl(\IL{\sss\olG(\Q) \bs \olG(\A)} \Phi_{g_1}(g)^2  dg\Bigr)^{1/2}
 \Bigl(\IL{\sss\olG(\Q) \bs \olG(\A)} |\phi(g)|^2 dg\Bigr)^{1/2} < \infty. \]
This proves that \eqref{2abs} is absolutely convergent.
\end{proof}

\begin{corollary}[Theorem \ref{tc}]\label{tcpf}
Suppose $f=f_\infty\times f_{\fin}\in C_c^m(G(\A),\ol\w)$ for $m\ge 2$
  and $f_\infty$ bi-$K_\infty$-invariant.
    Then $R_0(f)$ is a Hilbert-Schmidt operator.
\end{corollary}
\begin{proof} 
Because $m\ge 2$, $V$ is twice differentiable on the open interval $(0,\infty)$
(Proposition \ref{uydiff}).  Since it also has compact support, 
  it trivially satisfies \eqref{Vsi}.  Hence the result follows from
   Proposition \ref{Tcusp}.
\end{proof}

\begin{corollary}\label{genHS} Suppose $h$ satisfies \eqref{ht} with 
  $A>3$ and $B>4$. Let $f=f_\infty\times f_{\fin}$ for $f_\infty$ corresponding
  to $h$, and $f_{\fin}$ as described below \eqref{Vbound}.
  Then $R_0(f)$ is a Hilbert-Schmidt operator. 
\end{corollary}
    \begin{proof} By Proposition \ref{Vd}, $V$ satisfies \eqref{Vsi}.  Therefore
  the result follows by Proposition \ref{Tcusp}.
 \end{proof}

\begin{proposition} \label{ell=12}
  Let $r_\ell=\frac1{(1+|t|)^\ell}$.  Then in the notation of 
  Proposition \ref{hmain2}, $\operatorname{Spec}_1^a(r_\ell)<\infty$ if $\ell>9$.
 \end{proposition}
\begin{proof}
(See also the remark after Proposition \ref{ktfcusp}.)
Given $\ell>9$, fix any $A> 3$, and let $h(t) = \tfrac1{(4A^2+t^2)^{\ell/2}}$.
  (The purpose of $4A^2$ is to ensure that $h$ is holomorphic on $|\Im(t)|<A$.)
  Let ${f}=f_\infty\times f^1$ for $f_\infty$ corresponding
   to ${h}_0(t) = (4A^2+t^2)^{-(\ell-1)/4}$ and $f^1$ the identity Hecke operator on 
  $G(\Af)$ (corresponding to $\n=1$).  Then $R_0(f)\varphi_{u_j}=h_0(t_j)\varphi_{u_j}$
  for all Maass cusp forms $u_j$.  It is not hard to show that ${h}_0$ 
  satisfies \eqref{ht} with $B=\tfrac{\ell-1}2>4$.
   By equation \eqref{cuspsum},
   \[ \operatorname{Spec}_1^a(r_\ell) \ll
   \operatorname{Spec}_1^a(h) \ll \sum_{u_j\in\mathcal F}
   \frac{(1+|t_j|)}{|4A^2+t_j^2|^{\ell/2}}
   \ll \sum_{u_j}  \frac{1}{|4A^2+t_j^2|^{(\ell-1)/2}} \]
   \[ = \sum_{u_j} |{h}_0(t_j)|^2
   = \|R_0(f)\|_{HS}^2 < \infty. \]
   The last step follows from Corollary \ref{genHS}.
\end{proof}


\subsection{Proof of Proposition \ref{KTsi}}

Here we assume that \eqref{Vsi} holds.  Set
  \begin{equation} \label{K1}
      K_1(g_1, g_2) = \sum_{\sss \gamma \not \in \ol{B}(\Q)} f(g_1^{-1} \gamma g_2),
  \end{equation}
and
   \begin{equation} \label{KT2}
\hskip -.2cm    K^T_2(g_1, g_2) =
   \sum_{\sss \gamma  \in \ol{B}(\Q)} f(g_1^{-1} \gamma g_2)
   -  \hskip -.2cm\sum_{\sss \delta \in B(\Q) \bs G(\Q)} \sum_{\sss \mu \in \ol{M}(\Q)}
      \left( \int_{N(\A)} f(g_1^{-1} \mu n \delta g_2) dn \right) \tT(\delta g_2).
   \end{equation}
   Then
   \[ K^T(g_1, g_2) = K_1(g_1,g_2) + K_2^T(g_1, g_2). \]
We will show that each of these terms is square integrable over $\ol{\mathfrak F}
  \times \ol{\mathfrak F}$.

For $g_i\in\SL_2(\R)\times K_{\fin}\subset G(\A)$, we write
  \begin{equation} \label{gi}
  g_i = \mat{1}{x_i}{}{1} \mat{y_i^{1/2}}{}{}{y_i^{-1/2}} r_i\times k_i,
  \end{equation}
  where $x_i \in \R, y_i > 0, r_{i} \in K_\infty, k_{i} \in K_{\fin}$. 
  Note that if $g_i \in \ol{\mathfrak{F}}$, then
   $x_i \in [-\tfrac12,\tfrac12]$ and $y_i \geq \frac{\sqrt{3}}{2}$.

\begin{lemma} Given $\alpha>2$ as in \eqref{fbound}, then with notation as above, 
  for $g_1, g_2 \in \mathfrak{F}$ we have
   \[ \sum_{\sss \gamma \not \in \ol{B}(\Q)} |f(g_1^{-1} \gamma g_2)| \ll_{\alpha}
    \frac{1}{y_1^{\alpha/2-1} y_2^{\alpha/2-1}} + \frac{1}{y_1^{\alpha/2} y_2^{\alpha/2-1}}. \]
   \end{lemma}
    \begin{proof} In view of (18.7) of \cite{KL}, the result holds by
  (18.3) and (18.4) of \cite{KL} Lemma 18.3 with
  $C_1=\{g_1\}$, $C_2=\{g_2\}$, $L_1=(y_2/y_1)^{1/2}$, $L_2=(y_1/y_2)^{1/2}$,
  and $L_3=(y_1y_2)^{1/2}$.
   \end{proof}

\begin{proposition} \label{K1si} 
  \[\|K_1\|_{\sss {\mathfrak{F}} \times {\mathfrak{F}}}<\infty.\]
\end{proposition}
\begin{proof}  The square 
  $\|K_1\|^2_{\sss {\mathfrak{F}} \times {\mathfrak{F}}}$
  of the $L^2$-norm is
\[   \int_{\sss \ol{\mathfrak{F}}} \int_{\sss \ol{\mathfrak{F}}}
  \left| \sum_{\sss \gamma \not \in \ol{B}(\Q)} f(g_1^{-1} \gamma g_2) \right|^2 dg_1 dg_2 
  \le \int_{\sss \ol{\mathfrak{F}}} \int_{\sss \ol{\mathfrak{F}}}
  \left( \sum_{\sss \gamma \not \in \ol{B}(\Q)} |f(g_1^{-1} \gamma g_2)| \right)^2 dg_1 dg_2.\]
   By the above lemma, the latter expression is
\[   \ll \int_{\frac{\sqrt{3}}2}^\infty \int_{-1/2}^{1/2}
   \int_{\frac{\sqrt{3}}2}^\infty  \int_{-1/2}^{1/2}  
    \left(\frac{1}{y_1^{\alpha/2-1} y_2^{\alpha/2-1}} + 
  \frac{1}{y_1^{\alpha/2} y_2^{\alpha/2-1}}\right)^2
\frac{dx_1 dy_1 dx_2 dy_2}{y_1^2\, y_2^2} \]
\[   \ll \int_{\frac{\sqrt{3}}2}^\infty \int_{\frac{\sqrt{3}}2}^\infty  
  \frac{dy_1 dy_2}{y_1^{\alpha}\, {y_2^{\alpha}}} < \infty.
\qedhere \]
\end{proof}

It remains to treat $K^T_2(g_1,g_2)$, for which $\g\in B(\Q)$.
  When $g_1,g_2\in \mathfrak{F}$, we can assume that $\det \g>0$, since otherwise
  $f_\infty$ vanishes.  Thus, for $\mu \in M(\Q)^+$ we 
define\index{notations}{Fm@$F_{\mu,g_1,g_2}$}
    \[ F_{\mu, g_1, g_2}(t) = f(g_1^{-1} \mu \mat 1t{}1 g_2) \qquad(t\in\A).\]
Given $g_1,g_2\in\mathfrak F$, we will require bounds for the Fourier transform
\begin{equation}\label{Fhatbound}
\widehat{F}_{\mu,g_1,g_2}(r) = \int_{\A}F_{\mu,g_1,g_2}(t)\theta(rt)dt
 \le D \widehat{F}_{\mu,g_1,g_2,\infty}(r_\infty).
\end{equation}
Here we have bounded the finite part by $D$
as in the proof of Proposition \ref{Tcusp}, and 
 $\widehat{F}_{\mu,g_1,g_2,\infty}(r_\infty)
=\int_{\R}f_\infty(g_{1\infty}^{-1}\mu\smat1t01g_{2\infty})e(-r_\infty t)dt$
 is the archimedean part.


\begin{lemma} \label{Fhat}
   Let $g_1, g_2\in \SL_2(\R)$ be of the form of \eqref{gi} (but of course
  with no $G(\Af)$ component),
    and let $\mu = \smat a{}{}d\in M(\Q)^+$. Suppose $V$ satisfies \eqref{Vsi} for 
  $u>0$.  Then
   \begin{equation} \label{Fhat0}
   \widehat{F}_{\mu,g_1, g_2, \infty}(0) \ll \sqrt{\frac{dy_1 y_2}a}
  \left(\frac{ay_2}{dy_1}+\frac{dy_1}{ay_2}-1\right)^{-\frac12-\e}. 
   \end{equation}
  If, in addition, $ay_2\neq dy_1$, then for real $r\neq 0$ we have
   \begin{equation} \label{Fhat1}
    \widehat{F}_{\mu, g_1,  g_2, \infty}(r) \ll r^{-2}\sqrt{\frac{a}{dy_1 y_2}}.
    \end{equation}
If $V$ satisfies \eqref{Vsi} also at the endpoint $u=0$, then \eqref{Fhat1} holds
  even when $ay_2=dy_1$.
\end{lemma}
\begin{proof} We have
   \begin{align}\label{Vt2}
  \notag f_\infty(g_{1\infty}^{-1} \mu \mat 1t{}1 g_{2\infty})
&= f_\infty (\mat{y_1^{-1/2} y_2^{1/2}a}{y_1^{-1/2} y_2^{-1/2}(-dx_1 + ax_2 + at)}
  {0}{y_1^{1/2}y_2^{-1/2} d}) \\
   &= V(\frac{ay_2}{dy_1}  + \frac{dy_1}{ay_2} + 
  \frac{(-dx_1+ax_2+at)^2}{ady_1 y_2}-2).
\end{align}
    The Fourier transform is thus given by
  \[ \widehat{F}_{\mu, g_1, g_2,\infty}(r) =
     \int_\R V(\frac{ay_2}{dy_1}  + \frac{dy_1}{ay_2} + \frac{(-dx_1+ax_2+at)^2}{ady_1 y_2}-2) e(-r t) dt\]
  \[ = \int_\R V(\frac{ay_2}{dy_1}  + \frac{dy_1}{ay_2} + \frac{at^2}{dy_1 y_2}-2) e(-r (t-x_2+ \frac{dx_1}a)) dt\]
  \[ = e(r (x_2 -\frac{dx_1}a)) \int_\R V(\frac{ay_2}{dy_1}  + \frac{dy_1}{ay_2} + \frac{at^2}{dy_1 y_2}-2) e(-r t) dt\]
  \begin{equation}\label{Fhatcomp}
 = e(r (x_2 -\frac{dx_1}a)) \sqrt{\frac{dy_1 y_2}{a}}\int_\R V(A_{y_1,y_2} + t^2) 
  e(-r \sqrt{\tfrac{dy_1 y_2}{a}}t) dt,
\end{equation}
  where $A_{y_1,y_2} = \frac{ay_2}{dy_1}  + \frac{dy_1}{ay_2} - 2 \ge 0$.

If $r=0$, the estimate $V(u)\ll (1+u)^{-1-\e}$ (of \eqref{Vsi}) implies that
  $\widehat{F}_{\mu, g_1, g_2,\infty}(0)$ is
\[\ll  \sqrt{\frac{dy_1 y_2}{a}}\int_\R \frac{dt}{(1+A_{y_1,y_2} + t^2)^{1+\e}}
  =\frac{\sqrt{\frac{dy_1 y_2}{a}}}{(1+A_{y_1,y_2})^{1+\e}}
  \int_\R \frac{dt}{(1 + \frac{t^2}{1+A_{y_1,y_2}})^{1+\e}}\]
\[  ={\sqrt{\frac{dy_1 y_2}{a}}}\frac{(1+A_{y_1,y_2})^{1/2}}{(1+A_{y_1,y_2})^{1+\e}}
  \int_\R \frac{dt}{(1 + t^2)^{1+\e}}.\]
The estimate \eqref{Fhat0} follows.
If $r\neq0$, then by Proposition \ref{1911}, \eqref{Fhatcomp} is
\[ \ll  \sqrt{\frac{dy_1 y_2}{a}}\left(r\sqrt{\frac{dy_1y_2}a}\right)^{-2}
  \int_\R \left|\frac{d^2}{dt^2}V(A_{y_1,y_2} + t^2)\right|dt.\]
In order to prove \eqref{Fhat1}, it suffices to show that the above integral
  is bounded independently of $a,d,y_1,y_2$.  Using the bounds \eqref{Vsi}, we have
%
 \[ \int_\R \left| \frac{d^2}{dt^2} V(A_{y_1,y_2}+t^2) \right| dt
    = \int_\R |2V'(A_{y_1,y_2}+t^2) + 4t^2 V''(A_{y_1,y_2}+t^2)| dt \]
 \[ \ll \int_\R 2 (1+ A_{y_1,y_2}+t^2)^{-1-\e}dt + \int_\R 4t^2 (1+A_{y_1,y_2}+t^2)^{-3/2-\e}dt\]
 \[ \leq \int_\R 2 (1+t^2)^{-1-\e}dt + \int_\R 4t^2 (1+t^2)^{-3/2-\e}dt < \infty,\]
as needed.  Notice that when $ay_2\neq dy_1$, $A_{y_1,y_2}>0$ and the above
  involves only derivatives of $V$ on the open interval $(0,\infty)$.
\end{proof}

Recall the partition $\mathfrak{F}=\mathfrak{F}_T\cup \widetilde{\mathfrak{F}}_T$
  from \eqref{FT}.
It will be convenient to
  decompose the norm of $K_2^T$ accordingly as
\[ \|K_2^T\|_{\sss \mathfrak{F}\times\mathfrak{F}}=
  (\|K_2^T\|^2_{\sss \mathfrak{F}\times\mathfrak{F}_T}+
  \|K_2^T\|^2_{\sss \mathfrak{F}\times\widetilde{\mathfrak{F}}_T})^{1/2},\]
  and consider each piece separately.

\begin{proposition} \label{K2si} Under the hypotheses of Proposition \ref{KTsi},
\[ \|K_2^T\|_{\sss {\mathfrak{F}} \times {\mathfrak{F}}_T} < \infty. \]
\end{proposition}
\begin{proof}
    Suppose $g_2 \in \mathfrak{F}_T$.  Then
  $\tT(\delta g_2) \neq 0$ only if $\delta = 1$ (cf. Lemma 17.1 of \cite{KL}).
  Hence
    \[ K_2^T(g_1, g_2) = \sum_{\sss \gamma  \in \ol{B}(\Q)} f(g_1^{-1} \gamma g_2)
       -   \sum_{\sss \mu \in \ol{M}(\Q)}
       \int_{\sss N(\A)} f(g_1^{-1} \mu n g_2) dn  \]
    \[ =  \sum_{\sss \mu \in \ol{M}(\Q)}
      \left( \sum_{\sss \eta \in N(\Q)} f(g_1^{-1} \mu \eta  g_2) - \int_{\sss N(\A)} f(g_1^{-1} \mu n  g_2) dn \right) . \]
    The rearrangement is justified by Proposition \ref{KTabscon}.

We would like to apply the Poisson summation formula to the sum over $\eta$.
  To justify this, note that by \eqref{Vt2} and \eqref{Vsi},
   $F_{\mu,g_1,g_2,\infty}(t)\ll t^{-2}$, while
  $F_{\mu,g_1,g_2,\fin}$ is a Schwartz-Bruhat function on $\Af$ (see the proof of
  Proposition 19.10 of \cite{KL}).
  On the other hand, write $\mu=\smat a{}{}d$, take $g_1, g_2$ in the form of \eqref{gi},
  and suppose $ay_2\neq dy_1$.  Then by the above lemma,
  $\widehat{F}_{\mu,g_1, g_2,\infty}(t) \ll t^{-2}$ for $t \neq 0$.
  Hence by \cite{KL} Theorem 8.17, the adelic Poisson summation formula can be applied 
  to the global function $F_{\mu, g_1, g_2}$.

  Therefore for fixed $g_1$, using Lemma \ref{mu}
  we see that for {almost all} $g_2\in \mathfrak{F}_T$,
  \[K_2^T(g_1,g_2)=\sum_{\mu=\operatorname{diag}(a,d)\atop{ad=D,a>0}}
  \sum_{t\in \Q^*}\widehat{F}_{\mu,g_1,g_2}(t).\]
  (Poisson sumation may fail to hold
  on the set of $g_2$ with $y_2\in \{\tfrac{dy_1}a|ad=D\}$, but this is of measure $0$.)
  Because
   ${F}_{\mu,g_1,g_2,\fin}$ is a Schwartz-Bruhat function,
   $\widehat{F}_{\mu,g_1,g_2,\fin}$ is as well (cf. \cite{KL}, Proposition
  8.13).  Therefore its support is contained in $\tfrac1M\Zhat$ for some
  integer $M>0$.
 By \eqref{Fhatbound} and \eqref{Fhat1},
\begin{equation}\label{tb} 
\sum_{\mu=\operatorname{diag}(a,d)\atop{ad=D,a>0}}
  \sum_{t\in \Q^*}\widehat{F}_{\mu,g_1,g_2}(t)
  \ll\sum_{\mu=\operatorname{diag}(a,d)\atop{ad=D,a>0}}
  \sum_{t\in \frac1M\Z-\{0\}}t^{-2}\sqrt{\frac a{dy_1y_2}} \ll 
  \frac1{\sqrt{y_1y_2}}
\end{equation}
 for a.e. $g_2\in\mathfrak{F}_T$.
    Therefore
   \[ \|K_2^T\|_{\sss {\mathfrak{F}} \times {\mathfrak{F}}_T}^2
    = \int_{\ol{\mathfrak{F}}}  \int_{\ol{\mathfrak{F}}_T} |K^T_2(g_1, g_2)|^2dg_2 dg_1 
     \ll  \int_{\frac{\sqrt3}2}^{\infty} \int_{-\frac 12}^{\frac 12}
   \int_{T}^\infty \int_{-\frac 12}^{\frac 12}
    \frac{dx_2 dy_2 dx_1 dy_1}{y_1^3 y_2^3},\]
which is clearly finite.
\end{proof}

\begin{lemma}\label{Fhat0int} Let $C=C_\infty\times C_{\fin}$ be a compact 
  subset of $\olG(\A)$.
  Then for any $\mu\in M(\Q)$,
\[\int_C\int_{\ol{\mathfrak{F}}}|\widehat{F}_{\mu,g_1,g_2}(0)|^2
  dg_1dg_2<\infty.\]
\end{lemma}

\begin{proof}
The above integral factorizes as
\begin{equation}\label{2di} 
\int_{C_\infty}\int_{F}|\widehat{{F}}_{\mu,g_1,g_2,\infty}(0)|^2dg_{1\infty}dg_{2\infty}
\int_{C_{\fin}}\int_{K_{\fin}}|\widehat{{F}}_{\mu,k,g_2,\fin}(0)|^2dk\,dg_{2\!\fin},
\end{equation}
where $F$ is the archimedean part of $\ol{\mathfrak{F}}$, defined on page \pageref{Fdef}.
Observe that by \eqref{ffsup}, $f_{\fin}(k_1^{-1}\mu\smat1t01g_{2\!\fin})\neq 0$ only if
\[ \mat 1t01\in
  Z(\Af)\mu^{-1}K_{\fin}\delta K_{\fin}C_{\fin}^{-1}.\]
Since $\mu^{-1}K_{\fin}\delta K_{\fin}C_{\fin}^{-1}$ is compact, taking 
  the determinant of both sides we see that the $Z(\Af)$-part of $\smat 1t01$
  is also restricted to a compact set, i.e.
\[ \mat 1t01\in
  Z_0\,\mu^{-1}K_{\fin}\delta K_{\fin}C_{\fin}^{-1}.\]
for some compact subset $Z_0\subset Z(\Af)$.  The above set is compact, so
  it follows that
  $t$ is restricted to some compact subset $B\subset \Af$, and hence
\[|\widehat{F}_{\mu,k_1,g_2,\fin}(0)|\le \int_{B}|f_{\fin}(k_1^{-1}\mu\smat1t01g_2)|dt
  \le \meas(B).\]
Therefore, the non-archimedean double integral in \eqref{2di} is finite.

For the infinite part, without loss of generality we can assume that $C_\infty\subset
  \SL_2(\R)$, so it consists of elements $\smat1x01\smat{\sqrt y}{}{}{\sqrt{y}^{-1}}k_\infty$
  with $-L\le x\le L$ and $0<T_1\le y\le T_2$ for some constants $L,T_1,T_2$.
  By \eqref{Fhat0},
\[|\widehat{F}_{\mu,g_1, g_2,\infty}(0)|^2 \ll_{\e} \frac{dy_1 y_2}
  {a(\frac{dy_1}{ay_2}+\frac{ay_2}{dy_1}-1)^{1+\e}}
\ll \frac{y_1 y_2}
  {(\frac{dy_1}{ay_2}+\frac{ay_2}{dy_1})^{1+\e}},\]
 where the latter bound holds by the fact that $\frac{dy_1}{ay_2}+\frac{ay_2}{dy_1}\ge 2$.
Hence 
\[\int_{C_\infty}\int_{F}|\widehat{F}_{g_1,g_2,\mu,\infty}(0)|^2
  dg_{1\infty}dg_{2\infty}\]
\[\ll \int_{-L}^L \int_{-1/2}^{1/2} \int_{T_1}^{T_2}
\int_{\frac{\sqrt3}2}^\infty \frac{y_1 y_2}
  {(\frac{dy_1}{ay_2}+\frac{ay_2}{dy_1})^{1+\e}}\frac{dy_1dy_2}{y_1^2y_2^2}dx_1dx_2\]
\[\ll\int_{T_1}^{T_2} \left(\int_{\frac{\sqrt3}2}^\infty 
  \frac1{(\frac{dy_1}{ay_2}+\frac{ay_2}{dy_1})^{1+\e}}\frac{dy_1}{y_1}\right)\frac{dy_2}{y_2}
=\int_{T_1}^{T_2} \left(\int_{\frac{d\sqrt3}{2ay_2}}^\infty 
  \frac1{(y+y^{-1})^{1+\e}}\frac{dy}{y}\right)\frac{dy_2}{y_2}.\]
The inner integral is absolutely convergent, and defines a continuous function
  of $y_2$.  Therefore the outer integral converges as well.
\end{proof}

\begin{lemma}\label{mudel}
Given $\delta\in B(\Q)\bs G(\Q)$, there exists a finite subset $A_\delta\subset
  \ol{M}(\Q)$ such that
$\widehat{F}_{\mu,g_1,\delta g_2}(0)$
  is identically $0$ as a function of $(g_1,g_2)\in \mathfrak F\times
  \mathfrak F$ for all $\mu\in \ol{M}(\Q)$ which are not in $A_\delta$.
\end{lemma}

\begin{proof}
Write $g_{i\fin}=k_i\in K_{\fin}$.  The lemma follows by looking at the finite part
\[\widehat{F}_{\mu,g_1,\delta g_2,\fin}(0)=\int_{\Af}
  f_{\fin}(k_1^{-1}\smat a{}{}d\smat1t01\delta k_2)dt.\]
By the Bruhat decomposition $G(\Q)=B(\Q)\cup B(\Q)\smat{}11{}N(\Q)$, we can take
\[\delta\in \{1\}\cup\{\smat 011r|r\in\Q\}.\]
When $\delta=1$, the assertion follows from 
  Lemma \ref{mu}.  Hence, we may suppose that $\delta=\smat011r$.
  Suppose 
\[  f_{\fin}(k_1^{-1}\smat a{}{}d\smat1t01\smat 011rk_2)\neq 0.\]
  Taking the determinant and arguing as in the
 proof of Proposition \ref{mu}, we can assume that $a>0$, $ad=\pm D$, and
\[ \mat a{}{}d\mat1t01\mat 011r=\mat{at}{a+(at)r}{d}{dr}
  \in M_2(\Zhat).\]
In particular, $d\in\Z$ and $at\in \Zhat$.
   From the fact that the upper right-hand entry also belongs to $\Zhat$,
   it then follows that
\[\frac Dd=\pm a\in 
   (r\Zhat+\Zhat)\cap \Q = \frac1\beta\Z\]
if $r=\tfrac\alpha\beta$ for $\alpha,\beta\in\Z$ relatively prime.
   It follows that $d|{\beta}D$.   In particular, the set of such $d$ is finite.
\end{proof}

\begin{proposition} \label{K2sic} 
  \[ \|K_2^T\|_{\sss {\mathfrak{F}} \times {\widetilde{\mathfrak{F}}}_T} < \infty. \]
\end{proposition}
\begin{proof} By definition, for $g_1,g_2\in\mathfrak{F}$, we have
\[K_2^T(g_1,g_2)=\sum_{\mu\in\ol{M}(\Q)}\sum_{t\in\Q}F_{\mu,g_1,g_2}(t)
  - \sum_{\delta\in B(\Q)\bs G(\Q)}\sum_{\mu\in\ol{M}(\Q)}
  \widehat{F}_{\mu,g_1,\delta g_2}(0)\tT(\delta g_2).\]
  As in the proof of Proposition \ref{K2si}, for fixed $g_1$, we can apply
  Poisson summation to the sum over $t$ for a.e. $g_2$, so
   $K_2^T(g_1,g_2)$ is equal almost everywhere to the
  sum of the following three functions:
\begin{enumerate}
\item $\ds \sum_{\mu=\operatorname{diag}(a,d),\atop{ad=D,a>0}}\sum_{t\in \frac1M\Z-\{0\}}\widehat{F}_{\mu,g_1,g_2}(t)$
\item $\ds \sum_{\mu=\operatorname{diag}(a,d),\atop{ad=D,a>0}}
  \widehat{F}_{\mu,g_1,g_2}(0)$
\item $\ds  - \sum_{\delta\in B(\Q)\bs G(\Q)} \sum_{\mu\in A_\delta}
       \widehat{F}_{\mu,g_1,\delta g_2}(0)\tT(\delta g_2)$,
\end{enumerate}
where $A_\delta$ is the finite subset of $\ol M(\Q)$ given by Lemma \ref{mudel}.
  By Minkowski's inequality,
  it suffices to show that each of these three functions is square integrable
  over $(g_1,g_2)\in \ol{\mathfrak{F}}\times \ol{\widetilde{\mathfrak F}}_T$.

By \eqref{tb}, the integral of the square of first function over $\ol{\mathfrak{F}}
  \times \ol{\widetilde{\mathfrak{F}}}_T$ is
\[\ll\int_{\frac{\sqrt{3}}2}^T\int_{\frac{\sqrt{3}}2}^\infty
  \frac{dy_1}{y_1^3} \frac{dy_2}{y_2^3}<\infty.\]
The square integrability of each summand of the second function 
  was proven in Lemma \ref{Fhat0int} above, and it follows by Minkowski's
  inequality that the second function itself is square integrable over the given set.
For the third function, by \cite{KL} Proposition 17.2, there are only finitely many 
  $\delta \in B(\Q) \bs G(\Q)$ such that $\tT(\delta g) \neq 0$
for some $g \in \widetilde{\mathfrak{F}}_T$.
  Therefore  it suffices to show that 
  $\|\widehat{F}_{\mu,g_1,\delta g_2}(0)\tT(\delta g_2)\|_{\mathfrak{F}\times
  \widetilde{\mathfrak{F}}_T}$ is finite for fixed $\delta$ and $\mu$.
We have
\[ \left\|
    \widehat{F}_{\mu, g_1,\delta g_2}(0) \tT(\delta g_2)
   \right\|_{{\mathfrak{F}} \times {\widetilde{\mathfrak{F}}}_T}^2
   \le \int_{\ol{\mathfrak{F}}} \int_{\ol{\widetilde{\mathfrak{F}}}_T}
      |\widehat{F}_{\mu, g_1,\delta g_2}(0)|^2 dg_2 dg_1 \]
  \[= \int_{\ol{\mathfrak{F}}} \int_{\delta \ol{\widetilde{\mathfrak{F}}}_T}
      |\widehat{F}_{\mu, g_1, g_2}(0)|^2 dg_2 dg_1,\]
which is finite by Lemma \ref{Fhat0int}, since $\delta\ol{\widetilde{\mathfrak F}}_T$ is 
  compact and factorizable.
 \end{proof}

\noindent \emph{Proof of Proposition \ref{KTsi}.}
Since $K^T=K_1+K_2^T$, it suffices by Minkowski's inequality to show that the
  latter two functions are square integrable over 
$\ol{\mathfrak{F}} \times \ol{\mathfrak{F}}$.
 By Proposition \ref{K1si},
   $\|K_1\|_{\sss {\mathfrak{F}} \times {\mathfrak{F}}}<\infty$.
  By  Proposition \ref{K2si} and Proposition \ref{K2sic},
\[ \|K_2^T\|_{\sss {\mathfrak{F}} \times {\mathfrak{F}}}^2
   =  \|K_2^T\|_{\sss {\mathfrak{F}} \times {\mathfrak{F}}_T}^2
   + \|K_2^T\|_{\sss {\mathfrak{F}} \times {\widetilde{\mathfrak{F}}}_T}^2 < \infty. \]
   This completes the proof.  \hfill $\square$

\pagebreak
\section{Kloosterman sums}\label{Klsec}


Fix a modulus $N\in\Z^+$, and let $\chi$ be a Dirichlet character
  modulo $N$ of conductor $\c_{\chi}$.
We have defined the following generalized Kloosterman
  sum for any $c\in N\Z^+$ and nonzero 
  $\n\in\Z$:\index{keywords}{Kloosterman sum!generalized}
\begin{equation}\label{kloos} \index{notations}{S k@$S_{\chi}(a,b;\n;c)$}
S_{\chi}(a,b;\n;c)=\sum_{\scri x,x'\in \Z/c\Z,\atop{\scri xx'=\n}}
  \ol{\chi(x)}e(\frac{ax+bx'}{c}).
\end{equation}
Although $\gcd(\n,N)=1$ elsewhere in this paper, we make no such
  restriction in this section.
Note that when $\n>1$, $x$ need not be invertible in $\Z/c\Z$.
   Furthermore, $\chi$ is {\em not} generally
  a Dirichlet character modulo $c$, and should be viewed simply
  as a multiplicative function on $\Z/c\Z$.  In particular it
  can happen that $\chi(x)\neq 0$ when $(x,c)>1$.

In the special case where $\n=1$, we obtain
  the usual twisted Kloosterman sum with character $\chi$
  defined by
\begin{equation}\label{non} \index{notations}{S k@$S_{\chi}(a,b;c)=S_{\chi}(a,b;1;c)$}
S_{\chi}(a,b;c)=\sum_{x\in (\Z/c\Z)^*}\ol{\chi(x)}e(\frac{ax+b\ol{x}}{c}),
\end{equation}
where $x\ol{x}\equiv 1\mod c$.  
  If $\chi$ is the
  principal character modulo $N$, then we simply write $S(a,b;c)$,
  which is the classical Kloosterman sum.\index{keywords}{Kloosterman sum!classical}

Suppose $\n=\n_1\n_2$ where $(\n_1,c)=1$.
  Then replacing $x'$ by $\ol{\n_1}x'$, we have
 \begin{equation}\label{n1n2}
S_{\chi}(a,b;\n;c)=S_{\chi}(a,b\n_1;\n_2;c).
\end{equation}
  In particular, if $(\n,c)=1$ we have
\[S_{\chi}(a,b;\n;c)=S_{\chi}(a,b\n;c).\]
This holds in other situations as well; see \eqref{lockloo} below.
In his Ph.D. thesis, J. Andersson discusses the generalized
  Kloosterman sums \eqref{kloos}, which were apparently first defined by
  Bykovsky, Kuznetsov and Vinogradov (\cite{A}, \cite{BKV}).
  He gives elementary
  proofs of the following identities, special cases of which were
  given by \cite{BKV} and Selberg \cite{Sel}.

\begin{proposition}
If either $(N,\n)=1$ or $(N,b)=1$, then
\begin{equation}\label{BKVid}
  S_\chi(a,b;\n;c)=\sum_{d|(\n,b,c)}\ol{\chi(d)}\,d\,
  S_\chi(a,\tfrac{b\n}{d^2};\tfrac cd).
\end{equation}
The identity also holds if $\chi$ is taken to be the principal 
  character modulo $c$ (resp. $c/d$) on the left (resp. right).
In the case where $\chi$ is principal, we have
\begin{equation}
S(a_1,a_2;a_3;c)=S(a_{\sigma(1)},a_{\sigma(2)};a_{\sigma(3)};c)
\end{equation}
for any permutation $\sigma\in S_3$.
\end{proposition}
\noindent
  See \cite{A} for the proofs.  In his proof of \eqref{BKVid} (Theorem 1
  on page 109 of \cite{A}), some hypothesis on $N$ (such as
  $(N,\n)=1$ or $(N,b)=1$) is used implicitly in order for $\chi$
  to be well-defined modulo $c/d$ in $S_\chi(a,\tfrac{b\n}{d^2},\tfrac cd)$.\\

The purpose of this section is to prove the following Weil bound for
  the sum \eqref{kloos}.

\begin{theorem}\label{Weil} \index{keywords}{Kloosterman sum!Weil bound}
For integers $c\in N\Z$ and $a,b,\n\in\Z$ with $c,\n$ nonzero,
  we have the bounds
\[|S_{\chi}(a,b;\n;c)|\le \tau(\n)\,\tau(c)\,(a\n,b\n,c)^{1/2}
  \,c^{1/2}\,\c_\chi^{1/2}\]
and
\[|S_{\chi}(a,b;\n;c)|\le \tau(\n)\,\tau(c)\,(a\n,b\n,c)^{1/2}
  \,c^{1/2}\,\c_\chi^{1/4}\prod_{p|\c_\chi}p^{1/4}\]
for the divisor function $\tau$.
\end{theorem}
\noindent{\em Remark:}
  Bruggeman and Miatello produce a bound when $\n=1$, which
  is valid over any totally real field (cf. Section 2.4 of \cite{BM}).
  They use the trivial bound at primes $p|N$, which results in the estimate
\[|S_\chi(a,b;c)|=O(c^{\frac12+\e}\prod_{p|N}p^{c_p/2}).\]
This is somewhat weaker than the estimates in Theorem \ref{Weil}, whose full
  strength was required in the proof of Proposition \ref{Klbound}.

\subsection{A bound for twisted Kloosterman sums}

The proof of Theorem \ref{Weil} follows three steps: express \eqref{kloos} as
  a product of local factors, relate the local factors to twisted Kloosterman
  sums \eqref{non}, and apply a Weil bound to the latter.  The present section
  establishes the Weil bound needed for the last step.
The classical Kloosterman sums satisfy the Weil/Sali\'e bound
\begin{equation}\label{classicalWeil}
|S(a,b;c)|\le \tau(c)(a,b,c)^{1/2}c^{1/2}
\end{equation}
(cf. \cite{IK}, Corollary 11.12).
  It should be noted that
  the above bound does {\em not} hold for $S_{\chi}(a,b;c)$.
  See Example \ref{p3} below.
  In general, one must account for the conductor of $\chi$ as well.

\begin{theorem}\label{pcase}  Let $p$ be any prime.
Suppose $c=p^\ell$ and $\chi$ is a Dirichlet character of conductor
  $\c_\chi=p^\g$ for $\g\le \ell$.  Then for any integers $a,b$,
\begin{equation}\label{pcase1}
|S_{\chi}(a,b;c)|\le \tau(c)\,(a,b,c)^{1/2}\,c^{1/2}\,\c_\chi^{1/2}
\end{equation}
and
\begin{equation}\label{pcase2}
|S_{\chi}(a,b;c)|\le \tau(c)\,(a,b,c)^{1/2}\,c^{1/2}\,\c_\chi^{1/4}p^{1/4}.
\end{equation}
\end{theorem}

\noindent{\em Remarks:} The proof will occupy the remainder of this section.
  The methods are standard and in large part elementary,  but because there
  seems to be no proof in the literature,
 we will include the details.
 The general case (for $c$ not necessarily a prime power) is contained
  in Theorem \ref{Weil}, whose proof will follow later.\\


The case $\ell =1$ is the most difficult, but it is well-known.

\begin{proposition} \label{Chowla}
  Suppose $c=p$ is prime.
  Then $|S_{\chi}(a,b;p)| \leq 2(a,b,p)^{1/2}{p}^{1/2}$.
\end{proposition}
   \begin{proof}
   When $p=2$, the proposition is trivial.
   If $p$ is odd and $p\nmid ab$, this was proven by Weil for
  principal $\chi$, and extended to non-principal $\chi$ by Chowla
  (\cite{We}, \cite{Ch}; see also \cite{Co}).
  These sources deal only with the case $b=1$,
  but the general case follows easily by
  a change of variables.

If $p|a$ and $p\nmid b$ (or vice versa),
 then $S_{\chi}(a,b;p)$ is a character sum precisely of the
  kind discussed in Section \ref{charsum}.  In this case, if $\chi$ is the principal
  character modulo $p$, the value of the sum is $-1$.  If $\chi$ is non-principal,
  then $|S_\chi(a,b;p)|=p^{1/2}$ (\cite{Hua}, Theorem 7.4.4).

Lastly, if $p|a$ and $p|b$, then by the triangle inequality,
 $|S_{\chi}(a,b;p)|\le p= (a,b,p)^{1/2}p^{1/2}$.
   \end{proof}

The case $p=c^\ell$ with $\ell \ge 2$ is elementary, as first shown
  for the case of principal $\chi$ by Sali\'e \cite{Sal}, whose work was later
  refined by Estermann \cite{Es}.  We
  will follow the presentation in Section 12.3 of \cite{IK}.
  It requires a knowledge of the number of solutions to certain quadratic congruences,
  given as Lemma \ref{quad} below.  Although this is standard, we include
  the proof because of its central importance in what follows.

\begin{lemma} \label{landau}
Let $n,D>0$, with $p\nmid D$.  Let $M$ be the number of solutions of
\begin{equation}\label{Dcong}
x^2\equiv D\mod p^n.
\end{equation}
Then
\[M=\begin{cases}
  1& \text{if }p=2, n=1\\
  0& \text{if }p=2, n=2, D\equiv 3\mod 4\\
  2& \text{if }p=2, n=2, D\equiv 1\mod 4\\
  0& \text{if }p=2, n>2, D\not \equiv 1\mod 8\\
  4& \text{if }p=2, n>2, D\equiv 1\mod 8\\
1+(\frac{D}p)&\text{if }p>2.
\end{cases}\]
\end{lemma}

\begin{proof} See e.g. \cite{La}, Theorem 87.
\end{proof}

\begin{lemma}\label{quad}
  Let $a$ be an integer and $p\nmid a$ a prime.  Consider the congruence
\begin{equation}\label{cong}
ax^2+Bx+c\equiv 0\mod p^n
\end{equation}
for $n>0$.  Write $\Delta =B^2-4ac = p^\delta\Delta'$,
  where $p\nmid \Delta'$.  Let $M$ denote the number of
  solutions to \eqref{cong}.  Then if $p\neq 2$,
\[M=\begin{cases} p^{\lfloor \frac n2\rfloor}
  & \text{if } \delta \ge n,\\
  2p^{\delta'}&\text{if }\delta=2\delta' <n\text{ and }
   (\tfrac{\Delta'}p)=1,\\
  0&\text{otherwise, i.e. } \delta < n\text{ and $(\delta$ is odd or }
   (\tfrac{\Delta'}p)=-1).
\end{cases}\]
When $\delta > 0$, all solutions are prime to $p$ if $p\nmid B$,
  and divisible by $p$ otherwise.  
Suppose $\delta=0$ and $(\frac{\Delta}p)=1$. Then
  both solutions are prime to $p$ if $p\nmid c$ (in particular if $p|B$), but
  if $p|c$ then exactly one of the two solutions is divisible by $p$.

If $p=2$ and $B$ is even, then $\delta\ge 2$ and
\[M=\begin{cases} 2^{\lfloor \frac n2\rfloor}
  & \text{if } \delta \ge n+2,\\
2^{\min(n-\delta+1,2)}2^{\delta'-1}&\text{if }
  2\le \delta=2\delta'<n+2\text{ and }\Delta'\equiv 1\mod 2^{\min(n-\delta+2,3)}\\
  0 &\text{otherwise.}
\end{cases}\]
  By \eqref{cong}, all solutions have the same parity as $c$.  Furthermore,
when $\delta >2$, all solutions are odd if $4\nmid B$, and
  even otherwise.  When $\delta=2$, all solutions are even if $4\nmid B$, and odd otherwise.

If $p=2$ and $B$ is odd, then $M=\begin{cases}2&\text{if }
  \Delta\equiv 1\mod 8\\0&\text{otherwise.}
  \end{cases}$\\
Note that $\Delta\equiv 1\mod 8$ if and only if $c$ is even.
  In this case, exactly one of the two solutions is even.
\end{lemma}

\begin{proof}
First, suppose $p \neq 2$. Then \eqref{cong} is equivalent to
   \begin{equation} \label {quadcong2}
       (2ax+B)^2 \equiv \Delta \mod{p^n}.
   \end{equation}
If $\delta \ge n$, the solutions of \eqref{quadcong2}
  are given by $2ax+B \equiv 0 \mod{p^{\lceil\frac{n}2\rceil}}$.
  There is a unique solution $x$ modulo ${p^{\lceil\frac{n}2\rceil}}$, so there
  are $p^{n-\lceil\frac n2\rceil}=p^{\lfloor\frac n2\rfloor}$ solutions modulo $p^n$.
   The solutions $x$ are coprime to $p$ if and only if $p \nmid B$.

Suppose $\delta < n$.  If $\delta$ is odd, it is easy to see that
   \eqref{quadcong2} has no solution.
  Suppose $\delta$ is even and write $\delta = 2\delta'$.
  Then the solutions of the congruence are given by
  $2ax+B \equiv p^{\delta'}X\mod{p^n}$, where
  \[X^2 \equiv \Delta' \mod{p^{n-\delta}}.\]
  By Lemma \ref{landau},
   this congruence has solutions (necessarily two) if and only
  if $\Delta'$ is a quadratic residue modulo $p$.
  So if $(\frac{\Delta'}p)=-1$, \eqref{quadcong2} has no solution.
  Otherwise, the solutions of \eqref{quadcong2} are given by
\[2ax+B\equiv p^{\delta'}(X+p^{n-\delta}\alpha)\mod p^n,\]
  where $\alpha$ ranges through $(\Z/p^{\delta'}\Z)$.
  Therefore, in this case the number of solutions is $2p^{\delta'}$
  (two choices for $X$, and $p^{\delta'}$ choices for $\alpha$).
  If $\delta>0$, then a solution $x$ is divisible by $p$ if and only if $p|B$.
  Now suppose $\delta=0$.  Then $x\equiv (2a)^{-1}(X-B)\mod p^n$, where
  $X^2\equiv \Delta \mod p^n$.
  If $p\nmid c$ (which is the case if $p|B$),
  then $(X-B)$, and hence $x$, is prime to $p$ since
  \[(X-B)(X+B)\equiv X^2-B^2\equiv -4ac\mod p.\]
  When $p|c$ (so that $p\nmid B$), then exactly one of the solutions is divisible
  by $p$, as can be seen by considering
$ax^2+Bx+c\equiv x(ax+B) \equiv 0\mod p$.

Now consider $p=2$.  The congruence \eqref{cong} is equivalent to
   \begin{equation} \label {quadcong3}
       (2ax+B)^2 \equiv \Delta \mod{2^{n+2}}.
   \end{equation}
Suppose $B=2B'$ is even.  Then \eqref{quadcong3} has a solution only
   if $\delta \ge 2$.
   In that case, the congruence is equivalent to
\begin{equation} \label{quadcong4}
  (ax+B')^2 \equiv 2^{\delta-2} \Delta' \mod{2^n}.
 \end{equation}
   If $\delta-2 \geq n$, then $ax+B' \equiv 0 \mod{2^{\lceil\frac{n}2\rceil}}$.
   So as before, the congruence \eqref{quadcong4} has $2^{\lfloor\frac n2\rfloor}$ solutions.
  These solutions $x$ are prime to $2$ if and only if $2\nmid B'$, or
  equivalently, $4\nmid B$.
   Now suppose $\delta-2<n$.  Then \eqref{quadcong4} is possible only if
   $\delta=2\delta'$ is even, in which case we can write
  $ax+B' \equiv 2^{\delta'-1} X \mod{2^n}$,
   where $X^2 \equiv \Delta' \mod{2^{n-\delta+2}}$.
   By Lemma \ref{landau}, such $X$ exists if and only if
  $\Delta' \equiv 1 \mod{2^{\min(n-\delta+2,3)}}$, and
   the number of solutions $X$ modulo $2^{n-\delta+2}$ is
   $2^{\min(n-\delta+1,2)}$.  In this case we can take
\[ax+B'\equiv 2^{\delta'-1}(X+2^{n-\delta+2}\alpha)\mod 2^{n}\]
  for any $\alpha\in \Z/2^{\delta'-1}\Z$.
   Therefore \eqref{quadcong4} has $2^{\min(n-\delta+1,2)} 2^{\delta'-1}$ solutions.
   If $\delta>2$, then we see that $2|x\iff 2|B'\iff 4|B$.
   If $\delta =2$, then $\Delta'=(B')^2-ac$ is odd, and we see that
\[B'\text{ even }\implies c\text{ odd }\implies x\text{ odd},\]
\[B'\text{ odd }\implies c\text{ even }\implies x\text{ even}.\]
(The fact that $x$ and $c$ have the same parity when $2|B$
  is immediate from \eqref{cong}).


Lastly, suppose $p=2$ and $B$ is odd. Then $\eqref{quadcong3}$ is solvable
   only if $\delta=0$.
   In that case, the congruence $X^2 \equiv \Delta \mod{2^{n+2}}$ is solvable
   if and only if $\Delta \equiv 1 \mod{8}$.
   The solutions to the latter congruence can be denoted $X,\,X+2^{n+1},\,-X,\,-X+2^{n+1}$.
   Therefore $2ax + B \equiv \pm X \mod{2^{n+1}}$.  This means
  $x\equiv (\frac{\pm X-B}2)a^{-1}\mod 2^n$ has exactly two solutions.
  Because $x^2+x\equiv 0\mod 2$, we see that one solution is odd and one is even.
\end{proof}

\begin{proposition} \label{kbound}
  Let $c=p^{2\alpha}$ with $\alpha \geq 1$.  Let $\chi$ be a Dirichlet
  character modulo $c$, of conductor $p^\g$ ($\g\le 2\alpha$).
  Suppose $(a,c)=1$ and $c\nmid b$.  Then:
   \begin{enumerate}
    \item  \label{kbound1} If $p$ is odd,  then
  $|S_{\chi}(a,b;c)| \leq 2p^{3\alpha/2}$.
    \item  \label{kbound2} If $p=2$,  then
  $|S_{\chi}(a,b;c)| \leq 4p^{3\alpha/2}$.
    \item  \label{kbound3} If $p$ is odd and $\g \le 2\alpha -1$,
      then $|S_{\chi}(a,b;c)| \leq 2p^{\alpha}$.
     If further (i) $p|b$ or (ii) $p\nmid b$ and $ab$ is not
     a quadratic residue mod $p$, then $S_{\chi}(a,b;c)=0$.
    \item  \label{kbound4} If $p=2$ and $\g\le 2\alpha-2$,
     then $|S_{\chi}(a,b;c)| \leq 2^{\min(\alpha-1,2)} p^{\alpha} \leq 4p^{\alpha}$.
   \end{enumerate}
\end{proposition}

\begin{proof}
  We apply Lemma 12.2 of \cite{IK} with $f(y)=y$ and $g(y)=\frac{ay^2+b}y$,
  which gives
   \begin{equation} \label{IK1}
    S_{\chi}(a,b;c) = p^{\alpha} \sum_{y \in (\Z/p^{\alpha} \Z)^*
  \atop{h(y)\equiv 0 \mod p^{\alpha}}}
   \ol{\chi(y)} e(\frac{ay+b\ol y}{c}), \end{equation}
the summand being independent of the choice of representative for $y$, where
   \begin{equation} \label{h(y)}
 h(y) = a-b{y}^{-2} + B{y^{-1}}
  \end{equation}
  for $B$ determined by
   \begin{equation} \label{B1} \ol{\chi(1+zp^{\alpha})} = e(\frac{Bz}{p^{\alpha}}). \end{equation}
   This immediately gives
   $|S_{\chi}(a,b;c)| \leq p^{\alpha}M$, where $M$ is the number of
   solutions to
   \begin{equation} \label{hc}
ay^2+By-b \equiv 0 \mod{p^{\alpha}},\quad (y,p)=1.
  \end{equation}
   By Lemma \ref{quad}, $M\le 2\gcd(2,p)p^{\alpha/2}$.
  This proves \ref{kbound1} and \ref{kbound2}.

   Now suppose $p$ is odd and $\g \le 2\alpha - 1$.
  If $\g\le \alpha$, then $B=0$ by \eqref{B1}.
   If $\alpha<\g\le 2\alpha-1$, then taking $z=p^{\gamma - \alpha}$
   in \eqref{B1} we have $e(\frac{B}{p^{2\alpha-\gamma}})=1$.
   Hence $\frac{B}{p^{2\alpha-\gamma}} \in \Z$.  So we see that $p|B$ whenever
   $\g\le 2\alpha-1$.  Therefore by Lemma \ref{quad}, \eqref{hc} has no
  solutions $y$ which are prime to $p$, unless $p\nmid b$ and $4ab$ (and hence $ab$)
  is a quadratic residue modulo $p$.  In the latter case, there are exactly two
  such solutions, so that $|S_{\chi}(a,b;c)|\le 2p^\alpha$.
    This proves \ref{kbound3}.

   Next, assume $p=2$ and $\g\le 2\alpha -2$.
   If $\g \le \alpha$, then $B=0$ by \eqref{B1}.
   If $\alpha<\g\le 2\alpha - 2$, then taking $z=p^{\gamma - \alpha}$
   in \eqref{B1} gives $e(\frac{B}{p^{2\alpha-\gamma}})=1$.
   Hence $\frac{B}{p^{2\alpha-\gamma}} \in \Z$, so that $4|B$
  whenever $\g\le 2\alpha-2$.
  By Lemma \ref{quad}, \eqref{hc} has solutions $y$ only if $b$ is odd
  (and so $\delta=2$ in the notation of the lemma).
  The number of solutions is at most $2^{\min(\alpha-1,2)}$.  Assertion
  \ref{kbound4} follows.
\end{proof}

\begin{proposition} \label{kboundo}
    Let $c=p^{2\alpha+1}$ with $\alpha \geq 1$.
    Let $\chi$ be a Dirichlet
  character modulo $c$, of conductor $p^\g$ ($\g\le 2\alpha+1$).
   Suppose $(a,c)=1$ and $c\nmid b$.  Then:
    \begin{enumerate}
    \item \label{kboundo1} If $p$ is odd, then
  $|S_{\chi}(a,b;c)| \leq 2p^{3\alpha/2+1}$.
    \item \label{kboundo2} If $p=2$, then
  $|S_{\chi}(a,b;c)| \leq 4p^{3\alpha/2+1}$.
    \item \label{kboundo3} If $p$ is odd and $\g\le 2\alpha$, then
   $|S_{\chi}(a,b;c)| \leq 2p^{\alpha+1/2}$.
    Furthermore, if (i) $p|b$ or (ii) $p\nmid b$ and $ab$ is a quadratic residue modulo $p$,
    then $S_{\chi}(a,b;c)=0$.
    \item \label{kboundo4} If $p=2$ and $\g\le 2\alpha-1$, then
    $|S_{\chi}(a,b;c)| \le 2^{\min(3,\alpha)} p^{\alpha} \leq 8p^{\alpha}$.
    \end{enumerate}
\end{proposition}

\begin{proof} We apply Lemma 12.3 of [IK] with $f(y)=y$ and $g(y)=\frac{ay^2+b}y$, which gives
\begin{equation}\label{Seq}
 S_{\chi}(a,b;c) = p^{\alpha} \sum_{y \in (\Z/p^{\alpha} \Z)^*,
   \atop{h(y)\equiv 0\mod p^{\alpha}}}
   \ol{\chi(y)} e(\frac{ay+b\ol y}{c}) G_p(y).
 \end{equation}
   Here $h(y)$ is given by \eqref{h(y)} as before, but this time $B$ is
  defined by
   \begin{equation} \label{B2}
    \ol{\chi(1+zp^{\alpha})} = e\left( \frac{Bz}{p^{\alpha+1}} + (p-1) \frac{Bz^2}{2p} \right),
    \end{equation}
   and $G_p(y)$ is the Gauss sum
   \begin{equation} \label{Gp}
    G_p(y) = \sum_{z\mod p}e(\frac{d(y) z^2 + h(y) p^{-\alpha}z}{p})
    \end{equation}
   for
   \begin{equation}\label{d(y)}
    d(y) = by^{-3} + (p-1) \frac B2 y^{-2}.
    \end{equation}
   Because $|G_p(y)| \leq p$, we have
\begin{equation}\label{M}
|S_{\chi}(a,b;c)| \leq p^{\alpha+1}M,
\end{equation}
   where $M$ is the number of solutions to \eqref{hc}.  As before,
   $M\le 2\gcd(2,p)p^{\alpha/2}$, so \ref{kboundo1} and \ref{kboundo2} follow.

   Suppose $p$ is odd and $\g\le 2\alpha$. If $\g\le \alpha$, then $B=0$ by \eqref{B2}.
   If $\alpha<\g\le 2\alpha$, then setting $z=p^{\gamma-\alpha}$ in \eqref{B2} gives
   \[1 = e\left( \frac{B}{p^{2\alpha+1-\gamma}} + \frac{(p-1)}2 Bp^{2(\gamma-\alpha)-1} \right)=
   e\left( \frac{B}{p^{2\alpha+1-\gamma}} \right).\]
   Thus $p|B$ whenever $\g\le 2\alpha$.
   As in the previous proof, the congruence \eqref{hc} has solutions (necessarily two in number)
  only if $p\nmid b$ and $ab$ is a quadratic residue modulo $p$.
   Because $p|B$, $d(y) \equiv by^{-3} \not \equiv 0 \mod{p}$.
   Hence by (12.37) of \cite{IK}, $|G_p(y)| = p^{1/2}$.
   It now follows that $|S_\chi(a,b;c)|\leq 2p^{\alpha}p^{1/2}$, which proves
   \ref{kboundo3}.

   Now suppose $p=2$ and $\g\le 2\alpha-1$.
   If $\g\le \alpha$, then  then $B=0$ by \eqref{B2}.
   If $\alpha<\g\le 2\alpha - 1$, then setting $z=p^{\gamma-\alpha}$ in \eqref{B2} gives
   \[1=e\left(\frac{B}{p^{2\alpha+1-\g}}+\frac{Bp^{2(\g-\alpha)}}{p^2}\right)=
  e\left(\frac{B}{p^{2\alpha+1-\gamma}}\right).\]
   Because $2\alpha+1-\gamma \geq 2$, we see that $4|B$. As in the previous proof,
   the number of solutions to \eqref{hc} is $M\le 2^{\min(\alpha-1,2)}\le 4$.
   Assertion \ref{kboundo4} now follows immediately by \eqref{M}.
\end{proof}

\begin{example}\label{p3}
Let $p$ be an odd prime, and let
 $\chi$ be a primitive Dirichlet character of 
  modulus $p^3$.  Then there exist $a,b\in(\Z/p^3\Z)^*$ such that
\[S_\chi(a,b;p^3)=p^2.\]
  In particular, if $c=p^3$ for $p\ge 17$,
\[|S_\chi(a,b;c)|>\tau(c)(a,b,c)^{1/2}c^{1/2}.\]
\end{example}

\begin{proof}
We apply the above proposition with $\alpha=1$.
  If, in \eqref{B2}, $p|B$, then 
\[\ol{\chi}(1+zp)=e(\frac{Bz}{p^2}),\]
which implies that $\ol{\chi}(1+zp^2)=1$, and hence $\c_{\chi}|p^2$.
 Thus assuming $\chi$ is primitive, $p\nmid B$.
 Take $a=\tfrac{p-1}2B$ and $b=-\tfrac{p-1}2B$, and consider $S_\chi(a,b;p^3)$ for
  $\chi$ primitive.
  In the notation of the previous proof,
\[h(y) = \tfrac{p-1}2B+\tfrac{p-1}2By^{-2}+By^{-1}\equiv 0\mod p\]
\[\iff y^2-2y+1\equiv 0\mod p \iff y\equiv 1\mod p.\]
Therefore since $a+b=0$, \eqref{Seq} gives
\[S_\chi(a,b;p^3)=pG_p(1).\]
In the notation of \eqref{d(y)}, we have
\[d(1)=-\tfrac{p-1}2B+\tfrac{p-1}2B=0.\]
Since $h(1)=0$ as well, we have $G_p(1)=p$.  Thus $S_\chi(a,b;p^3)=p^2$.
%
\end{proof}

\begin{proposition} \label{kboundg}
   Suppose $c=p^\ell$ and $\c_\chi=p^\g$ for $\g\le \ell$. If $(a,c)=1$ and $c|b$, then
   $|S_{\chi}(a,b;c)|= \begin{cases}p^{\ell/2} &\text{if } \g=\ell,\\0&\text{otherwise.}
\end{cases}$
\end{proposition}
\begin{proof} When $c|b$, $S_{\chi}(a,b;c)=\sum_{d \in (\Z/c\Z)^*} \ol{\chi(d)}
   e(\frac {ad}{c})$ is a Gauss sum.
   If $\gamma < \ell$, then the Gauss sum vanishes
  (\cite{Hua}, Theorem 7.4.2).  If
   $\gamma=\ell$, then the absolute value of the Gauss sum is $p^{\ell/2}$.
\end{proof}

\begin{corollary} \label{kboundfinal}
   Suppose $c=p^\ell$ for $\ell\ge1$,
  $\c_\chi=p^\g$ for $\g\le \ell$, and $(a,c)=1$.  Then:
   \begin{itemize}
   \item $|S_{\chi}(a,b;c)| \leq 2\gcd(2,p)^2\, c^{1/2} p^{1/4} \c_\chi^{1/4}$,
   \item $|S_{\chi}(a,b;c)| \leq \tau(c)c^{1/2}p^{1/4} \c_\chi^{1/4}$,
   \item $|S_{\chi}(a,b;c)| \leq 2\gcd(2,p)^2\, c^{1/2} \c_\chi^{1/2}$,
   \item $|S_\chi(a,b;c)|\le \tau(c)c^{1/2}\c_{\chi}^{1/2}$.
   \end{itemize}
\end{corollary}
\begin{proof}
  This follows directly from what we have proven above.  We just need to
  examine each case.
  In view of Proposition \ref{kboundg}, we can assume that $c\nmid b$.
   First, suppose $p$ is odd and $\ell=2\alpha$ is even.
   If $\g=2\alpha$, then
   by Proposition \ref{kbound} \eqref{kbound1},
\[|S_{\chi}(a,b;c)| \le 2p^{\alpha} p^{\alpha/2}
  =  2c^{1/2} \c_\chi^{1/4}.\]
   If $\g\le 2\alpha -1$, then by Proposition \ref{kbound} \eqref{kbound3},
   we have $|S_\chi(a,b;c)|\le 2c^{1/2}$.  Now consider $\ell$ odd.
   If $\ell=1$, then the bounds hold by Proposition \ref{Chowla}.
   Suppose $\ell=2\alpha+1$ for $\alpha \ge 1$.
   If $\g=2\alpha + 1$, then by Proposition \ref{kboundo} \eqref{kboundo1},
 \[|S_{\chi}(a,b;c)| \leq 2p^{\alpha+\frac12} p^{\frac\alpha2+\frac12} =2p^{1/4}c^{1/2}\c_\chi^{1/4}
   \leq 2c^{1/2}\c_{\chi}^{1/2}.\]
   The last step holds since $\gamma \geq 1$.
   If $\g\le 2\alpha$, then $|S_\chi(a,b;c)|\le 2p^{\alpha+\frac12}=2c^{1/2}$ by
   Proposition \ref{kboundo} \eqref{kboundo3}.  This establishes the
  bounds when $p$ is odd.

   Now consider the case $p=2$, and suppose $\ell=2\alpha$ is even.
  When $\alpha=1$, the bounds are trivial because $S_\chi(a,b;c)$ is a
  sum over $(\Z/4\Z)^*$ and is hence bounded by $2$.  So we can assume
  that $\alpha>1$.
  If $\g=2\alpha$, then by Proposition \ref{kbound} \eqref{kbound2},
   \[|S_{\chi}(a,b;c)| \le 4p^{\alpha} p^{\alpha/2} =4c^{1/2}\c_\chi^{1/4}
  \le (2\alpha+1)c^{1/2}\c_\chi^{1/4}=\tau(c)c^{1/2}\c_\chi^{1/4}.\]
  If $\g=2\alpha-1$, then
\[   |S_{\chi}(a,b;c)| \le 4p^{\alpha} p^{\alpha/2} =4c^{1/2} p^{(\gamma+1)/4}
  \le \tau(c) p^{1/4}c^{1/2}\c_\chi^{1/4}
 \leq \tau(c)c^{1/2}\c_{\chi}^{1/2}.\]
   The last step holds because $\gamma \geq 1$.
   If $\g\le 2\alpha -2$, then by Proposition \ref{kbound} \eqref{kbound4},
\[|S_\chi(a,b;c)|\le 4p^\alpha=4c^{1/2}\le (2\alpha+1)c^{1/2}=\tau(c)c^{1/2},\]
  since $\alpha >1$.
   Now consider $\ell$ odd.  If $\ell=1$, then the bounds are obvious since
  the summation only has one term.
  Suppose $\ell=2\alpha+1$ with $\alpha \ge 1$.
  If $\g=2\alpha+1$, then by Proposition \ref{kboundo} \eqref{kboundo2},
\[|S_\chi(a,b;c)|\le 4p^{\alpha+\frac12}p^{\frac\alpha2+\frac12}
  =4p^{1/4}c^{1/2}\c_\chi^{1/4}\le 4c^{1/2}\c_{\chi}^{1/2}
   \le \tau(c)c^{1/2}\c_\chi^{1/2},\]
  since $\g\ge 1$ and $\tau(c)=2\alpha+2\ge 4$.
   If $\g=2\alpha$, then
   by Proposition \ref{kboundo} \eqref{kboundo2},
\[|S_{\chi}(a,b;c)| \leq 4p^{\alpha+\frac12} p^{\frac\alpha2+\frac12}
  =4p^{1/4}c^{1/2}p^{1/4}\c_{\chi}^{1/4}.\]
If $\alpha \ge 2$, then $4p^{1/4}\le 2\alpha+2=\tau(c)$, and the first two
  inequalities follow.  That the remaining ones also hold follows from
  $p^{\frac12+\frac{\alpha}2}\le p^{\alpha}=\c_\chi^{1/2}$.
  If $\alpha=1$, then $S_\chi(a,b;c)$ is a sum over
  $(\Z/8\Z)^*$, so it is bounded by 4, and the inequalities clearly hold
  in this case as well.
Finally, if $\g\le 2\alpha-1$, then by Proposition \ref{kboundo}
   \eqref{kboundo4},
  \[|S_\chi(a,b;c)|\le 2^{\min(3,\alpha)}p^\alpha\le (2\alpha+2)p^{\alpha+\frac12}
  =\tau(c)c^{1/2}.\qedhere\]
\end{proof}


\begin{proposition}\label{swap}
  The results of Propositions \ref{kbound}, \ref{kboundo} and Corollary \ref{kboundfinal}
  hold if we exchange the roles of $a$ and $b$.
\end{proposition}
\begin{proof} This follows from the fact that $S_{\chi}(a,b;c) = S_{\ol{\chi}}(b,a;c)$.
\end{proof}


We now have all of the pieces in place to prove Theorem \ref{pcase}.

\begin{proof}[Proof of Theorem \ref{pcase}]
 Suppose $c=p^\ell$, and $\c_\chi=p^\g$ for $\g\le \ell$.
  We need to show first that for any $a,b$,
\[|S_{\chi}(a,b;c)| \leq \tau(c)\,(a,b,c)^{1/2}c^{1/2}\c_\chi^{1/2}.\]
  If $c|a$ and $c|b$, this is trivial.
  Suppose $(a,b,c)=p^{a_p}$ for $a_p=\ord_p(a)$, and write
   $a'=p^{-a_p} a$ and $b'= p^{-a_p} b$.
  Then by Corollary \ref{kboundfinal},
\[|S_{\chi}(a,b,c)| = p^{a_p} |S_{\chi}(a',b',p^{\ell-a_p})|
  \leq \tau(c) p^{a_p} p^{(\ell-a_p)/2}\c_\chi^{1/2}\]
\[= \tau(c)(a,b,c)^{1/2}c^{1/2}\c_\chi^{1/2}.\]
  If $(a,b,c)=p^{\ord_p(b)}$, the inequality can be proven in the
  same way after applying Proposition \ref{swap}.

The second assertion, that
\[|S_{\chi}(a,b;c)| \leq \tau(c)\,(a,b,c)^{1/2}c^{1/2}\c_\chi^{1/4}p^{1/4},\]
 follows in the same manner.
\end{proof}

\subsection{Factorization}

Now we turn our attention back to the generalized Kloosterman
  sum $S_\chi(a,b;\n;c)$ \eqref{kloos}, expressing it as a product 
  of local factors.  These factors will in turn be expressed in terms
  of the sums $S_\chi(a,b;c)$ studied in the previous section.

  Let $\chi$ be any multiplicative function $\Z/c\Z\longrightarrow \C$.
  Suppose $c=qr$ with $(q,r)=1$. Then using
  \[\Z/c\Z=(\Z/q\Z)\times(\Z/r\Z),\]
  we see that $\chi$ has a canonical factorization
  $\chi(x)=\chi_q(x)\chi_r(x)$,
  where $\chi_q$ and $\chi_r$ are multiplicative functions on $\Z/q\Z$ and
  $\Z/r\Z$ respectively.
  If $\chi$ is a Dirichlet character modulo $N$, viewed as a function
  on $\Z/c\Z$, and if $(r,N)=1$,
  then $\chi_r =\mathbf{1}$ is the constant function $1$ on $\Z/r\Z$
  (not to be confused with the principal character modulo $r$).

\begin{proposition}  Suppose $\chi$ is a multiplicative function modulo $N$, and
  $q,r\in \Z^+$ with $(q,r)=1$ and $qr\in N\Z$.
Write $\chi(x) = \chi_{q}(x)\chi_r(x)$ as above.
Then
\[S_{\chi}(a,b;\n;qr) = S_{\chi_q}(a\ol{r},b\ol{r};\n;q)
  S_{\chi_r}(a\ol{q},b\ol{q};\n;r),\]
where $\ol{r}r\equiv 1\mod q$ and $\ol qq\equiv 1\mod r$.
\end{proposition}

\begin{proof}
By the Chinese reminder theorem, $x=r\ol{r} t + q \ol{q} d$ runs
through a complete residue system mod $qr$ when $t$ and $d$ run through complete
residue systems mod $q$ and $r$ respectively.

For fixed $x=r\ol{r} t + q \ol{q} d$, an integer
  $x'$ satisfies $xx'\equiv \n \mod{qr}$
  if and only if
\[tx' \equiv \n \mod{q}\quad\text{ and }\quad
dx'  \equiv \n \mod{r}.\]
Again by the Chinese reminder theorem,
the set of all such $x'$ is parametrized by
$x' =  r \ol{r}t' + q\ol{q}d'$, as $t'$ and $d'$ run
   through all solutions of $t t' \equiv \n \mod{q}$ and
  $dd' \equiv \n \mod{r}$, respectively.

Therefore
\[S_{\chi}(a,b;\n;qr)=\sum_{tt'\equiv \n\atop{\mod q}}
\sum_{dd'\equiv \n\atop{\mod r}}\ol{\chi(r\ol rt+q\ol qd)}
  e\Bigl(\frac{a(r\ol rt+q\ol qd)+b(r\ol rt'+q\ol qd')}{qr}\Bigr)\]
\[=
\Bigl(\sum_{tt'\equiv\n\atop{\mod q}}\ol{\chi_q(t)}
  e(\frac{a\ol rt+b\ol rt'}q)\Bigr)
\Bigl(\sum_{dd'\equiv\n\atop{\mod r}}\ol{\chi_r(d)}
  e(\frac{a\ol qd+b\ol q d'}r)\Bigr).
\qedhere\]
\end{proof}

For $p|c$, write $c=p^{c_p}c^{(p)}$ and $\n=p^{\n_p}\n^{(p)}$,
  \index{notations}{13@$c^{(p)}=p^{-c_p}c$, $\n^{(p)}=p^{-\n_p}\n$}
   where $p\nmid c^{(p)}\n^{(p)}$.
  Then by successive applications of the proposition and \eqref{n1n2},
   we obtain the following.

\begin{corollary}\label{Sfac} With notation as above,
\begin{equation}\label{kloofac}
S_{\chi}(a,b;\n;c)=\prod_{p|c}S_{\chi_{p}}(a\ol{c^{(p)}},
  b\ol{c^{(p)}}\n^{(p)}; p^{\n_p};p^{c_{p}}).
\end{equation}
If $p|N$, then $\chi_p$ is the Dirichlet character mod $p^{c_p}$
  defined as in \eqref{chi_p}, so that $\chi_p(d)=0$
  if $p|d$.  If $p\nmid N$, then 
  $\chi_p=\mathbf 1$\index{notations}{1@$\mathbf 1$ (constant function $1$)}
  is the constant function $1$ on $\Z/p^{c_p}\Z$.
\end{corollary}


Each local factor in \eqref{kloofac} can be expressed in terms of the
  familiar twisted Kloosterman sums \eqref{non}, as the next proposition shows.

\begin{proposition}\label{klcomp}
 Fix integers $k\ge 0$ and $\ell\ge 1$,  and
 let $\chi_p$ be a Dirichlet character modulo $p^\ell$.
  Then    
\begin{equation}\label{lockloo}
S_{\chi_p}(a,b;p^k;p^\ell)= S_{\chi_p}(a,bp^{k};p^\ell),
\end{equation}
If instead of a Dirichlet character, 
  $\chi_p=\mathbf{1}$
   is the constant function $1$ on $\Z/p^\ell\Z$, then when $k<\ell$,
\begin{equation}\label{lockloo3}
S_{\mathbf 1}(a,b;p^k;p^\ell)=\begin{cases}\ds
  p^k\sum_{i=\max(0,k-a_p)}^{\min(b_p,k)}S(\frac{a}{p^{k-i}},\frac{b}{p^i};p^{\ell-k})
  & \text{if }k\le a_p+b_p \\ 0&\text{otherwise,}\end{cases}
\end{equation}
where as usual $a_p=\ord_p(a)$ and $b_p=\ord_p(b)$.
For the $k\ge \ell$ case, the sum is evaluated in \eqref{ngec} below. It
  vanishes unless $\ell\le a_p+b_p+1$.
\end{proposition}

\begin{proof}
The left-hand side of \eqref{lockloo} is a sum over $xx'=p^k$ in
  $(\Z/p^\ell\Z)$.  If $p|x$, then $\chi_p(x)=0$.  Therefore we can take
  $x\in (\Z/p^\ell\Z)^*$, and $x'=\ol{x}p^k$.  Eq. \eqref{lockloo} follows.

For the case $\chi_p=\mathbf{1}$, suppose first that $k < \ell$.
  Group the sum in $S_{\mathbf 1}(a,b;p^k;p^\ell)$ according to $i=\ord_p(x)\le k$.
  Suppose
\begin{equation}\label{xx'}
xx' \equiv p^k \mod{p^{\ell}}.
\end{equation}
  Then $x = p^i t$ and $x' = p^{k-i} \ol{t}$ for some
  $t \ol{t} \equiv 1 \mod{p^{\ell-k}}$.
  For given $t$, $\ol{t}$, we have solutions
  $x = p^{i}(t+p^{\ell-k} d)$ and $x'=p^{k-i}(t'+p^{\ell-k}d')$.
  As $d$, $d'$ and $t$ range through $\Z/p^{k-i}\Z$, $\Z/p^i\Z$ and
   $(\Z/p^{\ell-i}\Z)^*$ respectively,
  $x$ and $x'$ give all incongruent solutions to \eqref{xx'} modulo $p^{\ell}$.  Thus
\[ S_\mathbf{1}(a,b;p^k;p^\ell) =
  \sum_{i=0}^k \sum_{t\in(\Z/p^{\ell-k}\Z)^*}   
  \sum_{d=1}^{p^{k-i}} \sum_{d'=1}^{p^i}
  e(\frac{ap^i(t+p^{\ell-k}d) + bp^{k-i}(\ol{t}+p^{\ell-k}d')}{p^{\ell}}) \]
\[ =\sum_{i=0}^k
\sum_{t \in(\Z/p^{\ell-k}\Z)^*}
  e(\frac{ap^it + bp^{k-i}\ol{t}}{p^{\ell}})
   \sum_{d=1}^{p^{k-i}} e(\frac{ad}{p^{k-i}})
   \sum_{d'=1}^{p^i}e(\frac{bd'}{p^{i}}).\]
The $i^{\text{th}}$ summand is non-zero only if $p^{k-i}|a$ and $p^i|b$.
  In this situation, write $a=p^{k-i}a'$ and $b=p^ib'$.  Then
   the above is
\[ = p^k \sum_{0\le i\le k,\atop{k-a_p \le i \le b_p}}
  \sum_{t \ol{t} \equiv 1 \atop \mod{p^{\ell-k}}}
  e(\frac{a't + b'\ol{t}}{p^{\ell-k}})
  =p^k\sum_{0\le i\le k,\atop{k-a_p\le i\le b_p}}S(a',b';p^{\ell-k}). \]
This proves \eqref{lockloo3}.

Now suppose $k\ge \ell$.  Then $xx'\equiv 0\mod p^\ell$, and we write $x=p^it$, $x'=p^{\ell-i}t'$,
 for $t\in (\Z/p^{\ell-i}\Z)^*$ and $t'\in \Z/p^{i}\Z$.  Thus
\[S_{\mathbf{1}}(a,b;p^k;p^\ell)=\sum_{i=0}^{\ell}
  \sum_{t}\sum_{t'}e(\frac{ap^it+bp^{\ell-i}t'}{p^\ell})\]
\begin{equation}\label{ngec}
=\sum_{i=0}^\ell\sum_{t\in(\Z/p^{\ell-i}\Z)^*}e(\frac{at}{p^{\ell-i}})
  \sum_{t'\in \Z/p^i\Z}e(\frac{bt'}{p^i}).
\end{equation}
The sum over $t$ can be evaluated explicitly using
\[\sum_{t\in(\Z/p^r\Z)^*}e(\frac{at}{p^r})=\sum_{t=1}^{p^r}e(\frac{at}{p^r})
  -\sum_{t=1}^{p^{r-1}}e(\frac{apt}{p^r})=\begin{cases}
  p^r-p^{r-1}&\text{if $0<r\le a_p$}\\
  -p^{r-1}&\text{if $r=a_p+1$}\\
  0&\text{ if $r>a_p+1$,}\end{cases}
\]
  and the sum over $t'$ is $p^i$ or $0$ according to whether or not $i\le b_p$.
In particular, the $i^{\text{th}}$ term of \eqref{ngec}
  vanishes unless $\ell-i\le a_p+1$ and $i\le b_p$,
  i.e. $\ell-a_p-1\le i\le b_p$.  Thus the whole expression vanishes unless
 $\ell\le a_p+b_p+1$.
\end{proof}

\comment{
  Suppose $i\le \tfrac k2$.  Write $x=p^it$ for $t\in (\Z/p^{\ell-i}\Z)^*$.
  Define $\ol{t}\in (\Z/p^{\ell-i}\Z)^*$ by $t\ol{t}\equiv 1\mod p^{\ell-i}$.
  If $xx'\equiv p^k\mod p^\ell$ for some $x'\in\Z/p^\ell\Z$,
  then $x'=p^{k-i}t'$ for $t'\in (\Z/p^{\ell-(k-i)}\Z)^*$.
  Because $tt'\equiv1\mod p^{\ell-k}$, we have
  $t'=\ol{t}+dp^{\ell-k}$ for unique $d\in \Z/p^i\Z$.
  (Note that $\ol t$ is defined modulo $p^{\ell-(k-i)}$ since
  $p^{\ell-(k-i)}|p^{\ell-i}$ when $i\le k/2$.)
  The corresponding piece of $S_{\mathbf 1}(a,b;p^k;p^\ell)$ is then
\[\sum_{t\in (\Z/p^{\ell-i}\Z)^*}\sum_{d\in\Z/p^i\Z}
   e(\frac{ap^it+bp^{k-i}(\ol{t}+dp^{\ell-k})}{p^\ell})\]
\[=S(a,bp^{k-2i};p^{\ell-i})\sum_{d\mod p^i}e(\frac{bd}{p^i})
=\begin{cases}p^iS(a,bp^{k-2i};p^{\ell-i})&\text{if }p^i|b\\
  0&\text{otherwise.}\end{cases}\]
Similarly, if $i\ge \tfrac k2$,
  then we write $x'=p^{k-i}t$ for $t\in (\Z/p^{\ell-k+i}\Z)^*$,
  so that $x=p^i(\ol{t}+dp^{\ell-k})$ for unique $d\in\Z/p^{k-i}\Z$,
  and summing over $t,d$, we obtain $p^{k-i}S(ap^{2i-k},b;p^{\ell-(k-i)})$ if
  $p^{k-i}|a$, and $0$ otherwise.  Writing $j=k-i$, we obtain the second sum
  for the range $\lfloor\tfrac k2\rfloor+1\le i\le k$.
\end{proof}
}

\subsection{Proof of Theorem \ref{Weil}}

We will bound each term of \eqref{kloofac}.
  Suppose $p|N$.  Then by Proposition \ref{klcomp},
\[S_{\chi_{p}}(a\ol{c^{(p)}},b\ol{c^{(p)}}\n^{(p)}; p^{\n_p};p^{c_{p}})=
  S_{\chi_p}(a\ol{c^{(p)}},b\ol{c^{(p)}}\n;p^{c_p}).\]
  Applying Theorem \ref{pcase} to the latter sum,
\begin{align}\label{b1}
\hskip -.5cm\notag|S_{\chi_{p}}(a\ol{c^{(p)}},b\ol{c^{(p)}}\n^{(p)}; p^{\n_p};p^{c_{p}})|
  &\le \tau(p^{c_p})\,(a,b\n,p^{c_p})^{1/2}p^{c_p/2}\c_{\chi_p}^{1/2}\\
 &\le \tau(p^{c_p})\,(a\n,b\n,p^{c_p})^{1/2}p^{c_p/2}\c_{\chi_p}^{1/2}.
\end{align}
Now suppose $p|c$ but $p\nmid N$.  Then $\chi_p=\mathbf 1$, and if $\n_p<c_p$,
  by \eqref{lockloo3} we have
\[S_{\mathbf 1}(a\ol{c^{(p)}},b\ol{c^{(p)}}\n^{(p)}; p^{\n_p};p^{c_{p}})=
\sum_{i=\max(0,\n_p-a_p)}^{\min(\n_p,b_p)} p^{\n_p}S(\frac{a\ol{c^{(p)}}}{p^{\n_p-i}},
  \frac{b\ol{c^{(p)}\n^{(p)}}}{p^i};p^{c_p-\n_p})\]
  Therefore applying the Weil bound $|S(a,b;c)|\le \tau(c)(a,b,c)^{1/2}c^{1/2}$ to each term
  in the sum, we find (still assuming $\n_p<c_p$)
\begin{align}\label{b2}
\notag
|S_{\mathbf 1}(a\ol{c^{(p)}},b\ol{c^{(p)}}\n^{(p)}; p^{\n_p};p^{c_{p}})|
&\le \sum_i\tau(p^{c_p-\n_p})p^{\n_p}(\frac{a}{p^{\n_p-i}},\frac{b}{p^i},p^{c_p-\n_p})^{1/2}
  p^{c_p/2-\n_p/2}\\
  &\le(\n_p+1)(c_p+1)(a\n,b\n,p^{c_p})^{1/2}p^{c_p/2},
\end{align}
since the sum has at most $(\n_p+1)$ terms.
If $\n_p\ge c_p$, the bound \eqref{b2} also holds, since from \eqref{ngec},
\begin{align*}
|S_{\mathbf 1}(a\ol{c^{(p)}},b\ol{c^{(p)}}\n^{(p)}; p^{\n_p};p^{c_{p}})|
&\le \sum_{i=0}^{c_p}p^{c_p-i}p^i\\
&\le (\n_p+1)p^{c_p}=\tau(p^{\n_p})(a\n,b\n,p^{c_p})^{1/2}
  p^{c_p/2}.
\end{align*}
Multiplying the local bounds \eqref{b1} and \eqref{b2} together, by \eqref{kloofac} we have
\[|S_\chi(a,b;\n;c)|\le \tau(\n)\tau(c)(a\n,b\n,c)^{1/2}c^{1/2}
  \c_\chi^{1/2},\]
which proves the first inequality in Theorem \ref{Weil}.
The proof of the second inequality is identical, using the second inequality
 of Theorem \ref{pcase} for \eqref{b1} in the case that $p|\c_\chi$, and using
 the classical Weil bound \eqref{classicalWeil} in place of \eqref{b1} 
  in the case that $p|N$ but $\chi_p$ is principal, i.e. $p\nmid \c_\chi$.

\pagebreak
\section{Equidistribution of Hecke eigenvalues}\label{Equi}

The Hecke eigenvalues attached to cusp forms have many interesting statistical
  properties.  On one hand, there is the ``horizontal"
  Sato-Tate problem of fixing a newform $u(z)$ and determining the distribution
  of the Hecke eigenvalues at all primes away from the level.
  If $u$ is non-dihedral, then conjecturally the
  normalized eigenvalues $\nu_p^u$ are equidistributed relative to the
  Sato-Tate measure\index{keywords}{Sato-Tate measure}
\begin{equation}\label{STmeas}
d\mu_\infty(x) =\begin{cases}
\frac1\pi\sqrt{1-\frac{x^2}4}\,dx & \text{if }-2\le x\le 2\\ 
  0 &\text{otherwise}.\index{notations}{muinfty@$\mu_\infty$}
 \end{cases}
 \end{equation}
  This problem is very deep, being tied to the analytic properties of the
  symmetric power $L$-functions of $u$.   It has now been proven
  if $u$ is holomorphic of weight $\k\ge 2$ by Barnet-Lamb,
  Geraghty, Harris, and Taylor, \cite{BLGHT}.

Another point of view is the ``vertical" problem of fixing the prime $p$
  and determining the distribution of the eigenvalues of $T_p$ on a parametric
  family of cusp forms, as the parameter (level, weight) tends to infinity.
  This question has been addressed independently by several authors:
  for Maass forms by Bruggeman \cite{Brug} and Sarnak \cite{Sa},
  and for holomorphic forms by Serre \cite{Se} and Conrey/Duke/Farmer \cite{CDF}.
  Strikingly, the relevant measure in this case is not the Sato-Tate
  measure, but the $p$-adic measure\index{notations}{mup@$\mu_p$}
\[d\mu_p(x) = \frac{p+1}{(p^{1/2}+p^{-1/2})^2-x^2}\,d\mu_\infty(x).\]
  Serre's article discusses many interesting applications of this result.
  Effective versions have been given by Murty and Sinha (\cite{MS}) and
  Lau and Wang (\cite{LW}).

In the holomorphic case, one obtains a different vertical result
  using Peter\-sson's trace formula in which
  each Hecke eigenvalue has an analytic weight coming from
  Fourier coefficients and the Petersson norm of the cusp form.  When weighted
  in this way, the eigenvalues of $T_p$ become equidistributed relative to the Sato-Tate
  measure itself (independent of $p$), as the level $N\to\infty$ (\cite{Li}, \cite{KL3}).

In this section we treat the case of Maass forms from
  the latter perspective, using the Kuznetsov trace formula.
  We will prove that for a fixed prime $p\nmid N$, the eigenvalues of $T_p$ on
  the Maass eigenforms, when given the weights that arise naturally
  in the Kuznetsov formula, become equidistributed relative to the
  Sato-Tate measure as the level goes to infinity.
  An interesting feature is that the weights depend on the choice of
 $f_\infty$ (or equivalently, its Selberg transform $h(t)$),
  while the measure is independent of this choice.

Fix an integer $m>0$ and a function $h(t)$ as in Theorem \ref{hmain}.
  We will apply the Kuznetsov formula with $m_1=m_2=m$.
 Fix a prime $p\nmid N$ and an exponent $\ell\ge 0$.
  For a Maass eigenform $u\in\mathcal{F}$,
  define the normalized Hecke eigenvalue
\[\nu_{p^\ell}^u=\w'(p)^{\ell/2} \lambda_{p^\ell}(u)\in\R.\]
The value is real because $\w'(p)^{\ell/2}T_{p^\ell}$ is self-adjoint, and it is
  bounded in absolute value by a number depending only on $p^\ell$ (see p. \pageref{KS}).
  For all $\ell\ge 0$,
\[\nu_{p^\ell}^u = X_\ell(\nu_{p}^u),\]
  where \index{notations}{Xl@$X_\ell$, Chebyshev polynomial}\index{keywords}{Chebyshev polynomial}
\[X_\ell(2\cos\theta)=\frac{\sin((\ell+1)\theta)}{\sin\theta}
  =e^{i\ell\theta}+e^{i(\ell-2)\theta}+\cdots+e^{-i\ell\theta}\]
 is the Chebyshev polynomial of degree $\ell$
  (see e.g. Proposition 29.8 of \cite{KL}).

Now for each $u_j\in \mathcal{F}$, define a weight
\begin{equation}\label{wh}
w_{u_j}=
  \frac{\left|a_m(u_j)\right|^2}{\|u_j\|^2}\frac{h(t_j)}{\cosh(\pi t_j)},
\end{equation}
where $t_j$ is the spectral parameter of $u_j$.
Note that at this point, $w_{u_j}$ may be a complex number.
   However, in the equidistribution result below
  (Theorem \ref{dist}), we shall impose an extra hypothesis
  to ensure that $w_{u_j}$ is a nonnegative real number for all $j$.

\begin{proposition}\label{101} With $h(t)$ as in Theorem \ref{hmain}, we have
\[\sum_{u\in \mathcal{F}}X_\ell(\nu_p^u)w_u
  =\begin{cases}J\psi(N)+O(N^{\frac12+\e})&\text{if $\ell=2\ell'$
  with $0\le \ell'\le \ord_p(m)$}\\
  O(N^{\frac12+\e})&\text{otherwise}\end{cases}\]
as $N\to\infty$,
where
\begin{equation}\label{J}
J=\frac{1}{\pi^2}\int_{-\infty}^\infty h(t)\tanh(\pi t)\,t\,dt
  = \frac4\pi V(0)=\frac4\pi f_\infty(1).
\end{equation}
Here, $V$ and $f_\infty$ are the functions attached to $h$ in \eqref{Vu} and
  \eqref{Vf2} respectively, and the equalities in \eqref{J} are from \eqref{V0}.
%
\end{proposition}
\noindent{\em Remark:} This demonstrates the existence of cusp forms
  with nonvanishing $m^{\text{th}}$ Fourier coefficient for
  all sufficiently large $N$.

\begin{proof}
Taking $m_1=m_2=m$ and $\n=p^\ell$ in Theorem \ref{main},
the cuspidal term is
\begin{equation}\label{cusp}
\sum_{u\in\mathcal{F}}\lambda_{p^\ell}(u) w_u
  = \ol{\w'(p)}^{\ell/2}\sum_u X_\ell(\nu_p^u) w_u,
\end{equation}
the sum converging absolutely.
This is equal to the first geometric term
\[T(m,m,p^\ell){\psi(N)}\ol{\w'(p^{\ell/2})}\frac1{\pi^2}\int_\R
  h(t) \tanh(\pi t)\, t\,dt
 =T(m,m,p^\ell)\psi(N)\ol{\w'(p^{\ell/2})}J\]
plus the remaining geometric terms and minus the continuous term.
By Proposition \ref{bound} and Proposition \ref{Klbound},
the latter terms are both $O(N^{1/2+\e})$.\footnote{These bounds were proven
  for $h\in PW^{12}(\C)^{\operatorname{even}}$, but they hold as well for $h$ as in
  Theorem \ref{hmain} so long as $A>\tfrac14$ and $B>2$,
   as shown in the proof of Proposition \ref{absconv}.}
It is easy to see that $T(m,m,p^\ell)=1$ if and only if
  $\ell = 2\ell'$ for some $0\le \ell'\le \ord_p(m)$.
Multiplying through by $\w'(p)^{\ell/2}$, the result follows.
\end{proof}

\begin{theorem}\label{dist}
Fix a prime $p$ and let $m>0$ be an integer.
   For each $n=1,2,\ldots$,
\begin{itemize}
\item let $N_n$ be a positive integer
   coprime to $p$, such that $\lim\limits_{n\to\infty}N_n=\infty$
\item let $\w'_n$ be a Dirichlet character modulo $N_n$
\item let $\mathcal{F}_n$ be an orthogonal basis
   for $L^2_0(N_n,\w'_n)$ consisting of Maass eigenforms.
\end{itemize}
 Define weights $w_u$ as in \eqref{wh}.
  Suppose $h(t)$ is chosen as in Theorem \ref{hmain} so that $J$ in \eqref{J} is nonzero,
  and $h(t_j)\ge 0$ for all spectral parameters $t_j$. (The latter condition will
  be discussed afterwards.)
  For each $n$, define the multiset
\[S_n=\{\nu_p^u|\,  u\in \mathcal{F}_n\}.\]
  Then the sequence $\{S_n\}$ is $w_u$-equidistributed with respect
  to the measure
\begin{equation}\label{dmu}
d\mu(x) = \sum_{\ell'=0}^{\ord_p(m)}X_{2\ell'}(x)\,d\mu_\infty(x),
\end{equation}
 where $d\mu_\infty(x)$ is the Sato-Tate measure \eqref{STmeas}.
  This means that for any continuous function $f$ on $\R$, we have
\begin{equation}\label{lim}
\lim_{n\to\infty}\frac{\sum_{u\in \mathcal{F}_n}f(\nu_p^u)w_u}
  {\sum_{u\in\mathcal{F}_n}w_u} = \int_{\R}f(x)d\mu(x).
\end{equation}
\end{theorem}

\noindent{\em Remarks:} (1) If we choose $m$ so that $p\nmid m$,
  then $d\mu=d\mu_\infty$ is the Sato-Tate measure itself.  In this case,
  the measure is independent of $p$, $m$ and $h$.\vskip .2cm

\noindent (2) The theorem illustrates in particular the fact that the
  normalized Hecke eigenvalues $\nu_p^u$ are dense in the interval $[-2,2]$.
  Thus the Ramanujan Conjecture, if true, is optimal.  In the other direction,
  the theorem provides evidence for the conjecture, by virtue of the fact 
  that the measure is supported on $[-2,2]$.
  Any counterexamples to the Ramanujan conjecture are sparse enough to be undetectable 
  in \eqref{lim}.

\begin{proof} Setting $\ell=0$ in Proposition \ref{101} gives
\begin{equation}\label{wn}
\sum_{u\in\mathcal{F}}w_u = J\psi(N)+o(N).
\end{equation}
  In particular, the denominator in \eqref{lim} is nonzero when $n$ is sufficiently large.
  We may assume that this is the case for all $n$.
  By \eqref{wn}, for all $\ell\ge 0$ we have
\begin{align*}
\lim_{n\to\infty}\frac{\sum_{u\in\mathcal{F}_n}X_\ell(\nu_p^u)w_u}
{\sum_{u\in\mathcal{F}_n}w_u}&=\begin{cases}1&\text{if }\ell=
2\ell', \text{ with }0\le \ell'\le \ord_p(m)\\
0&\text{otherwise}\end{cases}\\
 &=\int_\R X_\ell(x) d\mu(x).
\end{align*}
The latter equality holds because the polynomials $X_\ell(x)$
  are orthonormal with respect to the Sato-Tate measure (see e.g. \cite{KL}, Proposition 29.7).
  By linearity, \eqref{lim} holds for all polynomials.
Let $I\supseteq[-2,2]$ be a compact interval containing $\nu_p^u$ for all $u$.
  (According to the Ramanujan conjecture, we can take $I=[-2,2]$, but we do not
  assume this here. See \cite{appendix} Proposition 2.9 for an elementary proof of the 
  existence of $I$.)  As one can show, both sides of \eqref{lim} define continuous
  linear functionals on $C(I)$, relative to the sup-norm topology.
  Since the set of polynomials is dense, it follows that \eqref{lim} holds for
  all continuous functions, as required.

  In more detail, suppose $f$ is any continuous function on $I$. Given $\e>0$,
  let $P$ be a polynomial approximating $f$ to within $\e$ on the interval $I$.
  Then for any $n$,
\[\left|\frac{\sum_{u\in \mathcal{F}_n}f(\nu_p^u)w_u}{\sum_{u\in \mathcal{F}_n}w_u}
  -\int_\R f(x)d\mu(x)\right|\le 
\left|\frac{\sum_{u\in \mathcal{F}_n}(f(\nu_p^u)-P(\nu_p^u))w_u}{\sum_{u\in \mathcal{F}_n}w_u}
\right|
\]
\[+\left|\frac{\sum_{u\in \mathcal{F}_n}P(\nu_p^u)w_u}{\sum_{u\in \mathcal{F}_n}w_u}
  -\int_\R P(x) d\mu(x)\right|+\left|\int_\R(P(x)-f(x))d\mu(x)\right|\]
\begin{equation}\label{ein}
\le \e
+\left|\frac{\sum_{u\in \mathcal{F}_n}P(\nu_p^u)w_u}{\sum_{u\in \mathcal{F}_n}w_u}
  -\int_\R P(x) d\mu(x)\right|+\e \int_\R d\mu(x).
\end{equation}
In the first term of \eqref{ein}, we have used the fact that $w_u\ge0$ for all $u$,
  which holds because of the hypotheses imposed on $h$
   and the fact that $\cosh(\pi t_j)\ge 0$ for all $t_j$.
  The latter assertion is clear when $t_j\in \R$ by the definition of $\cosh$.
  The hypothetical exceptional parameters are of the form $t_j=ix_j$ for 
  $x_j\in (-\tfrac12,\tfrac12)$, so that 
  $\cosh(\pi t_j)=\cosh(i\pi x_j)=\cos(\pi x_j)\ge 0$ as well.

As shown in the first part of the proof, the middle
term of \eqref{ein} has the limit $0$ as $n\to\infty$.  Therefore
\[\limsup_{n\to\infty}\left|\frac{\sum_{u\in \mathcal{F}_n}f(\nu_p^u)w_u}{\sum_{u\in \mathcal{F}_n}w_u}
  -\int_\R f(x)d\mu(x)\right|\le \e(1+\int_{-2}^2 d\mu(x)).\]
 Letting $\e\to0$, we obtain \eqref{lim} as needed.
\end{proof}

In the theorem, we assumed that $h(t_j)\ge0$ for all spectral parameters $t_j$.
  Since $t_j\in \R\cup i(-\frac12,\frac12)$, the condition holds if $h$ is nonnegative on the
  real and imaginary axes.
  Examples of allowable $h$ include the Gaussian $h(t)=e^{-t^2}$
  and the function $h_R(t)=e^{-(t^2-R^2)^2}$.

\includegraphics{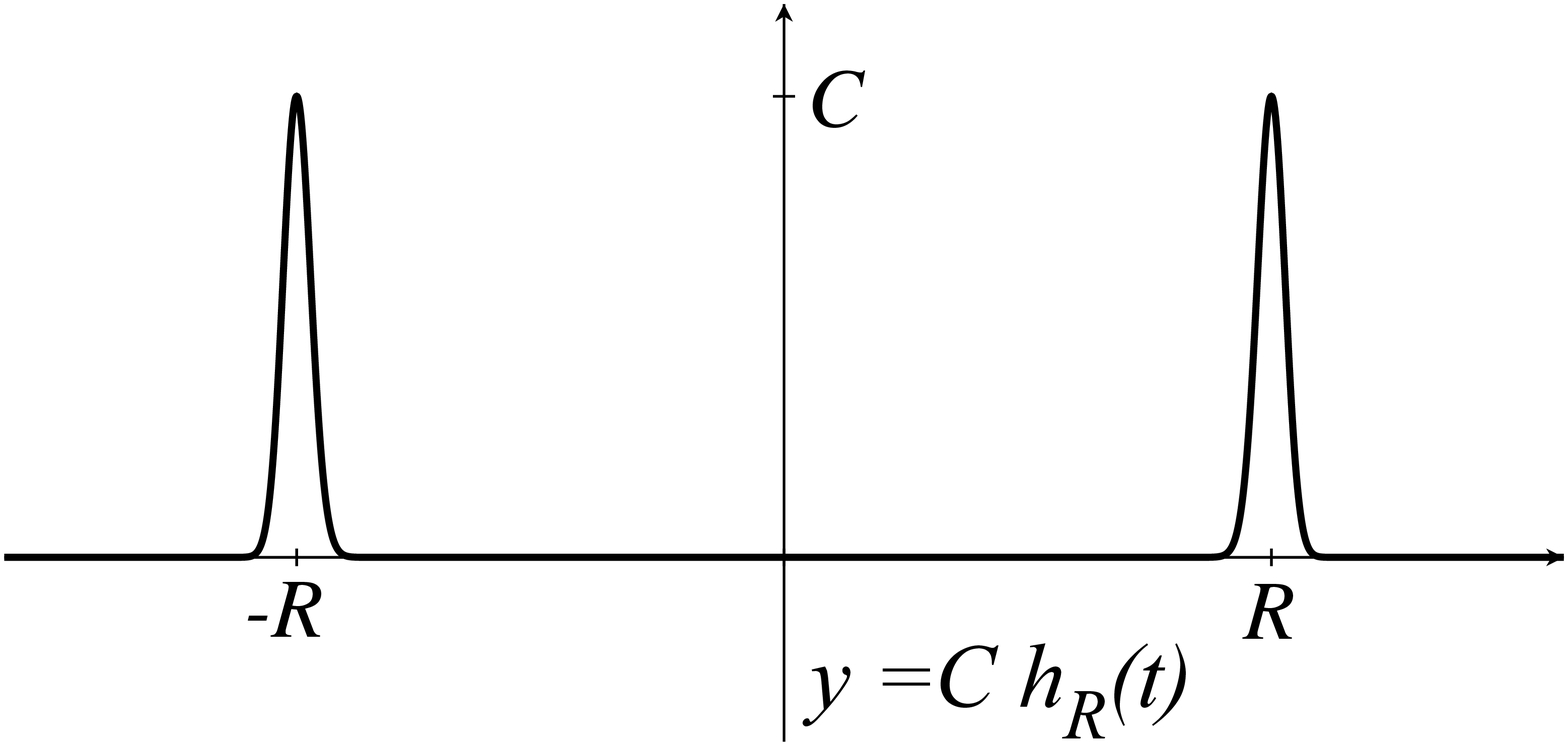}
\vskip 2.4cm
  \noindent The latter detects
  just those Maass forms with spectral parameter close to $\pm R$.
  When we apply the theorem to $h_R$, the fact that the result is independent of $R$
  shows that the equidistribution holds even when we restrict to a small piece of the
  spectrum.

Other functions $h$ satisfying the hypotheses of the theorem may be constructed as follows.
Let $h$ be the Selberg transform of a function $f_\infty=F^**F$,
  where $F\in C_c^m(G^+//K_\infty)$ for $m\ge12$.
  Let $f^1:G(\Af)\rightarrow\C$ be the identity
  Hecke operator, corresponding to $\n=1$.  
  Then if $u$ is a Maass cusp form with spectral parameter $t$, 
  by Proposition \ref{diag} we have
\[h(t) = \frac{\sg{R(f)\varphi_u,\varphi_u}}{\sg{\varphi_u,\varphi_u}}
  =\frac{\sg{R(F\times f^1)\varphi_u,R(F\times f^1)\varphi_u}}{\|u\|^2}\ge 0.\]

\pagebreak

\printindex{notations}{Notation index}
\printindex{keywords}{Subject index}

\end{document}